\documentclass[12pt]{amsart}    % use "amsart" instead of "article" for AMSLaTeX format

\usepackage[margin=1in]{geometry}  % set the margins to 1in on all sides
\usepackage[all]{xy}    
\usepackage{amsmath,amsthm,amssymb,mathrsfs,mathtools,bm,eucal,tensor}                  
\usepackage{bbm}

\usepackage[colorlinks]{hyperref}
\hypersetup{linkcolor=black, citecolor=black}

\usepackage{fancyhdr}
\usepackage{xcolor}

\usepackage[utf8]{inputenc}

\usepackage{chngcntr}
\numberwithin{equation}{subsection}
\newcounter{keepeqno}
\newenvironment{num}
 {\setcounter{keepeqno}{\value{equation}}%
  \begin{list}{(\theequation)}{\usecounter{equation}}%
  \setcounter{equation}{\value{keepeqno}}}
 {\end{list}}
 
\usepackage[colorinlistoftodos,prependcaption,textsize=tiny]{todonotes}

\newcommand{\BA}{{\mathbb {A}}}

\newcommand{\BC}{{\mathbb {C}}}

\newcommand{\BF}{{\mathbb {F}}}
\newcommand{\BG}{{\mathbb {G}}}

\newcommand{\BL}{{\mathbb {L}}}

\newcommand{\BN}{{\mathbb {N}}}

\newcommand{\BR}{{\mathbb {R}}}
\newcommand{\BS}{{\mathbb {S}}}

\newcommand{\BX}{{\mathbb {X}}}

\newcommand{\BZ}{{\mathbb {Z}}}

\newcommand{\CA}{{\mathcal {A}}}

\newcommand{\CC}{{\mathcal {C}}}
\newcommand{\CE}{{\mathcal {E}}}
\newcommand{\CF}{{\mathcal {F}}}
\newcommand{\CG}{{\mathcal {G}}}
\newcommand{\CH}{{\mathcal {H}}}
\newcommand{\CI}{{\mathcal {I}}}

\newcommand{\CM}{{\mathcal {M}}}
\newcommand{\CN}{{\mathcal {N}}}
\newcommand{\CO}{{\mathcal {O}}}
\newcommand{\CP}{{\mathcal {P}}}

\newcommand{\CS}{{\mathcal {S}}}
\newcommand{\CT}{{\mathcal {T}}}

\newcommand{\CW}{{\mathcal {W}}}

\newcommand{\CZ}{{\mathcal {Z}}}

\newcommand{\FD}{{\mathfrak {D}}}

\newcommand{\FI}{{\mathfrak {I}}}

\newcommand{\FS}{{\mathfrak {S}}}

\newcommand{\Fc}{{\mathfrak {c}}}

\newcommand{\Fo}{{\mathfrak {o}}}
\newcommand{\Fp}{{\mathfrak {p}}}

\newcommand{\RD}{{\mathrm {D}}}

\newcommand{\RF}{{\mathrm {F}}}

\newcommand{\RI}{{\mathrm {I}}}
\newcommand{\RJ}{{\mathrm {J}}}

\newcommand{\RM}{{\mathrm {M}}}

\newcommand{\RO}{{\mathrm {O}}}

\newcommand{\RS}{{\mathrm {S}}}
\newcommand{\RT}{{\mathrm {T}}}
\newcommand{\RU}{{\mathrm {U}}}

\newcommand{\Ad}{{\mathrm{Ad}}}

\newcommand{\cusp}{{\mathrm{cus}}}

\newcommand{\disc}{{\mathrm{disc}}}

\newcommand{\End}{{\mathrm{End}}}
\newcommand{\el}{{\mathrm{ell}}}

\newcommand{\Gal}{{\mathrm{Gal}}}
\newcommand{\GL}{{\mathrm{GL}}}

\newcommand{\Hom}{{\mathrm{Hom}}}

\renewcommand{\Im}{{\mathrm{Im}}}
\newcommand{\Ind}{{\mathrm{Ind}}}

\newcommand{\Id}{{\mathrm{Id}}}

\newcommand{\Ker}{{\mathrm{Ker}}}

\newcommand{\ord}{{\mathrm{ord}}}

\renewcommand{\Re}{{\mathrm{Re}}}
\newcommand{\reg}{{\mathrm{reg}}}
\newcommand{\Res}{{\mathrm{Res}}}

\newcommand{\SL}{{\mathrm{SL}}}

\newcommand{\Spec}{{\mathrm{Spec}}}
\newcommand{\SO}{{\mathrm{SO}}}

\newcommand{\Sym}{{\mathrm{Sym}}}
\newcommand{\sgn}{{\mathrm{sgn}}}
\newcommand{\Sp}{{\mathrm{Sp}}}

\newcommand{\st}{{\mathrm{st}}}

\newcommand{\tr}{{\mathrm{tr}}}

\newcommand{\ud}{\,\mathrm{d}}
\newcommand{\vol}{{\mathrm{vol}}}

\newcommand{\wt}{\widetilde}
\newcommand{\wh}{\widehat}

\newcommand{\bs}{\backslash}

\def\alp{{\alpha}}

\def\bet{{\beta}}

\def\del{{\delta}}
\def\Del{{\Delta}}
\def\diag{{\rm diag}}

\def\eps{{\epsilon}}

\def\veps{{\varepsilon}}

\def\sig{{\sigma}}

\def\std{\rm std}
\def\ome{{\omega}}
\def\Ome{{\Omega}}
\def\lam{{\lambda}}
\def\Lam{{\Lambda}}

\def\gam{{\gamma}}
\def\Gam{{\Gamma}}

\def\wb{\overline} %widebar

\def\vpi{\varpi}
\def\LG{{}^{L}G}
\def\LT{{}^{L}T}

\def\vphi{\varphi}

\def\p{\prime}

\newcommand{\sslash}{\mathbin{/\mkern-6mu/}}

\def\scusp{\mathrm{scusp}}
\def\pv{\mathrm{pv}}
\def\Irr{\mathrm{Irr}}

\def\qd{\mathrm{qd}}
\def\Tr{\mathrm{Tr}}
\def\CD{\mathcal{D}}

\def\St{\mathrm{St}}
\def\fin{\mathrm{fin}}
\def\Laf{\mathrm{Laf}}

\def\triv{\mathrm{triv}}
\def\im{\mathrm{im}}
\def\Sat{\mathrm{Sat}}
\def\unr{\mathrm{unr}}
\def\LafSec{\mathrm{LafSec}}

\def\RJ{\mathrm{J}}
\def\RF{\mathrm{F}}

\newtheorem{thm}{Theorem}[subsection]
\newtheorem{defin}[thm]{Definition}

\newtheorem{pro}[thm]{Proposition}
\newtheorem{lem}[thm]{Lemma}

\newtheorem{conjec}[thm]{Conjecture}
\newtheorem{quest}[thm]{Question}

\newtheorem{cor}[thm]{Corollary}
\newtheorem{theorem}[thm]{Theorem}

\newtheorem{corollary}[thm]{Corollary}

\theoremstyle{remark}
\newtheorem{rmk}[thm]{Remark}

\makeatletter

\newcommand{\Rmnum}[1]{\expandafter\@slowromancap\romannumeral #1@}
\makeatother

\def\C{\mathbb{C}}

\def\Nr{\mathrm{Nr}}

\def\c{\mathbf{c}}

\def\v{\mathbf{v}}

\def\ind{\mathrm{ind}}

\begin{document}

\title[Nonabelian Fourier kernels on $\SL_2$ and $\GL_2$]{Nonabelian Fourier kernels on $\SL_2$ and $\GL_2$}

\author{Zhilin Luo}
\address{Department of Mathematics\\
University of Chicago\\
Chicago IL, 60637, USA}
\email{zhilinchicago@uchicago.edu}

\author{Bao Châu Ngô}
\address{Department of Mathematics\\
University of Chicago\\
Chicago IL, 60637, USA}
\email{ngo@uchicago.edu}


\date{\today}


\begin{abstract}
For $G = \SL_2$ or $\GL_2$, we present explicit formulas for the nonabelian Fourier kernels on $G$, as conjectured by A. Braverman and D. Kazhdan. Additionally, we furnish explicit formulas for the orbital Hankel transform on $G$, a topic investigated by the second author, and provide an explicit formula for the stable orbital integral of the basic function. These results are applicable to local fields with residual characteristics other than two.
\end{abstract}

\maketitle

\tableofcontents

% !TEX root = luo-ngo.tex

\section{Introduction}\label{intro}

\subsection{Historical context}
In his celebrated thesis (\cite{tatethesis}), J. Tate established the basic analytical properties (the Hecke theory) of the Hecke $L$-functions through Fourier analysis on the additive group. The Hecke $L$-functions $L(s,\chi)$ associated to Grössencharacters $\chi$ of the idele group $\BA^\times$ are expressed as the poles of the global zeta integrals attached to $\chi$ and Schwartz-Bruhat functions on the adele ring $\BA$. The meromorphic continuation and the functional equation are consequences of the Fourier transform and the Poisson summation formula.

Developments in the modern theory of automorphic forms rest in a set of conjectures formulated by R. Langlands in \cite{langlandsproblems} with the reciprocity and functoriality conjectures being the two pillars. The concept of the automorphic $L$-function $L(s,\pi,\rho)$ attached to an automorphic representation $\pi$ of a reductive group $G$ over a global field $k$ and an $L$-homomorphism $\rho:\LG\to \GL(V_\rho)$ of the Langlands dual group, which is a vast generation of the Riemann zeta function, plays a central role in these conjectures. The functions $L(s,\pi,\rho)$ are defined via the Euler product, which only converges on a half-plane. One of the major implications of the functoriality conjecture is that these $L$-functions have a meromorphic continuation to the whole complex plane satisfying a functional equation analogous to the Riemann zeta function.

It is natural to ask if one can establish the basic analytical properties for general automorphic $L$-functions $L(s,\pi,\rho)$ for general reductive groups $G$ over $k$ via harmonic analysis on $G(\BA_k)$ following the approach of the Tate thesis, thus bypassing the functoriality conjecture. By the Fourier analysis on the vector space of $n\times n$ matrices, R. Godement and H. Jacquet (\cite{gjzeta}) generalized the work of Tate to $G=\GL_n$ over a central simple algebra and $\rho$ the standard representation of its dual group. Around 2000, A. Braverman and D. Kazhdan (\cite{BK00}) formalized Tate and Godement-Jacquet's approaches in a series of conjectures aiming at the meromorphic continuation and functional equation general automorphic $L$-functions. These proposals were elaborated by L. Lafforgue (\cite{lafforgue2016principe}) and given in a more precise form by the second author (\cite{ngo2016hankel}).

The local aspect of the Braverman-Kazhdan program stipulated the existence of a space of $\rho$-Schwartz functions, generalizing the space of usual Schwartz functions on the space of matrices, and the $\rho$-Fourier transform generalizing the usual Fourier transform on the space of matrices. Over a local field, the expected form of the $\rho$-Fourier transform is the convolution with a stably invariant distribution $\RJ^\rho_\psi$ on $G(F)$. {\it The main focus of this work is to construct the distribution $\RJ^\rho_\psi$, which we call the nonabelian Fourier kernel.}  

The distribution $\RJ^\rho_\psi$ depends on a nontrivial additive character $\psi$ of the local field, such that after regularization, the action of $\RJ^\rho_\psi$ on any irreducible admissible representation $\pi$ of $G(F)$ is well-defined and is equal to the multiplication by the Langlands local gamma factor $\gam(\cdot,\pi,\rho,\psi)$. The distribution $\RJ^\rho_\psi$ is expected to be represented by a locally integrable function, which is smooth over the regular semisimple locus. Because it is stably invariant, it descends to a function on the Steinberg quotient, which, in the case $\SL_2$, is the affine line of the trace variable. The main result of this paper is to provide a formula for $\RJ^\rho_\psi$ in the case of  $\SL_2$ (and ${\GL}_2$), as a function of the trace (and the determinant) variables, thus answering a question of Braverman and Kazhdan in \cite{BK00}. Beyond the ${\rm SL}_2$ and ${\rm GL}_2$ cases, we elaborate a general framework based on Langlands' stable transfer (\cite{langlands-singularites-transfert}) and what we call the Lafforgue transform (\cite{lafforguegl2}). These transforms, which exist in general, can be made explicit in the case of ${\GL}_2$ and ${\SL}_2$. 

The finite fields analog of the distribution $\RJ^\rho_\psi$ is an equivariant perverse sheaf. Its construction was conjectured in \cite[\S 9]{BK00} and studied for $G$ of semi-simple rank one in \cite{braverman2003sheaves}. The conjecture has been verified by S. Cheng and the second author in \cite{cheng2017conjecture} for $G=\GL_n$, by G. Laumon and E. Letellier for general $G$ in \cite{laumon2019notes}, and by T.-H. Chen in \cite{MR4467308} for general $G$ when the characteristic of the underlying finite field is large and $\CD$-module setting.

Much less is known in the $p$-adic case. When $G=\GL_n$ and $\rho$ is the standard representation, this is the case studied by Godement and Jacquet (\cite{gjzeta}), and $\RJ^\rho_\psi = \psi(\tr(\cdot))$ is the standard Fourier kernel. The Fourier kernel is also known in the case of tori as suggested by Braverman and Kazhdan (\cite{BK00}) and made more precise by the second author (\cite{ngo2016hankel}). The first author calculates the spherical part of the distribution $\RJ^\rho_\psi$ for any $G$ and $\rho$ in \cite{MR3990815} over a local field of characteristic zero. In \cite{JLZ}, Jiang, Zhang, and the first-named author give a formula for $\RJ^\rho_\psi$ for $G=\BG_m\times \Sp_{2n}$ and the standard representation of $\LG$.

Motivated by the work of Godement and Jacquet, the second author (\cite{ngo2016hankel}) speculated that the distribution $\RJ^\rho_{\psi,G}$ should be patched from $\{\RJ^\rho_{\psi,T}\}_{T\subset G}$ where $T$ runs over all maximal tori of $G$. However, as pointed out by L. Lafforgue (\cite{lafforguegl2}), the naive patching construction fails to provide the correct kernel for the $\rho$-Fourier transform for $\rho=\Sym^2$ of $G=\BG_m\times \SL_2$ because it does not act on principal series representations with the correct gamma factors. For $G=\SL_2$ or $\GL_2$, Lafforgue proposed a recipe to correct the naive patching construction so that the action of the resulting distribution on the principal series representations is given by the multiplication by the correct Langlands local gamma factor. However, whether the resulting distribution acts on supercuspidal representations by the correct gamma factors or not is unclear. Comparing with the result of \cite{JLZ}, one can show that the recipe of Lafforgue works for $\rho=\Sym^2$ of $G=\BG_m\times \SL_2$, but it fails for supercuspidal representations and for higher symmetric powers $\rho$. The main result of this paper is to provide an explicit formula of $\RJ^\rho_\psi$ in the case $G=\SL_2$ and $\GL_2$ and for an arbitrary representation $\rho$ under the assumption that the residual characteristic of the underlying local field $F$ is not equal to two. When $\rho$ is the standard representation attached to $G=\BG_m\times \SL_2$, the formula found in this paper matches with the formula of \cite{JLZ}. We will show in this work that the naive patching construction and the Lafforgue transform, although not sufficient, are still essential ingredients in the construction of the correct kernel of the $\rho$-Fourier transform. 

We also discuss the $\rho$-orbital Hankel transform considered by the second author (\cite{ngo2016hankel}), which is the descent of the $\rho$-Fourier transform to the level of stable orbital integrals of $G(F)$. Everling has calculated the orbital Hankel transform for the usual Fourier transform, i.e., when $\rho$ is the standard representation, and found rather complicated formulae (\cite{Everling-Fourier}).  It is, therefore, a good surprise that the stable orbital $\rho$-Hankel transform can be given a nice formula. 

The nonabelian Fourier and Hankel transform will be given in completely explicit integral formulae, which can be better understood only from the representation theoretic perspective. In particular, they are given in terms of Langlands' stable transfers in combination with their canonical partial splittings, and also the Lafforgue transform, which are new ingredients we introduce in this work.

\subsection{Organization of the paper}

As we want to state our result in its proper context, we will review the contents of different sections of the paper in order. In Section \ref{BK:intro}, we review broadly the Braverman-Kazhdan proposal following the presentation of \cite{ngo2016hankel}. In subsection \ref{subsec:toruscase}, we rigorously build the local part of the Braverman-Kazhdan program for tori, particularly the Fourier kernel in this case. These abelian Fourier kernels will be later used to construct the nonabelian Fourier kernel. In subsection \ref{subsec:excepex}, we supply explicit formulae for the nonabelian Fourier kernels in some low symmetric power representations $\rho$. The case of the symmetric square representation follows the work of Jiang-Luo-Zhang (\cite{JLZ}), which treats the case of symplectic group and the standard representation of its dual group. In subsection \ref{subsec:aformulaFK}, we present explicit formulae for the nonabelian Fourier kernels when $G=\BG_m\times \SL_2$ and $\GL_2$.

In Section \ref{sec:orbital-cocenter}, we recall the normalization of orbital and stable orbital integrals of a reductive group $G$ over a local field $F$ following the work of Frenkel-Langlands-Ngô (\cite{FLN10}). They are, on one hand, interpreted both as functions on the space of conjugacy classes or the stable conjugacy classes, aka the Steinberg-Hitchin base, which are equipped with the canonical measure. On the other hand, they can be interpreted as elements of the cocenter $\CC^{\rm fin}=\CH/[\CH,\CH]$ of the Hecke algebra $\CH$ of $G(F)$. We introduce the Kazhdan cocenter $\CC$, which is a completion of $\CC^{\rm fin}$, in a similar way as the Bernstein center $\CZ$ is a completion of the subalgebra $\CZ^{\rm fin}$ of regular functions on the Bernstein variety $\Omega(G(F))$ supported by finitely many components. We introduce the stable cocenter $\CC^{\rm fin,st}$ and formulate Langlands' concept of stable transfer in terms of the stable cocenter that is a canonical linear map $$\CT_\varphi:\CC^{\rm fin,st}(G(F)) \to \CC^{\rm fin,st}(H(F))$$ in the presence of a homomorphism of $L$-groups $\varphi:\,\!\! ^L H \to\,\!\! ^L G$. Because elements of $\CC^{\rm fin,st}(G(F))$ can be interpreted as functions of the Steinberg-Hitchin space $\Fc_G(F)$, we expect that there is a nice integral giving rise to the stable transfer factor. In the case, $G=\SL_2$ and $H$ is a one-dimensional torus associated with a quadratic extension $E/F$, this integral transform 
 \begin{equation} \label{GG-transfer}
 	\CT_E(f)=\lambda_{E/F} \pi_E^*(\CF_\psi( \eta_E \CF_{\psi^{-1}}(f) ))
 \end{equation}
 where $\eta_E$ is the quadratic character of $F^\times$ associated with the quadratic extension $E/F$, $\pi_E:E^1\to F$ is the trace map, and $\lambda_{E/F}$ is the Weil constant appearing in \cite{jlgl2}. This integral formula for the stable transfer is equivalent to the Gelfand-Graev character formula to be recalled in the following section. We also review the stable Bernstein center in the case of $\SL_2$, and show that it is a subalgebra of the usual Bernstein center which is known to experts.
 
The Gelfand-Graev character formula first appears in \cite{Gelfand-Graev} but also in the famous book of Gelfand-Graev-Piatetsky-Shapiro (\cite{ggps}). In modern languages, it expresses the character of the dihedral representation of $\GL_2(F)$ or the corresponding stable character of $\SL_2(F)$ as an additive convolution in the trace variable. One of the side outcomes of this work is that from staring at different formulas, the additive structure in the trace variable plays a major role in the representation theory of $\SL_2(F)$ and $\GL_2(F)$, a fact which is not easy to explain from the theoretical standpoint. In Section \ref{sec:ggtransform}, we recall the construction of the dihedral representations from the Weil representation, and derive the Gelfand-Graev character formula, at least formally. We recall the Weil representation and the dihedral representations which will be later used in the proof of our descent (to tori) formula. In subsection \ref{subsec:ggpsstable:gl2}, we connect the Gelfand-Graev character formula with Langlands' stable transfer factor. 
 
In Section \ref{sec:Bernstein-center}, we construct the Lafforgue transform $\Laf:\CC \to \CZ$ from the Kazhdan cocenter to the Bernstein center. For every element $c\in \CC$ we can define $\tr(c,\pi)$ the trace of $c$ on every irreducible representation of $\pi$. Similarly, each element $z\in \CZ$ acts on every irreducible representation $\pi$ by a scalar $\gamma(z,\pi)$. The Lafforgue transform ${\rm Laf}:\CC \to \CZ$ is uniquely defined by the equation 
$$\gamma(\Laf(c),[\pi])=\tr(c,\pi)$$
for every point  $[\pi]$ in the Bernstein variety $\Omega(G)$ such the inertial equivalence class $[\pi]$ contains a unique irreducible representation $\pi$. In general, since $\Laf$ is not injective, there are different choices of sections for $\Laf$. When $G=\SL_2$, we introduce the following section $\LafSec:\CZ\to \CC$ which factors through the subspace $\CC^{\perp,\text{triv}}$ of $\CC$ that are orthogonal to the trivial character, i.e. $\tr(c,\triv) = 0$. Notice that after localization, $\Laf$ is an isomorphism of $\CZ$-module away from the Steinberg component, hence $\LafSec$ is uniquely determined. 
The Lafforgue transform $\Laf$ induces map from the stable Kazhdan center to the stable Bernstein center
$$
\Laf:\CC^\st\to \CZ^\st.
$$
The main claim in this section is that, when $G=\SL_2$ and restricted to the finite stable part $\CZ^{\st,\fin}$, $\LafSec$ is given by the integral transform which induces an isomorphism of $\CZ^\st$-modules $\CZ^{\st,\fin}\simeq \CC^{\perp,\triv,\fin}$,
\begin{equation} \label{integral-Laf}
 	\LafSec(z) = \CF_\psi\left(|.|^{-1}\CF_{\psi^{-1}}(z)\right)
\end{equation}
where both $z$ and $\LafSec(z)\in \CC^{\st,\fin}$ are regarded as functions of the trace variable and $\CF_\psi$ is the Fourier transform with respect to additive character $\psi$. We expect that a similar formula exists for all reductive groups. In this paper we prove it only for $\SL_2$ and $\GL_2$. Even in those cases, we don't know how to prove it directly by deriving it from what we call the descent formula. Although we prove it as a consequence of the descent formula, we stress that the above formula for the Lafforgue transform supplies a proper theoretical framework for the descent formula. It turns out that over a real field, the Lafforgue transform can be traced back to a simple relation between Chebyshev polynomials of first and second kind: Up to conjugation by the square-root of the Weyl discriminant, the derivative of Chebyshev polynomials of first kind is Chebyshev polynomials of second kind. Section \ref{sec:Bernstein-center} also contains the more or less known result that the stable Bernstein center for $\SL_2$ is a subalgebra. We also recall some explicit calculations of Moy-Tadic (\cite{Moy-Tadic-Bernstein}) to justify the convergence of the integral representation of the Lafforgue transform over the stable locus. 

Section \ref{sec:descent} contains the statement and the proof of the descent formula. Given an element $z\in \CZ^{\rm st}$ represented as a function of the trace variable, we want to calculate its action on a dihedral representation. Because $z$ lies in the stable Bernstein center, its action on the dihedral representation associated with a character $\chi$ of $E^1$ where $E/F$ is a quadratic extension is a scalar. For $G=\SL_2(F)$ that scalar can be computed by using first  the sections to the Lafforgue transform $\LafSec:\CZ^{\rm st}\to \CC^{\rm st}$ composed with the stable transfer $\CT_E: \CC^{\rm st} \to \CC^{\rm st}(E^1)$, and finally paired with $\chi$:
$$\int_{E^1} \chi(e) \big(\CT_E\circ \LafSec(z)\big)(e) \ud e.$$
The descent formula from $\SL_2(F)$ to the torus $E^1$ is equivalent to an explicit formula for $\CT_E\circ \LafSec(z)$
\begin{equation} \label{descent}
	\CT_E\circ \LafSec(z)= \lambda_{E/F} \pi_E^*\left (\CF_\psi\left( |.|^{-1} \eta_E  \CF_{\psi^{-1}}(z)\right)\right) .
\end{equation}
where $\pi_E:E^1\to F$ is the trace map, and $\eta_E$ is the quadratic character associated to the quadratic extension $E/F$. This formula can be derived from the combination of the Gelfand-Graev formula for the stable transfer \eqref{GG-transfer} and the integral formula for the Lafforgue transform \eqref{integral-Laf}. Still, we prove it directly using the Weil representation. In fact, we derive the integral formula \eqref{integral-Laf} from the descent formula \eqref{descent}. 

Section \ref{sec:summands} focuses on a new feature of the cocenter and stable transfers. We show that if $E/F$ is a nonsplit quadratic extension the stable transfer map $\CT_E: \CC(\SL_2(F)) \to \CC(E^1)^{\FS_2}$ admits a canonical right inverse $$\CE_E:\CC(E^1)^{\FS_2} \to \CC(\SL_2(F))$$
where $\FS_2$ is generated by the unique involution attached to $E/F$. 
The construction of this right inverse is based on the elliptic orthogonality of character and stable characters, which in particular demonstrates the existence of another unitary structure on the affine line. It gives rise to the direct summand of $\CC(\SL_2(F))$ corresponding to the quadratic extension $E^1$.  If $E'$ is another nonsplit quadratic extension of $F$, then $\CT_{E'} \circ \CE_E\neq 0$ but the image is generated by the trivial character and the nontrivial quadratic character of ${E'}^1$. If $E_0$ is the split quadratic extension, we have $\CT_{E_0} \circ \CE_E=0$. From these remarks, we is a unique right inverse 
$$\CE_{E_0}:\CC(E_0^1)^{\FS_2} \to \CC(\SL_2(F))$$
of $\CT_{E_0}: \CC(\SL_2(F)) \to \CC(E_0^1)^{\FS_2}$ such that for all nonsplit quadratic extensions $E$ of $F$ we have $\CT_E \circ \CE_{E_0}=0$. We also have an explicit albeit complicated integral formula for $\CE_{E_0}$. We thus obtain an ``almost" direct decomposition of $\CC(\SL_2(F))$ with summands corresponding to the four quadratic extensions of $F$. This is only an almost direct decomposition because the three summands corresponding to three nonsplit quadratic extension intersect in a two-dimensional vector space generated by the normalized elliptic character of the Steinberg representation and the normalized elliptic stable character of  the unique $L$-packet with four elements discovered by Labesse and Langlands in \cite{labesse-langlands} corresponding to the nontrivial quadratic character of $E^1$ for each nonsplit quadratic extension $E$ of $F$. 

Armed with the operators $\CT_E$ and $\CE_E$ and the almost direct decomposition of the cocenter, we can write down explicit formulae for the kernel of the nonstandard Fourier transform, the Hankel transform of stable orbital integrals, and also stable orbital integrals for all spherical functions, including the basic function. These are discussed in Section \ref{sec:inv}, Section \ref{sec:stableorbit} and Section \ref{sec:HKT}.

\subsection{Explicit formulae}\label{subsec:mainresult}
We now present explicit formulae that can be derived from the structures presented above. In fact, we sometimes need to prove the explicit formulae first and derive the structures from those explicit formulae.

\subsubsection{Descent formula}

The Bernstein center $\CZ$ of $G(F)$, defined in \cite{BD84}, consists of invariant distributions on $G(F)$ which are essentially of compact support. Every element $z\in \CZ$ acts on each irreducible representation $\pi$ of $G(F)$ by a scalar that we denote by $\gamma(z,\pi)$. For every $z\in \CZ$, the function $\pi\mapsto \gamma(z,\pi)$ is an algebraic function on the Bernstein variety $\Omega(G(F))$ classifying irreducible representations up to the inertial equivalence. We denote by $\CZ^{\rm fin}$ the nonunital subalgebra of $\CZ$ consisting of elements $z$ such that the function $\pi\mapsto \gamma(z,\pi)$ is supported by finitely many components of the Bernstein variety. 

Let $z$ be a stably invariant distribution on $G(F) = \SL_2(F)$ belonging to the nonunital subalgebra $\CZ^{\rm fin}$.
Using some results of Moy and Tadic in  \cite{Moy-Tadic-Bernstein}, we can prove that there exists a locally integrable function $z_\Fc$ on the Steinberg-Hitchin $\Fc(F)$ base such that
$$z = (z_\Fc\circ \c)\ud g$$ where $\c = \tr$ is the trace map. It turns out that based on the explicit calculations in \cite{Moy-Tadic-Bernstein}, the function $z_\Fc$ enjoys the following characterization:
\begin{num}
\item\label{num:intro:FTJ} $\CF_{\psi^{-1}}(z_\Fc)$ is globally bounded, smooth and vanishes at zero. Moreover, $\CF_{\psi^{-1}}(z_\Fc)/|\cdot|$ is absolutely integrable.
\end{num}
Based on the above characterization, we establish the following descent formula.

\begin{thm}\label{thm:mainresult1}
Let $E$ be an étale quadratic $F$-algebra with $\CW_E$ the Weil representation on the space $\CS(E)$ of Schwartz-Bruhat functions on $E$. For every character $\chi$ of $E^1 = \Ker (\Nr:E^\times\to F^\times)$ and for any $\phi\in \CS(\chi)$, the following identity holds
$$
\gam\big(z,\CS(\chi)\big)=
\int^\reg_{G(F)}
z(g)\big(
\CW_E(g)\phi
\big)
=\lam_{E/F}
\bigg(
\int_{E^1}
\chi(e)\big(\CT_E\circ \LafSec(z)\big)(\tr(e))\ud_{E,1}e
\bigg)\phi.
$$
Here $\int^\reg_{G(F)}$ is the regularized integral introduced in Section \ref{sec:descent}, $\lam_{E/F} = \lam_{E/F,\psi}
$ is the Weil constant, and 
$$
\LafSec(z) = \CF_{\psi}
\big(
|\cdot|^{-1}
\CF_{\psi^{-1}}(z_\Fc)
\big).
$$
\end{thm}
As a corollary, by the Gelfand-Graev character identity, we deduce the following fact:
\begin{cor}\label{cor:mainresult:1}
With the above notation the following identity holds
$$
\gam
\big(
z,\CS(\chi)
\big) = \tr\big(\LafSec(z),\CS(\chi)\big).
$$
\end{cor}
Parallel results hold for $G=\GL_2$. 

\subsubsection{Inversion formulae}

Given a stable distribution $z\in \CZ^\fin$, the descent formula above provides an element in the cocenter $c = \Laf(z)\in \CC^{\perp, \triv,\fin}$ that are connected by the following simple integral formulae
$$
c=\LafSec(z) = \CF_{\psi}
\big(
|\cdot|^{-1}
\CF_{\psi^{-1}}(z)
\big),\quad 
z = \Laf(c) = 
\CF_{\psi}
\big(
|\cdot|
\CF_{\psi^{-1}}(c)
\big).
$$
On the other hand, starting from $c\in \CC^\fin$, for any étale quadratic $F$-algebra $E$, the Langlands' stable transfer $\CT_E$ provides an element in the cocenter of the maximal torus $E^1$ determined by $E$,
$$
z_E=c_E=\CT_E(c)= \lam_{E/F}
\pi^*_E
\big(
\CF_{\psi}
(\eta_E \CF_{\psi^{-1}}(c))
\big)\in \CC(E^1)^\fin\simeq \CZ(E^1)^\fin.
$$
A natural question is to reproduce $z$ and $\LafSec(z) = c$ from the datum $\{z_E=c_E\}_{E/F}$. When the residual characteristic of $F$ is not equal to two, any irreducible representation of $\SL_2(F)$ is a subquotient of a dihedral representation. Hence the datum $\{c_E\}_{E/F}$ determines $z$ and $c=\LafSec(z)\in \CC^{\perp,\triv,\fin}$ uniquely. Moreover, $z=\Laf(c)$. Hence we are reduced to reproduce $c$ from $\{c_E = \CT_E(c)\}_{E/F}$. 

To state the inversion formulae, we introduce the following notation. Let $\CI$ be the set parametrizing stable conjugacy classes of maximal tori $\{E^1_\alp\}_{\alp\in \CI}$ in $\SL_2(F)$, 
$$
\CI = 
\Bigg\{
\begin{matrix}
\{0\} & F=\BC
\\
\{0,1\} & F=\BR
\\
\{0,1,\pm1/2\} & F \text{ non. archi.}
\end{matrix}
$$
Here $E_0 = F\times F$ corresponds to the split quadratic algebra, $T_1$ corresponds  to the unramified quadratic extension (resp. complex extension) when $F$ is non-archimedean (resp. real), and $T_{\pm 1/2}$ correspond to the two ramified extension of $F$. For every $\tau_\alp$-invariant function $c_\alp$ on $E^1_\alp$, where $\tau_\alp$ is the unique nontrivial involution attached to $E_\alp/F$, $c_\alp$ descends to a unique function $\nu_\alp(c_\alp)$ on $\tr(E^1_\alp)$ with the property that 
$$
c_\alp = \pi^*_\alp\big(
\nu_\alp(c_\alp)
\big)
$$
where $\pi_\alp = \pi_{E_\alp}$. 

\begin{thm}\label{thm:mainresult2}
With the above notation, for $c\in \CC^{\st,\fin}$ and $c_\alp = \CT_\alp(c)\in \CC(E^1_\alp)^{\FS_2}$, the following statements hold.
\begin{enumerate}
\item When $F=\BC$, $c=\nu_0(c_0)$;

\item When $F=\BR$, $c=\sum_{\alp\in \CI}\CE_\alp(c_\alp)$;

\item When $F$ is non-archimedean of odd residual characteristic, 
$$
c-c^\el_\St\cdot \v\Del \Theta^\el_\St = \sum_{\alp\in \CI}
\CE_\alp(c_\alp).
$$
\end{enumerate}
Here:
\begin{itemize}
\item For $\alp\in \CI\bs \{0\}$, 
\begin{align*}
\CE_\alp: \CC^\infty_c(E_1^\alp)^{\tau_\alp} &\to \CC^{\st,\fin}
\\
c_\alp &\mapsto 
\CE_\alp(c_\alp) = 
\mathbbm{1}_\el
\cdot 
\v\Del
\cdot \CG\CG_\alp
\bigg(
\frac{\nu_\alp(c_\alp)}{\v\Del}
\bigg)
\end{align*}
where $\CG\CG_\alp(\cdot) = \CF_{\psi}
\big(
\eta_\alp \CF_{\psi^{-1}}(\cdot)
\big),
$
$\Del$ is the square-root of Weyl discriminant, $\mathbbm{1}_\el$ is the characteristic function of elliptic locus, and $\v:\Fc^\el(F)\to \BC$ is the function supported on elliptic regular semisimple locus such that $\v(\xi) = \vol(E_\alp^1)^{-1}$ whenever $\xi\in E^1_\alp$;

\item $\CE_0$ is the map from $\CC^\infty_c(E_0^1)^{\FS_2}$ to functions on $\Fc(F)$ that is given by the following formulas
$$
c_0\mapsto 
\Bigg\{
\begin{matrix}
\nu_0(c_0) & F=\BC
\\
\nu_0(c_0)-\CE_1(\CT_1(\nu_0(J_0))) & F=\BR
\\
\nu_0(c_0)-\CE_1(\CT_1(\nu_0(c_0)))
-\sum_{\alp=\pm 1/2}\CE^{++}_{\alp}(\CT_\alp(\nu_0(c_0))^{++}) & 
F \text{ non-archi.}
\end{matrix}
$$
where $\CE^{++}_\alp$ and $\CT_\alp(\nu_0(c_0))^{++}$ are the restriction of $\CE_\alp$ and $\CT_\alp(\nu_0(c_0))$ to the subspace $\CC^\infty_c(E^1_\alp)^{\FS_2}$ consisting of functions whose Mellin transform evaluated at the trivial character and the unique quadratic character vanishes;

\item $\Theta^\el_\St$ is the elliptic part of the Steinberg character, and 
$$
c^\el_\St = \langle c,\Theta^\el_\St\rangle =
\int_{\Fc(F)}
c(\xi)
\Theta^\el_\St(\xi)
\ome_\Fc.
$$
\end{itemize}
\end{thm}
When $F$ is archimedean, it is also convenient to replace the space of smooth and compactly supported functions by Schwartz functions.

Theorem \ref{thm:mainresult2} reproduces $c\in \CC^{\st,\fin}$ in turns of the descent datum $\{c_\alp = \CT_\alp(c)\}_{\alp\in \CI}$, except when $F$ is non-archimedean (of odd residual characteristic), the descent datum is not sufficient due to the existence of the Steinberg representation. 

In the following, we introduce the following two extended sections from $\CC^\infty_c(E^1_0)^{\FS_2}$ to functions on $\Fc(F)$:

\begin{thm}\label{thm:mainresult3}
With the notation from Theorem \ref{thm:mainresult2}, when $F$ is non-archimedean of odd residual characteristic, consider the following map from $c_0\in \CC^\infty_c(E^1_0)^{\FS_2}$ to functions on $\Fc(F)$:
\begin{align*}
\CE^{\perp \St}_0(c_0) &=
\CE_0(c_0)-\v\Del \mathbbm{1}_\el 
\cdot 
\langle \nu_0(c_0),1-\phi_{\del^{1/2}_B}\rangle
\\
\CE^{\perp \triv}_0(c_0) &=
\CE_0(c_0)
-\v\Del \mathbbm{1}_\el
\cdot 
\langle \nu_0(c_0),1\rangle 
\end{align*}
where $1$ is the constant function on $\Fc(F)$ and $\phi_{\del_B^{1/2}}=\nu_0\big(\frac{\del_B^{1/2}+\del_B^{-1/2}}{\Del}\big)$ with $\del_B$ the modular character for the Borel subgroup of $\SL_2(F)$. Then 
\begin{enumerate}
\item For $c\in \CC^{\st,\fin,\perp \St}$,
$$
c= \CE^{\perp \St}_0(c_0)+\sum_{\alp\neq 0}
\CE_\alp(c_\alp);
$$

\item For $c\in \CC^{\st,\fin,\perp \triv}$,
$$
c=\CE^{\perp \triv}_0(c_0)+\sum_{\alp\neq 0}
\CE_\alp(c_\alp).
$$
\end{enumerate}
where $\CC^{\st,\fin,\perp\St}$ and $\CC^{\st,\fin,\perp \triv}$ are the subspaces that are perpendicular to the Steinberg character and the trivial character, respectively.
\end{thm}

As a corollary, for distributions $z\in \CZ^{\st,\fin}$, by \eqref{num:intro:FTJ}, as functions on the Steinberg-Hitchin base, they are orthogonal to the trivial character. Hence the following theorem holds.

\begin{thm}\label{thm:mainresult:2}
With the above notation, for $z\in \CZ^{\st,\fin}$ with $\CT_\alp\circ \LafSec(z) = z_\alp\in\CC(E^1_\alp)^{\FS_2}$, the following identities hold.
\begin{enumerate}
\item When $F=\BC$, 
$$
z = \Laf
\big(
\nu_0(z_0)
\big);
$$

\item When $F=\BR$,
$$
z=\Laf
\bigg(
\sum_{\alp\in \CI}
\CE_\alp
\big(
\nu_\alp(z_\alp)
\big)
\bigg);
$$

\item When $F$ is non-archimedean of odd residual characteristic, 
$$
z = 
\Laf
\bigg(
\CE^{\perp, \triv}_0
\big(
\nu_0(z_0)
\big)
+
\sum_{\alp\in \CI\bs \{0\}}
\CE_\alp
\big(
\nu_\alp(z_\alp)
\big)
\bigg).
$$
\end{enumerate}
\end{thm}

\subsubsection{Nonabelian Fourier kernels}

Up to multiplying a smooth and compactly supported function along the determinant factor, one can assume that the conjectural nonabelian Fourier kernels $\RJ^\rho_\psi$ land in $\CZ$. Moreover, through comparing the (nonstandard) Langlands local gamma factors for diehdral representations and abelian Langlands local gamma factors on maximal tori, one can show that up to constant multiples of the Weil constant depending only on $\rho$, $\CT_\alp(\RJ^\rho_\psi) = \RJ^{\rho}_{\psi,\alp}$ for any $\alp\in \CI$, where $\RJ^\rho_{\psi,\alp}$ (up to constant multiples of the Weil constant) are smooth functions on the corresponding maximal tori constructed in subsection \ref{subsec:toruscase}. Hence Theorem \ref{thm:mainresult3} implies the following corollary.

\begin{cor}\label{cor:mainresult:3}
The following identities hold.
\begin{enumerate}
\item When $F=\BC$, 
$$
\RJ^\rho_\psi = \Laf
\big(
\nu_0(\RJ^\rho_{\psi,0})
\big);
$$

\item When $F=\BR$,
$$
\RJ^\rho_\psi=\Laf
\bigg(
\sum_{\alp\in \CI}
\CE_\alp
\big(
\nu_\alp(\RJ^\rho_{\psi,\alp})
\big)
\bigg);
$$

\item When $F$ is non-archimedean of odd residual characteristic, 
$$
\RJ^\rho_{\psi} = 
\Laf
\bigg(
\CE^{\perp, \triv}_0
\big(
\nu_0(\RJ^\rho_{\psi,0})
\big)
+
\sum_{\alp\in \CI\bs \{0\}}
\CE_\alp
\big(
\nu_\alp(\RJ^\rho_{\psi,\alp})
\big)
\bigg).
$$
\end{enumerate}
\end{cor}
Notice that Theorem \ref{thm:mainresult3} works for distributions supported only on finitely many Bernstein components. Although nonabelian Fourier kernels are supported on the whole Bernstein variety, for the purpose of local $L$-factors, it suffices to examine the action of $\RJ^\rho_\psi$ on $K$-finite vectors for smooth  and compactly supported (or Schwartz if $F$ is archimedean) functions. Hence we are reduced to the situation where the distribution lands in $\CZ^{\st,\fin}$. For details see subsection \ref{subsec:aformulaFK}.

\subsubsection{Stable orbital integral of spherical functions}

Let $h\in \CC^\infty_c(\SL_2(F))$ be a spherical function with stable orbital integral $\RS\RO(h)\in \CC^{\st,\fin}$ with the normalization introduced by Frenkel-Langlands-Ngô (\cite{FLN10}) revisited in subsection \ref{subsec:normalization-SOI-OI}. It is straightforward to see that $\RS\RO(h)$ is orthogonal to the Steinberg character, i.e. $\RS\RO(h)\in \CC^{\st,\fin,\perp \St}$. It turns out that the following theorem holds.

\begin{thm}\label{thm:mainresultSBC}
With the above notation, the following statements hold:
\begin{enumerate}
\item Let $h_\alp = \CT_\alp(\RS\RO(h))$. Then 
\begin{enumerate}
\item $h_0=\Sat(h)$;

\item $h_\alp = 0$ for $\alp\in \CI\bs \{0,1\}$;

\item When $F$ is non-archimedean of odd residual characteristic, $\RS\RO(h)\in \CC^{\st,\fin,\perp \St}$.
\end{enumerate}

\item 
$$
\RS\RO(h) = 
\Bigg\{
\begin{matrix}
h_0=\Sat(h) & F=\BC
\\
\CE_0(h_0) & F=\BR
\\
\CE^{\perp,\St}_0(h_0) & F\text{ non-archi.}
\end{matrix}
$$
\end{enumerate}
\end{thm}

\begin{rmk}
The fact that $\RS\RO(h) = \CE_0^{\perp, \St}(h_0)$ and $h_0=\RS\RO(h)|_{\c(E^1_0)}$ reflects the heuristic that, at least over the local function fields, there should exist a perverse $t$-structure on the stacky quotient $[G/\Ad(G)]$, such that $\RS\RO(h)$ can be realized as the trace of Frobenius of a perverse sheaf on $[G/\Ad(G)]$ which is the intermediate extension of its restriction to the split locus (\cite{BKVPevInf}).
\end{rmk}

\subsubsection{Orbital Hankel transform}

For a reductive group $G$ with $\rho$ an $L$-homomorphism of the Langlands dual group of $G$, let $\CS^\rho(G(F))$ be the conjectural $\rho$-Schwartz space (see subsection \ref{subsec:proposal} for a description) intertwined by the conjectural $\rho$-Fourier transform $\CF^\rho_\psi$. In \cite{ngo2016hankel}, the second author asks if one plugs the global $\rho$-Schwartz functions into the Arthur-Selberg trace formula, can one establish the functional equation of the global automorphic $L$-functions attached to the pair $(G,\rho)$ from the geometric side of the Arthur-Selberg trace formula. 
To understand the geometric side of the Arthur-Selberg trace formula for these $\rho$-Schwartz functions, one first need to understand the following commutative diagram
\begin{align}\label{eq:mainresult:HKT:1}
\xymatrix{
\CS^\rho(G(F))\ar[r]^{\CF^\rho_\psi} \ar[d]^{\mathrm{st. orb.}} & \CS^\rho(G(F))\ar[d]^{\mathrm{st. orb.}}
\\
\CS\CO^\rho(G(F)) \ar[r]^{\CH^{\rho}_\psi} & \CS\CO^\rho(G(F))
}
\end{align}
Here the vertical arrows are given by taking the stable orbital integrals, and $\CS\CO^\rho(G(F))$ is defined to be the space of stable orbital integrals of the conjectural $\rho$-Schwartz functions. The lower horizontal arrow is given by the induced transform $\CH^\rho_\psi$, which we refer to the \emph{$\rho$-orbital Hankel transform}. 
A natural question is if one can derive an explicit formula for the operator $\CH^\rho_\psi$. Motivated by Theorem \ref{thm:mainresult2}, we are going to study the descent of $\CS\CO^\rho(G(F))$ to the maximal tori when $G=\SL_2$ or $\GL_2$. Precisely, we show that the following diagram is commutative and the operator $\CH_z$ is well-defined for any $z\in \CZ$
\begin{align*}
\xymatrix{
\CS(G(F)) \ar[r]^z \ar[d]^{\text{st. orb.}} & \CS(G(F))\ar[d]^{\text{st. orb.}}
\\
\CS\CO(G(F)) \ar[r]^{\CH_z} \ar[d]^{\CT_\alp} & \CS\CO(G(F)) \ar[d]^{\CT_\alp}\\
\CS(T_\alp(F)) \ar[r]^{z_\alp} & \CS(T_\alp(F))
}
\end{align*}
Here $\CS(G(F))$ is the space of smooth and compactly supported functions on $G(F)$ when $F$ is non-archimedean, and is the space of Schwartz algebra consisting functions of rapidly decreasing when $F$ is archimedean. The upper (resp. lower) horizontal operator is given by 
$$
h\in \CS(G(F))\mapsto 
z*h^\vee\quad 
(\text{resp. $h_\alp\in \CS(T_\alp(F))\mapsto z_\alp*h^\vee_\alp$})
$$
where $z_\alp = \CT_\alp\circ \LafSec(z)$. Specializing to the $\rho$-orbital Hankel transform, it turns out that we have the following commutative diagram 
\begin{equation}\label{eq:mainresult:HKT:2}
\xymatrix{
\CS\CO^\rho(G(F)) \ar[d]^{\CT_\alp} \ar[r]^{\CH^\rho_\psi} & \CS\CO^\rho(G(F)) \ar[d]^{\CT_\alp}
\\
\CS^\rho(T_\alp(F)) \ar[r]^{\CF^\rho_{\psi,\alp}} & \CS^\rho(T_\alp(F))
}
\end{equation}
where the pair $(\CS^\rho(T_\alp(F)),\CF^\rho_{\psi,\alp})$ is introduced in subsection \ref{subsec:toruscase} that is of abelian nature.

Finally, based on Theorem \ref{thm:mainresult2}, we derive the following explicit formula for $\CH_z$ in turns of the descent datum. We present the formula for $G=\SL_2$.

\begin{thm}\label{thm:mainresult:5}
With the above notation, the following identities hold.
\begin{enumerate}
\item When $F=\BC$, 
$$
\CH_z(f) = \nu_0\big(z_0*\CT_0(f)^\vee\big);
$$

\item When $F=\BR$, 
$$
\CH_z(f) = 
\sum_{\alp\in \CI}
\CE_\alp
\big(z_\alp*
(
\CT_\alp(f)
)^\vee
\big);
$$

\item When $F$ is non-archimedean of odd residual characteristic,
$$
\CH_z(f) = 
\CE_0^{\rm HK}
(z_0,f)
+
\sum_{\alp\in \CI\bs\{0\}}
\CE_\alp
\big(z_\alp*
(
\CT_\alp(f)
)^\vee
\big)
$$
where 
$$
\CE^{\rm HK}_0(z_0,f) = 
\CE_0
\big(
z_0*(\CT_0(f))^\vee
\big)
-
\big(\CH_z(f)\big)^\el_\St
\cdot 
\v\Del \mathbbm{1}_\el.
$$
$\big(\CH_z(f)\big)^\el_\St$ is the constant given by the following formula
$$
=
\gam(z_0,\del_B^{1/2})
\langle f,\Theta_\St\rangle 
-
\langle 
\nu_0
\big(
z_0*\CT_0(f)^\vee
\big),
\phi_{\del^{1/2}_B}-1\rangle.
$$
\end{enumerate}
\end{thm}
It would be desirable if the formula above can be used to establish a global Poisson summation formula for $\CH^\rho_{\psi}$ on the Steinberg-Hitchin base.

\subsection{Notation and conventions}\label{subsec:notation}
Throughout the paper, fix a local field of residual characteristic not equal to two. In the archimedean case, $F$ is either $\BR$ or $\BC$. When $F$ is nonarchimedean, let $\Fo_F$ be the ring of integers with maximal ideal $\Fp_F$ generated by a fixed uniformizer $\vpi_F$, and $\Fo_F/\Fp_F = \kappa_F \simeq \BF_{q_F}$ is a finite field with $q_F$ elements. 

We normalize the additive invariant measure $\ud x$ on $F$ as follows: In the archimedean case $F=\BR$ we choose the usual Lebesgue measure. For $F=\BC$, we choose the measure $\ud x \ud y$ where $x$ and $y$ are real and imaginary parts. When $F$ is nonarchimedean, we normalize $\ud x$ so that the ring of integers $\Fo_F$ of $F$ has volume one.  

We fix a nontrivial additive character $\psi$ of $F$. There is then a unique additive Haae measure $\ud x$ on $F$, which is autodual with respect to $\psi$ in the following sense. 
Let $\CS(F)$ be the space of Schwartz-Bruhat functions on $F$. We have the Fourier transform 
$$\CF_\psi: \CS(F)\to \CS(F)$$
given by 
$$\CF_\psi(f)(y)=\int_F f(x)\psi(xy) \ud x.$$
The additive measure $\ud x$ is said to be auto-dual if we have $\CF_{\psi}\circ \CF_{\psi^{-1}}={\rm id}_{\CS(F)}$. 
In the case $F$ is nonarchimedean, we assume that $\psi$ is trivial on $\Fo_F$ but nontrivial on $\Fp^{-1}_F = \vpi^{-1}_F\Fo_F$ then the auto-dual measure $\ud x$ assigns $\Fo_F$ the measure one. If $F$ is Archimedean, we assume that
$$
\psi(x) = 
\begin{cases}
	\exp(2\pi ix), & F=\BR;\\
\exp(2\pi i(x+\wb{x})), &F=\BC.
\end{cases}
$$
Then the auto-dual measure will be the usual Lebesgue measure on $\BR$ and $\BC$. 

The action of $F^\times$ on $F$ induces an action of $F^\times$ on additive invariant measures given rise to the norm $x\mapsto |x| \in \BR_+$. When $F$ is archimedean, we have
$$|x| = 
\begin{cases}
	\text{absolute value of }x, &F=\BR\\
x\wb{x}, &F=\BC
\end{cases}
.$$
For nonarchimedean local fields $F$, we have $|x| = q_{F}^{-\ord_F(x)}$ where $\ord_F:F\to \BZ\cup \{\infty\}$ is the valuation on $F$ such that $\ord_F(\vpi_F) = 1$. In both archimedean and nonarchimedean cases, we denote $$\ud^\times x =|x|^{-1} {\ud x},$$ which is a Haar measure on $F^\times$. When $F$ is non-archimedean, and if $\vol(\Fo_F,\ud x)=1$, then we have $\vol(\Fo^\times_F,\ud^\times x) = 1-q^{-1}_F$. 

Let $E$ be a quadratic étale $F$-algebra. Let $\Tr = \Tr_{E/F} = \tr$ and $\Nr = \Nr_{E/F} = \det$ be the trace and norm map, respectively. Let $\eta =\eta_E= \eta_{E/F}:F^\times \to \{\pm 1\}$ be the corresponding quadratic character attached to $E/F$ via class field theory. Let $\lam_{E} = \lam_{E/F} = \lam_{E/F,\psi}$ be the Weil constant defined in \cite[Thm.~2]{weil64unitary} and \cite[Lem.~1.1]{jlgl2}. Let $\iota =\iota_{E} = \iota_{E/F}$ be the nontrivial involution on $E$ fixing $F$. Let $\psi_E = \psi\circ \Tr_{E/F}$ and let $\ud_E x$ be the additive Haar measure on $E$ that is self-dual with respect to $\psi_E$. In particular when $F$ is non-archimedean, $\vol(\Fo_E,\ud_E x) = |\FD_{E/F}|^{1/2}$ where $\FD_{E/F}$ is the discriminant for $E/F$. Let $E^1 = \ker(\Nr:E^\times \to F^\times)$ which fits into a short exact sequence $$1\to E^1\to E^\times \to \Nr(E^\times)\to 1.$$ Equip $\Nr(E^\times)$ with the restriction of the multiplicative Haar measure $\ud^\times x$ on $F^\times \supset \Nr(E^\times)$. Combining with the multiplicative Haar measure $\ud^\times _E x = {|\Nr(x)|}^{-1}{\ud_E x}$ on $E^\times$, the induced Haar measure $\ud_{E,1}e$ on $E^1$ is equal to $\ud^\times x$ if $E/F$ is split, and has the property that 
\begin{equation}\label{eq:notation:volumeE1}
\vol(E^1,\ud_{E,1}e)= \frac{\vol(\Fo^\times_E,\ud_E)[\Fo^\times_F:\Nr(\Fo^\times_E)]}{\vol(\Fo^\times_F,\ud x)}=\frac{(1-q_E^{-1})|\FD_{E/F}|^{\frac{1}{2}}e(E\mid F)}{(1-q_F^{-1})}
\end{equation}
when $E/F$ is a field extension and $F$ is non-archimedean. Here $e(E\mid F)$ is the ramification index of $E/F$.

Let $G=\GL_2$ (resp. $\SL_2$) be the group scheme of $2\times 2$ invertible matrices (resp. matrices with determinant one). Let $\c:G\to \Fc_G = G\sslash \Ad(G)\simeq \BG_a\times \BG_m$ (resp. $\BG_a$) be the Steinberg map \cite[p87]{Steinberg-conjugacy} sending $g\in G$ to $(c,a)=(\tr(g),\det(g))$ (resp. $c=\tr(g)$).  Let $\RD^G$ be the norm of the Weyl discriminant of $G$ and $\Del$ be its square root which can be considered as functions on $\mathfrak c_G(F)$
$$\Delta(c,a)=|c^2-4a|^{1/2} \mbox{ resp } \Delta(c)=|c^2-4|^{1/2}.$$

Haar measures on $\SL_2(F)$ can be expressed by coordinates given by the triangular decomposition.
For a generic element $g\in \SL_2(F)$ lying in the open dense Bruhat cell, we write
\begin{equation}\label{eq:notation:sl2mes:1}
g=  
\begin{pmatrix}
1 & \\
v & 1
\end{pmatrix}
\begin{pmatrix}
t & \\
  &t^{-1}
\end{pmatrix}
\begin{pmatrix}
1 & u\\
 &1
\end{pmatrix}\quad u,v\in F,t\in F^\times.
\end{equation}
Fix the following measure $\ud g$ on $\SL_2(F)$ inherited from the decomposition \eqref{eq:notation:sl2mes:1}
\begin{equation}\label{eq:notation:sl2mes:2}
\ud g = |t|^2\ud^\times t\ud u\ud v.
\end{equation}
By direct calculation, we see that this formula provides a Haar measure on $\SL_2(F)$. About this Haar measure, 
the following volume calculation will be needed later.

\begin{lem}\label{lem:notation:sl2:vol}
Let $F$ be non-archimedean. Then 
$$
\vol(\SL_2(\Fo_F),\ud^\times g) = 1-q_F^{-2}
$$
\end{lem}
\begin{proof}
Write 
$$
g = \begin{pmatrix}
1 & \\
v &1 
\end{pmatrix} 
\begin{pmatrix}
t & \\
  &t^{-1}
\end{pmatrix}
\begin{pmatrix}
1 & u\\
  &1
\end{pmatrix}
=
\begin{pmatrix}
t & tu \\
tv & t^{-1}+tuv
\end{pmatrix}.
$$
Since $\det g=1$, $g\in \SL_2(\Fo_F)$ if and only if all of its coordinates lie in $\Fo_F$. Equivalently 
$$
t\in \Fo_F,\quad tu,tv\in \Fo_F,\quad 1+t^2uv\in t \Fo_F.
$$
When $t\in \Fo^\times_F$, $u,v\in \Fo_F$, and hence under the measure \eqref{eq:notation:sl2mes:2}, it provides a total mass $\vol(\Fo^\times_F,\ud^\times t) = 1-q^{-1}_F$. When $t\in \Fp^k_F\bs \Fp^{k+1}_F= \vpi^k_F \Fo^\times_F$ for some $k\geq 1$,
$$
u,v\in \vpi^{-k}_F \Fo^\times_F,\quad \vpi^{2k}_F uv\in -1+\Fp^{k}_F
$$
and hence under the measure \eqref{eq:notation:sl2mes:2}, it provides a total mass 
$$
|\vpi_F^{k}|^2 |\vpi^{-k}_F|^2 \vol(\Fo^\times_F,\ud^\times t)\vol(\Fo^\times_F,\ud u) \vol(-1+\Fp^k,\ud v)=(1-q_F^{-1})^2 q_F^{-k}.
$$
In conclusion, we have
$$
\vol(\SL_2(\Fo_F),\ud g) = 
1-q_F^{-1}+
(1-q_F^{-1}) ^2
\sum_{k\geq 1}q_F^{-k} = 1-q_F^{-2},
$$
which is the desired formula.
\end{proof}

Similarly, for $g\in \GL_2(F)$ lying in the open dense Bruhat cell, write 
\begin{equation}\label{eq:notation:gl2:1}
g= 
\begin{pmatrix}
a & \\
  & 1
\end{pmatrix}
\begin{pmatrix}
1 & \\
v & 1
\end{pmatrix}
\begin{pmatrix}
t & \\
  & t^{-1}
\end{pmatrix}
\begin{pmatrix}
1 & u\\
  & 1
\end{pmatrix},\quad u,v\in F,\quad a,t\in F^\times.
\end{equation}
Fix the following measure $\ud g$ on $\GL_2(F)$ inherited from the decomposition \eqref{eq:notation:gl2:1}
\begin{equation}\label{eq:notation:gl2:2}
\ud g = |t|^2\ud^\times a\ud^\times t\ud u\ud v.
\end{equation}
By direct calculation the trivial extension of \eqref{eq:notation:gl2:2} to $\GL_2(F)$ provides a Haar measure on $\GL_2(F)$. Based on Lemma \ref{lem:notation:sl2:vol}, when $F$ is non-archimedean,
\begin{equation}\label{eq:notation:gl2:3}
\vol(\GL_2(\Fo_F), \ud g) = (1-q_F^{-1})(1-q_F^{-2}).
\end{equation}
In particular, $\ud g$ coincides with ${|\det x|^{-2}}{\ud x}$ where $\ud x$ is the restriction of the additive Haar measure on the space of $2\times 2$ matrices $\RM_{2\times 2}(F)$ normalized by $\vol(\RM_{2\times 2}(\Fo_F),\ud g) =1$.

\subsection{Acknowledgement}\label{acknowledgement}
We thank Jayce Getz, Hervé Jacquet, Dihua Jiang, Daniel Johnstone, David Kazhdan,  Robert Kottwitz, Robert Langlands, Gérard Laumon, Yiannis Sakellaridis, Freydoon Shahidi, Zhiwei Yun, and Don Zagier for stimulating conversations related to this work. Both authors were partially supported by the second author's Simon Investigator grant. The second author was also partially supported by the NSF grants 1702380 and 2201314.

% !TEX root = luo-ngo.tex

\section{Local aspects of the Braverman-Kazhdan program}\label{BK:intro} \label{sec:BK-intro}

In this section, we review the local aspects of the  Braverman-Kazhdan program. In subsection \ref{subsec:toruscase}, we establish the proposal for the torus case.

\subsection{An abstract of the local proposal}\label{subsec:proposal}

Let $G$ be a reductive group over a local field $F$. Following \cite[\S 4.2]{ngo2016hankel}, we will assume that $G$ fits into the following short exact sequence
\begin{equation}\label{eq:proposal:ses}
1\to G^\p\to G\to \BG_m\to 1
\end{equation}
with $\nu: G\to \BG_m$ a nontrivial homomorphism playing the role of the determinant. Dually there is a homomorphism $\nu^\vee:\BC^\times \to \LG$ where $\LG$ is the Langlands dual group of $G$. Let us restrict to the case when $\rho:\LG \to \GL(V_\rho)$ is irreducible and $\rho\circ \nu^\vee$ is the scalar multiplication of $\BC^\times$ on $V_\rho$ as in \cite[\S4.2]{ngo2016hankel}, which can always be achieved by adding an extra $\BG_m$ factor to $G$. Set 
\begin{equation}\label{eq:proposal:1}
L(\pi,\rho) = L(-\frac{n_\rho-1}{2},\pi,\rho),\quad 
n_\rho = \langle 2\eta_G,\lam_\rho\rangle +1
\end{equation}
where $2\eta_G$ is the sum of positive roots of $G$ and $\lam_\rho$ is the highest weight of $\rho$.
Here $$L(s,\pi,\rho)= L(\pi\otimes |\nu|^{s+\frac{n_\rho-1}{2}},\rho)$$ is the unramified local $L$-factor attached to an unramified representation $\pi$ of $G(F)$, which is defined via $L(s,\pi,\rho) = \det\big(\Id_{V_\rho}-\rho(c(\pi))q^{-s}\big)^{-1}$ with $c(\pi)$ the Satake parameter of $\pi$,
$2\eta_G$ is the sum of positive roots of $G$ and $\lam_\rho$ is the highest weight of $\rho$. 

The local aspect of the conjecture of Braverman-Kazhdan can be stated as follows:

\begin{conjec}\label{conjec:bknproposal}
\begin{enumerate}
\item There exists a Schwartz space $\CS^\rho(G(F))$ which contains the space of smooth and compactly supported functions $\CC^\infty_c(G(F))$ and is contained in the space of smooth functions $\CC^\infty(G(F))$, such that the zeta integral 
\begin{equation}\label{eq:proposal:zeta}
\CZ(s,\phi,\vphi_\pi) = 
\int_{G(F)}
\phi(g)\vphi_\pi(g)|\nu(g)|^{s+\frac{n_\rho-1}{2}}\ud g,\quad 
\phi\in \CS^\rho(G(F)),\vphi_\pi\in \CC(\pi)
\end{equation}
is absolutely convergent for $\Re(s)$ sufficiently large, with a meromorphic continuation to $s\in \BC$. Here $\ud g$ is a fixed Haar measure on $G(F)$ and $\CC(\pi)$ is the set of smooth matrix coefficients of an irreducible admissible representation $\pi$ of $G(F)$. 

The zeta integral $\CZ(s,\phi,\vphi_\pi)$ is a holomorphic multiple of the Langlands local $L$-factor $L(s,\pi,\rho)$. Moreover, when $F$ is non-archimedean, the set
$$
\CI_\pi = 
\{\CZ(s,\phi,\vphi_\pi)\mid \phi\in \CS^\rho(G(F)),\vphi_\pi\in \CC(\pi)\}
$$
is a finitely generated nonzero fractional ideal of $\BC[q^s,q^{-s}]$ with generator $L(s,\pi,\rho)$. When $F$ is archimedean, let $S_{a,b} = \{s\in \BC\mid a<\Re(s)<b\}$ be the vertical strip for any $a<b$. If $P(s)$ is a polynomial in $s$ such that $s\in \BC\mapsto P(s)L(s,\pi,\rho)$ is bounded in the vertical strip $S_{a,b}$, then $s\in \BC\mapsto P(s)\CZ(s,\phi,\vphi_\pi)$ is also bounded in the same vertical strip $S_{a,b}$;

The Schwartz space $\CS^\rho(G(F))$ contains a distinguished vector, the $\rho$-basic function $\BL^\rho\in \CS^\rho(G(F))$, which satisfies the following properties 
\begin{enumerate}
\item $\{\CZ(s,\BL^\rho,\vphi_\pi)\mid \vphi_\pi\in \CC(\pi)\}$ is nonzero only when $\pi$ is unramified;

\item For $\pi$ unramified and $\vphi_\pi^\circ$ the zonal spherical function of $\pi$
$$
\CZ(s,\BL^\rho,\vphi_\pi^\circ) = L(s,\pi,\rho).
$$
\end{enumerate}

\item There exists an invariant distribution $\RJ^{\rho}_{\psi}$ on $G(F)$ which is locally integrable and smooth over generic locus, such that the following properties hold
\begin{enumerate}
    \item
When $F$ is non-archimedean, for any $n\in \BZ$, the distribution $\RJ^{\rho}_{\psi,n} = \RJ^{\rho}_{\psi}\cdot \mathbbm{1}_{G_n}$ lies in the Bernstein center of $G(F)$, where $\mathbbm{1}_{G_n}$ is the characteristic function of the set $G_n = \{g\in G(F)\mid |\nu(g)| = q^{-n}\}$. Let $\gamma_{n}(\pi,\rho,\psi)$ be the regular function on the Bernstein variety of $G(F)$ defined by $\pi(\RJ^{\rho}_{\psi,n}) = \gamma_{n}(\pi,\rho,\psi)\Id_\pi$, then the following Laurent series
$$
s\in \BC\mapsto 
\sum_{n}\gamma_{n}(\pi\otimes |\nu|^s,\rho,\psi)
$$
is convergent for $\Re(s)$ sufficiently large, with a meromorphic continuation to $s\in \BC$, and is equal to the Langlands local gamma factor $\gam(-\frac{n_\rho-1}{2}-s,\pi^\vee,\rho,\psi)$. Here $\pi^\vee$ is the smooth contragredient of $\pi$;

    \item 
When $F$ is archimedean, for any delta sequence $\{\Fc_n\}_{n\geq 1}\subset \CC^\infty_c(G(F))$ tending to the delta mass supported at identity of $G(F)$, the following limit holds 
$$
\lim_{n\mapsto \infty} (\pi\otimes |\nu|^s)(\RJ^{\rho}_{\psi}*\Fc_n^\vee)
 \mapsto \gam(-\frac{n_\rho-1}{2}-s,\pi^\vee,\rho,\psi)\Id_\pi
$$ 
as long as $\Re(s)$ is sufficiently large depending only on $\pi$;
\end{enumerate}

\item The $\rho$-Fourier transform 
$$
\CF^{\rho}_{\psi}(\phi)(g) := 
|\nu(g)|^{-n_\rho}
(\RJ^{\rho}_{\psi}*\phi^\vee)(g),\quad \phi\in \CC^\infty_c(G(F))
$$
extends to an invertible operator sending $\CS^\rho(G(F))$ to itself, and enjoys the following properties 
\begin{enumerate}
    \item
$\CF^{\rho}_{\psi}\circ \CF^{\rho}_{\psi^{-1}} = \Id$;

    \item 
$\CF^{\rho}_{\psi}$ extends to a unitary operator on $L^2(G(F),|\nu(g)|^{n_\rho}\ud g)$;

    \item 
$\CF^{\rho}_{\psi}(\BL^\rho) = \BL^\rho$;
\end{enumerate}
After meromorphic continuation, the following functional equation holds 
\begin{equation}\label{eq:proposal:fe}
\CZ(1-s,\CF^{\rho}_{\psi}(\phi),\vphi_\pi^\vee) = 
\gam(s,\pi,\rho,\psi)\CZ(s,\phi,\vphi_\pi),\quad \phi\in \CS^\rho(G(F)),\vphi_\pi\in \CC(\pi).
\end{equation}
\end{enumerate}
\end{conjec}

\begin{rmk}\label{rmk:proposal}

\begin{enumerate}
\item The Schwartz space is expected to be connected with the  reductive monoid $M^\rho$ attached to the pair $(G,\rho)$. Over an algebraically closed field, the theory of reductive monoid is developed by E. Vinberg, M. Putcha, and L. Renner $($\cite{vinberg1995reductive}\cite{MR2134980}$)$. Over a non-algebraically closed field, an explicit construction of $M^\rho$ is given in \cite{ngo2016hankel}. To make it precise, for every $L$-homomorphism $\rho:\LG\to \GL(V_\rho)$, as shown in \cite[Prop.~5.1]{ngo2016hankel}, one can construct a monoid $M^\rho$ containing $G$ as the open subset of invertible elements. Over a non-archimedean local field of positive characteristic, the basic function $\BL^\rho$ can be interpreted as the trace of Frobenius on the sought-for intersection complex of the arc space of $M^\rho$ as shown in \cite{MR3462881}\cite{MR3619916}. In general, it is expected that a basis of the space of $\rho$-Schwartz functions can be constructed by nearby cycles (\cite{bezrukavnikov2023schwartz}).

\item In the short exact sequence \eqref{eq:proposal:ses}, if one set $G^\p = \SL_2$, then for any dominant cocharacter $\lam_n = (a\in \BG_m\mapsto a^n\in \BG_m)$, Vinberg's theory of universal monoid provides the following explicit description of $M_\rho$ and $G$ fitting into \eqref{eq:proposal:ses}:
\begin{enumerate}
\item 
\begin{align*}
M^{\lam_n} &= \{(a,x)\in \BA\times \RM_{2\times 2}\mid \lam_n(a) = \det g\}
\\
 G^{\lam_n} &= \{(a,g)\in \BG_m\times \GL_2\mid \lam_n(a) = \det g\};
\end{align*}

\item 
$$
G^{\lam_n}
\simeq 
\bigg\{
\begin{matrix}
\BG_m\times \SL_2 & n\text{ odd}\\
\GL_2 & n\text{ even}
\end{matrix}
$$

\item 
When $n$ is odd, ${}^LG^{\lam_n}\simeq \GL_2(\BC)$ and $
\rho_{\lam_n}
:
(\det)^{-\frac{n-1}{2}}\otimes \Sym^n: \GL_2(\BC)\to \GL_{n+1}(\BC) 
$; When $n$ is even, ${}^LG^{\lam_n}\simeq \BC^\times  \times \SO_3(\BC)$ and $\rho_{\lam_n}$ is the tensor product of the scaling action of $\BC^\times$ and the unique irreducible representation of $\SO_3(\BC)$ of dimension $n+1$. The result is known to experts, and explicit calculations can be found in \cite[Ex.~5.1.7]{Zhilinsingapore}.
\end{enumerate}

\item 
For the normalization of the action of $\RJ^{\rho}_{\psi}$ on $\pi$, i.e. the property (3) above, and how is it connected with the functional equation \eqref{eq:proposal:fe}, see \cite{MR3990815} and \cite[(5.2.1)]{Zhilinsingapore} for a precise derivation. 
Up to unramified twist there is no harm to replace $n_\rho$ by any complex number. In section \ref{sec:HKT} and section \ref{sec:stableorbit}, we will take $n_\rho = 0$ in order to present a cleaner formula for the orbital Hankel transforms and stable orbital integral of basic functions for $G=\BG_m\times \SL_2$ and $\GL_2$ and arbitrary $\rho$.

\item 
The conjecture holds for $G=\GL_n$ over an $F$-central simple algebra and $\rho$ the standard representation, which is due to Godement-Jacquet \cite{gjzeta} and \cite{jacquetarchimedean} $($Over an archimedean local fields, \cite{gjzeta} only treated $K$-finite vectors in $\CS^\rho(G(F))$ and $\CC(\pi)$, which was completed in \cite{jacquetarchimedean}$)$. Based on \cite{JLZ}, the conjecture is also known for $G=\BG_m\times \Sp_{2n}$, $\rho=\std$ and $F$ is non-archimedean of charateristic zero. Finally, the conjecture is also known when $G$ is a torus with mild assumptions on $\rho$, which is announced in \cite[\S 5.2]{ngo2016hankel}, and the details will be provided in subsection \ref{subsec:toruscase}.

\item 
When $F$ is archimedean, similar as the non-archimedean case, we expect that for any test function $f\in \CC^\infty_c(F^\times)$, $(f\circ \nu) \cdot \phi$ lies in $\CS(G(F))$, which is the space of Schwartz functions on $G(F)$ in the sense of \cite[7.1.2]{Wallach-RealI} that are rapidly decreasing under the action of any algebraic differential operators on $G(F)$ (and also \cite{AGSchwartznashmd}). Moreover, $\big((f\circ \nu)\cdot \RJ^\rho_{\psi}\big)*\CS(G(F))\subset \CS(G(F))$. 

\item Following Godement-Jacquet (\cite{gjzeta}), when $F$ is archimedean, it suffices to consider $K\times K$-finite functions in $\CS^\rho(G(F))$ (and $\CS(G(F))$), which is enough to produce all the local $L$-factors. Here $K$ is a fixed maximal compact subgroup of $G(F)$.

\item 
There are conjectures on the construction of the general $\rho$-Fourier kernels $\RJ^{\rho}_{\psi}$ proposed in \cite{BK00} and \cite[\S 6.2]{ngo2016hankel}, which will be briefly revisited in subsection \ref{subsec:excepex}.
\end{enumerate}
\end{rmk}

\subsection{Torus case}\label{subsec:toruscase}

When $G$ is a torus, Conjecture \ref{conjec:bknproposal} is known for any $\rho$. Here, we follow the exposition in \cite[\S 5.2,\S 6.1]{ngo2016hankel} and provide more details.

\subsubsection{General setup}
First let us consider the case when $T$ is a split torus over $F$, and let $\rho:T^\vee\to \GL(V_\rho)$ be a finite dimensional representation where $T^\vee$ is the dual torus of $T$ over $\BC$. Suppose that $\rho$ decomposes as a direct sum of characters 
$$
V_\rho = \oplus_{i=1}^r\BC_{\mu_i}
$$
where $\mu_i\in \Lam = \Hom(\BG_m,T)\simeq \Hom(T^\vee,\BG_m)$ are characters of $T^\vee$ which are not necessarily distinct. Following \cite[Prop.~5.1]{ngo2016hankel}, we assume that the cone $\xi(\rho)\in \Lam_\BR$ generated by $\{\mu_i\}_{i=1}^r$ is strictly convex. This assumption is implied by the exact sequence \eqref{eq:proposal:ses} and the assumption on the dual action of $\BC^\times$ on the representation $\rho$. Then there is a toric variety $M^\rho_T$ containing $T$ as a dense open subset, which is characterized by the property that a homomorphism $\lam:\BG_m\to T$ extends to a morphism $\BA^1\to M^\rho_T$ if and only if $\lam\in \Lam\cap \xi(\rho)$. 

The cocharacters $\{\mu_i\}_{i=1}^r$ induce a homomorphism of tori $\rho_T:\BG_m^r\to T$ via 
\begin{equation}\label{eq:toricase:1}
\rho_T(x_1,...,x_r) = \mu_1(x_1)...\mu_r(x_r)
\end{equation}
Each $\mu_i$ extends to a homomorphism of monoids $\mu_i:\BA^1\to M^\rho_T$ and \eqref{eq:toricase:1} extends to a homomorphism of monoids $\rho_{M_T}:\BA^r\to M^\rho_T$ such that the following commutative diagram is Cartesian
$$
\xymatrix{
\BG_m^r \ar[r] \ar[d]_{\rho_T} & \BA^r \ar[d]^{\rho_{M_T}}\\
T\ar[r] & M^\rho_T
}
$$

For a general torus $T$ over $F$, let $\Lam$ be the group of cocharacters of $T$ over an algebraic closure of $F$, then the Galois group $\Gam_F$ of $F$ acts on $\Lam$ through a finite quotient. The Langlands dual group $\LT$ is equal to $= T^\vee\rtimes \Gam_F$ where $T^\vee= \Hom(\Lam,\BC^\times)$. Let $\rho:\LT\to \GL(V_\rho)$ be an $r$-dimensional algebraic representation of $\LT$ satisfying the assumptions in \cite[Prop.~5.1]{ngo2016hankel}. The restriction $\rho|_{T^\vee}$ is a direct sum of characters, possibly with multiplicities
$$
\rho|_{T^\vee}
 = \oplus_{i=1}^m\lam_1^{r_i}
$$
where $\{\lam_i\}_{i=1}^m\in \Lam$ are distinct characters of $T^\vee$ and the multiplicities $\sum_{i=1}^m r_i = r$. The weights $\{\lam_i\}_{i=1}^m$ given with multiplicities $\{r_i\}_{i=1}^m$ determine a finite subset 
$$
R_\rho = \{(\lam_1,1),...,(\lam_1,r_1),...,(\lam_m,1),...,(\lam_m,r_m)\}
$$
of $\Lam\times \BN$. Since the representation $\rho|_{T^\vee}$ extends to $T^\vee\rtimes \Gam_F$, the subset $R_\rho$ is invariant under the action of $\Gam_F$ on $\Lam\times \BN$. Over an algebraic closure $\wb{F}$, there is a homomorphism $\rho_T:\BG_m^{R_\rho}\to T_{\wb{F}}$ by \eqref{eq:toricase:1}. As this homomorphism is equivariant with respect to the action of $\Gam_F$, it can be descended to a homomorphism $
\rho_T:D^\rho\to T
$ between tori over $F$
where $D^\rho$ is the unique torus defined over $F$ satisfying $D^\rho\otimes_F \wb{F} = \BG_m^{R_\rho}$ and such that the induced action of $\Gam_F$ on $\BG_m^{R_\rho}$ coincides with the one derived from the action of $\Gam_F$ on $R_\rho$. Let $\BA^\rho$ denote the affine space over $F$ satisfying $\BA^\rho\otimes_F \wb{F} = \BA^{R_\rho}$ and such that the induced action of $\Gam_F$ on $\BA^{R_\rho}$ coincide with the action of $\Gam_F$ on the set of indices $R_\rho$. The homomorphism $\rho_T$ extends to a morphism of monoids $\rho_{M_T}:\BA^\rho\to M^\rho_T$ such that the following commutative diagram is Cartesian:
\begin{equation}\label{eq:toricase:2}
\xymatrix{
D^\rho \ar[r] \ar[d]_{\rho_T} & \BA^\rho\ar[d]^{\rho_{M_T}}\\
T\ar[r] & M^\rho_T
}
\end{equation}

\subsubsection{$\rho$-Schwartz space}

We explain the construction of $\rho$-Schwartz space for the pair $(T,\rho)$.

Let $U$ be the kernel of $\rho_T$ and consider the quotient stack $\CM^\rho_T = [\BA^\rho/U]$. In general, as explained in \cite[\S 5.2, \S 6.1]{ngo2016hankel}, the geometry of $\CM^\rho_T$ gives rise to the space of Schwartz functions and the nonabelian Fourier transforms attached to $\rho$. For the purpose of this paper, let us consider the $F$-points of the diagram \eqref{eq:toricase:2} and consider the following commutative diagram, which in general constitutes only part of the $F$-point of the stack $\CM^\rho_T$:
\begin{equation}\label{eq:toricase:3}
\xymatrix{
    \CS(\BA^\rho(F)) \ar[d]_{(\rho_{M_T})_{!}} \ar[r]^{\CF_{\BA^\rho,\psi}}  & 
    \CS(\BA^\rho(F)) \ar[d]^{(\rho_{M_T})_{!}} \\
\CS( M^\rho_T(F)) \ar[r]^{\CF^{\rho}_{\psi}}    & 
\CS(M^\rho_T(F))
}
\end{equation}
Here by construction, $\BA^\rho(F)$ is an étale $F$-algebra. $\CS(\BA^\rho(F))$ is the space of Schwartz-Bruhat functions on $\BA^\rho(F)$ and $\CF_{\BA^\rho,\psi}$ is the standard Fourier transform intertwining the space $\CS(\BA^\rho(F))$. The morphism $(\rho_{M_T})_{!}$ is the push-forward map and $\CF^{\rho}_{\psi}$ is the induced nonabelian Fourier kernel. In particular, $\CS(M^\rho_T(F))$ can be viewed as $\CS( \BA^\rho(F))_{U(F)}$, the $U(F)$-coinvariants of $\CS( \BA^\rho(F))$. Not surprisingly, we define $\CS^\rho(T(F)):=\CS( M^\rho_T(F))$.

Alternatively, for any $t\in T(F)\subset M^\rho_T(F)$, by the Cartesian diagram \eqref{eq:toricase:2}, $(\rho_{M_T})^{-1}(t) = (\rho_T)^{-1}(t)$ is a $U(F)$-torsor, which is a closed subset of $\BA^\rho(F)$. The restriction of a Schwartz-Bruhat function to a closed subset is still a Schwartz function (\cite{BZ76}\cite{AGSchwartznashmd}). Hence
$$
t\in T(F)\mapsto (\rho_{M_T})_{!}(\phi)(t) =  \int_{x\in (\rho_T)^{-1}(t)}
\phi(x)\ud_tx,\quad \phi\in \CS(\BA^\rho(F))
$$
is absolutely convergent and defines a smooth function on $T(F)$, where $\ud_tx$ is the measure on the fiber $(\rho^{-1}_T)(t)$ which is inherited from the $U(F)$-torsor structure. Then $\CS^\rho(T(F)):=\big\{(\rho_{M_T})_{!}(\phi)\mid \phi\in \CS\big(\BA^\rho(F) \big)\big\}$. Through pulling back the zeta integrals and $\rho$-Schwartz functions to $D^\rho(F)\subset \BA^\rho(F)$, it is immediate to see that $\CS^\rho(T(F))$ satisfies Part (1) of Conjecture \ref{conjec:bknproposal}. 

\begin{rmk}\label{rmk:rhoschwartz:easyrmk}
We have the following remarks about $\CS^\rho(T(F))=\CS(M^\rho_T(F))$:
\begin{enumerate}
    \item
In general, functions in $\CS(M^\rho_T(F))$ are only defined on $T(F)$. They may tend to infinity near the boundary $M^\rho_T(F)\bs T(F)$;

    \item
As the group homomorphism $\rho_T$ might not be surjective on $F$-points, when restricted to $T(F)$, the functions in $\CS(M^\rho_T(F))$ are only supported on $\rho_T(D^\rho(F))$. Comparing with Part (1) of Conjecture \ref{conjec:bknproposal}, we may artificially add all the space of smooth and compactly supported functions $\CC^\infty_c(T(F))$ into $\CS^\rho(T(F))$. But for the purpose of the local theory of zeta integrals for $(T,\rho)$, this is actually not necessary;

    \item
In the following, when discussing test functions on $T(F)$, we restrict ourselves to those functions contained in $\CC^\infty_c(\rho_T(D^\rho(F)))$.

\end{enumerate}
\end{rmk}

\subsubsection{$\rho$-Fourier transform}
We explain the construction of $\rho$-Fourier transform for the pair $(T,\rho)$.

For any $\Phi\in \CS^\rho(T(F))$, by definition, there exists $\phi\in \CS(\BA^\rho(F))$ such that $\Phi = (\rho_{M_T})_!(\phi)$. Let us define $\CF^{\rho}_{\psi}(\Phi):=(\rho_{M_T})_!(\CF_{\BA^\rho,\psi}(\phi))\in \CS^\rho(T(F))$. 

First let us show that $\CF^{\rho}_{\psi}(\Phi)$ is independent of the choice of the representative $\phi$. Suppose that there are $\phi_i\in \CS(\BA^\rho(F))$, $i=1,2$, such that $\Phi = (\rho_{M_T})_!(\phi_i)$. For any character $\chi:T(F)\to \BC^\times$, based on the assumption made in \eqref{eq:proposal:ses}, the following identity holds
\begin{equation}\label{eq:toruscase:pullbackzeta}
\CZ(s,\Phi,\chi)  =
\int_{T(F)}
\Phi(t)\chi(t)|\nu(t)|^s\ud^*t = 
\int_{D^\rho(F)}
\phi_i(x)\big(\chi\circ \rho_T\big)(x)
|\det(x)|^s\ud^*x
\end{equation}
where $\ud^*t$ and $\ud^*x$ are fixed Haar measures on $T(F)$ and $D^\rho(F)$, respectively, and $n_\rho$ is taken to be $1$. In particular if the character $\chi$ is taken to be unitary, then by Tate's thesis (\cite{tatethesis}), the above integral is absolutely convergent for $\Re(s)>1$ with a meromorphic continuation to $s\in \BC$, and the poles are captured by the local $L$-factor $L(s,\chi\circ \rho_T)$. After meromorphic continuation, the following functional equation holds, 
\begin{equation}\label{eq:nonarchikernel:tori:fe}
\gam(s,\chi\circ \rho_T,\psi)
\CZ(s,\Phi,\chi) = 
\CZ(1-s,(\rho_{M_T})_!\big(\CF_{\BA^\rho,\psi}(\phi_i) \big),\chi^{-1}),\quad i=1,2.
\end{equation}
Taking the subtraction of the above identities for $i=1,2$, we get, after meromorphic continuation and changing variables
$$
\CZ(s,(\rho_{M_T})_!\big(\CF_{\BA^\rho,\psi}(\phi_1-\phi_2) \big),\chi)=0
$$
for any unitary character $\chi$ of $T(F)$. Hence the Mellin transform of the smooth function $(\rho_{M_T})_!\big(\CF_{\BA^\rho,\psi}(\phi_1-\phi_2) \big)|\nu(\cdot)|^s$ is identically zero for $\Re(s)$ large independent of the unitary character $\chi$. By Pointryagin duality and smoothness, $(\rho_{M_T})_!\big(\CF_{\BA^\rho,\psi}(\phi_1-\phi_2) \big)=0$ identically. Therefore the definition of $\CF^{\rho}_{\psi}(\Phi)$ is independent of the choice of $\phi_i$ and hence is well-defined. In particular, properties (a), (b), (c) and \eqref{eq:proposal:fe} in Part (3) of Conjecture \ref{conjec:bknproposal} follows immediately after pulling back the zeta integrals to $D^\rho(F)$. 

\begin{rmk}\label{rmk:rhofourier:easy}
Following Part (2) and (3) of Remark \ref{rmk:rhoschwartz:easyrmk}, we may artificially extend $\CF^\rho_\psi$ to $\CC^\infty_c(T(F))$ via zero extension, or using the $\rho$-Fourier kernel constructed below via convolution. For convenience we still restrict ourselves to test functions supported on $\rho_T(D^\rho(F))\subset T(F)$.
\end{rmk}

\subsubsection{$\rho$-Fourier kernel}

It remains to show the existence of the invariant distribution $\RJ^{\rho}_{\psi}$ representing $\CF^{\rho}_{\psi}$ as conjectured in Part (3) of Conjecture \ref{conjec:bknproposal} and establish the properties listed in Part (2) of Conjecture \ref{conjec:bknproposal}.  We follow the techniques developed in \cite{MR4474366}.

\subsubsection{$\rho$-Fourier kernel: Non-archimedean case}

First let us consider the case when $F$ is non-archimedean. Let us write $\BA^\rho(F) = \prod_{i=1}^m E_i$ with $E_i$ an $F$-étale algebra. Let $\Fo_{E_i}$ be the ring of integers of $E_i$ with the maximal ideal $\Fp_i$ and a fixed uniformizer $\vpi_i$. Let $\Fp^{\rho} = \prod_{i=1}^m \Fp_i$ which is an ideal of $\Fo^\rho = \prod_{i=1}^m \Fo_i$. Fix an additive character $\psi^\rho$ of $\BA^\rho(F)$ that is trivial on $\Fo^\rho$ but nontrivial on $(\Fp^\rho)^{-1} = \prod_{i=1}^m \Fp_i^{-1}$. In particular $\psi^\rho = \prod_{i=1}^m\psi_i$ where $\psi_i$ is an additive character of $E_i$ that is trivial on $\Fo_{i}$ but nontrivial on $\Fp_i^{-1}$. Fix the Haar measure $\ud x$ on $\BA^\rho(F)$ that is self-dual with respect to $\psi^\rho$. Consider the following sequence $\{\Fc_n\}_{n\geq 1}\subset \CC^\infty_c(\BA^\rho(F))$:
\begin{equation}\label{eq:nonarchikernel:tori:1}
\Fc_n (x) = \frac{1}{\vol(K_n,\ud x)}
\mathbbm{1}_{K_n}(x),\quad K_n = \Id_{D^\rho(F)}+(\Fp^\rho)^n.
\end{equation}
Then $\lim_{n\mapsto\infty}\Fc_n = \del_{\Id_{D^\rho(F)}}$ which is the delta mass supported at the identity of $D^\rho(F)$. By direct calculation, 
$$
\CF_{\BA^\rho,\psi}(\Fc_n)(x) = 
\psi^\rho(x)\mathbbm{1}_{(\Fp^\rho)^{-n}}(x).
$$
Since $\CF_{\BA^\rho,\psi}(\Fc_n)\in \CC^\infty_c(\BA^\rho(F))$, the function 
$$
(\rho_{M_T})_!\big(\CF_{\BA^\rho,\psi}(\Fc_n) \big)
$$
lies in $\CC^\infty(T(F))$. Following \cite[\S 3]{MR4474366}, we consider the following limit:
$$
\lim_{n\mapsto \infty}
(\rho_{M_T})_!\big(\CF_{\BA^\rho,\psi}(\Fc_n) \big).
$$
The following lemma is analogous to \cite[Lem.~3.3.]{MR4474366}.
\begin{lem}\label{lem:nonarchikernel:tori}
For any $t\in T(F)$, the limit 
$$
\lim_{n\mapsto \infty}
(\rho_{M_T})_!\big(\CF_{\BA^\rho,\psi}(\Fc_n) \big)(t)
$$
is stably convergent.
\end{lem}

\begin{proof}
For any $n_1\geq n_2$ sufficiently large, consider the subtraction 
\begin{equation}\label{eq:nonarchikernel:tori:2}
(\rho_{M_T})_!\big(\CF_{\BA^\rho,\psi}(\Fc_{n_1}-\Fc_{n_2}) \big)(t).
\end{equation}
In order to establish the lemma, it suffices to show that for fixed $t\in T(F)$, \eqref{eq:nonarchikernel:tori:2} vanishes as long as $n_1\geq n_2$ is sufficiently large. 

Following \eqref{eq:nonarchikernel:tori:1}, \eqref{eq:nonarchikernel:tori:2} is equal to the following absolutely convergent integral
\begin{equation}\label{eq:nonarchikernel:tori:3}
\int_{
\substack{
x\in (\rho_T)^{-1}(t)
\\
x\in (\Fp^\rho)^{-n_1}\bs (\Fp^\rho)^{-n_2}
}
}
\psi^\rho(x)
\ud_tx.
\end{equation}
If $t\in T(F)$ does not lie in the image of $\rho_T$, or the kernel $U(F)$ is compact, the lemma holds automatically. Hence in the following let us assume that $t\in T(F)$ is in the image of $\rho_T$ and $U(F)$ is non-compact. 

Set $U_1 = K_1\cap U(F)$, which is an open compact subgroup of $U(F)$. For any $u\in U_1$, since $u\in K_1$ (resp. $\in U(F)$), its left multiplication fixes $(\Fp^\rho)^{-n}$ (resp. $(\rho_T)^{-1}(t)$). It follows that for any $u\in U_1$, 
\begin{equation}\label{eq:nonarchikernel:tori:4}
\int_{
\substack{
x\in (\rho_T)^{-1}(t)
\\
x\in (\Fp^\rho)^{-n_1}\bs (\Fp^\rho)^{-n_2}
}
}
\psi^\rho(x)
\ud_tx
=\int_{
\substack{
x\in (\rho_T)^{-1}(t)
\\
x\in (\Fp^\rho)^{-n_1}\bs (\Fp^\rho)^{-n_2}
}
}
\psi^\rho(ux)
\ud_tx.
\end{equation}
Integrating both sides of \eqref{eq:nonarchikernel:tori:4} over $U_1$ with respect to the Haar measure $\ud u$ on $U(F)$, we get
\begin{equation}\label{eq:nonarchikernel:tori:5}
\vol(U_1,\ud u)
\int_{
\substack{
x\in (\rho_T)^{-1}(t)
\\
x\in (\Fp^\rho)^{-n_1}\bs (\Fp^\rho)^{-n_2}
}
}
\psi^\rho(x)
\ud_tx
=
\int_{u\in U_1}
\ud u
\int_{
\substack{
x\in (\rho_T)^{-1}(t)
\\
x\in (\Fp^\rho)^{-n_1}\bs (\Fp^\rho)^{-n_2}
}
}
\psi^\rho(ux)
\ud_tx.
\end{equation}
Notice that one can switch the integration order on the right hand side above due to the compactness of $U_1$:
\begin{equation}\label{eq:nonarchikernel:tori:6}
\int_{u\in U_1}
\ud u
\int_{
\substack{
x\in (\rho_T)^{-1}(t)
\\
x\in (\Fp^\rho)^{-n_1}\bs (\Fp^\rho)^{-n_2}
}
}
\psi^\rho(ux)
\ud_tx=
\int_{
\substack{
x\in (\rho_T)^{-1}(t)
\\
x\in (\Fp^\rho)^{-n_1}\bs (\Fp^\rho)^{-n_2}
}
}
\ud_tx
\int_{u\in U_1}
\psi^\rho(ux)
\ud u.
\end{equation}
Therefore, to show that \eqref{eq:nonarchikernel:tori:3} (equivalently, the left hand side of \eqref{eq:nonarchikernel:tori:5}) vanishes for $n_1\geq n_2$ sufficiently large, it suffices to establish the following statement:

\begin{num}\label{num:nonarchikernel:tori:1}
\item For any $x\in (\rho_T)^{-1}(t)\cap \big((\Fp^\rho)^{-n_1}\bs (\Fp^\rho)^{-n_2} \big)$, the integral 
$$
\int_{u\in U_1}\psi^\rho(ux)\ud u
$$
vanishes.
\end{num}
By \cite[Prop.~1.8]{gl2llc}, as long as $n_1\geq n_2$ is large, for $x\in \big( (\Fp^\rho)^{-n_1}\bs (\Fp^\rho)^{-n_2}\big)$,
\begin{equation}\label{eq:torusFK:nonarchi}
u\in K_1\mapsto \psi^\rho\big(x(u-1)\big)
\end{equation}
is a nontrivial character whose kernel is contained in $K_{n_1}$. Since $\{K_n\}_{n\geq 1}$ is a family of open compact neighborhoods of $\Id_{D^\rho(F)}$ tending to $\Id_{D^\rho(F)}$, $\{K_n\cap U(F)\}_{n\geq 1}$ is also a family of open compact neighborhoods of $\Id_{D^\rho(F)} = \Id_{U(F)}$ contained in $U(F)$ tending to identity. Therefore as long as $n_2$ is sufficiently large, the restriction of the character \eqref{eq:torusFK:nonarchi} to the closed subgroup $U_1 =U(F)\cap K_1$ is still a nontrivial character since $U(F)$ is assumed to be noncompact and $n_2$ can be arbitrarily large. It follows that for $n_2\geq n_1$ sufficiently large and $x\in \big((\Fp^\rho)^{-n_1}\bs (\Fp^{\rho})^{-n_2} \big)$, 
$$
\int_{u\in U_1}
\psi^\rho\big(x(u-1) \big)\ud u=0\quad \text{and hence}\quad 
\int_{u\in U_1}\psi^\rho(xu)\ud u=0.
$$
It follows that the statement \eqref{num:nonarchikernel:tori:1} holds and we finish the proof of the lemma.
\end{proof}

As a corollary, the stably convergent limit 
$
\lim_{n\mapsto \infty}
(\rho_{M_T})_!\big(\CF_{\BA^\rho,\psi}(\Fc_n) \big)
$
 defines a smooth function on $T(F)$. 
\begin{defin}\label{defin:nonarchikernel:tori}
With the above notation, define
$$
\RJ^{\rho}_{\psi}(t) = \lim_{n\mapsto \infty}
(\rho_{M_T})_!\big(\CF_{\BA^\rho,\psi}(\Fc_n) \big)(t),\quad t\in T(F).
$$
\end{defin}
It remains to show that on $\CC^\infty_c(T(F))\cap \CS^\rho(T(F)) = \CC^\infty_c(\rho_T(D^\rho(F)))$, $\CF^{\rho}_{\psi}$ is represented by the smooth function $\RJ^{\rho}_{\psi}$ as mentioned in Part (3) of Conjecture \ref{conjec:bknproposal}, and enjoys the properties listed in Part (2) of Conjecture \ref{conjec:bknproposal}. 

To show that $\CF^{\rho}_{\psi}$ is represented by $\RJ^{\rho}_{\psi}$, we make the following observation: For any $\Phi\in \CC^\infty_c(T(F))\cap \CS^\rho(T(F))$, we can assume that $\Phi=(\rho_{M_T})_!(\phi)$, $\phi\in \CC^\infty_c(D^\rho(F))$. By definition, 
$$
\CF^{\rho}_{\psi}(\Phi)(t) = 
(\rho_{M_T})_!\big(
\CF_{\BA^\rho,\psi}(\phi)
\big)(t),\quad 
\CF_{\BA^\rho,\psi}(\phi)(x) = |\det(x)|^{-1}(\psi^\rho*\phi^\vee)(x).
$$
Since $\phi\in \CC^\infty_c(D^\rho(F))$, it is fixed under the translation action by open compact subgroups $\{K_n\}_{n\geq n_1}$ for $n_1$ large. Therefore 
$$
\CF_{\BA^\rho,\psi}(\phi)(x) = 
|\det(x)|^{-1}
\big(\psi^\rho*(\Fc_{n_1}*\phi^\vee)\big)(x) = 
|\det(x)|^{-1}\big(
(\psi^\rho*\Fc_{n_1})*\phi^\vee
\big)(x).
$$
Since $\rho_T:D^\rho(F)\to T(F)$ is a group homomorphism, by direct calculation, the push-forward map $(\rho_{M_T})_!$ respects the convolution of functions and the involution $\phi\mapsto \phi^\vee$. Therefore 
\begin{align*}
(\rho_{M_T})_{!}(\CF_{\BA^\rho,\psi}(\phi))(t) =& 
|\nu(t)|^{-1}
\big(
(\rho_{M_T})_{!}
(
(\psi^\rho*\Fc_{n_1})
)
*(\rho_{M_T})_{!}(\phi^\vee)\big)(t)
\\
=&
|\nu(t)|^{-1}
\big(
(\rho_{M_T})_{!}
(
(\psi^\rho*\Fc_{n_1})
)
*\Phi^\vee\big)(t)
\end{align*}
Now let us take the limit $n_1\mapsto \infty$. Since $\Phi\in \CC^\infty_c(T(F))$, we derive the desired identity 
\begin{equation}\label{eq:nonarchikernel:tori:7}
\CF^{\rho}_{\psi}(\Phi)(t) = 
|\nu(t)|^{-1}(\RJ^{\rho}_{\psi}*\Phi^\vee)(t),\quad t\in T(F).
\end{equation}
It follows that the operator $\CF^{\rho}_{\psi}$ is represented by $\RJ^{\rho}_{\psi}$ on $\CC^\infty_c(T(F))\cap \CS^\rho(T(F))$. 

\begin{rmk}\label{rmk:nonarchikernel}
For a general function $\Phi\in \CS^\rho(T(F))$, the operator $\Phi\mapsto \CF^{\rho}_{\psi}(\Phi)$ can still be represented by the smooth function $\RJ^{\rho}_{\psi}$. But the convolution on the right hand side of \eqref{eq:nonarchikernel:tori:7} need regularization. See \cite{JLZ} and \cite{MR4474366} for particular cases treated.
\end{rmk}
Next let us establish the following lemma.

\begin{lem}\label{lem:nonarchikernel:tori:BCenter}
Let $\RJ^{\rho}_{\psi,n} = \RJ^{\rho}_{\psi}\cdot\mathbbm{1}_{T_n}$, where $\mathbbm{1}_{T_n} =\{t\in T(F)\mid |\nu(t)| = q^{-n}\}$. Then $\RJ^{\rho}_{\psi,n}$ lies in the Bernstein center of $T(F)$.
\end{lem}
\begin{proof}
For any $\Phi\in \CC^\infty_c(T(F))$, since $\RJ^\rho_\psi$ is supported on $\rho_T(D^\rho(F))$, there is no harm to assume that $\Phi$ is also supported on $\rho_T(D^\rho(F))$, and hence $\Phi=(\rho_{M_T})_!(\phi)$ with $\phi\in \CC^\infty_c(D^\rho(F))$. It suffices to show that 
$$
|\nu(t)|^{-1}(\RJ^{\rho}_{\psi,n}*\Phi^\vee)(t)\in \CC^\infty_c(T(F)). 
$$
Up to translation and a finite linear combination, there is no harm to assume that $\Phi$ is supported in a maximal open compact subgroup of $T(F)$ with $\mathrm{supp}(\Phi)\subset \{t\in T(F)\mid |\nu(t)| = 1\}$. By definition, 
$$
(\RJ^{\rho}_{\psi,n}*\Phi^\vee)(t) = 
\int_{T(F)}
\RJ^{\rho}_{\psi}(yt)
\mathbbm{1}_{T_n}(yt)\Phi(y)\ud^*y
$$
By the support condition of $\Phi$, $\mathbbm{1}_{T_n}(yt) = \mathbbm{1}_{T_n}(t)$. Therefore 
$$
(\RJ^{\rho}_{\psi,n}*\Phi^\vee)(t)
=
\mathbbm{1}_{T_n}(t)
\int_{T(F)}
\RJ^{\rho}_{\psi}(yt)\Phi(y)\ud^*y
=\mathbbm{1}_{T_n}(t)
(\RJ^{\rho}_{\psi}*\Phi^\vee)(t).
$$
By \eqref{eq:nonarchikernel:tori:7}, it is equal to 
$$
\mathbbm{1}_{T_n}(t)
|\nu(t)|
\CF^{\rho}_{\psi}(\Phi)(t) = 
\mathbbm{1}_{T_n}(t)
|\nu(t)|
(\rho_{M_T})_!(\CF_{\BA^\rho,\psi}(\phi))(t)
$$
which is equal to 
$$
(\rho_{M_T})_!(\mathbbm{1}_{D^\rho_n}|\det|\CF_{\BA^\rho,\psi}(\phi))(t)
$$
where $\mathbbm{1}_{D^\rho_n} = \{x\in D^\rho(F)\mid |\det x| = q^{-n}\}$. Since $\CF_{\BA^\rho,\psi}(\phi)\in \CS(\BA^\rho(F))$, 
$$
\mathbbm{1}_{D^\rho_n}|\det|\CF_{\BA^\rho,\psi}(\phi)
\in \CC^\infty_c(\BA^\rho(F))
$$
and hence 
$$
(\rho_{M_T})_!(\mathbbm{1}_{D^\rho_n}|\det|\CF_{\BA^\rho,\psi}(\phi))\in \CC^\infty_c(T(F)).
$$
It follows that $\RJ^{\rho}_{\psi,n}$ lies in the Bernstein center of $T(F)$. 
\end{proof}
It follows that there exists a regular function $\gam_n$ on the admissible dual of $T(F)$, such that for any character $\chi$ of $T(F)$, 
$$
\chi(\RJ^{\rho}_{\psi,n}) = \gam_n(\chi,\rho,\psi).
$$
It remains to show that 
$$
s\in \BC\mapsto 
\sum_{n}\gam_n(\chi\otimes |\nu|^s,\rho,\psi)
$$
is convergent for $\Re(s)$ large, with a meromorphic continuation to $\BC$, and is equal to $\gam(-s,\chi^{-1},\rho,\psi)$. But it follows directly from the functional equation \eqref{eq:nonarchikernel:tori:fe} and the same calculation as \cite[Thm.~3.6]{MR4474366}. To make it precise, choose a family of open compact subgroups $K_{T,k}$ of $T(F)$ tending to identity and let $\Fc_{T,k}$ be the normalized characteristic function of $K_{T,k}$. By the functional equation \eqref{eq:nonarchikernel:tori:fe}, 
$$
\gam(s,\chi\circ \rho_T,\psi)\CZ(s,\Fc_{T,k},\chi) = 
\CZ(1-s,\CF^{\rho}_{\psi}(\Fc_{T,k}),\chi^{-1}).
$$
Notice that since $\Fc_{T,n}$ is smooth and compactly supported, and $\chi$ is fixed, 
$
\CZ(s,\Fc_{T,k},\chi) = 1
$
for any $s\in \BC$ as long as $k$ is bigger than the level/depth of $\chi$. Therefore for $\Re(s)$ sufficiently small, the zeta integral defining $\CZ(1-s,\CF^{\rho}_{\psi}(\Fc_{T,k}),\chi^{-1})$ is absolutely convergent and is equal to $\gam(s,\chi\circ \rho_T,\psi)$. Equivalently, for $\Re(s)$ sufficiently large, the zeta integral defining $\CZ(s,\RJ^{\rho}_{\psi}*\Fc^\vee_{T,k},\chi)$ is absolutely convergent and
$$
\CZ(s,\RJ^{\rho}_{\psi}*\Fc^\vee_{T,k},\chi) = 
\gam(-s,\chi^{-1},\rho,\psi).
$$
Therefore 
$$
\CZ(s,\RJ^{\rho}_{\psi}*\Fc^\vee_{T,k},\chi) = 
\sum_{n}
\CZ(s,\RJ^{\rho}_{\psi,n}*\Fc^\vee_{T,k},\chi) = 
\sum_n \gam_n(\chi\otimes |\nu|^s,\rho,\psi)
$$
is convergent for $\Re(s)$ large and is equal to $\gam(-s,\chi^{-1},\rho,\psi)$.

It follows that we have established all the desired properties listed in Conjecture \ref{conjec:bknproposal} for $T$ a torus over a non-archimedean local field.

\subsubsection{$\rho$-Fourier kernel: Archimedean case}

Now let us assume that $F$ is archimedean. Following the same idea as the non-archimedean case, let us fix a delta sequence $\{\Fc_n\}_{n\geq 1}\subset \CC^\infty_c(D^\rho(F))$ tending to the delta mass supported at the identity element of $D^\rho(F)$. Consider the limit 
$$
\lim_{n\mapsto \infty}
(\rho_{M_T})_!\big(\CF_{\BA^\rho,\psi}(\Fc_n)\big)(t),\quad t\in T(F).
$$
Following the argument of \cite[Prop.~4.3]{MR4474366}, we are going to show that the limit (and all of its derivatives in $t$) is uniformly convergent for $t$ lying in a compact subset, from which we deduce that the limit converges to a smooth function on $T(F)$. 

Take the Casimir element $\Del_{\mathrm{cas}}$ of $D^\rho(F)$. By \cite[Lem.~3.7]{bkglobalization}, for any integer $m$ with $2m>\dim D^\rho(F)$, there exists a function $f_1\in \CC^{2m-\dim D^\rho(F)-1}_c(D^\rho(F))$ and $f_2\in \CC^\infty_c(D^\rho(F))$, such that 
\begin{equation}\label{eq:archi:kernel:tori:limit}
\Del_{\mathrm{cas}}^m*f_1+f_2 = \del_{\Id_{D^\rho(F)}}.
\end{equation}
Let $L$ be the translation action of $T(F)$ on a space of functions on $T(F)$. Then 
$$
L(\Del_{\mathrm{cas}}^m*f_1+f_2) \text{ is equal to the identity operator}.
$$
Therefore 
\begin{equation}\label{eq:archi:kernel:tori:1}
(\rho_{M_T})_!(\CF_{\BA^\rho,\psi}(\Fc_n))(t) = 
\int_{x\in (\rho_T)^{-1}(t)}
\CF_{\BA^\rho,\psi}(\Fc_n)(x)
L(\Del_{\mathrm{cas}}^m*f_1+f_2)(\mathbbm{1}_{D^\rho})(x)
\ud_tx
\end{equation}
where $\mathbbm{1}_{D^\rho}$ is the identity function (which can also be viewed as the trivial character) on $D^\rho(F)$. Now the Carsimir element act on $\mathbbm{1}_{D^\rho}$ via a scalar. Hence up to constant, \eqref{eq:archi:kernel:tori:1} is equal to the sum of the following two integrals 
$$
\int_{x\in (\rho_T)^{-1}(t)}
\CF_{\BA^\rho,\psi}(\Fc_n)(x)f_i(x)\ud_tx,\quad i=1,2.
$$
Now by absolutely convergence and the compactly supportedness of $f_i$, 
\begin{align}\label{eq:archi:kernel:tori:2}
\int_{x\in (\rho_T)^{-1}(t)}
\CF_{\BA^\rho,\psi}(\Fc_n)(x)f_i(x)\ud_tx
=&
\int_{x\in (\rho_T)^{-1}(t)}
f_i(x)
\ud_tx
\int_{y\in \BA^\rho(F)}
\psi(xy)\Fc_n(y)\ud y   \nonumber
\\
=&
\int_{y\in \BA^\rho(F)}
\Fc_n(y)
\ud y
\int_{x\in (\rho_T)^{-1}(t)}
f_i(x)\psi(xy)\ud_tx
\end{align}
which, as $n\mapsto \infty$, tends to the absolutely convergent integral
$$
\int_{x\in (\rho_T)^{-1}(t)}
f_i(x)\psi(xy)\ud_tx.
$$
As $f_1\in \CC^{2m-\dim D^\rho(F)-1}_c(D^\rho(F))$ and $f_2\in \CC^\infty_c(D^\rho(F))$, and $m$ can be arbitrary large, we deduce that the limit \eqref{eq:archi:kernel:tori:2} is uniformly convergent on compact subset and converges to a smooth function on $T(F)$, which is still denoted as $\RJ^{\rho}_{\psi}$. Finally for Part (2)(b) of Conjecture \ref{conjec:bknproposal}, it follows immediately from the functional equation after plugging one side a family of delta sequence. The proof is the same as \cite[Thm.~4.6]{MR4474366} and we omit.

It follows that we have established all the desired properties listed in Conjecture \ref{conjec:bknproposal} for $T$ a torus over an archimedean local field.

\begin{thm}\label{thm:BK:torus}
Conjecture \ref{conjec:bknproposal} holds for $G=T$ a torus over any local fields.
\end{thm}

%\subsection{Preliminary remarks on the nonabelian Fourier kernel}\label{subsec:nonabFK}

\subsection{Low symmetric power examples}\label{subsec:excepex}

In subsection \ref{subsec:toruscase}, the Conjecture \ref{conjec:bknproposal} is fully established when $G$ is a torus. The next question is to construct $\RJ^{\rho}_{\psi}$ for $G$ a general reductive group over a local field $F$ with a fixed $L$-homomorphism $\rho:\LG\to \GL(V_\rho)$. 

Let $\c_G:G\to G\sslash \Ad(G)$ be the Chevalley quotient map. For convenience let us assume that $G$ is split, semisimple  and simply connected. We fix a maximal split torus $T$ and let $W = W(G,T)$ be the Weyl group attached to $T$. Following \cite{Steinberg-conjugacy}, $\Fc_G=G\sslash \Ad(G)\simeq T\sslash W$ is an affine space. By restricting $\LG$ to the diagonal torus $\LT$, we obtain a homomorphism $\rho_T:\BG_m^{\dim V_\rho}\to T$. The idea of Braverman and Kazhdan (\cite{BK00}) is that one should be able to define a (possibly birational) $W$-action on $\BG_m^{\dim V_\rho}$ such that the $W$-action on $\BG_m^{\dim V_\rho}$ preserves the trace function $\tr:\BG_m^{\dim V_\rho}\to \BA^1$, acts via sign character $\sgn:W\to \{\pm 1\}$ on the standard volume form on $\BA_n^{\dim V_\rho}$, and makes $\rho_T$ a $W$-equivariant morphism. If such a $W$-action is ensured, there is an induced morphism $\rho_{T\sslash W}:\BG_m^{\dim V_\rho}\sslash W\to T\sslash W$. Through fiber product there is a morphism $\BG_m^{\dim V_\rho}\times_{T\sslash W}G\to G$ and one can push-forward the datum on $(\BG_m^{\dim V_\rho}\times_{T\sslash W}G)(F)$ to $G(F)$ to get the desired distribution. We are not able to find the $W$-action that Braverman and Kazhdan stipulated in general but only in some particular cases.

Instead, we will pursue an idea speculated in \cite[\S 6]{ngo2016hankel}. It is expected that the distribution $\RJ^\rho_\psi$ should be glued from the family of smooth functions $\{\RJ^\rho_{\psi,T}\}_{T}$, where $T\subset G$ runs over the stable conjugacy classes of maximal tori of $G$ and $\RJ^\rho_{\psi,T}$ is the Fourier kernel constructed in subsection \ref{subsec:toruscase} for the pair $(T,\rho)$. The first naive idea is to patch the functions $\{\RJ^\rho_{\psi,T}\}_{T}$ directly to obtain the desired function $\RJ^\rho_\psi$ on $\Fc(F)$, as this works in the Godement-Jacquet case for $G=\GL_n$ and $\rho$ the standard representation of $\GL_n$ (\cite{gjzeta}). 

In the following, for convenience, let us consider $G=\SL_2$ over a non-archimedean local field of odd residual characteristic. Actually based on Part (2) of Remark \ref{rmk:proposal} we shall take $G=\BG_m\times \SL_2$, but for introductory purpose let us stick to $G=\SL_2$. There are four stable conjugacy classes of one-dimensional tori $T_\alp$ with $T_\alp(F) = E_\alp^1$ for $\alp\in \{0,1,\pm 1/2\}$: $E_\alp$ are the étale quadratic $F$-algebras with $E_0 = F\times F$ the split extension, $E_1$ the unramified extension, $E_{\pm 1/2}$ the two nonisomorphic ramified quadratic extensions of $F$, and $E^1_\alp$ is the subgroup of elements of norm one. When $\alp\neq 0$, $T_\alp(F)$ are compact commutative groups. For each $\alp$, we denote by $\pi_\alp:T_\alp(F)\to \Fc(F)$ given by the trace map, and $\tr(E_\alp^1)$ the image of $\pi_\alp$. We have an almost partition of $\Fc(F) = F$ as 
$$
\Fc(F)=F = \bigcup_{\alp\in \{0,1,\pm 1/2\}}
\tr(E^1_\alp)
$$
with $\tr(E^1_\alp)\cap \tr(E^1_\bet) = \{\pm 2\}$ if $\alp\neq \bet$. We denote by $\tau_\alp$ the involution of $T_\alp$ corresponding to the Galois involution of $E_\alp/F$. Every $\tau_\alp$-invariant smooth function $J_\alpha$ on $T_\alpha(F)$ descends to a unique function $\nu_\alpha(J_\alpha)$ on $\tr(E_\alpha^1)$:
\begin{equation}\label{eq:pialp-nualp}
J_\alpha = \pi_\alpha^*(\nu_\alpha(J_\alpha)).
\end{equation}
The naive guest would be $\RJ^\rho_{\psi,\rm Ngo}=\sum_{\in \{0,1,\pm 1/2\}} \nu_\alpha(\RJ^\rho_{\psi,\alpha})$ where $\RJ^\rho_{\psi,\alp} = \RJ^\rho_{\psi,T_\alp}$ is constructed in subsection \ref{subsec:toruscase}. It turns out that $\RJ^\rho_{\psi,\rm Ngo}$ conincides with the Godement-Jacquet kernel based on the fact that $\RJ^{\rho=\std}_{\psi} = \psi(\tr(\cdot))$ and $\RJ^\rho_{\psi,\alp} = \psi(\tr_{E_\alp/F}(\cdot))$.

But this naive approach does not work in other cases, as was shown by Lafforgue in \cite{lafforguegl2}. He showed that in general, the naive patching construction fails to yield the correct gamma factor, even for principal series representations. He further gave an integral transform that, when applied to the naive patching function $\RJ^\rho_{\psi,\rm Ngo}$, gives rise to the correct Langlands local gamma factors for principal series representations. His transform
$$\RJ^\rho_{\psi,\mathrm{Laf}}:= \mathrm{Laf}(\RJ^\rho_{\psi,\mathrm{Ngo}})$$ 
is given by the following identity
$$
\RJ^\rho_{\psi,\mathrm{Laf}}=\CF_{\psi}(|.| \CF_{\psi^{-1}}(\RJ^\rho_{\psi,\mathrm{Ngo}}))
$$
where $\CF_\psi$ is the Fourier transform in the trace variable, and $|.|$ is the absolute value of the dual variable. In \cite{lafforguegl2}, Lafforgue claimed without proof that, disregarding potential convergence issue in its definition, $\RJ^\rho_{\psi,\mathrm{Laf}}$ would act on principal series representations with the correct Langlands local gamma factors. The next question is whether the Lafforgue recipe applied to the naive patching function gives rise to the correct $\gamma$ factor for supercuspidal representations. The answer is still negative in general, but as we will see, the Lafforgue integral transform is still an essential device for formulating the correct construction. 

In the following, we discuss three explicit examples related to the conjectural construction of the second-named author and the correction provided by  Lafforgue.

\subsubsection{$\RI$} 
The first example we discuss is $G=\GL_2$ and $\rho$ the standard representation. We have seen from the above discussion that $\RJ^\rho_\psi = \RJ^\rho_{\psi,\rm Ngo} = \psi(\tr(\cdot))$. On the other hand, since $\CF_{\psi^{-1}}(\psi(\cdot)) = \del_{1}$, the delta mass supported at $x=1$. It follows that in this case, $\RJ^{\rho}_{\psi,\rm Laf} = \RJ^\rho_{\psi,\rm Ngo} = \RJ^\rho_\psi$.

\subsubsection{$\RI\RI$}
The second example we discuss is $G=\BG_m\times \SL_2$ and $\rho = \Sym^2$, i.e. $\LG = \BC^\times \times \SO_3$ and $\rho$ is the standard embedding of $\SO_3$ into $\GL_3$ together with the central scaling action of $\BC^\times$. This is the case treated in \cite{JLZ} (based on \cite{doubling}). 

Based on \cite[(1.9)]{JLZ}, up to unramified twist on the determinant factor, 
$$
\RJ^\rho_{\psi}(a,\tr(g)) = \eta_{\psi}(a(2+\tr(g)))
$$
where $\eta_\psi$ is a smooth function on $F^\times$ with the property that for any character $\chi$ of $F^\times$, the following regularized Mellin transform holds as meromorphic function of $s\in \BC$
$$
\int^\pv_{F^\times}
\chi^{-1}(x)|x|^{-s}\eta_\psi(x)\ud^\times x = 
\gam(s+1,\chi,\psi)\gam(s+2,\chi^2,\psi)
$$
where $\gam(s,\chi,\psi)$ is the local Tate gamma factor (\cite{tatethesis}). For the definition of the regularization, see \cite[\S 4.3]{JLZ} for details. Equivalently, $\eta_\psi$ is the (multiplicative) convolution of the following two distributions on $F^\times$:
\begin{equation}\label{eq:lowSym:II:etaconvolution}
\eta_\psi = \Phi_{1,\psi}*\Phi_{2,\psi}
\end{equation}
where $\Phi_{1,\psi} = \psi(\cdot)$ and $\Phi_{2,\psi} = \lam_{!}(\psi(\cdot)|\cdot|^{-1})(\cdot)$ where $\lam:F^\times \to F^\times$ is the map $x\mapsto x^2$ and $\lam_!$ is the push-forward of $\psi(\cdot)|\cdot|^{-1}$ along $\lam$. Formally, we can also write $\RJ^\rho_\psi$ by the following integral formula 
\begin{equation}\label{eq:lowSym:II:correct}
\RJ^\rho_{\psi}(a,\tr(g)) = 
\int^\pv_{F^\times}
\psi
\bigg(
z+\frac{a(2+\tr(g))}{z^2}
\bigg)
|z|^{-2}\ud^\times z.
\end{equation}
In the following, let us compute $\RJ^\rho_{\psi,\rm Ngo}$, $\RJ^\rho_{\psi,\rm Laf}$ and compare them with $\RJ^\rho_\psi$.

We first compute $\RJ^\rho_{\psi, \rm Ngo}$. We follow the presentation in subsection \ref{subsec:toruscase}. The restriction of $\rho$ to the diagonal torus of $\LG$ is given by 
$$
\rho|_{T^\vee}:(a,
\left(
\begin{smallmatrix}
 t& & \\
  & 1&\\
  && t^{-1}
\end{smallmatrix}
\right)
)\mapsto 
\left(
\begin{smallmatrix}
 at& & \\
  & a&\\
  && at^{-1}
\end{smallmatrix}\right).
$$
Passing to the dual side, we obtain the following homomorphisms
\begin{num}
    \item 
For any étale quadratic $F$-algebra $E$, let $\BG_{m,E} = \Res_{E/F}\BG_m$ and $\BG_{m,E}^1 = \Ker(\Nr:\BG_{m,E}\to \BG_m)$. Let $T_E = \BG_m\times \BG_{m,E}^1$ be the maximal torus in $G=\BG_m\times \SL_2$ determined by $E$, then the homomorphism dual to $\rho|_{T^\vee}$ attached to $E$ is given by 
\begin{align*}
\rho_E:\BG_m\times \BG_{m,E}&\to \BG_m\times \BG_{m,E}^1
\\
(d,e)
&\mapsto 
(d\Nr(e),e/\tau_E(e)).
\end{align*}
\end{num}
It is worth pointing out that in this situation, $\ker \rho_E$ is always isomorphic to $\BG_m$. 

Passing to $F$-points, for $(a,g)\in F\times E^1 = T_E(F)$, we are going to push-forward the standard Fourier kernel $\psi(d+\tr_{E/F}(e))$ on $F\times E$ along the fiber $\rho^{-1}_E(a,g)$:
$$
\RJ^\rho_{\psi,E}(a,g)= 
\int^\pv_{
\substack{
d\Nr(e) = a
\\
e/\tau_E(e) = g
\\
d\in F^\times, e\in E^\times
}
}
\psi(d+\tr_{E/F}(e))
(\ud^\times d \ud^\times e)_{(a,g)}.
$$
The integral need regularization. Before discussing the regularization let us simplify the algebraic equations for the fiber $\rho^{-1}_E(a,g)$ given by the equations
\begin{equation}\label{eq:lowsymex:II:1}
d\Nr(e) =a,e/\tau_E(e) = g,\quad d\in F^\times, e\in E^\times.
\end{equation}
Choose $e_g\in E^\times$ such that $e_x/\tau_E(e_g) = g$. Then all other solutions to the equation $e/\tau_E(e) = x$ are given by 
$e=ze_g$ for $z\in F^\times$. It follows that \eqref{eq:lowsymex:II:1} can be reformulated as 
$$
dz^2\Nr(e_g)= a,\quad z\in F^\times,d\in F^\times
$$
and we are pushing forward $\psi(d+\tr_{E/F}(e)) = \psi(\frac{a}{z^2\Nr(e_g)}+z\tr_{E/F}(e_g))$ along $z\in F^\times$. The measure on $z\in F^\times$ used is the multiplicative Haar measure. Changing variable $z\mapsto z/\tr(e_g)$, and using the identity $\tr(e_g)^2/\Nr(e_g) = 2+\tr(g)$, we deduce the following fact:
\begin{num}
\item\label{num:lowSym:II:ngokernel} 
The following identity holds for any étale quadratic $F$-algebra $E$ and $(a,g)\in F^\times \times E^1$,
$$
\RJ^\rho_{\psi,E}(a,g) = 
\int^\pv_{z\in F^\times}
\psi
\bigg(
z+
\frac{a(2+\tr(g))}{z^2}
\bigg)\ud^\times z.
$$
The regularization is understood as follows: 
$$
\RJ^\rho_{\psi,E}(a,g) = \wt{\eta}_\psi(a(2+\tr(g)))
$$
where $\wt{\eta}_\psi$ is the smooth function on $F^\times$ whose regularized Mellin transform enjoys the following identity in the sense of \cite[\S 4.3]{JLZ}
$$
\int^\pv_{F^\times}
\chi^{-1}(x)
|x|^{-s}
\wt{\eta}_\psi(x)\ud^\times x = \gam(s+1,\chi,\psi)
\gam(s+1,\chi^2,\psi).
$$
Equivalently, $\wt{\eta}_\psi = \wt{\Phi}_{1,\psi}*\wt{\Phi}_{2,\psi}$ where $\wt{\Phi}_{1,\psi} = \psi(\cdot )$ and $\wt{\Phi}_{2,\psi} = \lam_!(\psi(\cdot))$ where $\lam:F^\times \to F^\times$ is given by $x\mapsto x^2$. As a result 
$$
\RJ^\rho_{\psi,\rm Ngo}(a,g) = 
\int^\pv_{F^\times}
\psi
\bigg(
z+
\frac{a(2+\tr(g))}{z^2}
\bigg)
\ud^\times z,\quad (a,g)\in G(F).
$$
\end{num}
Comparing with \eqref{eq:lowSym:II:correct}, we see that $\RJ^\rho_\psi\neq \RJ^\rho_{\psi,\rm Ngo}$.

Next, we are going to compute $\RJ^\rho_{\psi,\rm Laf}$ based on the above explicit formula for $\RJ^\rho_{\psi,\rm Ngo}$. Following Part (2) of Remark \ref{rmk:proposal}, instead of performing the Lafforgue transform on the trace variable of $G$, we would like to embed $G$ into the reductive monoid $M^{\lam_2}$ and perform the Lafforgue transform on the trace variable of $M^{\lam_2}$, which is suggested by Lafforgue in \cite{lafforguegl2}. Precisely, $M^{\lam_2} =\{(a,x)\in \BA\times\RM_{2\times 2}\mid a^2 = \det x\}$ and $G$ embeds into $M^{\lam_2}$ via 
$$
(a,g)\in G\mapsto (a,ag).
$$
The analogue of the Chevalley quotient map is given by 
\begin{align*}
\c: M^{\lam_2} &\to \BA^2\\
(a,x) &\mapsto (a,\tr(x))
\end{align*}
As a result, we have the following fact:
\begin{num}
\item With the above notation, for $(a,x)\in M^{\lam_2}(F)$, 
$$
\RJ^\rho_{\psi}(a,x) = 
\eta_\psi(2a+\tr(x)),\quad 
\RJ^{\rho}_{\psi,\rm Ngo}(a,x) = 
\wt{\eta}_\psi(2a+\tr(x)).
$$
After written the above kernel functions in $(a,t=\tr(x))$-variable,
$$
\RJ^\rho_{\psi}(a,t) = 
\eta_\psi(2a+t),\quad 
\RJ^{\rho}_{\psi,\rm Ngo}(a,t) = 
\wt{\eta}_\psi(2a+t).
$$
\end{num}
Now we calculate the Fourier inversion of $\RJ^\rho_{\psi}$ and $\RJ^\rho_{\psi,\rm Ngo}$ in trace variable, as distributions on the affine line. By definition,
\begin{align*}
\CF_{\psi^{-1}}(\RJ^\rho_{\psi,\rm Ngo})(a,\xi)=
\psi(2a\xi)
\CF_{\psi^{-1}}(\wt{\eta}_\psi)(\xi),\quad 
\CF_{\psi^{-1}}(\RJ^\rho_{\psi})(\xi)(a,\xi) = 
\psi(2a\xi)
\CF_{\psi^{-1}}(\eta_\psi)(\xi).
\end{align*}
Hence to show that $\RJ^\rho_{\psi,\rm Laf} = \RJ^\rho_\psi$, it suffices to establish the following identity:
\begin{equation}\label{eq:lowSym:II:identityLaftrue}
|\xi|\CF_{\psi^{-1}}(\eta_\psi)(\xi) = \CF_{\psi^{-1}}(\wt{\eta}_\psi)(\xi).
\end{equation}
We first calculate $\CF_{\psi^{-1}}(\wt{\eta}_\psi)$. Following \eqref{num:lowSym:II:ngokernel}, write $\wt{\eta}_\psi$ as the multiplicative convolution $\wt{\eta}_\psi = \wt{\Phi}_{1,\psi}*\wt{\Phi}_{2,\psi}$ with $\wt{\Phi}_{1,\psi} = \psi(\cdot)$ and $\wt{\Phi}_{2,\psi} = \lam_!(\psi(\cdot))$. 
In general, for a Schwartz-Bruhat function $f\in \CS(F)$, its Fourier transform can be written as the following multiplicative convolution
$$
\CF_{\psi}(f)(t) = \int_{y\in F}\psi(ty)f(y)\ud y = \psi*(|\cdot |f)^\vee(t).
$$
It follows that as distributions, 
$\wt{\eta}_\psi=\wt{\Phi}_{1,\psi}*\wt{\Phi}_{2,\psi} = \CF_{\psi}(|\cdot|^{-1}\wt{\Phi}_{2,\psi}^\vee)$. Hence 
\begin{align*}
\CF_{\psi^{-1}}(\wt{\eta}_\psi)(f) = 
\int_{t\in F}\wt{\eta}_\psi(t)\CF_{\psi^{-1}}(f)(t)
\ud t
=&
\int_{t\in F}
\CF_{\psi}(|\cdot|^{-1}\wt{\Phi}_{2,\psi}^\vee)(t)
\CF_{\psi^{-1}}(f)(t)\ud t,
\\
=&
\int_{t\in F}
\big(|\cdot|^{-1}\wt{\Phi}_{2,\psi}^\vee\big)(t)
f(t)\ud t.
\end{align*}
In other words, $\CF_{\psi^{-1}}(\wt{\eta}_\psi) = |\cdot|^{-1}\wt{\Phi}^\vee_{2,\psi}$. Similarly, $\CF_{\psi^{-1}}(\eta_\psi) = |\cdot|^{-1}\Phi^\vee_{2,\psi}$. Hence, to derive \eqref{eq:lowSym:II:identityLaftrue}, it is equivalent to show the following identity 
$$
|\xi|
\Phi^\vee_{2,\psi}(\xi) = \wt{\Phi}^\vee_{2,\psi}(\xi)
$$
which follows directly from the definition of $\Phi_{2,\psi}$ and $\wt{\Phi}_{2,\psi}$.

\begin{rmk}\label{rmk:exceptionaleg}
Unfortunately the remedy of Lafforgue still fails in general, which will be shown in the next example. Before preceding, let us explain why the remedy of Lafforgue works for $(G=\GL_2,\rho=\std)$ and $(G=\BG_m\times \SL_2,\rho=\Sym^2)$ intuitively. 

Precisely, for the above two situations, the distribution $\RJ^\rho_\psi$ enjoys the following two characterizations:
\begin{num}
\item\label{num:lowSymII:twopro1} The distributions $\CF_{\psi^{-1}}(\RJ^\rho_{\psi})$ and $\CF_{\psi^{-1}}(\RJ^\rho_{\psi,\rm Ngo})$ are always supported on $(F^\times)^2\subset F$;

\item\label{num:lowSymII:twopro2} From \cite[Thm.~4.7]{jlgl2}, for any étale quadratic $F$-algebra $E$, the following identity for the standard local Langlands gamma factors hold
$$
\gam(s,\Ind^{W_F}_{W_E}\chi,\psi) = \lam_{E}\cdot \gam(s,\chi,\psi_E)
$$
for any character $\chi$ of $T_E(F)$. Here $\Ind^{W_E}_{W_F}\chi$ is the automorphic induction from $F$ to $E$ of the character $\chi$, and $\lam_{E} = \varepsilon(\frac{1}{2},\eta_{E},\psi)$ is the Weil constant attached to $E/F$.

\end{num}
It turns out that based on Theorem \ref{thm:descent:sl2} that we are going to establish later, for any character $\chi$ of $T_E(F)$ with automorphic induction $\pi_{\chi}$ as a representation of $G(F)$, up to unramified twist and potential convergence issue the following identity holds 
$$
\gam(\pi_\chi,\rho,\psi) = 
\lam_{E}
\cdot 
\int_{T_E(F)}
\CF_{\psi}
\bigg(
\CF_{\psi^{-1}}(\RJ^\rho_\psi)
\frac{\eta_E(\cdot)}{|\cdot|}
\bigg)
\chi(e)\ud e
$$
When $\CF_{\psi^{-1}}(\RJ^\rho_\psi)$ is supported on $(F^\times)^2\subset F$, $\eta_E$ is equal to $1$ identically. Hence the right hand side above can be reduced to 
$$
\lam_{E}
\cdot 
\int_{T_E(F)}
\CF_{\psi}
\bigg(
\frac{\CF_{\psi^{-1}}(\RJ^\rho_\psi)}{|\cdot|}
\bigg)
\chi(e)\ud e.
$$
Furthermore, by the difference between the Langlands local gamma factors attached to $\pi_\chi$ and $\chi$, the above identity reduces to 
$$
\gam(\chi,\rho,\psi) = 
\int_{T_E(F)}
\CF_{\psi}
\bigg(
\frac{\CF_{\psi^{-1}}(\RJ^\rho_\psi)}{|\cdot|}
\bigg)
\chi(e)\ud e.
$$
which indicates that 
$$
\CF_{\psi}
\bigg(
\frac{\CF_{\psi^{-1}}(\RJ^\rho_\psi)}{|\cdot|}
\bigg) = \RJ^\rho_{\psi,\rm Ngo}\Longleftrightarrow 
\RJ^\rho_\psi = 
\CF_{\psi}(|\cdot|\CF_{\psi^{-1}}(\RJ^\rho_{\psi,\rm Ngo})).
$$

In conclusion, the fact that the remedy of Lafforgue works for the above two cases crucially relies on \eqref{num:lowSymII:twopro1} and \eqref{num:lowSymII:twopro2}. However, in general as shown in the next example, there do exist nonabelian Fourier kernels whose Fourier inversions in trace variable are not supported on $(F^\times)^2\subset F$, and the discrepancy between the Langlands local gamma factors on $T_E(F)$ and their automorphic induction on $G(F)$ differs more than just the constant $\lam_{E}$. Hence a more sophisticated inversion formula is needed. 
\end{rmk}

\subsubsection{$\RI\RI\RI$}

The third example we are going to discuss is $G=\GL_2\times_{\det} \GL_2 = \{(g_1,g_2)\mid g_1,g_2\in \GL_2,\det g_1=\det g_2\}$ and $\rho = \otimes$ is the tensor lifting. The integral representation for the corresponding $L$-functions can be obtained via a slight variant of the doubling method (\cite{doubling}). Following the same argument as \cite{JLZ}, one can show that the non-abelian Fourier kernel in this situation is given by 
$$
\RJ^\rho_{\psi}(g_1,g_2) = 
\wb{\eta}_\psi(\tr(g_1)+\tr(g_2))
$$
with $\wb{\eta}_\psi$ the smooth function on $F^\times$ whose regularized Mellin transform is given by 
$$
\int^\pv_{F^\times}
\chi^{-1}(x)|x|^{-s}
\wb{\eta}_\psi(x)\ud^*x = 
\gam(s+1,\chi,\psi)\gam(s+3,\chi,\psi).
$$
Formally, we can also write $\RJ^\rho_\psi$ as follows
$$
\RJ^\rho_{\psi}(g_1,g_2) = 
\int^\pv_{F^\times}\psi\bigg(z+\frac{\tr(g_1)+\tr(g_2)}{z}\bigg)
|z|^{-2}\ud^\times z.
$$
The analytical discussions for this example is the same as $\rho=\Sym^2$ case and hence we only briefly explain the ideas below without diving into the details. Following the same argument for $\rho=\Sym^2$, the Fourier inversion of $\RJ^\rho_\psi$ in $\tr(g_1)$ and $\tr(g_2)$ two variables is given by 
$$
\CF_{\psi^{-1}}
(\RJ^\rho_{\psi})(\xi_1,\xi_2) = 
\int_{\xi_1z=\xi_2z = 1}
\psi(z)
|z|^{-2}\ud^\times z
$$
which is \textbf{not} supported on the square class. As a result, $\RJ^\rho_\psi \neq \RJ^\rho_{\psi,\rm Laf}$ and hence Lafforgue's modification breaks down for this example. Precisely, suppose that Lafforgue's modification works for this case, then we would have 
$$
\mathrm{Laf}^\p(\RJ^\rho_{\psi,\rm Ngo}) = \RJ^\rho_{\psi}
$$
and hence 
$$
\RJ^\rho_{\psi,\rm Ngo}(g_1,g_2) = 
\int^\pv_{F^\times}
\psi\bigg(z+\frac{\tr(g_1)+\tr(g_2)}{z}\bigg)
\ud^\times z.
$$
However, by Theorem \ref{thm:descent:sl2} and \eqref{num:lowSymII:twopro2}, for $E_i$ $(i=1,2)$ two étale quadratic $F$-algebras, 
$$
\RJ^\rho_{\psi,E_1\times E_2}(e_1,e_2)=
\int^\pv_{F^\times}
\psi\bigg(z+\frac{\tr(e_1)+\tr(e_2)}{z}\bigg)\eta_{1}(z)\eta_{2}(z)
\ud^\times z
$$
where $\eta_i = \eta_{E_i}$, which is clearly not the same as the above formula.

\begin{rmk}\label{rmk:exceptionaleg:2}
Based on the above example, we see that neither the naive construction $\RJ^\rho_{\psi,\rm Ngo}$ nor its Lafforgue transform $\RJ^\rho_{\psi,\rm Laf}$ provides the correct nonabelian Fourier kernel in general. In this paper, we are going to construct the correct nonabelian Fourier kernel for $G=\SL_2$ or $\GL_2$ from the abelian pieces from tori $\nu_\alpha(\RJ^\rho_{\psi,\alpha})$ using the Langlands' stable transfer factor and the Lafforgue transform. \end{rmk}

\subsection{An explicit formula}\label{subsec:aformulaFK}

Let $F$ be a local field of residual characteristic not equal to two. In this subsection, following Part (2) of Remark \ref{rmk:proposal}, we are going to take $(G,\rho)$ to be the following cases
\begin{itemize}
\item $G=\BG_m\times \SL_2$, and $\rho = \rho_{2n}$ the tensor product of the scaling action of $\BC^\times$ and the unique irreducible representation of $\SO_3(\BC)$ of dimension $2n+1$;

\item $G=\GL_2$, and $\rho = \rho_{2n+1} = (\det)^{-n}\otimes \Sym^{2n+1}$. 
\end{itemize}
Our goal is to provide an explicit integral formula for $\RJ^\rho_{\psi}$ on $\Fc(F)$ such that 
\begin{equation}\label{eq:explicit-formula:1}
\pi(\RJ^\rho_\psi\circ \c) = \gam(1/2,\pi^\vee,\rho,\psi)\Id_\pi,\quad \pi\in \Irr(G(F)).
\end{equation}
In particular we take $n_\rho = 0$ in Conjecture \ref{conjec:bknproposal}.

Let $\CI$ be the following set parametrizing stable conjugacy classes of maximal tori $\{T_\alp\}_{\alp\in \CI}$ in $G(F)$ 
\begin{align}\label{eq:indexset-I}
\CI = 
\Bigg\{
\begin{matrix}
\{0\} & F=\BC \\
\{0,1\} & F=\BR \\
\{0,1,\pm 1/2\} & F \text{ non. archi.}
\end{matrix}
\end{align}
Here $T_0$ corresponds to the split quadratic $F$-algebra $E_0 = F\times F$, $T_1$ corresponds to the unramified quadratic extension (resp. complex extension) when $F$ is non-archimedean (resp. real), and $T_{\pm 1/2}$ correspond to the two ramified extensions of $F$. 

Now we are ready to state our explicit formula for $\RJ^\rho_\psi$.

\begin{thm}\label{thm:explicitformula-Jrho}
With the above notation, the following identities hold.
\begin{enumerate}
\item When $F=\BC$,
$$
\RJ^\rho_\psi = 
\Laf\big(\nu_0(\wt{\RJ}^{\rho}_{\psi,0})\big);
$$

\item When $F=\BR$, 
$$
\RJ^\rho_{\psi} = 
\Laf
\bigg(
\sum_{\alp\in \CI}
\CE_\alp
\big(
\nu_\alp(\wt{\RJ}^{\rho}_{\psi,\alp})
\big)
\bigg);
$$

\item
When $F$ is non-archimedean of odd residual characteristic, 
$$
\RJ^\rho_{\psi} = 
\Laf
\bigg(
\CE^{\perp,\triv}_0(\nu_0(\wt{\RJ}^\rho_{\psi,0}))
+
\sum_{\alp\in \CI\bs \{0\}}
\CE_\alp
\big(
\nu_\alp(\wt{\RJ}^{\rho}_{\psi,\alp})
\big)
\bigg).
$$
\end{enumerate}
\end{thm}

We are going to explain the theorem below.

(1) The transform $\Laf$, as discussed in subsection \ref{subsec:excepex}, is defined as follows: For a function $z$ on $\Fc(F)$, 
$$
\Laf(z) = \CF_{\psi}
\big(
|\cdot|\CF_{\psi^{-1}}(z)
\big)
$$
where the Fourier transform and absolute value are taken in trace variable. From the discussion in subsection \ref{subsec:Laf-Reptheoretic-def}, $\Laf$ realizes the natural map from (a subspace of) the stable cocenter of $G(F)$ to its stable Bernstein center;

(2)
$$
\wt{\RJ}^\rho_{\psi,\alp} = 
\RJ^\rho_{\psi,\alp}\cdot \kappa^\rho_\alp,
,\quad \kappa^\rho_\alp = 
\bigg\{
\begin{matrix}
\lam_{\alp}^{n}  & G=\BG_m\times \SL_2,\rho=\rho_{2n}\\
\lam_\alp^{n+1} & G=\GL_2,\rho=\rho_{2n+1}
\end{matrix}
$$
where $\lam_\alp = \lam_{E_\alp/F}$ is the Weil constant attached to $E_\alp/F$ (\cite{jlgl2}), and $\RJ^\rho_{\psi,\alp}$ is the smooth function on $T_\alp(F)$ studied in subsection \ref{subsec:toruscase} with the property that for any character $\chi_\alp$ of $T_\alp(F)$,
$$
\chi_\alp(\RJ^\rho_{\psi,\alp}) = 
\gam(1/2,\chi^{-1}_\alp,\rho|_{{}^L T_\alp},\psi).
$$

(3) In particular, 
$$
\chi_\alp(\wt{\RJ}^\rho_{\psi,\alp}) = 
\gam(1/2,\CW(\chi^{-1}_\alp),\rho,\psi)
$$
where $\CW(\chi_\alp)$ is the dihedral lifting of $\chi_\alp$ that is going to be reviewed in subsection \ref{subsec:dihedral rep}. 

To be precise, since the residual characteristic of $F$ is not equal to two, every irreducible admissible representation of $G(F)$ is a subquotient of some dihedral representation of $G(F)$ (\cite{jlgl2}\cite{Casselma-Quadratic}). For any character $\chi$ of a maximal torus $T(F)$ attached to a quadratic $F$-algebra $E$, let $\CW(\chi)$ be the dihedral lifting of $\chi$ to $G(F)$ that is going to be reviewed in subsection \ref{subsec:dihedral rep}. From the local Langlands correspondence for $\GL_2$ (\cite{gl2llc}) $\CW(\chi)$ has local $L$-parameter given by $\Ind^{W_F}_{W_{E}}\chi$. Moreover, $(\Ind^{W_F}_{W_{E}}\chi)^\vee\simeq \Ind^{W_F}_{W_{E}}\chi^{-1}$. 

Without loss of generality, let us for convenience assume that $G=\GL_2$ and $\rho=(\det)^{-n}\otimes \Sym^{2n+1}$. Since $\LG = \GL_2(\BC)$, up to direct sum and twist by determinant, by \cite[Thm.~1.0.2]{DanielZhilin}, $\rho\circ \Ind^{W_F}_{W_{E}}\chi^{-1}$ is a direct sum of one dimensional character or two dimensional dihedral representations attached to $E$ and the induced character depends only on $\chi$. The description is uniform over local fields of residual characteristic not equal to two. In particular, from \cite[Thm.~4.7]{jlgl2}, 
$$
\gam(s,\Ind^{W_F}_{W_E}\chi,\psi_F) = 
\gam(s,\chi,\psi_E)\cdot \lam_{E/F}
$$
where $\lam_{E/F} = \varepsilon(1/2,\eta_E,\psi)$ is the Weil constant. It follows from the decomposition of local $L$-parameters as in \cite[Thm.~1.0.2]{DanielZhilin} that there exists a constant $\kappa^{\rho}_{E/F}= \lam_{E/F}^{n}$ depending only on the quadratic extension $E/F$, $\rho$ and $\psi$ only, such that 
\begin{equation}\label{eq:nonabfourier:gl2:1}
\gam(1/2,\CW(\chi^{-1}),\rho,\psi) = 
\gam(1/2,\chi^{-1},\rho,\psi_E)\cdot 
\kappa^\rho_{E/F}.
\end{equation}

(4)
The map $\nu_\alp$ is introduced in \eqref{eq:pialp-nualp}, which is the descent of the $\tau_\alp$-invariant function $\wt{\RJ}^\rho_{\psi,\alp}$ to $\c(T_\alp(F))\subset \Fc(F)$. Indeed, $\RJ^\rho_{\psi,\alp}\circ \tau_\alp = \RJ^\rho_{\psi,\alp}$ which follows from the fact that $\tr_{\alp}\circ \tau_\alp = \tr_\alp$, and the fact that $\RJ^\rho_{\psi,\alp}$ is the push-forward of the standard Fourier kernel along $\rho^\vee_\alp$:
\begin{itemize}
\item If $G=\BG_m\times \SL_2$, 
\begin{align*}
\rho^\vee_\alp:
(E^\times_\alp)^{n}\times F^\times
&\to 
F^\times\times E^1_\alp
\\
\big((e_i)_{i=1}^n,a\big)&\mapsto 
\bigg(a\prod_{i=1}^{n}\Nr_\alp(e_i),
\prod_{i=1}^{n}\frac{e_i}{\tau_\alp(e_i)}
\bigg)
\end{align*}

\item If $G=\GL_2$, 
\begin{align*}
\rho^\vee_\alp:
(E_\alp^\times)^{n+1}
&\to E_\alp^\times
\\
\big(
(e_i)_{i=1}^{n+1}
\big)
&\mapsto 
\prod_{i=1}^{n+1}
e_i
\bigg(\frac{e_i}{\tau_\alp(e_i)}\bigg)^{i-1}
\end{align*}
\end{itemize}

(5) For $\alp\in \CI$, $\CE_\alp$ and $\CE_0^{\perp,\triv}$ are sections from functions on $\c(T_\alp(F))$ to the stable cocenter of $G(F)$ that are realized as the (elliptic part of the) Langlands stable transfer map between $G(F)$ and its maximal tori that is reviewed in subsection \ref{subsec:STandDescent}. In particular they can be realized as Gelfand-Graev transforms on $\Fc(F)$ as shown in Theorem \ref{thm:explicitsection} and Theorem \ref{thm:explicitsection:distinguishedtwo}, which are Fourier type transforms on $\Fc(F)$ that is going to be reviewed in subsection \ref{subsec:ggpsstable:gl2}.

(6) The convergence issue is addressed as follows: Up to multiplying a smooth and compactly supported function along the determinant, we may assume that the distribution $\RJ^\rho_\psi\circ \c$ that we want to construct actually lies in the stable Bernstein center and its archimedean variant, which is going to be reviewed in subsection \ref{subsec:mainresult:SBC}. On the other hand, for the purpose of local $L$-factors, it suffices to examine the action of $\RJ^\rho_{\psi}\circ \c$ on $K$-finite vectors for smooth and compactly supported (or Schwartz if $F$ is archimedean) functions on $G(F)$. In particular, we may reduce the construction of $\RJ^\rho_{\psi}$ which lie in the stable Bernstein center of $G(F)$, to distributions in the stable Bernstein center that are actually supported only on finitely many Bernstein components. Hence we are reduced to assume that $\{\RJ^{\rho}_{\psi,\alp}\}_{\alp\in \CI}$ are all supported on finitely many components of the admissible dual of $\{T_\alp(F)\}_{\alp\in \CI}$ where they now become test (Schwartz if $F$ is archimedean) functions on $T_\alp(F)$. Now for this situation the convergence are properly addressed in section \ref{sec:descent}, Section \ref{sec:summands} and Section \ref{sec:inv}.

% !TEX root = luo-ngo.tex

\section{Bernsten center, cocenter and orbital integrals} \label{sec:orbital-cocenter}

We recall the normalization of \cite{FLN10} of stable orbital integrals as a function over the Steinberg-Hitchin base. We will connect this construction to Kazhdan's theory of the cocenter and its stable version.

\subsection{Normalization of measures} \label{subsec:measure}

In this subsection, over a local field $F$, we define a measure on the space of conjugacy classes of $G(F)$ of $F$-points of a semisimple simply connected group $G$, essentially following \cite{FLN10}. We define first a measure on the space of stable conjugacy classes using a canonical volume form of the Steinberg-Hitchin base then derive a measure on the space of conjugacy classes. The discussion can be generalized to general reductive groups following the treatment of \cite[\S 3.3]{FLN10}.

Following Steinberg (\cite{Steinberg-conjugacy}), if $G$ is a split semisimple simply connected group, the invariant theoretic quotient of $G$ acting on itself by conjugation is the $r$-dimensional affine space
\begin{equation} \label{eq:Steinberg-base}
	\Fc=\mathfrak{c}_G = G\sslash G= \Spec (F[G]^G)=\mathbb{A}^r
\end{equation}
 with coordinates $c_1,\ldots,c_r$ given by $c_i(g)=\tr(\rho_{\omega_i}(g))$ where $\rho_{\omega_1},\ldots,\rho_{\omega_r}$ are the fundamental representations of $G$. Following \cite{FLN10}, we call the invariant theoretic quotient $\Fc=\Fc_G$ the Steinberg-Hitchin base because in \emph{loc. cit.} the normalization of stable orbital integrals was inspired by the Hitchin fibration.  For  $G=\SL_2$, the invariant theoretic quotient $\mathfrak{c}=\mathbb{A}^1$ is the affine line with the coordinate function $\c(g)=\tr(g)$, the trace of the $2\times 2$-matrix $g$ with determinant one. 

By the Chevalley restriction theorem, we have an isomorphism 
$$
\Fc = T \sslash W=\Spec (F[T]^W)
$$ 
where $T$ is a maximal split torus and $W$ is the Weyl group. We define the Weyl discriminant $D\in F[\Fc]$ as the $W$-invariant function on $T$ given by 
\begin{equation}
	D(t)=\prod_{\alpha\in \Phi} (1-\alpha(t))
\end{equation}
where $\Phi = \Phi(G,T)$ is the set of roots. In the case $G=\SL_2$, $T$ is torus of diagonal matrices of entries $(t,t^{-1})$ with $t\in \mathbb{G}_m$. The Weyl discriminant is then 
$$
D(t)=(1-t^2) (1-t^{-2}).
$$ 
The absolute value $|D(t)|$ of the Weyl discriminant $D(t)$ admits a square root constructed as follows. Choose a $\BZ$-basis $\chi_1,\ldots,\chi_r$ of the group $\mathbb{X}^*(T)$ of rational characters of $T$. We then have an invariant $r$-form on $T$ given by 
\begin{equation} \label{top-form-torus}
	\frac{\ud \chi_1}{\chi_1} \wedge \cdots \wedge \frac{\ud \chi_r}{\chi_r}
\end{equation}
which, up to a sign, is independent of the choice of the $\BZ$-basis $\chi_1,\ldots,\chi_r$ of  $\mathbb{X}^*(T)$. By pulling back the $r$-form $\ud c_1 \wedge \cdots \wedge \ud c_r$ along the morphism  $\pi_T:T\to \mathfrak{c}=T \sslash W$, we have the formula
\begin{equation}\label{def:Delta}
	\pi_T^* \ud c_1 \wedge \cdots \wedge \ud c_r = \Delta_T(t) \frac{\ud \chi_1}{\chi_1} \wedge \cdots \wedge \frac{\ud \chi_r}{\chi_r}
\end{equation}
where $\Delta_T(t)$ is an algebraic function on $T$ well defined up to a sign. Its absolute value 
$$\Delta(t)=|\Delta_T(t)|\in \BR_+$$
is then a real-valued function invariant with respect to the action of the Weyl group $W$. 
 By \cite[Prop.~3.29]{FLN10} we have the formula
\begin{equation} \label{square-Delta}
	|D(t)|=\Delta(t)^2.
\end{equation}
For $G=\SL_2$, we have $c=t+t^{-1}$, and hence
$$\ud c= (t-t^{-1}) \frac{\ud t}{t}.$$
We have $\Delta(t)=|t-t^{-1}|$ and $\Delta(t)^2=|D(t)|$.

The $r$-form $\ud c_1 \wedge \cdots \wedge \ud c_r$ on the $r$-dimensional affine space $\mathfrak{c}$ determines a measure 
\begin{equation}\label{def:omega-c}
	\omega_\Fc=|\ud c_1 \wedge \cdots \wedge \ud c_r|
\end{equation}
on $\mathfrak{c}(F)=F^r$. It follows from the formula \eqref{def:Delta} that $\Delta^{-1} \pi_T^* \omega_\Fc$ is an invariant measure on $T(F)$ which coincides with the measure associated with the invariant $r$-form \eqref{top-form-torus}. If $T$ is now a maximal torus of $G$ but not necessarily split, we still have the finite map $\pi_T:T \to \mathfrak{c}$ and the formula $\Delta^{-1} \pi_T^* \omega_\Fc$ still defines an invariant measure on $T(F)$. We will denote the corresponding measure by
\begin{equation}\label{invariant-measure-torus}
	\omega_T= \Delta^{-1} \pi_T^* \omega_\Fc.
\end{equation} 
Following \eqref{eq:indexset-I}, let 
\begin{align}\label{indexset:parametrizequadratic}
\CI = 
\Bigg\{
\begin{matrix}
\{0\} & F=\BC \\
\{0,1\} & F=\BR \\
\{0,1,\pm1/2\} & F\text{ non-archi. of odd residual characteristic}
\end{matrix}
\end{align}
parametrizing quadratic $F$-algebras and the stable conjugacy of maximal tori in $G(F)=\SL_2(F)$. We have 
$$
T_\alp(F)=E_\alpha^{1}= \ker(\Nr_\alp:E_\alpha^\times \to F^\times)
$$
which is compact for $\alp\neq 0$. For $\alp\neq 0$ indexing non-split torus of $\SL_2$ we have the constants
\begin{equation}\label{eq:volume:def:sl2}
	\mathbf{v}_\alpha=\vol(E_\alpha^1, \omega_{T_\alpha})^{-1}
\end{equation}
where $\ud t_\alpha=\omega_{T_\alpha}$ is the invariant measure on $T_\alpha(F)$ defined by the formula \eqref{invariant-measure-torus}. These constants will play an important role in our discussion of the Weyl integration formula.

We denote by $\Fc'$ the complement of the divisor defined by the Weyl discriminant function $D(t)$. Its inverse image $G'$ in $G$ consists of regular semisimple elements of $G$. Following Kottwitz (\cite{kott2}), the map 
\begin{equation} \label{eq:Steinberg-map}
	\c:G'(F)\to \Fc'(F)
\end{equation}
is surjective, and its fibers are stable conjugacy classes so that one can identify $\Fc'(F)$ with the set  of stable regular semisimple conjugacy classes. As a result, the set of stable regular semisimple conjugacy classes can thus be equipped with a canonical structure of $F$-analytic variety induced from $\Fc(F)$, and the measure induced from $\omega_\Fc$. 

Let $C'(F)$ denote the set of regular semisimple conjugacy classes in $G'(F)$. The fibers of the quotient map 
\begin{equation} \label{def:d}
	\mathbf{d}:G'(F) \to C'(F)
\end{equation}
are regular semisimple conjugacy classes. We consider the map 
$$\mathbf{e}:C'(F)\to \Fc'(F)$$
from the set of regular semisimple conjugacy classes of $G'(F)$ into the set of stable regular semisimple conjugacy classes. This is a finite-to-one map since in every stable regular semisimple conjugacy class, there are only finitely many conjugacy classes. This is also a covering in the sense that for every $x\in C'(F)$, there exists an open neighborhood $U$ of $x$, and an open neighborhood $V$ of $\mathbf{e}(x)$ such that $\mathbf{e}$ induces a homeomorphism $U\to V$. The canonical measure on $\omega_{\Fc}$ on $\Fc'(F)$ gives rise to a canonical measure $\omega_C$ on $C'(F)$. We also note that the map \eqref{def:d} is an $F$-analytic submersion. Since $\Fc'(F)$ is equipped with a structure of $F$-analytic manifold, its covering $C'(F)$ is also an $F$-analytic manifold. The measure $\omega_\Fc$ on $\Fc(F)$ pulls back to a measure on $C'(F)$. This extends to a measure on the set  $C(F)$ of conjugacy classes of $G(F)$ by declaring that the complement of $C^{\p}(F)$ in $C(F)$ has measure zero.

Finally the above results can be generalized to general reductive groups over $F$ without hard based on the discussion in \cite[\S 3.3]{FLN10}. We omit the details and refer the readers to \emph{loc. cit.} for a detailed discussion.

\subsection{Normalization of stable orbital integrals and orbital integrals}\label{subsec:normalization-SOI-OI}

Let $\mathcal{H}$ denote the algebra of smooth and compactly supported functions on $G(F)$, where $G$ is a reductive group over $F$. When $F$ is archimedean, it is also convenient to work with the Schwartz algebra of rapidly decreasing functions $\CS$ on $G(F)$ (\cite[p.230]{wallachredgpI}). In the following, we will focus on functions in $\CH$, and the discussion applies to functions in $\CS$ when $F$ is archimedean. 

For every $h\in \mathcal{H}$, we consider the pushforward measure $\c_! (h \ud g)$ defined so that for every continuous function $f:\Fc(F)\to \BC$ we have the following identity, here $\ud g$ is a Haar measure on $G(F)$
$$\langle \c_! (h\ud g), f \rangle = \langle h\ud g, \c^* f \rangle .$$
Here $\langle \cdot ,\cdot \rangle$ is the standard pairing between compactly supported measures and continuous functions and $\c^*f = f\circ \c$ is the pull-back of $f$ along $\c$. Based on the work of Harish-Chandra, we know that there exists a function ${\rm SO}(h)$ locally integrable with respect to $\omega_\Fc$ such that
$$
\c_! (h\ud g)= {\rm SO}(h) \omega_\Fc.
$$ 

When restricted to the (strongly) regular semisimple locus, 
$$
{\rm SO}(h):\Fc'(F) \to \BC
$$
is a smooth function that we call the stable orbital integral of $h$. This is the normalization of stable orbital integrals adopted in \cite[(3.31)]{FLN10}.

We can apply the same construction for the map $G'(F)\to C'(F)$ and thus define a function $${\rm O}(h): C'(F) \to \BC,$$ 
such that 
$${\mathbf d}_!(h\ud g)={\rm O}(h) \omega_C,$$
and we call it the orbital integral of $h$. This is equivalent to requiring that for every continuous function $f:C(F)\to \BC$ we have
$$\langle \mathbf{d}_! (h\ud g), f \rangle = \langle h\ud g, \mathbf{d}^* f \rangle .$$

We now compare the above normalization of orbital and stable orbital integrals with the more usual definition. In the customary setting, we define the orbital integral for an element $\gamma\in G'(F)$ with the formula
\begin{equation}\label{O gamma h}	
{\rm O}_\gamma(h)=\int_{G(F)/G_\gamma(F)} h(g\gamma g^{-1}) \frac{\ud g}{\ud t} 
\end{equation}
where $\ud t$ is a Haar measure of the of the group of $F$-points $T(F)$ of the torus $T=G_\gamma$. In general, this definition requires the assumption that the centralizer of $\gamma$ is a maximal torus $T=G_\gamma$, that is $\gamma$ is strongly regular, but under the running assumption that $G^{\rm der}$ is simply connected, all regular semisimple elements are strongly regular. We will denote $G^{\rm sr}$ the Zariski open subset of $G$ consisting of strongly regular elements and $T^{\rm rs}=G^{\rm rs}\cap T$.

This definition is more suitable if we assume that $\gamma$ varies in a fixed maximal torus $T(F)$. Fix a maximal torus $T$ and an invariant measure $\ud t$ on $T(F)$, we then have a function 
\begin{equation}\label{O T h}
	{\rm O}_T(h): T^{\rm sr}(F)\to \BC
\end{equation}
such that for every $\gamma\in T^{\rm sr}(F)$, the value ${\rm O}_T(h)(\gamma)={\rm O}_\gamma(h)$ is supplied by the formula \eqref{O gamma h}.

As the maximal torus $T$ of $G$ varies, we must consistently choose the Haar measure $\ud t$ on $T(F)$. One way to choose the Haar measure $\ud t$ on $T(F)$ is to use the formula \eqref{invariant-measure-torus}. With the Haar measures on maximal tori being chosen as in \eqref{invariant-measure-torus}, we can compare the FLN-normalization of orbital integrals with the traditional normalization with the aid of the Weyl integration formula. 

For every maximal $F$-torus $T$ of $G$, we consider the $F$-analytic map 
$$G(F)/T(F) \times T(F) \to G(F)$$
given by $(gT(F),t)\mapsto gtg^{-1}$. Its image is the subset of $G(F)$ consisting of elements conjugate to an element of $T(F)$. If $T'(F)$ is the intersection $T(F)\cap G^{\rm sr}(F)$ then the image of 
$$q_T:G(F)/T(F) \times T^{\rm sr}(F) \to G^{\rm sr}(F)$$ 
is an open subset of $G'(F)$. The latter map is locally an isomorphism on the source. It induces a cover of its image with $|{N}_T(F)/T(F)|$ sheets where ${N}_T$ is the normalizer of $T$ in $G$. 

Let $\mathcal{T}$ denote the set of $G(F)$-conjugacy classes of maximal tori of $G(F)$.
We then have a map
\begin{equation} \label{q alpha}
	q=\bigsqcup_{\alpha \in \mathcal{T}} q_\alpha: \bigsqcup_{\alpha \in \mathcal{T}} G(F)/T_\alpha(F) \times T_\alpha(F) \to G(F)
\end{equation}
whose image is the subset of semisimple elements in $G(F)$. If two maximal tori $T_1$ and $T_2$ which are not $G(F)$-conjugate, the images of $G(F)/T_1(F) \times T_1'(F)$ and $G(F)/T_2(F) \times T_2'(F)$ in $G'(F)$ are disjoint open subsets. It follows that the map 
$$ q:\bigsqcup_{\alpha \in \mathcal{T}} G(F)/T(F) \times T_\alpha^{\rm sr}(F) \to G^{\rm sr}(F).$$
defines a partition into open and closed subsets 
\begin{equation} \label{image q alpha}
	G^{\rm sr}(F)=\bigsqcup_{\alpha\in \mathcal{T}}G_\alpha^{\rm sr}(F)
\end{equation}
and the map 
$$G(F)/T_\alpha(F) \times T_\alpha^{\rm sr}(F) \to G^{\rm sr}_\alpha (F)$$
is a covering with $N_{T_\alpha}(F)/T_\alpha(F)$ sheets. 

\begin{thm} For every $h\in\mathcal{H}$ the following identity holds
	\begin{equation} \label{Weyl integration}
		\int_{G(F)} h(g)\ud g = \sum_{\alpha\in \mathcal{T}} \frac{1}{|N_\alpha(F)/T_\alpha(F)|} \int_{t\in T_\alpha(F)}  |D(t)| {\rm O}_{T_\alpha}(h)(t) \ud t 
	\end{equation}
\end{thm}
When $F$ is archimedean the theorem also holds for $h\in \CS$.
\begin{proof} Following \eqref{q alpha} and \eqref{image q alpha}, it is enough to prove that the Jacobian of the morphism $q_\alpha:G(F)/T_\alpha(F) \times T_\alpha(F) \to G(F)$ at a point $(g,t)\in G(F)/T_\alpha(F) \times T_\alpha^{\rm sr}(F)$ is equal to $D(t)$. Because $q_\alpha$ is $G$-equivariant, we can assume that $g$ is the identity element of $G$. The induced map on the tangent space at $(1,t)\in G/T \times T$ is now 
$$\ud q_\alpha: \mathfrak{g}/\mathfrak{t} \times \mathfrak{t} \to \mathfrak{g}$$
given by $$(x,y)\mapsto ({\rm ad}(t)-1)(x)+y.$$ The Jacobian determinant can now be identified with the Weyl discriminant $|D(t)|$. 	
\end{proof}

For every $\alpha\in \mathcal{T}$, the map $p_\alpha:T_\alpha^{\rm sr}(F) \to C^{\rm sr}(F)$ is a covering on its image $p_\alpha(T_\alpha^{\rm sr}(F))$ with the deck group $N_{T_\alpha}(F)/T_\alpha(F)$. We also remark that the function ${\rm O}_{T_\alpha}(h):T_\alpha^{\rm sr}(F) \to \BC$ descends to a function ${\rm O}^\alpha(h):p_\alpha(T_\alpha^{\rm sr}(F))\to \BC$ so that for every $t\in T_\alpha(F)$ we have
$$ {\rm O}_{T_\alpha}(h) (t)= {\rm O}^{\alpha}(h) (p_\alpha(t)).$$
Because the sets $p_\alpha(T_\alpha^{\rm sr}(F))$ form a partition of $C^{\rm sr}(F)$, we have a unique function 
$${\rm O}^{\mathcal{T}}(h): C^{\rm sr}(F)\to \C$$
by gluing the functions ${\rm O}^{\alpha}(h)$ so that we have the identity
$$ {\rm O}_{T_\alpha}(h) (t)= {\rm O}^{\mathcal T}(h) (p_\alpha(t))$$
for all $\alpha\in\mathcal T$ and $t\in T_\alpha^{\rm sr}(F)$. 
Because of the Jacobian relation $\Delta(t) \ud t= \ud a$ with $|D(t)|=\Delta(t)^2$, we have 
\begin{equation} \label{descent to C}
	\sum_{\alpha\in \mathcal{T}} \frac{1}{|N_{T_\alpha}(F)/T_\alpha(F)|} \int_{t\in T_\alpha(F)} |D(t)| {\rm O}_{T_\alpha}(h)(t) \ud t = \int_{C^{\rm sr}(F)} \Delta(a) {\rm O}^{\mathcal T}(h)(a) \ud a.
\end{equation}

\begin{corollary}\label{cor:normalizationoforb}
We have the identities ${\rm O}(h)(a)= \Delta(a) {\rm O}^{\mathcal T}(h)(a)$.
\end{corollary}

\begin{proof}		
It is enough to prove that for every continuous function $f:C(F) \to \C$ we have the equality
$$\int_{C(F)}  {\rm O}(h) (a) f(a) \ud a= \int_{C(F)}  |\Delta(a)| {\rm O}^{\mathcal T}(h)(a) f(a) \ud a.$$
But this equality can be derived from the Weyl integration formula \eqref{Weyl integration} and \eqref{descent to C}.
\end{proof}

\subsection{Cocenter and orbital integrals}\label{subsec:Cocenter-OI}

Following Harish-Chandra (\cite{HC70}\cite{varrealHC}), the function 
$$
{\rm O}(h):C(F)^{\rm sr}\to \BC
$$
is bounded for every $h\in \CH$. We thus have a map $\CH \to L^\infty (C(F))$ given by $h\mapsto {\rm O}(h)$ from the algebra of smooth and compactly supported functions $\CH$ to the space of bounded functions on $C(F)$ defined almost everywhere. When $F$ is non-archimedean, by Kazhdan (\cite{kazhdan_representations_1986}), the kernel of this map is the commutator $[\CH,\CH]$ of $\CH$. This induces an isomorphism between the cocenter $\CH/[\CH,\CH]$ with the image ${\rm O}(\CH)$ of $\CH$ in $L^\infty (C(F))$. In general we let $\RO(\CH)$ be the image of $\CH$. 
Similarly, the map $h\mapsto {\rm SO}(h)$ defines a map $\CH\to L^\infty(\Fc(F))$ from the Hecke algebra to the space of defined almost everywhere bounded functions on the Steinberg-Hitchin base whose image will be denoted ${\rm SO}(\CH)$. We will call ${\rm SO}(\CH)$ the stable cocenter of $\CH$. The above discussion applies to $F$ archimedean and functions $h\in \CS$. In particular, for $F$ archimedean, one can also introduce $\RO(\CS)$ and $\RS\RO(\CS)$ for functions in $\CS$. 

For $\pi$ a smooth irreducible representation of $G(F)$, by definition, its distribution character $\Theta_\pi$ is the generalized function on $G(F)$ such that
$$
\tr(\pi(h))=\int_{G(F)} h(g) \Theta_\pi(g),\quad h\in \CH.
$$
The trace distribution extends to $\CS$ when $F$ is archimedean (\cite[\S 8.1]{wallachredgpI}).

From the work of Harish-Chandra, it is known that $\Theta_\pi$ is an invariant locally integrable function, which is smooth on the regular semisimple locus. We abuse the notation and write $\Theta_\pi$ the corresponding invariant function on $G(F)$.

We denote $\Theta_\pi^C:C^{\rm sr}(F)\to \BC$ the induced function on the space $C^{\rm sr}(F)$ of strongly regular conjugacy classes of $G(F)$. The trace of $h\in \CH$ can now written as an integral over $C(F)$:
\begin{equation}
	\tr_\pi(h)=\int_{G(F)} h(g)\Theta_\pi(g) \ud g = \int_{C(F)} {\rm O}(h)(a) \Theta_\pi^C(a) \ome_C= \langle {\rm O}(h), \Theta_\pi^C \rangle .
\end{equation}
where $\ome_C$ is the measure on $C(F)$ introduced in last subsection, and $\langle \cdot,\cdot \rangle$ is the standard pairing between functions on $C(F)$ with respect to the measure $\ome_C$.

Similarly, for every $L$-packet of representations $[\pi]$ of $G(F)$ with attached stable distribution character $\Theta^\st_{[\pi]}$ (\cite{Langlands-StableConjugacy}\cite{ar13}), $\Theta^{\st}_{[\pi]}$ descends to a function $\Theta^{\Fc}_{[\pi]}$ on $\Fc(F)$ and we have 
\begin{equation}
	\tr^{\rm st}_{[\pi]}(h)=\int_{G(F)} h(g)\Theta_{[\pi]}^{\rm st}(g) \ud g= \int_{\Fc(F)} {\rm SO}(h)(c) \Theta_{[\pi]}^{\Fc}(c) \ome_\Fc=\langle {\rm SO}(h), \Theta_{[\pi]}^{\Fc} \rangle.
\end{equation}

\subsection{Bernstein center and Kazhdan cocenter}\label{subsec:BCandKC}

Let $G$ be a reductive group over a non-archimedean local field $F$. Following \cite{BD84}, the Bernstein center $\CZ=\CZ_G=\CZ(G(F))$ has four equivalent definitions:
\begin{itemize}
	\item[-] It is the ring of endomorphism of the identity functor of the category of smooth representations of $G(F)$;
	\item[-] It is the projective limit of the centers of Hecke algebras of finite levels $\CH_K=e_K\CH e_K$ where $e_K$ is the characteristic measures of the compact open subgroups $K$, as $K$ get smaller and smaller;
	\item[-] It is the ring of invariant distributions $J$ on $G(F)$ which are essentially of compact support, that is for every compact open subgroup $K$, we have $e_KJ=Je_K\in \CH$;
	\item[-] Finally, it is the ring of regular functions on the algebraic variety $\Omega=\Ome_G=\Omega(G(F))$, the Bernstein variety whose points are irreducible representations of $G(F)$ up to inertial equivalence.
\end{itemize} 

The Bernstein variety has infinitely many connected components. We denote $\CZ^{\rm fin}$ the nonunital 
subalgebra of $\CZ$ whose elements are regular functions on the Bernstein variety which are supported on the union of finitely many connected components. For every finite subset $I$ of the set $\pi_0(\Omega)$ of connected components of $\Omega$, we will denote $\epsilon_I\in \CZ$ the regular function on $\Omega$ supported on $\bigcup_{i\in I}\Omega_i$ and taking the value one at those components. We abuse the notation and still write $\eps_I$ for the corresponding distribution on $G(F)$. For every $I \subset \pi_0(\CZ)$, $\epsilon_I$ is a finite idempotent of $\CZ$, i.e. supported on a finite number of components of $\Omega$. We denote $\epsilon^I=1-\epsilon_I$, which is a cofinite idempotent. If $I\subset J$ are two finite subsets of $\pi_0(\Omega)$, we have $\epsilon_I=\epsilon_I \epsilon_J$ and therefore a homomorphism of unital algebras $\CZ \epsilon_J \to \CZ \epsilon_I$ given by $f\mapsto f\epsilon_I$ and a homomorphism of nonunital algebras $\CZ \epsilon_I \to \CZ \epsilon_J$. In other words, $\CZ \epsilon_I$ and $\CZ \epsilon_J$ are both unital algebras but only the surjection  $\CZ \epsilon_J \to \CZ \epsilon_I$  which preserves units, while the injection $\CZ \epsilon_I \to \CZ \epsilon_J$ does not preserve units. The Bernstein center $\CZ$ can be presented as the projective limit of $\CZ \epsilon_I$, which is then an unital commutative algebra, and its finite part $\CZ^{\rm fin}$ can be presented as the injective limit of $\CZ \epsilon_I$, which is then a nonunital algebra. For every compact open subgroup $K$ of $G(F)$, the quotient $\CZ_K$ can be decomposed as a product 
$$\CZ_K=\CZ_K \epsilon_I \times \CZ_K \epsilon ^I.$$
For a fixed compact open subgroup $K$ we have $\CZ_K=\CZ_K \epsilon_I$ for all finite subsets $I$ of $\pi_0(\Omega)$ large enough. In this case, we say that $I$ dominates $K$ and we have $e_K \epsilon_I=e_K$. On the other hand, for every finite subset $I$ of $\pi_0(\Omega)$, for all compact open subgroup $K$ small enough, we have $e_K \epsilon_I=\epsilon_I$, and as a result $\CZ_K \epsilon_I=\CZ_I$. In this case, we say that the open compact subgroup $K$ dominates the index set $I$. 

By construction, the Bernstein center $\CZ$ acts on the Hecke algebra $\CH$. Because the action of $\CZ$ preserves the commutator $[\CH,\CH]$, it acts on the cocenter $\RO(\CH) = \CH/[\CH,\CH]$. For every compact open subgroup $K$, $\CZ$ acts on $\CH_K=e_K \CH e_K$ through the quotient $\CZ e_K$ and the commutator $[\CH_K,\CH_K]$ is a $\CZ_K$-submodule of $\CH_K$. As a result, the quotient $\CH_K/[\CH_K,\CH_K]$ is a $\CZ_K$-module. If $K'$ is a compact open subgroup contained in $K$, then we have inclusions $\CH_K \subset \CH_{K'}$ and $[\CH_K,\CH_K]\subset [\CH_{K'},\CH_{K'}]$ but the inclusion 
$$[\CH_K,\CH_K] \subset [\CH_{K'},\CH_{K'}] \cap \CH_K$$
may be strict. In other words, the induced maps on the quotients 
$$\CH_K/  [\CH_K,\CH_K] \to \CH_{K'}/[\CH_{K'},\CH_{K'}] $$
may or may not be injective. Nevertheless, the inductive limit of $\CH_K/  [\CH_K,\CH_K]$ as the compact open subgroup $K$ gets smaller is $\CH/[\CH,\CH]$ and it is equipped with a canonical structure of $\CZ$-module. 

We note that we don't have a natural map $\CH_{K'}\to \CH_{K}$ let alone a natural map between the quotients $\CH_K/  [\CH_K,\CH_K]$
and $\CH_{K'}/[\CH_{K'},\CH_{K'}]$. However, instead of taking a projective limit on $K$, we can take a projective limit on the set of finite subsets $I$ of $\pi_0(\Omega)$. For every such a finite subset $I$, we denote $\CH_I=\CH \epsilon_I=\epsilon_I \CH$, which is a $\CZ_I$-module.  The commutator $[\CH_I,\CH_I]$ is a $\CZ_I$-submodule, and as a result, the quotient $\CH_I/[\CH_I,\CH_I]$ is a $\CZ_I$-module. If $I'$ is a finite subset of $\pi_0(\Omega)$ containing $I$ then $\epsilon_{I'} \epsilon_I=\epsilon_I$, and as a result, we have the inclusion $\CH_{I}\to \CH_{I'}$ as well as a surjective $\CZ$-linear map $\CH_{I'} \to \CH_I$ given by $f\mapsto \eps_I f$. With respect to the inclusion, we have $[\CH_{I'},\CH_{I'}] \cap \CH_I=[\CH_I,\CH_I]$ (\cite[Lem.~3.2]{kazhdan_representations_1986}), and the surjection $f\mapsto \eps_I f$ sends $[\CH_{I'},\CH_{I'}]$ onto $[\CH_I,\CH_I]$. We derive an injective map $\CH_I/[\CH_I,\CH_I] \to \CH_{I'}/[\CH_{I'},\CH_{I'}]$ from the inclusion, and the surjective map $\CH_{I'}/[\CH_{I'},\CH_{I'}]\to \CH_I/[\CH_I,\CH_I]$ from the surjection. The inductive limit of $\CH_I/[\CH_I,\CH_I]$ as the finite $I$ gets bigger is again the cocenter $\CH/[\CH,\CH]$. The projective limit of $\CH_I/[\CH_I,\CH_I]$ as the finite $I$ gets bigger is what we will call the {\em Kazhdan cocenter} and denote by $\CC=\CC_G=\CC(G(F))$, which is a completion of the cocenter $\CH/[\CH,\CH]$. Indeed, we have
$$\CH/[\CH,\CH] = \CC^{\rm fin}$$
where $\CC^{\rm fin}$ are the $\CZ$-submodule consisting of elements $f\in \CC$ such that $\epsilon_I f=f$ for some finite set $I$ of $\pi_0(\Omega)$. 

It is often enlightening to think of the cocenter $\CH/[\CH,\CH] = \CC^{\rm fin}$ and its completion $\CC$ in terms of the Lafforgue variety, constructed in \cite{psaromiligkos_lafforgue_2023}. In \cite{psaromiligkos_lafforgue_2023}, the set ${\rm Irr}(G(F))$ of irreducible representations of $G(F)$ is given a canonical structure of algebraic variety which has a finite-to-one morphism to the Bernstein variety $\Ome$. Following the trace Paley-Wiener theorem of Bernstein-Deligne-Kazhdan \cite{Bernstein-Deligne-Kazhdan}, a linear form $\ell$ on the vector space $R(G(F))$ with basis ${\rm Irr}(G(F))$ is of the form 
$$\ell(\pi)=\langle {\rm O}(h),\Theta_\pi^C \rangle$$
for some $h\in \CH$ if and only if the function $\pi \mapsto \ell(\pi)$ is an algebraic function on the Lafforgue variety and it is supported on finitely many components of the Lafforgue variety. From this angle, we can reinterpret the Kazhdan cocenter $\CC$ as the space of linear form $\ell:R(G(F))\to \C$ such that the function $\pi\mapsto \ell(\pi)$ defines an algebraic function on the Lafforgue variety, which may or may not be supported on finitely many components of the Lafforgue variety. 

\subsubsection{Archimedean variant}

Next we briefly discuss the archimedean analogue of the Bernstein center. Let $F$ be archimedean. We will work with the category of Casselman-Wallach representations of $G(F)$, which are admissible smooth Fréchet representations of moderate growth. By \cite[Prop.~2.20]{bkglobalization}, the category of smooth Fréchet representations of moderate growth is equivalent to the category of continuous non-degenerate $\CS$-modules, where we recall that $\CS$ is the Schwartz algebra of rapidly decreasing functions on $G(F)$. 

A natural candidate for the archimedean analogue of the Bernstein center is the algebra of multipliers of $\CS$, i.e. the ring of $\CS\times \CS$-equivariant continuous endomorphisms of $\CS$ which is defined as $\End_{\CS\times \CS}(\CS)$. Similarly, the Kazhdan cocenter can also be defined as the projective limit over finite subsets of the set of $K$-types of $G(F)$. Precisely, let $\hat{K}$ be the set of equivalence classes of irreducible finite dimensional representations of a fixed maximal compact subgroup of $G(F)$. For any $\sig\in \hat{K}$, let $e_\sig$ be the idempotent attached to $\sig$. For any finite subset $I$ of $\hat{K}$. Let $\eps_I = \sum_{\sig\in I}e_\sig$. Then following the same idea as the non-archimedean case, we may define the Kazhdan cocenter to be the projective limit over the collection of finite subsets of $\hat{K}$ of the level $I$ cocenter $\CS_I/[\CS_I,\CS_I]$, where $\CS_I = \eps_I\CS\eps_I$ and both the quotient $\CS_I/[\CS_I,\CS_I]$ and the projective limit are taken in the category of Fréchet spaces. We will treat the details of the above more algebraic construction in a separate paper. For the purpose of this paper, we would like to follow the spectral characterization in the non-archimedean case and introduce the following more restricted version of the center and cocenter, which are sufficient for the purpose of this paper. 

A spectral characterization of this ring relies on the matrix Paley-Wiener theorem for $\CS$, which seems not known to us for general reductive groups. However, for $G=\SL_2$, we do have the matrix Paley-Wiener theorem for $\CS$. When $F=\BR$ it follows from \cite[p.100]{Barker}, and when $F=\BC$, it follows from \cite[Chap.~IV,\S~5.6]{GGV}. The results can be generalized to $G=\GL_2$ without difficulties. 

Precisely, let $B=TN$ be the standard upper triangular Borel of $G=\SL_2$ with Cartan decomposition $G(F)=B(F)K$. For any character $\chi:T(F)\to \BC^\times$, the unitary dual $\wh{T}$ of $T(F)$ has a real manifold structure given by 
$$
\wh{T} = 
\bigg\{
\begin{matrix}
\{\pm 1\}\times i\BR, & F=\BR\\
\BZ\times i\BR, & F=\BC
\end{matrix}
$$
When $F=\BR$, for $(\pm,r)\in \wh{T}$, the associated unitary character is given by $\chi_{(\pm,r)}(z) = \sgn(z)|z|^{r}$; When $F=\BC$, for $(n,r)\in \wh{T}$, the associated unitary character is given by $\chi_{n,r}(z) = (\frac{z}{|z|})^n|z|^{r}$. Consider the normalized induced representation $\pi_\chi = \Ind_B^G\chi\simeq \Ind^{K}_{K\cap B}\chi$. 

Notice that $\pi_{\chi_{(\pm,r)}}\simeq \pi_{\chi_{(\pm,-r)}}$ when $F=\BR$, and $\pi_{\chi_{(n,r)}}\simeq \pi_{\chi_{(-n,-r)}}$ when $F=\BC$. The matrix Paley-Wiener theorem from \cite[p.100]{Barker} and \cite[Chap.~IV,\S~5.6]{GGV} implies that the functions $h\in \CS$ are characterized by the family of operators $\chi\in \wh{T}\mapsto \pi_\chi(h):\Ind^K_{K\cap B}\chi\to \Ind^K_{K\cap B}\chi$ extending to entire functions in $\wh{T}\otimes_\BR \BC$ that are of rapidly decay in bounded vertical region and invariant under the Weyl group permutation, together with the compatibility with the intertwining operators between principal series. For details, see the statement of theorems in \emph{loc. cit.}

\begin{rmk}\label{rmk:mainresult:complexMPW}
For $F=\BC$, the result from \cite[Chap.~IV,\S~5.6]{GGV} reads as follows. Following the notation from \emph{loc. cit.} the operators $\pi_\chi$ are represented by the kernels $K(z_1,z_2;n_1,n_2)$, with $n_1,n_2$ characters of $F^\times$ satisfying $n_1-n_2\in \BZ$. In particular, comparing with our notation, $n_1-n_2 = n$ and $n_1+n_2 = r$ (\cite[Chap.~III,\S~2.2]{GGV}). Hence the analytical properties of the operators $\chi\in \wh{T}\mapsto \pi_\chi(h)$ are captured by $(n_1,n_2)\mapsto K(z_1,z_2;n_1,n_2)$. Following the proof of the main theorem in \cite[Chap.~IV,\S~5.6]{GGV}, $K(z_1,z_2;n_1,n_2)$ is the Mellin transform of $\vphi(z_1,z_2;\lam)$ via
$$
K(z_1,z_2;n_1,n_2) = 
\frac{i}{2}
\int_{\lam\in \BC}\vphi(z_1,z_2;\lam)
\lam^{n_1-1}\bar{\lam}^{n_2-1}\ud \lam\ud \bar{\lam}
$$
and $\vphi(z_1,z_2;\lam) = |\lam|^2\Phi(z_1,1;\lam z_2,\lam)$, which, by the main theorem in \cite[Chap.~IV,\S~5.6]{GGV}, is of rapidly decay in $\lam\in \BC^\times$ with extra symmetries capturing its compatibility with the intertwining operators. Essentially, the theorem says that besides the compatibility with intertwining operators, the analytical behavior of the kernel functions for the operators $\chi\mapsto \pi_\chi$ are exactly given by the Mellin transform of Schwartz functions on $\BC^\times$ whose image are invariant under the Weyl group action. For a precise characterization of the Mellin transform of Schwartz functions on $\BC^\times$ (i.e. smooth functions on $\BC^\times$ that are both rapidly decay near infinity and zero), see \cite[Thm.~4.3]{igusa_higher} or a quick sumary in \cite[Defin.~2.2]{jiang2021certain}.
\end{rmk}

For the purpose of this paper, we introduce the following algebra of multipliers $\CZ=\CZ_G$ on $\CS$:
\begin{defin}\label{defin:multiplierarchimedean}

Let $G=\SL_2$.
\begin{itemize}
    \item
Let $F=\BR$. Let $\CZ$ be the space of holomorphic functions $\gam$ on $\BC\times \BZ_2$ that are even on $\BC$-factor, of rapidly decay in bounded vertical strips, and are of polynomial growth on the subset $\BZ\subset \BC$;

    \item
Let $F=\BC$. Let $\CZ$ be the space of functions $\gam$ on $\BZ\times \BC$ that are holomorphic on $\BC$-factor with symmetry $\gam(n,z) = \gam(-n,-z)$, and are of rapidly decay on bounded vertical strips that is uniform in $\BZ$-factor.
\end{itemize}

Let $G=\GL_2$.
\begin{itemize}
    \item 
Let $F=\BR$. Let $\CZ$ be the space of holomorphic functions on $(\BC\times \BZ_2)^2$ that are invariant under switching the two factors, of rapidly decay in bounded vertical strips, and are of moderate growth on the subset $\BZ\subset \BC$;

    \item
Let $F=\BC$. Let $\CZ$ be the space of functions on $(\BZ\times \BC)^2$ that are holomorphic on $\BC$-factor, invariant under switching the two factors,and are of rapidly decay on bounded vertical strips that is uniform in $\BZ$-factor.
\end{itemize}
Let $\CZ_G^\fin$ be the subalgebra of $\CZ_G$ that are supported on finitely many $K$-types.
\end{defin}
Notice that the archimedean Langlands local gamma factors are of rapidly decay with neighborhoods of possible poles removed via Stirling's formula (see \cite[Lem.~3.4.6]{MR3990815} for instance). Hence following Part (5) of Remark \ref{rmk:proposal}, after multiplying a smooth and compactly supported function on the determinant factor for the $\rho$-Fourier kernels (which, after regularization, are the Plancherel inversion of archimedean Langlands gamma factors), the associated holomorphic functions on the admissible dual belong to $\CZ$. In particular $\CZ$ is enough for our purpose. When $F=\BR$ the extra moderate growth condition on $\BZ\subset \BC$ is needed due to the existence of discrete series representations of $G(\BR)$. It is clear that $\CZ$ is an algebra under multiplication, and any $\gam\in \CZ$ gives rise to a multiplier from $\CS$ to itself. Moreover, the matrix Paley-Wiener theorem for smooth and compactly supported functions (\cite{Delorme-Paley-Wiener}) ensures that $\gam$ is represented by an invariant distribution $J=J_\gam$ on $G(F)$.

For the purpose of next subsection, let $\CZ^\st$ be the subset of $\CZ$ given by stable distributions in the sense of \cite{Langlands-StableConjugacy}, and let $\CZ^{\st,\fin} = \CZ^\st\cap \CZ^{\fin}$.

\subsection{Stable Bernstein center}\label{subsec:mainresult:SBC}

Let $G$ be a reductive group over a non-archimedean local field $F$. Let $\CZ_G$ be the Bernstein center of $G(F)$. Following Bezrukavnikov, Kazhdan and Varshavsky (\cite{bezrukavnikov2013categorical}), let $\CZ^\st_G$ be the stable Bernstein center of $G(F)$ which is the subset of $\CZ_G$ consisting of stable distributions on $G(F)$ in the sense of \cite{Langlands-StableConjugacy}. We recall that an invariant distribution $J$ on $G(F)$ is called stable if for any test function $h\in \CH$, we have $J(h) = 0$ whenever the stable orbital integral of $h$ vanishes identically. Let $\CZ^{\st,\fin}_G = \CZ^\st_G\cap \CZ^\fin_G$. The stable center conjecture of \cite{bezrukavnikov2013categorical} stipulates that $\CZ^\st_G$ is a subalgebra of $\CZ_G$.

Let $\Ome^\st_G$ be the variety of infinitesimal characters of $G$ (\cite[\S 5]{Haines-stable-Bernstein}\cite[\S 6]{scholze-shin}). For $G=\SL_2$, since the local Langlands conjecture is known (\cite{labesse-langlands}), there is a finite to one map $p:\Ome_G\to \Ome^\st_G$ sending the cuspidal datum in $\Ome_G$ to the corresponding $\wh{G}$-conjugacy classes of semi-simple local $L$-parameters. For $J\in \CZ^\st_G$ and a connected component $\Ome_0^\st\subset \Ome_G^\st$, we call the projection of $J$ to $\Ome^\st_0$ to be the distribution in $\CZ_G$ that is the projection of $J$ to $p^{-1}(\Ome_0^\st)$. When $G=\GL_2$, there is no stability issue and the stable Bernstein center is the same as the Bernstein center.

In this section, we establish the following theorem. Notice that when $F$ is archimedean, we have introduced $\CZ,\CZ^\st$ and $\CZ^{\st,\fin}$ in subsection \ref{subsec:BCandKC} when $G=\SL_2$ and $\GL_2$.

\begin{thm}\label{thm:main:SBC}
Let $G=\SL_2$ and the residual characteristic of $F$ is other than two. Then the following statements hold.

\begin{enumerate}
\item The stable tempered characters are dense in the space of stable distributions;

\item $\CZ_G^\st$ is a subalgebra of $\CZ_G$;

\item For any $J_G\in \CZ^\st_G$, its projection to any connected components of $\Ome^\st_G$ lands in $\CZ^{\st,\fin}_G$;
\end{enumerate}
\end{thm}

The above theorem also holds for $G=\GL_2$ over any local field $F$. Part (2) and Part (3) are manifest, while Part (1) is based on the density theorem of Kazhdan (\cite[Thm.~0]{KazCus}) when $F$ is non-archimedean, and Shelstad when $F$ is archimedean (\cite[Lem.~5.3]{shelsteadinner}).

\begin{proof}
(1) When $F$ is archimedean, the density of stable tempered characters follows from \cite[Lem.~5.3]{shelsteadinner}. It remains to focus on the case when $F$ is non-archimedean.

To show that the stable tempered characters are dense in the space of stable distributions for $G$, we only need to show that for a fixed test function $h\in \CH$, if the stable tempered traces of $h$ vanish, then all the stable orbital integrals of $h$ vanish. 

Notice that the stable tempered characters for $G=\SL_2$ are given by the restriction of tempered characters of $\GL_2$ (\cite{labesse-langlands}). For $\GL_2$, there is no stability issue, and the density theorem of Kazhdan (\cite{KazCus}) shows that for a test function $\wt{h}$ on $\GL_2$, its tempered trace vanishes if and only if its orbital integral vanishes. Now for the fixed test function $h\in \CC^\infty_c(G(F))$, since $G(F)$ is closed in $\wt{G}(F) = \{g\in \GL_2(F)\mid \det g\in F^{\times 2}\}$, and $\wt{G}(F)$ is open in $\GL_2(F)$, we can lift $h$ to a test function $\wt{h}\in \CC^\infty_c(\GL_2(F))$ such that $\wt{h}|_{G(F)} = h$. Applying the density theorem to $\wt{h}$, we get
\begin{num}
\item $\Theta_\Pi(\wt{h}) =0$ for any tempered representation $\Pi$ of $\GL_2(F)$ if and only if $\RO(\wt{h}) =0$. Here $\Theta_\Pi$ is the trace distribution of $\Pi$, and $\RO(\wt{h})$ is the orbital integral of $\wt{h}$.
\end{num}
By Weyl integration formula and the support condition of $\wt{h}$, we can rewrite $\Theta_\Pi(\wt{h})$ in terms of (trace, determinant) variables
$$
\Theta_\Pi(\wt{h}) = 
\int_{(t,d^2)\in F\times F^\times}
\Theta_\Pi(t,d^2)\RO(\wt{h})(t,d^2)\ud t\ud^\times d
$$
Twisting $\Pi$ by a unitary character $\chi$ of $F^\times$, and using the Mellin inversion in $\chi$, we get 
\begin{num}
\item $\Theta_\Pi(\wt{h}) = 0$ for any tempered representation $\Pi$ of $\GL_2(F)$ if and only if the following function vanishes identically (on regular semi-simple locus)
$$
(t,d)\in F\times F^{\times 2}
\mapsto 
\int_{t\in F}
\Theta_\Pi(t/d,1)\RO(\wt{h})(t,d^2)\ud t =0 = 
\int_{t\in F}
\Theta_\Pi(t,1)\RO(\wt{h})(td,d^2)\ud t,
$$
if and only if the orbital integral of $\wt{h}$ vanishes.
\end{num}
But $\RO(\wt{h})(td,d^2)$ is the orbital integral of $\wt{h}(d\cdot)$ evaluated at any regular semi-simple element with trace $t$. Since $\wt{h}$ is a test function, $d\in F^\times \mapsto \wt{h}(dg)$ is uniformly locally constant in $d$ for any $g$. Hence we may multiply $\wt{h}$ by a characteristic function along the center that is supported in a small neighborhood $\CN_{\wt{h}}$ of identity, such that $\RO(\wt{h})(td,d^2)$ is non-vanishing only for $d\in \CN_{\wt{h}}$, and is identically equal to $\RO(\wt{h})(t,1)$ whenever $d\in \CN_{\wt{h}}$. It follows that we deduce the following fact for the above choice of $\wt{h}$:
\begin{num} 
\item 
$$
\int_{t\in F}\Theta_\Pi(t,1)\RO(\wt{h})(t,1)\ud t=0
$$
for any tempered representation $\Pi$ of $\GL_2(F)$ if and only if $\RO(\wt{h})$ vanishes. 
\end{num}
But $\RO(\wt{h})(t,1)$ is exactly equal to the stable orbital integral of $h$ evaluated at any regular semi-simple element of trace $t$, hence 
$$
\int_{t\in F}\Theta_\Pi(t,1)
\RO(\wt{h})(t,1)\ud t = \tr(\Pi|_{G})(h)= 0
$$
if and only if $\RS\RO(h) = 0$. This completes the proof of the density of stable tempered characters in the space of stable distributions.

(2) To show that $\CZ_G^\st$ is a subalgebra of $\CZ_G$, we only need to show that for any $J\in \CZ_G^\st$, whenever the stable orbital integral of $h\in \CH$ vanishes, the stable orbital integral of $J*h$ also vanishes. By Part (1) above, we only need to show that the stable tempered trace of $J*h$ vanishes. Using the Plancherel inversion formula for distributions in the Bernstein center, and the fact that the Plancherel measure for $G=\SL_2$ is constant on tempered local $L$-packets (\cite[Chap.2,\S 6]{ggps}), we only need to show that the function on the Bernstein variety determined by $J\in \CZ_G^\st$ is constant on tempered local $L$-packets. This follows from the following statement:
\begin{num}
\item For a fixed tempered local $L$-packet of $G=\SL_2(F)$, the only stable distributions that are given by the linear combination of distribution characters inside the $L$-packet are given by the constant multiples of the stable tempered characters.
\end{num}
The proof for archimedean situation is the same as non-archimedean case. In the following we outline the proof for non-archimedean situation. We only need to treat the supercuspidal local $L$-packets. When the supercuspidal local $L$-packet has two elements, the above fact is manifest since a single supercuspidal character is not stable, which can be checked by hand using the character table from \cite{SallyShalika}. For the unique supercuspidal local $L$-packet with four elements, following \cite[Thm.~1.9]{Casselma-Quadratic}, there are two (reducible) representations $\Pi^\pm$ with $\Pi^\pm = \pi_1^\pm\oplus \pi_2^\pm$, such that $\{\pi_i^\pm\}_{i=1}^2$ forms the packet, and conjugation by $\GL_2(F)$ permutes $\Pi^+$ and $\Pi^-$. Hence the result follows from the linear independence of the distribution characters.

(3) Following the discussion in Part (1) and Part (2), it suffices to show that $J_{\Ome_0^\st}$ is a stable distribution, where $\Ome^\st_0\subset \Ome_G^\st$ is a fixed connected component, and $J_{\Ome_0^\st}\in \CZ_G^\fin$ is the distribution corresponding to the characteristic function of $p^{-1}(\Ome_0^\st)\subset \Ome_G$. But this follows from the explicit calculation in subsection \ref{subsec:analytic-properties-Jc}, which shows that $J_{\Ome_0^\st}$ is locally integrable and hence is given by a stable function in the sense that $J_{\Ome_0^\st}(g)= J_{\Ome_0^\st}(g^\p)$ whenever $g$ and $g^\p$ are regular semi-simple and stably conjugate to each other.
\end{proof}

\begin{rmk}\label{rmk:anotherdef:SBC}

Let $G$ be a reductive group over a non-archimedean local field $F$. Let $\CH$ be its Hecke algebra. Consider the map $p_\st:\CH\to \RS\RO(\CH)$ sending a test function to its stable orbital integral. Let 
$
\FI_\st = \{f\in \CH\mid \RS\RO(f) = 0\} = p^{-1}_\st(\{0\}),
$
which is the subset of $\CH$ whose stable orbital integral vanishes identically. Following \cite{scholze-shin}, one define $\CZ^{\st,\p}$ to be the subalgebra of the usual Bernstein center $\CZ$ consisting of distributions that stabilizes $\FI_\st$, i.e. 
$$
\CZ^{\st,\p} = \{z\in \CZ\mid z*\FI_\st\subset \FI_\st\}.
$$
Comparing with the definition for $\CZ^\st$, the definition for $\CZ^{\st,\p}$ has the benefit that it is immediately a subalgebra of $\CZ$ by definition. However, in general it is not clear about the relation between $\CZ^{\st,\p}$ and $\CZ^\st$. In the following, we establish the following fact:

\begin{pro}\label{pro:SBCprime=SBC}
The following statements hold:
\begin{enumerate}
\item In general, $\CZ^{\st,\p}\subset \CZ^\st$;

\item When $G=\SL_2$ and $F$ is non-archimedean of odd residual characteristic, $\CZ^{\st,\p}\supset \CZ^\st$, and hence $\CZ^\st = \CZ^{\st,\p}$.
\end{enumerate}
\end{pro}
\begin{proof}
(1) For any $z\in \CZ^{\st,\p}$ and $f\in \FI_\st$, by definition, $z*f\in \FI_\st$. Equivalently, $\RS\RO(z*f) = 0$. It is clear that $\FI_\st$ is invariant under the involution $g\mapsto g^{-1}$. Hence $\RS\RO(z*f^\vee) = 0$. Now to show that $z\in \CZ^\st$, we need to show that $z$ is a stable distribution, which, by definition, is equivalent to show that $z(f) = 0$ as long as $\RS\RO(f) = 0$. But by definition, $z(f) = \del_e(z*f^\vee)$. Since $\RS\RO(z*f^\vee) =0$ and $\del_e$ is a stable distribution, we deduce that $\del_e(z*f^\vee) = z(f) = 0$. It follows that $\CZ^{\st,\prime}\subset \CZ^\st$.

(2) For any $f\in \FI_\st$ and $z\in \CZ^\st$, we only need to show that $\RS\RO(z*f)=0$. By the density of stable tempered characters for $G(F)=\SL_2(F)$ that is proved in Theorem \ref{thm:main:SBC}, we only need to show that for any tempered local $L$-packet $[\pi]$ of $G(F)$, $\tr^{\st}_{[\pi]}(z*f) = 0$. Following the proof of Theorem \ref{thm:main:SBC}, the regular function $\gam(z,\cdot)$ on the Bernstein variety attached to $z\in \CZ^{\st}$ is constant on the tempered local $L$-packets of $G(F)$. Hence we are reduced to show that $\tr^\st_{[\pi]}(f) = 0$. But it follows from the fact that $f\in \FI_\st$. 

\end{proof}

As a corollary, we have the following fact.
\begin{cor}\label{cor:SBC=SBCprime}
Let $G=\SL_2$ or $\GL_2$ over a non-archimedean local field $F$ of odd residual characteristic. Then $\CZ^\st$ acts on $\RS\RO(\CH)$.
\end{cor}
Similar statement is true for $F$ archimedean, based on the trace Paley-Wiener theorem for $\CS$ (\cite{Delorme-limites}). Since our definition of $\CZ^\st$ is pretty limited in the archimedean case, we will not discuss the details here but leave it to a future work.
\end{rmk}

\subsection{Integral transforms on the stable cocenter}\label{subsec:integrlatransformonSC}

By \cite[Thm.~2.5]{Moy-Tadic-Bernstein}, for any reductive group $G$ over a non-archimedean local field $F$, any distribution $J\in \CZ_G^{\rm fin}$ can be represented by an invariant locally integrable function on $G(F)$ which is smooth on  $G'(F)$. It follows that $J$ is the pullback of a smooth function $J_C$ on the set $C'(F)$ regular semisimple conjugacy classes in $G(F)$. Similarly, any stably invariant distribution $J\in \CZ_G^{\st,\fin}$ is represented by a locally integrable stable function on the Steinberg-Hitchin base, and hence can be represented as $J=({\bf c}^* J_\Fc) \ud g$ where $J_\Fc$ is a function on the Steinberg-Hitchin base $\Fc(F)$ whose necessary analytic properties will be established in subsection \ref{subsec:analytic-properties-Jc} when $G=\SL_2$ and $\GL_2$.

In general, assuming the stable center conjecture, $\CZ^\st$ has a natural algebra structure under convolution. Moreover, assuming the analogue of Corollary \ref{cor:SBC=SBCprime} for general reductive groups, the stable cocenter ${\rm SO}(\CH)= \CC^{\rm st, fin}$ carries a natural structure of $\CZ^{\rm st}$-module. By construction, elements of ${\rm SO}(\CH)$ are of the form $f={\rm SO}(h)$ with $h\in \CH$ and $f$ is a locally integrable function on the Steinberg-Hitchin base $\Fc(F)$, well-defined as a smooth function on $\Fc'(F)$. At least in the case $\SL_2$ and $F$ is non-archimedean of odd residual characteristic, we know that $f$ is represented by continuous functions on $\Fc(F)$ (\cite[(2.2.10)]{langlands-singularites-transfert}). In our framework, a crucial problem is to represent the action of an element $J\in \CZ^\st$ on an element $f\in {\rm SO}(\CH)$ via integral transform over the Steinberg-Hitchin base. This problem will be solved in section \ref{sec:HKT}.

\subsection{Langlands stable transfer and descent}\label{subsec:STandDescent}
  
In \cite{langlands-singularites-transfert}, Langlands introduced the notion of stable transfer that we now reformulate with the aid of the Kazhdan cocenter. Let $H$ be another reductive group over $F$ such that we have a homomorphism of $L$-groups $\rho: \,\!\!^L H\to \,\!\!^L G$. Following the functoriality conjecture of Langlands, $\rho$ induces a map $[\pi_H]\mapsto [\pi]$ from the set of $L$-packets of $H$ to the set of $L$-packets of $G$. It follows that by adjunction there exists a map
$$
\CT_\rho:{\rm SO}(\CH_G) \to {\rm SO}(\CH_H)
$$
that Langlands calls the stable transfer such that for every $L$-packet $[\pi_H]$ of $H(F)$ mapping to the $L$-packet $[\pi]$ of $G(F)$, we have 
$$
\langle f, \Theta^{\Fc_G}_{[\pi]}\rangle = \langle \CT_\rho(f), \Theta_{[\pi_H]}^{\Fc_H} \rangle
$$
for every $f\in {\rm SO}(\CH_G)$. It is expected that Langlands' stable transfer extends to a linear map between the Kazhdan cocenters
$$
\CT_\rho: \CC^{\rm st}(G(F)) \to \CC^{\rm st}(H(F)).
$$
Since ${\rm SO}(\CH_G)$ can be realized as a suitable space of function on the $F$-analytic manifold $\Fc_G(F)$, it is desirable that $\CT_\rho$ is an integral transform whose kernel is of algebraic nature. 

The map $\rho_*:[\pi_H]\mapsto [\pi]$ is expected to induce an algebraic map $\rho_*:\Omega^{\rm st}_H \to \Omega^{\rm st}_G$ which is equivalent to a homomorphism of algebras 
\begin{equation}
	\rho^*: \CZ_{G}^{\rm st} \to \CZ_{H}^{\rm st}
\end{equation}
between the stable Bernstein centers, assuming the stable center conjecture à la \cite{bezrukavnikov2013categorical}. We call the map {\it descent}. Langlands' stable transfer map $\CT_\rho: \CC^{\rm st}_{G} \to \CC^{\rm st}_{H}$ is then $\CZ_{G}^{\rm st}$-linear.

%When $H$ is a one-dimensional torus over $F$, there is no difference between the Bernstein center and the Kazhdan cocenter: in this case, the Lafforgue map $\Laf_H$ is the identity map. In particular, we have $$ \rho^* = \CT_\rho\circ \Laf_G$$ is what we call the descent map from $\SL_2$ to the torus $H$. 
  
In the case of $G=\SL_2$ and $H$ is a one-dimensional torus, the stable transfer $\CT_\rho$ was discussed in great detail in  Langlands' paper (\cite{langlands-singularites-transfert}), which is a source of inspiration for our work. It is rather surprising that Langlands did not seem to be aware that his stable transfer has an integral representation discovered fifty years earlier by Gelfand and Graev (\cite{Gelfand-Graev}). The second-named author learned about the Gelfand-Graev formula from R. Kottwitz and cannot thank him enough for sparkling his curiosity about this formula.

% !TEX root = luo-ngo.tex

\section{Gelfand-Graev stable transfer}\label{sec:ggtransform}

In this section, we review the construction of dihedral representations of $\GL_2(F)$ via Weil representation following \cite{weil64unitary} and \cite[\S 1]{jlgl2}. We will recall the Gelfand-Graev formula for the character of dihedral representations of $\GL_2$ following \cite{Gelfand-Graev}. For $\SL_2$, this gives rise to the stable character, which induces an explicit formula for Langlands' stable transfer \cite{langlands-singularites-transfert}. We will call this formula the Gelfand-Graev transform and explore its properties.

\subsection{The Weil representation}\label{subsec:weilrep}

Let $E$ be an étale quadratic $F$-algebra where $F$ is a local field of residual characteristic not two. We have the additive character $\psi_E:E\to \BC^1$ given by $\psi_E(x)=\psi(\tr_{E/F}(x))$ where $\psi:F\to \BC^1$ is the additive character of $F$ chosen in \ref{subsec:notation}. The Weil representation $\CW_E$ of $\SL_2(F)$ on the space $\CS(E)$  of Schwartz-Bruhat functions on $E$, given in \cite[Prop.1.3]{jlgl2}, is defined on generators of $\SL_2(F)$ by the following formulae: for any test function $\phi\in \CS(E)$ and $z\in E$, 
\begin{enumerate}
\item 
$
\big(
\CW_E
\begin{pmatrix}
\alp & \\
  & \alp^{-1}
\end{pmatrix} 
\phi\big)(z) = 
\eta_E(\alp) |\alp|\,
\phi(\alp z),\quad \alp\in F^\times;
$

\item 
$
\big(
\CW_E
\begin{pmatrix}
1 & u\\
  & 1
\end{pmatrix}
\phi
\big)(z) = 
\psi\big(u \Nr(z)\big)\,\phi(z),\quad u\in F^\times;
$

\item
$
\big(
\CW_E
\begin{pmatrix}
 & 1\\
-1 & 
\end{pmatrix} 
\phi
\big)(z) = 
\lam_{E/F}\, \CF_{\psi_E}(\phi)(\iota(z))
$
\end{enumerate}
Here we have 
$$
\CF_{\psi_E}(\phi)(z) = 
\int_{x\in E}
\psi_E(x z)\phi(x)\ud_E x,
$$
and 
\begin{equation} \label{eq:Weil constant}
	\lam_{E/F} = \lam_{E/F,\psi} = \veps(1/2,\eta_E,\psi)
\end{equation}
is the Weil constant \cite[Lem.~1.1]{jlgl2}.
We will need an explicit formula for the Weil representation $\CW_E$ as in Lemma \ref{lem:ell:sl2:3}, which will be derived from the following lemma.

\begin{lem}\label{lem:lowerunip:sl2}
For every Schwartz-Bruhat function $\phi\in \CS(E)$, we have
$$
\big(
\CW_E
\begin{pmatrix}
1 & \\
v & 1
\end{pmatrix}
\phi
\big)(z) = 
\lam_{E/F}
\frac{\eta(-v)}{|v|}
\int_{x\in E}
\psi
\bigg(
\frac{\Nr(z-x)}{v}
\bigg)
\phi(x)\ud_E x.
$$
\end{lem}
\begin{proof}
Write 
$$
\begin{pmatrix}
1 & \\
v & 1
\end{pmatrix} = 
\begin{pmatrix}
-1 & \\
  & -1
\end{pmatrix}
\begin{pmatrix}
 & 1\\
-1  & 
\end{pmatrix}
\begin{pmatrix}
1 & -v\\
  & 1
\end{pmatrix}
\begin{pmatrix}
 &1 \\
-1  & 
\end{pmatrix},
$$
then 
$$
\big(
\CW_E
\begin{pmatrix}
1 &\\
v & 1
\end{pmatrix}
\phi
\big)
(z) = 
\eta(-1)
\lam_{E/F}\,
\CF_{\psi_E}
\bigg(
\CW_E
\begin{pmatrix}
1 & -v\\
 & 1
\end{pmatrix}
\begin{pmatrix}
 & 1\\
-1& 
\end{pmatrix}
\phi
\bigg)
\big(-\iota(z)\big).
$$
Expanding the definition of $\CF_{\psi_E}$ on the right hand side, the above equation can be rewritten as 
\begin{align*}
=& \ \eta(-1)
\lam_{E/F}
\int_{x\in E}
\psi_E\big(-x \iota(z)\big)
\CW_E
\bigg(
\begin{pmatrix}
1 & -v\\
  & 1
\end{pmatrix}
\begin{pmatrix}
 & 1\\
-1 &
\end{pmatrix} 
\phi
\bigg)(x)\ud_E x
\\
=& \
\eta(-1) 
\lam^2_{E/F}
\int_{x\in E}
\psi_E\big(
-x\iota(z)
\big)
\psi\big(
-v \Nr(x)
\big) 
\CF_{\psi_E}(\phi)\big(\iota(x)\big)\ud_E x.
\end{align*}
By \cite[Prop.~29.4~(3)]{gl2llc},
$\lam^2_{E/F} = \veps(\frac{1}{2},\eta,\psi)^2 = \eta(-1)$. Hence the above identity becomes 
\begin{equation}\label{eq:1:weil}
=\int_{x\in E}
\psi
\bigg(
-v \Nr(x)-\Tr\big(
x\iota(z)
\big)
\bigg)\, 
\CF_{\psi_E}(\phi)\big(
\iota(x)
\big)
\ud_E x.
\end{equation}
Since 
$$
-v \Nr(x)-\Tr\big(x\iota(z)\big) = 
-v 
\Nr\bigg(
x+\frac{z}{v}
\bigg)+\frac{\Nr(z)}{v}
$$
equation \eqref{eq:1:weil} becomes 
\begin{align*}
=\ &
\psi
\bigg(
\frac{\Nr(z)}{v}
\bigg)
\int_{x\in E}
\psi
\bigg(
-v 
\Nr\bigg(
x+\frac{z}{v}
\bigg)
\bigg)
\CF_{\psi_E}(\phi)
\big(
\iota(x)
\big)
\ud_E x
\\
=\ &\psi
\bigg(
\frac{\Nr(z)}{v}
\bigg)
\int_{x\in E}
\psi
\big(
-v \Nr(x)
\big)
\CF_{\psi_E}(\phi)
\left(
\iota(x)-\frac{\iota(z)}{v}
\right)
\ud_E x
\end{align*}
Following Weil's formula for the Fourier transform of $\psi\circ \Nr$ (\cite[Lem.~1.1]{jlgl2}), the above equation is equal to 
$$
\psi
\left(
\frac{\Nr(z)}{v}
\right)
\frac{\eta(-v)}{|v|} 
\lam_{E/F}
\int_{x\in E}\phi(x)
\psi_E\left(-\frac{\iota(z)x}{v}\right)
\psi
\left(
\frac{\Nr(x)}{v}
\right)
\ud_E x.
$$
The final formula follows from the following identity
$$
\frac{\Nr(z)}{v}
-\Tr\left(\frac{\iota(z)x}{v}\right)+\frac{\Nr(x)}{v} = 
\frac{\Nr(z-x)}{v}.
$$
\end{proof}

\begin{lem}\label{lem:ell:sl2:3}
For any $\phi\in \CS(E)$, we have
\begin{align*}
&\bigg( 
\CW_E
\bigg(
\begin{pmatrix}
1 &\\
v & 1
\end{pmatrix}
\begin{pmatrix}
t &\\
 & t^{-1}
\end{pmatrix}
\begin{pmatrix}
1 &u\\
 & 1
\end{pmatrix}
\bigg)
 \phi\bigg)(z) 
 \\ & =
\frac{\eta(-v) \lam_{E/F}}{|v|}
\int_{x\in E}
\psi
\bigg(
\frac{\Nr(z-x)}{v}
\bigg)
\eta(t) 
|t|
\psi\big(ut^2\Nr(x)\big)
\phi(tx)\ud_E x.
\end{align*}
\end{lem}

\begin{proof}
By Lemma \ref{lem:lowerunip:sl2}, 
\begin{align*}
&\bigg( 
\CW_E
\bigg(
\begin{pmatrix}
1 &\\
v & 1
\end{pmatrix}
\begin{pmatrix}
t &\\
 & t^{-1}
\end{pmatrix}
\begin{pmatrix}
1 &u\\
 & 1
\end{pmatrix}
\bigg)
\phi\bigg)(z)
\\
&=
\frac{\eta(-v)\lam_{E/F}}{|v|}
\int_{x\in E}
\psi
\bigg(
\frac{\Nr(z-x)}{v}
\bigg)
\bigg(
\CW_E
\bigg(
\begin{pmatrix}
t & \\
  &t^{-1}
\end{pmatrix} 
\begin{pmatrix}
1 & u\\
  &1
\end{pmatrix}
\bigg)
\phi
\bigg)(x) \ud_E x
\end{align*}
By definition, we have the formula
$$
\bigg(
\CW_E
\bigg(
\begin{pmatrix}
t & \\
  &t^{-1}
\end{pmatrix}
\begin{pmatrix}
1 & u\\
  &1
\end{pmatrix}
\bigg)
\phi
\bigg)(x)
=\eta(t) 
|t| 
\psi\big(
ut^2 \Nr(x)
\big)
\phi(tx),
$$
which proves the lemma.
\end{proof}

\subsection{Dihedral representations} \label{subsec:dihedral rep}

Fix a unitary character $\chi$ of $E^\times$ and let $\CS(\chi) = \CS(E,\chi)$ be the $(\chi,E^1)$-coinvariant space of $\CS(E)$ consisting of functions $\Phi$ on $E$ given by
$$
\Phi(z)=\CP_\chi(\phi)(z) = \int_{E^1}\phi(ze)\chi(e)\ud_{E^1}e,\quad z\in E^\times,\quad \phi\in \CS(E).
$$
When $E/F$ is a nonsplit quadratic extension and hence $E^1$ is compact, $\CS(\chi)$ can also be viewed as the space of functions $\phi\in \CS(E)$ such that $\phi(ze)  =\chi(e)\phi(z)$ for any $e\in E^1,z\in E$, which is then a direct factor of $\CS(E)$. Whether $E/F$ is split or not, the action of $\SL_2(F)$ on $\CS(E)$ commuting with the action of $E^1$  and induces a representation $\CW_E^+(\chi)$ of $\SL_2(F)$ on $\CS(\chi)$, see \cite[Prop.~1.5]{jlgl2}. 

The Galois conjugation $\iota\in \Gal(E/F)$, induces an involution of $\CS(E)$ commuting with the action of $\SL_2(F)$ while normalizing the action of $E^1$ mapping by the unitary character $\chi:E^1\to \BC^1$ to its inverse $\chi^{-1}$. It induces an involutive intertwiner 
$$\iota: \CW_E^+(\chi)\to \CW_E^+(\chi^{-1}).$$
We recall the well-known fact \cite[Thm.~1.7]{Casselma-Quadratic}:

\begin{pro}\label{pro:dihedral:Cas72:1}
	Let $E/F$ be a quadratic extension then $\CW_E^+(\chi)$ is irreducible if $\chi^2\neq 1$. It is also irreducible for $\chi=1$ in which case it is a principal series representation. There is a unique nontrivial character $\chi_E$ of $E^1$ with $\chi_E^2=1$. In this case, the involution $\iota: \CW_E^+(\chi_E)\to \CW_E^+(\chi_E)$ decomposes $\CW_E^+(\chi_E)$ into a direct sum of of two irreducible representations $$\CW_E^+(\chi_E)=\CW_E^+(\chi_E)_+ \oplus \CW_E^+(\chi_E)_-.$$ 
\end{pro}

Let $\GL_2(F)^+$ be the subgroup of $\GL_2(F)$ consisting of $g\in \GL_2(F)$ such that 
$$\det g\in \Im(\Nr:E^\times \to F^\times).$$ Then the representation $\CW_E^+(\chi)$ of $\SL_2(F)$ can be extended to a representation of $\GL_2(F)^+$, also denoted as $\CW^+_E(\chi)$, via 
$$
\bigg(
\CW^+_E(\chi)
\begin{pmatrix}
a & \\
  &1
\end{pmatrix}
 \phi
\bigg)(z) = 
|a|^{\frac{1}{2}} 
\chi(e_a) \phi(ze_a),\quad z\in E
$$
where $a=\Nr(e_a)$. Notice that $\phi\in \CS(\chi)$ and hence the right-hand side is independent of the choice of the representative $e_a$. 
	We denote 
\begin{equation} \label{eq:dihedral}
	\CW_E(\chi) = \Ind^{\GL_2(F)}_{\GL_2(F)^+}\CW^+_E(\chi),
\end{equation}	
 and call it the {\it dihedral representation} of $\GL_2(F)$ attached to the unitary character $\chi$ of $E^\times$.

\begin{thm}\label{thm:L-packet}
	If $\chi\neq 1$, the restriction of the dihedral representation $\CW_E(\chi)$ to $\SL_2$ decomposes into a direct sum or two irreducible supercuspidal representations if $\chi\neq \chi_E$, and four irreducible supercuspidal representations if $\chi=\chi_E$. In both cases, those irreducible summands form an $L$-packet of $\SL_2(F)$.
\end{thm}

Following \cite[\S 4]{jlgl2}, $\CW_E(\chi)$ is an irreducible admissible representation of $\GL_2(F)$, which is independent of the choice of $\psi$. Moreover, from \cite[Thm.~4.7,~Thm.~5.15,~Thm.~6.4]{jlgl2}, 
\begin{equation}\label{eq:weilrep:2}
L(s,\CW_E(\chi))  =L(s,\chi),\quad \veps(s,\CW_E(\chi),\psi)  =\veps(s,\chi,\psi_E) \lam_{E/F}
\end{equation}
where the local factors on the left hand side are defined via Hecke integrals as in \cite[Thm.~2.18,~Thm.~5.15,~Thm.~6.4]{jlgl2}, and the right hand side are defined via Tate integrals (\cite{tatethesis}).

Up to a finite linear combination, we can assume that $\phi\in \CW_E(\chi)$ is either supported on $\GL_2(F)^+$ or $\GL_2(F)^-=\GL_2(F)\bs \GL_2(F)^+$. When $\phi$ is supported on $\GL_2(F)^+$, we can assume that $\phi$ is actually represented by a vector in $\CS(\chi)$. The following lemma is similar to Lemma \ref{lem:ell:sl2:3}.
\begin{lem}\label{lem:elliptic:2}
With the above notation, for $\phi\in \CS(\chi)$
$$
\bigg(
\CW_E(\chi)
\bigg(
\begin{pmatrix}
a & \\
  & 1
\end{pmatrix}
\begin{pmatrix}
1 & \\
v  & 1
\end{pmatrix}
\begin{pmatrix}
t & \\
  & t^{-1}
\end{pmatrix}
\begin{pmatrix}
1 & u\\
  & 1
\end{pmatrix}
\bigg)
\phi\bigg)(z)
$$
is nonzero only when $a\in \Nr(E^\times)$, and is equal to 
\begin{align*}
=
\frac{|a|^{\frac{1}{2}} \chi(e_a) \eta(-v) \lam_{E/F}}{|v|}
\int_{x\in E}
\psi
\bigg(
\frac{\Nr(e_az-x)}{v}
\bigg)
\eta(t) |t|
\psi(u t^2 \Nr(x))
\phi(tx)
\ud_E x
\end{align*}
when $a=\Nr(e_a)$ for some $e_a\in E^\times$.
\end{lem}

\subsection{Gelfand-Graev stable character formula} \label{subsec:ggpsstable:gl2}

In this subsection, we review the stable character formula of $\SL_2(F)$, following Gelfand and Graev \cite{Gelfand-Graev} and \cite[Chap.~2,\S 5]{ggps}. We also slightly generalize their work to $\GL_2(F)$. 

Let $E$ be an étale quadratic field extension of $F$ and $\chi:E^\times \to \BC^\times$ a character. We will not discuss the stable transfer when $E$ is split since the distribution character for the principal series is easier and will not be used in subsequent. Let $\CW_E(\chi)$ be the dihedral representation \eqref{eq:dihedral} constructed in the previous subsection. We will denote $\theta_\chi^{\rm st}$ its distribution character, which is a generalized function determined by a locally integrable function on $\GL_2(F)$ smooth over the regular semisimple locus. 

As we noted in theorem \ref{thm:L-packet} for $\chi\neq 1$, the restriction of the dihedral representation $\CW_E(\chi)$ to $\SL_2$ decomposes into a direct sum of two irreducible supercuspidal representations if $\chi\neq \chi_E$, and four irreducible supercuspidal representations if $\chi=\chi_E$. Those irreducible summands form an $L$-packet of $\SL_2(F)$ in both cases. The sum of characters in the $L$-packet is the restriction of $\theta_\chi^{\rm st}$ to $\SL_2(F)$, which is the stable character of the $L$-packet. We have a beautiful formula for those stable characters, following Gelfand and Graev \cite{Gelfand-Graev} (see also \cite[Chap.~2,\S 5,~p.204]{ggps}).

\begin{thm}[\cite{Gelfand-Graev}]\label{thm:ggps:sl2}
For every regular semi-simple element $g\in \SL_2(F)$, we have
\begin{equation}\label{eq:ggps:gl2:1}
\theta_\chi^{\rm st}(g) = 
\frac{2}{\vol(E^1,\ud_{E^1}e)}
\int^*_{E^1}
\frac{\eta_E\big(\tr(g)-\tr(e)\big)}{|\tr(g)-\tr(e)|}
\chi(e)\ud_{E^1}e.
\end{equation}
\end{thm}
As we notice immediately that the above expression is independent of the choice of the Haar measure $\ud_{E^1}e$ on $E^1$ because it appears in both the numerator and the denominator. We will nonetheless use the Haar measure fixed in subsection \ref{subsec:measure}.

\begin{rmk}
 We also notice that the integral over $E^1$ is not absolutely convergent, and some regularization is needed. Following \cite[Chap.~2,\S 5]{ggps} we can make sense of the integral as the holomorphic continuation at $s=1$ of the following $s$-family of integrals \begin{align}\label{eq:ggps:sfamily}
\int_{E^1}\frac{\eta_E\big(\tr(g)-\tr(e)\big)}{|\tr(g)-\tr(e)|^s}\chi(e)\ud_{E^1}e
\end{align}
which are absolutely convergent for $\Re(s)<1$.
\end{rmk}

\begin{rmk}\label{rmk:ggps:gl2}
Alternatively, the regularized integral \eqref{eq:ggps:gl2:1} can be understood as follows. Let $\chi^\iota = \chi\circ \iota$ where $\iota$ is the nontrivial Galois involution attached to $E/F$. Then 
$$
\theta_\chi^{\rm st}(g) = 
\int^*_{E^1}\frac{\eta_E\big(\tr(g)-\tr(e) \big)}{|\tr(g)-\tr(e)|}\frac{(\chi+\chi^\iota)(e)}{\vol(E^1,\ud_{E^1})}\ud_{E^1}e.
$$
From \cite[Prop.~3.29]{FLN10}, under the trace map $\tr:\SL_2(F)\supset E^1\to F$ which is generically a $2$-fold cover, the additive Haar measure $\ud x$  and $\ud_{E^1}e$ for $x=\tr(e)$ are related as in subsection \ref{subsec:measure} by the following identity 
\begin{equation}\label{eq:ggps:measure}
\ud x= \Del(e)\ud_{E^1}e
\end{equation}
where $\Del(e)$ is the square root of the norm of the Weyl discriminant. Let $\phi_\chi$ be the function supported on $\tr(E^1)\subset F = \tr(\SL_2(F))$ defined via 
\begin{equation}\label{eq:phi-chi}
	\phi_\chi(\tr(e)) = \frac{\chi(e)+\chi^\iota(e)}{\Del(e)}
\end{equation}
then 
\begin{align}\label{eq:ggps:gl2:2}
\theta_\chi^{\rm st}(g) =&
\frac{2}{\vol(E^1,\ud_{E^1})}
 \int^*_{\xi\in F}
\frac{\eta_E\big(\tr(g)-\xi \big)}{|\tr(g)-\xi|}
\phi_\chi(\xi)\ud \xi  \nonumber
\\
=& \frac{2}{\vol(E^1,\ud_{E^1})}\bigg(\frac{\eta_E}{|\cdot|}*_+\phi_\chi\bigg)(\tr(g)).
\end{align}
Here $*_+$ is the additive convolution. 
Using Tate's thesis we can provide an alternative expression for \eqref{eq:ggps:gl2:2}. Precisely, the functional equation of Tate integrals can be reinterpreted as the following identity for meromorphic families of tempered distributions on $F$, 
$$
\gam(s,\eta_E,\psi)
\CF_{\psi^{-1}}(\eta_E|\cdot|^{s-1})
=\eta_E|\cdot|^{-s}.
$$
Therefore \eqref{eq:ggps:gl2:2} is equal to 
\begin{align*}
\theta_\chi^{\rm st}(g)	&=
\frac{2\gam(1,\eta_E,\psi)}{\vol(E^1,\ud_{E^1})}
\bigg(
\CF_{\psi^{-1}}(\eta_E)*_+\phi_\chi
\bigg)(\tr(g))
\\&=
\frac{2\gam(1,\eta_E,\psi)}{\vol(E^1,\ud_{E^1})}
\bigg(
\CF_{\psi^{-1}}\big(\eta_E \CF_{\psi}(\phi_\chi)\big)
\bigg)(\tr(g)).
\end{align*}
Here $\CF_{\psi}(\phi_\chi)$ is understood as the Fourier transform of a compactly supported distribution, which is represented by a globally bounded smooth function since $\phi_\chi$ is absolutely integrable. After multiplying by $\eta_E$, it is still a tempered distribution. Then the Fourier transform of the tempered distribution $\eta_E \CF_{\psi}(\phi_\chi)$ is represented by a function which up to constant is exactly the desired stable distribution character. 

Finally, from \cite[p.204]{ggps} and the functional equation of local gamma factors, $$a_E c_E = \frac{2}{\gam(0,\eta_E,\psi)} = 2\eta_E(-1)\gam(1,\eta_E,\psi)$$ with
$$
a_E = \vol(E^1,\ud_{E^1}),\quad c_E= \frac{\gam(1,\mathbbm{1}_E,\psi_E)}{\gam(1,\mathbbm{1}_F,\psi_F)\cdot \gam(1,\eta_E,\psi_F)} = \lam_{E/F}^{-1}
$$
where the formula for $a_E$ comes from \cite[p.136]{ggps} and $c_E$ from \cite[p.155~(3)]{ggps} together with \cite[p.76]{jlgl2}. Moreover from \cite[p.151]{ggps}, $\lambda_{E/F}$ is the Weil constant. 
It follows that $$\frac{2\gam(1,\eta_E,\psi)}{\vol(E^1,\ud_{E^1})} = \lam_{E/F}$$ and therefore
\begin{equation}\label{eq:ggps:gl2:interpret:sl2}
\theta_\chi^{\rm st}(g) = \lam_{E/F}
\bigg(
\CF_{\psi^{-1}}\big(\eta_E\cdot \CF_{\psi}(\phi_\chi)
\big)
\bigg)(\tr(g)).
\end{equation}
Notice that the above formula is also valid for $E$ split over $F$. Precisely, when $E$ is split over $F$, $\lam_{E/F} = 1 = \eta_E$, and hence $\theta^\st_\chi = \phi_\chi$ which recovers the distribution character for principal series representations.
\end{rmk}

In the following, we will adapt Theorem \ref{thm:ggps:sl2} to $\GL_2(F)$.

\begin{thm}\label{thm:ggps:gl2}
The distribution character $\theta_\chi^{\rm st}$ of the dihedral representation $\CW_E(\chi)$ is supported on $g\in \GL_2(F)^+ = \{g\in \GL_2(F)\mid \det g\in \Nr(E^\times)\}$. For any $g\in \GL_2(F)^+$ with $\det g=a$, 
$$
\theta_\chi^{\rm st}(g) = 
\frac{2|\det(g)|^{\frac{1}{2}}}{\vol(E^1,\ud_{E^1})}
\int^*_{E^{a}}
\frac{\eta(\tr(g)-\tr(e))}{|\tr(g)-\tr(e)|}\chi(e)\ud_{E,a}e.
$$
Here $E^a = \{e\in E^\times \mid \Nr(e)  =a\}$ and $\ud_{E^a}e$ is the measure on $E^a$ induced from the $E^1$-torsor structure. The regularized integral is understood as \eqref{eq:ggps:sfamily} and Remark \ref{rmk:ggps:gl2}.
\end{thm}

\begin{proof}
The proof is very similar to \cite{ggps}, so we only sketch it below. 

Since the stable distribution character $\theta_\chi^{\rm st}$ is locally integrable, it suffices to determine its value on the open dense Bruhat cell. Following \cite{ggps}, we are going to write the operator $\CW_E(\chi)(g)$ as an integral operator on the underlying space of $\CW_E(\chi) = \Ind^{\GL_2(F)}_{\GL_2(F)^+}\CS(\chi)$, which decomposes $\CS(\chi)\oplus \CS(\chi)^-$ as a dense subspace. Here $\CS(\chi)$ (resp. $\CS(\chi)^-$) consists of vectors in $\CW_E(\chi)$ supported on $\GL_2(F)^+$, resp. $\GL_2(F)^-$.

For $g\in \GL_2(F)^-$, the operator $\CW_E(\chi)(g)$ permutes $\CS(\chi)$ and $\CS(\chi)^-$, hence $\theta^\st_\chi$ vanishes for $g\in \GL_2(F)^-$. It remains to determine $\theta^\st_\chi(g)$ for $g\in \GL_2(F)^+$.

In the following, let us calculate the contribution to the trace of the operator $\CW_E(\chi)(g)$ on $\CS(\chi)$, with 
\begin{equation}\label{eq:ggps:gbruhatcell}
g=\begin{pmatrix}
a & \\
  & 1
\end{pmatrix}
\begin{pmatrix}
1 & \\
v  & 1
\end{pmatrix}
\begin{pmatrix}
t & \\
  & t^{-1}
\end{pmatrix}
\begin{pmatrix}
1 & u\\
  & 1
\end{pmatrix}.
\end{equation}
For $\phi\in \CS(\chi)$, by Lemma \ref{lem:elliptic:2}, 
$$
\bigg(
\CW_E(\chi)
\bigg(
\begin{pmatrix}
a & \\
  & 1
\end{pmatrix}
\begin{pmatrix}
1 & \\
v  & 1
\end{pmatrix}
\begin{pmatrix}
t & \\
  & t^{-1}
\end{pmatrix}
\begin{pmatrix}
1 & u\\
  & 1
\end{pmatrix}
\bigg)
\phi\bigg)(z)
$$
is nonzero only when $a\in \Nr(E^\times)$, and is equal to 
\begin{align*}
=
\frac{|a|^{\frac{1}{2}} \chi(e_a) \eta(-v) \lam_{E/F}}{|v|}
\int_{x\in E}
\psi
\bigg(
\frac{\Nr(e_az-x)}{v}
\bigg)
\eta(t) |t|
\psi(u t^2 \Nr(x))
\phi(tx)
\ud_E x
\end{align*}
when $a=\Nr(e_a)$ for some $e_a\in E^\times$. Changing variable $x\mapsto \frac{x}{t}$, the above identity is equal to 
\begin{align*}
=
\frac{|a|^{\frac{1}{2}} \chi(e_a) \eta(-vt) \lam_{E/F}}{|vt|}
\int_{x\in E}
\psi
\bigg(
\frac{\Nr(e_az-\frac{x}{t})}{v}
\bigg)
\psi(u \Nr(x))
\phi(x)
\ud_E x.
\end{align*}
Using the fact that $\phi\in \CS(\chi)$, the integral $x\in E$ can be separated according to the short exact sequence $1\to E^1\to E^\times \to \Nr(E^\times)\to 1$ and is equal to 
\begin{align*}
=
\frac{|a|^{\frac{1}{2}} \chi(e_a) \eta(-vt) \lam_{E/F}}{|vt|}
\int_{y\in \Nr(E^\times)}
\ud y
\int_{e\in E^1}
\psi
\bigg(
\frac{\Nr(e_az-\frac{e_y e}{t})}{v}
\bigg)
\psi(u y)
\chi^{-1}(e) 
\phi(e_y)
\ud e
\end{align*}
where $y=\Nr(e_y)$ for some $e_y\in E^\times$ and $\ud y$ is the additive Haar measure on the affine line. In particular, identifying $\phi$ as a function on $\Nr(E^\times)$, the above integral transformation admits an integral kernel 
$$
K_{\CW_E(\chi)(g)}(z,e_y e) = 
\frac{|a|^{\frac{1}{2}} \chi(e_a) \eta(-vt) \lam_{E/F}}{|vt|}
\int_{e\in E^1}
\psi
\bigg(
\frac{\Nr(e_az-\frac{e_y e}{t})}{v}
\bigg)
\psi(uy) 
\chi^{-1}(e)\ud_{E^1} e.
$$
Hence the trace of the operator $\CW_E(\chi)(g)$ on $\CS(\chi)$ can be calculated through integrating the kernel along the diagonal, which provides
\begin{equation}\label{eq:ggps:trace:1}
=\frac{|a|^{\frac{1}{2}} \chi(e_a) \eta(-vt) \lam_{E/F}}{|vt|}
\int_{y\in \Nr(E^\times)}^*
\ud y
\int_{e\in E^1}
\psi
\bigg(
\frac{\Nr(e_a e_y-\frac{e_y e}{t})}{v}
\bigg)
\psi(uy) 
\chi^{-1}(e)\ud_{E^1} e.
\end{equation}
Notice that the integral in $y\in \Nr(E^\times)$ is not absolutely convergent and regularization is needed, which are addressed carefully in \cite[Chap.~2,~\S 5]{ggps}.

Similarly, for vectors supported on $\GL_2(F)^-$ and hence lying in $\CS(\chi)^-$, the operator $\CW_E(\chi)(g)$ with $g$ in the open dense Bruhat cell as in \eqref{eq:ggps:gbruhatcell} has trace
\begin{equation}\label{eq:ggps:trace:2}
=-\frac{|a|^{\frac{1}{2}} \chi(e_a) \eta(-vt) \lam_{E/F}}{|vt|}
\int_{y\in E^\times \bs\Nr(E^\times)}^*
\ud y
\int_{e\in E^1}
\psi
\bigg(
\frac{\Nr(e_a e_y-\frac{e_y e}{t})}{v}
\bigg)
\psi(uy) 
\chi^{-1}(e)\ud_{E^1} e
\end{equation}
on the space $\CS(\chi)^-$.
Combining \eqref{eq:ggps:trace:1} and \eqref{eq:ggps:trace:2} together, the trace of the operator $\CW_E(\chi)(g)$ is given by
\begin{align*}
\theta_\chi^{\rm st}(g)=&\frac{|a|^{\frac{1}{2}} \chi(e_a) \eta(-vt) \lam_{E/F}}{|vt|}
\int_{y\in F}^*
\eta(y)
\ud y
\int_{e\in E^1}
\psi
\bigg(
\frac{y \Nr(e_a-\frac{e}{t})}{v}
\bigg)
\psi(uy) 
\chi^{-1}(e)\ud_{E^1}e
\\
=&
\frac{|a|^{\frac{1}{2}} \chi(e_a) \eta(-vt) \lam_{E/F}}{|vt|}
\int_{y\in F}^*
\eta(y)
\ud y
\int_{e\in E^1}
\psi
\bigg(
y
\bigg(
\frac{\Nr(e_a-\frac{e}{t})}{v}
+u
\bigg)
\bigg)
\chi^{-1}(e)\ud_{E^1}e
\end{align*}
where we use the fact that $\eta|_{\Nr(E^\times)} = 1$ and $\eta_{F^\times \bs \Nr(E^\times)} = -1$. 

Changing variable $y\mapsto (-vt)y$, the above integral becomes 
\begin{align*}
=|a|^{\frac{1}{2}} \chi(e_a)  \lam_{E/F}
\int_{y\in F}^*
\eta(y)
\ud y
\int_{e\in E^1}
\psi
\bigg(
-y
\big(
t\Nr(e_a-\frac{e}{t})
+uvt
\big)
\bigg)
\chi^{-1}(e)\ud_{E^1}e.
\end{align*}
Since 
\begin{align*}
t\Nr(e_a-\frac{e}{t})+uvt=
ta+\frac{1}{t}+tuv-\tr(e_a\iota(e))
=\tr(g)-\tr(e_a\iota(e))
\end{align*}
and $\iota(e) = e^{-1}$ for $\Nr(e) =1$, it follows that 
\begin{align*}
\theta_\chi^{\rm st}(g) =& 
|a|^{\frac{1}{2}} \chi(e_a)  \lam_{E/F}
\int_{y\in F}^*
\eta(y)
\ud y
\int_{e\in E^1}
\psi
\bigg(
-y
\big(
\tr(g)-\tr(e_a e)
\big)
\bigg)
\chi(e)\ud_{E^1}e
\\
=&
|a|^{\frac{1}{2}}\lam_{E/F}
\int_{y\in F}^*
\eta(y)
\psi\big(-y\tr(g)\big)
\ud y
\int_{e\in E^a}
\psi
\big(
y\tr(e)
\big)
\chi(e)\ud_{E,a}e.
\end{align*}
Now following the same idea as Remark \ref{rmk:ggps:gl2}, after descending down to the trace variable, the integration in $e\in E^a$ can be viewed as the Fourier transform of a compactly supported absolutely integrable function, which in particular is smooth and globally bounded in $y$. Hence the regularized integral $\int^*_{y\in F}$ can be understood as the Fourier inversion of a tempered distribution. Following the same argument as Remark \ref{rmk:ggps:gl2}, by Tate's thesis, 
$$
\theta_\chi^{\rm st}(g) = 
\frac{|a|^{\frac{1}{2}}
\lam_{E/F}}{\gam(1,\eta,\psi)}
\int_{E^a}^*
\frac{\eta(\tr(g)-\tr(e))}{|\tr(g)-\tr(e)|}
\chi(e)\ud_{E,a}e.
$$
For details on the convergence issue we refer to \cite[Chap.~2,~\S 5,~p.207]{ggps}. The final equality follows from $$\frac{2\gam(1,\eta,\psi)}{\vol(E^1,\ud_{E^1})} = \lam_{E/F},$$ which can be found from the equation above \eqref{eq:ggps:gl2:interpret:sl2}.
\end{proof}

\begin{rmk}
Returning to the formula \eqref{eq:ggps:gl2:2}, we note that the formula expressing the stable character of $\SL_2(F)$ is naturally in the form of an additive convolution in the trace variable. Because after the Fourier transform, the additive convolution becomes a multiplication operator, we can rewrite the Gelfand-Graev stable character for $\SL_2(F)$ in a convenient form 
\begin{equation}\label{eq:ggps:interpret:1}
\theta_\chi^{\rm st}(g) = 
\lam_{E/F}\bigg(
\CF_{\psi^{-1}}
\big(
\eta_E\cdot \CF_\psi(\phi_{\chi})
\big)
\bigg)\big(\tr(g)\big)
\end{equation}
where $\phi_{\chi} = \frac{\chi+\chi^\iota}{\Del}$ is the function supported in the subset $\tr(E^1)$ of $F$ defined in \eqref{eq:phi-chi}. The formula for the character of the dihedral representations of $\GL_2(F)$ now becomes
\begin{equation}\label{eq:ggps:interpret:2}
\theta_\chi^{\rm st}(g) = 
\lam_{E/F}|\det g|^{\frac{1}{2}}\bigg(
\CF_{\psi^{-1}}
\big(
\eta_E\cdot \CF_\psi(\phi_{\chi})
\big)
\bigg)\big(\tr(g),\det (g)\big)
\end{equation}
where the Fourier transform is taken over the trace variable. It is also worth noticing that the above stable character formula holds automatically when $E$ is split over $F$ since the transforms on the right hand side of \eqref{eq:ggps:interpret:1} becomes the identity map.
\end{rmk}

\begin{rmk}[Gelfand-Graev formula for the stable transfer]\label{rmk:GGformulaforstabletransfer}
Let $G=\SL_2$ (resp. $\GL_2$). For any étale quadratic $F$-algebra $E$, let $T_E\simeq E^1$ (resp. $E^\times$) be a fixed maximal torus determined by $E$. The adjoint of the Gelfand-Graev transfer $\chi\mapsto \theta_\chi$ is then a linear map 
\begin{equation}
 \CT_E: \mathrm{SO}(\CH) \to \CH_{T_E}
\end{equation} 
given by the same kernel but read in the opposite direction. When $G=\SL_2$, it maps every function $f \in {\rm SO}(\SL_2(F))$ to the function
\begin{equation}
  \CT_E(f)(e)=  \lambda_{E/F} \CF_{\psi} \big(
\eta_E \CF_{\psi^{-1}}(f)\big)(\tr_E(e)).
\end{equation}
In other words, we have
\begin{equation} \label{eq:GG formula}
\CT_E(f) = \lambda_{E/F} \tr_E^* \left(\CF_{\psi} \big(
\eta_E\cdot \CF_{\psi^{-1}}(f)\big) \right).
\end{equation}
We note that the latter formula also works for the split torus. In the split case $E=F\times F$ and $\eta_E$ is the trivial character of $F^\times$. As a result, we have the stable transfer is simply given by the restriction to the locus of hyperbolic trace $\tr(E_0^1)$ i.e. elements $c\in F$ of the form $c=t+t^{-1}$ with $t\in F^\times$. 
\end{rmk}

% !TEX root = luo-ngo.tex

\section{The Lafforgue transform}\label{sec:Bernstein-center}

In this section, we construct natural maps between the Bernstein center and the Kazhdan cocenter based on representation theoretic considerations. We stipulate that when we restrict this construction to the stable part, it can be expressed as an integral transform that we call the Lafforgue transform. In the remaining subsections, we investigate the convergence property of the Lafforgue transform in the case of $\SL_2$ and $\GL_2$. We show that all distributions in the stable Bernstein center of $\SL_2$ and $\GL_2$ supported on finitely many connected components of the Bernstein variety enjoy a vanishing property based on explicit calculations of Moy and Tadic, which implies the convergence of the Lafforgue transform. We also establish the archimedean variant.

\subsection{Representation theoretic definition}\label{subsec:Laf-Reptheoretic-def}

For a reductive group $G$ over a non-archimedean local field $F$, the set ${\rm Irr}(G(F))$ of irreducible representations of $G(F)$ can be equipped with a canonical structure of algebraic variety over $\C$, called the Lafforgue variety. This construction has been realized in Psaromiglikos' thesis \cite{psaromiligkos_lafforgue_2023}. We have an algebraic map ${\Irr}(G(F)) \to \Omega(G(F))$ from the Lafforgue variety ${\Irr}(G(F))$ to the Bernstein variety $\Omega(G(F))$ whose $\C$-points are irreducible representations of $G(F)$ up to inertial equivalence. The Bernstein center $\mathcal{Z}$ is the ring of regular functions on $\Omega(G(F))$ and a $\CZ$-linear surjective map from the Kazhdan cocenter to the ring of regular functions on ${\rm Irr}(G(F))$. 

By construction, there exists an open dense subvariety $\Omega'(G)$ of $\Omega(G)$ consisting of supercuspidal representations of Levi subgroup (up to inertial equivalence) whose parabolic induction to $G(F)$ is irreducible, its projection from the inverse image ${\rm Irr}'(G(F))$ in ${\rm Irr}(G(F))$ is an isomorphism. For every element $f\in \mathcal C$ viewed as a regular function on ${\rm Irr}(G(F))$, we obtain by restriction a function $f'$ on ${\rm Irr}'(G(F))=\Omega'(G(F))$. One can show that the function $f'$ extends uniquely to a regular function on $\Omega(G(F))$ that we denote $\Laf(f)$. We thus have a map
\begin{equation} \label{Laf'}
	\Laf: \mathcal C\to \mathcal Z
\end{equation}
which is $\mathcal Z$-linear that is an isomorphism after restricting to $\Ome^\p(G(F))\simeq \Irr^\p(G(F))$. Explicitly, for any cuspidal datum $(M,\sig)\in \Ome$ with $M$ a Levi subgroup of $G$ and $\sig$ a supercuspidal representation of $M(F)$, 
$$
\Laf(f)(M,\sig) = f\big(\Ind(\sig)\big).
$$
where $\Ind$ is the parabolic induction from $M(F)$ to $G(F)$. In general, $\Laf$ is only surjective but not injective. For instance, when $G=\SL_2$, $\Laf$ has one dimensional kernel generated by the indicator function on the Steinberg representation of $\SL_2(F)$, which, as an element in $\CC$, is given by the normalized elliptic character of the Steinberg representation (= orbital integral of the pseudo coefficient of the Steinberg representation). In particular for $G=\SL_2$, after localization, $\Laf$ is an isomorphism of $\CZ$-module away from the Steinberg representation component.

After restricting to $\CC^\fin$, $\Laf$ induces a $\CZ$-linear morphism
$$
\Laf:\CC^{\fin}\to \CZ^{\fin}.
$$
Moreover, $\Laf$ sends $\CC^{\st,\fin}$ into $\CZ^{\st,\fin}$. To make it precise, elements in $\CC^{\st,\fin}$ are represented by the stable orbital integral of test functions. On the other hand, distributions in $\CZ^\fin$ are represented by invariant locally integrable functions on $G(F)$ (\cite[Thm.~2.5]{Moy-Tadic-Bernstein}). It follows that the elements in $\Laf(\CC^{\st,\fin})$ are represented by invariant locally integrable functions on $G(F)$ that are also stably invariant, which therefore are represented by stably invariant functions on $G(F)$.

In general, since $\Laf$ is not injective, there are different choices of sections for $\Laf$. For the purpose of this paper, when $G=\SL_2$ and $\GL_2$, we consider the following section $\LafSec:\CZ\to \CC$: 
\begin{itemize}
\item Let $\CC^{\perp,\triv}$ be the subspace of $\CC$ that are orthogonal to one-dimensional characters of $G(F)$, i.e. 
$$
\langle f,\chi\rangle =0,\quad \text{ $\chi$ a $1$-dimensional character of $G(F)$}.
$$
Then there exists a unique section $\LafSec: \CZ\to \CC^{\perp,\triv}\subset \CC$ such that 
$$
\LafSec(z)
\big(
\Ind(\sig)
\big) = z(M,\sig)
$$
for any cuspidal datum $(M,\sig)$. Indeed, the induced representations $\Ind(\sig)$ are either irreducible or admit a filtration given by twisted Steinberg representations and $1$-dimensional characters, which furnishes the proof.
\end{itemize}

On the other hand, when restricted to finite and stable locus, elements in $\CC^{\st,\fin}$ and $\CZ^{\st,\fin}$ can be realized as functions on the Steinberg-Hitchin base $\Fc(F)$. One of our main results in this paper is to show that, for $G=\SL_2$ and $\GL_2$, when restricted to $\CZ^{\st,\fin}$ and $\CC^{\st,\fin,\perp,\triv}$, both $\LafSec$ and $\Laf$ can be realized as explicit integral transforms on the Steinberg-Hitchin base $\Fc(F)$. The integral transforms appear first in an unpublished note of Lafforgue for a different purpose (\cite{lafforguegl2}). We use the name of Lafforgue for the linear map $\Laf$, which has deep representation theoretic significance, because in the case of $\SL_2$, its stable part is given by Lafforgue's integral formula, and in fact, attempting to understand conceptually Lafforgue's mysterious formula was a motivation for this work. 

\begin{thm}\label{thm:stableLaf}
    Let $G=\SL_2$ or $\GL_2$ over a local field $F$ of residual characteristic not equal to two. For any distribution $z\in \CZ^{\rm st,fin}$ lying in the stable Bernstein center and supported on finitely many Bernstein components, there exists a stably invariant function $\RJ_G = \RJ_\Fc\circ \c$ such that $z=\RJ_G \ud g$. Then $\LafSec(z)$ is given as a function $f$ on the Steinberg-Hitchin base by the following formula, where the Fourier transform and $|\cdot|$ are in trace variable,
      \begin{equation} \label{eq:Laf formula}
        \LafSec(z)=\CF_\psi\left(\frac{\CF_{\psi^{-1}}(\RJ_\Fc)}{|.|}\right).
      \end{equation}
\end{thm}
It is not at all obvious that the expression on the right-hand side makes sense. This is what we will prove in the next subsection based on some explicit formulas in the case $\SL_2$ due to Moy and Tadic. The proof of this theorem will be achieved in subsection \ref{subsec:proofsofdescent}.

For general reductive group $G$, we expect that there always exists a distinguished section $\LafSec:\CZ\to \CC$ of the Lafforgue transform $\Laf:\CC\to \CZ$, such that when restricted to $\CZ^{\st,\fin}$, both transforms can be realized by explicit integral transforms on the Steinberg-Hitchin base of $G$. In a forthcoming work, we propose a conjectural formula for $G=\GL_n$ generalizing Theorem \ref{thm:stableLaf}.

\subsection{Supercuspidal component of the Lafforgue transform}

In this subsection, we study the image under $\LafSec$ of the supercuspidal idempotent for $G=\SL_2$. A supercuspidal representation of $G(F)$ is represented by an isolated point in the Bernstein variety $\Omega$. The associated idempotent $\epsilon_\pi\in \CZ$ is defined as the function on $\Omega$ taking the value one at $\pi$ and $0$ on other components. We can express $\epsilon_\pi$ and its Lafforgue transform using the distributional character of the representation $\pi$.  

\begin{theorem}\label{thm:imageLafsupercuspidal}
\begin{enumerate}
	\item The idempotent $\epsilon_\pi$ is given by the formula $$\epsilon_\pi=  \Theta_\pi d_\pi \ud g$$
where the distributional character $\Theta_\pi$ of $\pi$ is a generalized function, $d_\pi$ is the formal degree which depends on the choice of the Haar measure $\ud g$ but the product $d_\pi \ud g$ is independent of that choice. 
	\item As function on the set of regular semisimple conjugacy classes $C'(F)$ we have 
$$\LafSec(\epsilon_\pi)= \Delta {\bf v} {\mathbbm 1}_{\rm ell}  \Theta_\pi$$
where $\Delta$ is the square root of the absolute value of the Weyl discriminant, ${\mathbbm 1}_{\rm ell}$ is the characteristic function of the elliptic regular semi-simple locus, and $\v$ is the inverse of the volume function defined on the elliptic regular semi-simple locus determined by \eqref{eq:volume:def:sl2}, i.e. following \eqref{eq:volume:def:sl2}, for a quadratic field extension $E_\alp$ of $F$ and $e\in E_\alp^1$, $\v(e) = \vol(E^1_\alp,\ome_{T_\alp})^{-1}$. We will call the right-hand side the normalized elliptic character of $\pi$. 

\end{enumerate}
\end{theorem}

\begin{proof}
(1) This is a result of Moy and Tadic \cite[3.1]{Moy-Tadic-Bernstein}.

(2) By definition, $\mathrm{LafSec}(\eps_\pi)$ is the unique element in the cocenter of $G(F)$ that vanishes on all subrepresentations of principal series representations and supercuspidal representations, except takes value one on the supercuspidal representation $\pi$. As a result, $\mathrm{Laf}(\eps_\pi)$ is equal to the orbital integral of normalized matrix coefficient attached to $\pi$. The final identity follows from \cite[Thm.~5.1]{MR1237898} specialized to $M=G$. Notice that $\Del$ and $\v$ show up due to the normalization of the measures introduced in subsection \ref{subsec:measure}.
\end{proof}

\subsection{Discrete series representations and Chebyshev polynomials}\label{subsec:descent:chebyshev}

In this section, we would like to study the analogue of Theorem \ref{thm:imageLafsupercuspidal} over the real field. In particular we would like to highlight the connection between Theorem \ref{thm:descent:sl2} for discrete series representations of $\SL_2(\BR)$ and the Chebyshev polynomials. Let $F=\BR$ and $E=\BC$. For any integer $\ell\geq 1$, let $\Theta_\ell = \theta^\st_\ell$ be the stable discrete series character of weight $\ell$, and let $d_{\ell}$ be its formal degree which is equal to $\ell$. 

By the orthogonality of elliptic discrete series characters, it is immediate to see that Corollary \ref{cor:descent:sl2} is equivalent to the following fact:
\begin{thm}\label{thm:descent:DS:sl2}
The Lafforgue transform of $d_\ell \Theta_\ell$ equals $\v\Del\Theta_\ell^\el$. Equivalently, 
    $$
    \CF_{\psi}
    \bigg(
    \frac{\CF_{\psi^{-1}}(d_\ell \Theta_\ell)}{|\cdot|}
    \bigg) = \v\Del\Theta_\ell^\el.
    $$
Here $\v=2\pi$, and $\Del$ is the square-root of the norm of Weyl discriminant.
\end{thm}

\begin{proof}
Let $\phi_\ell$ be the function on the Steinberg-Hitchin base supported on the trace of elliptic maximal torus of $\SL_2(\BR)$ such that 
$$
\phi_\ell(x) = \frac{e^{i\ell \theta}+e^{-i\ell\theta}}{|e^{i\theta}-e^{-i\theta}|},\quad 
x=e^{i\theta}+e^{-i\theta}.
$$
The Gelfand-Graev character identity \eqref{eq:ggps:gl2:interpret:sl2} shows that (switching $\psi$ to $\psi^{-1}$)
\begin{equation}\label{eq:descent:DS:1}
\Theta_\ell = \wb{\lam}_{\BC/\BR}
\bigg(
\CF_{\psi}
\big(
\eta_\BC\cdot 
\CF_{\psi^{-1}}(\phi_\ell)
\big)
\bigg).
\end{equation}
Plugging \eqref{eq:descent:DS:1} into the statement of Theorem \ref{thm:descent:DS:sl2}, and using the fact that $\eta_\BC(x)|x| = \sgn(x)|x| = x$, Theorem \ref{thm:descent:DS:sl2} can be reformulated as follows:
$$
\CF_\psi\bigg(
\frac{\CF_{\psi^{-1}}(\phi_\ell)}{(\cdot)}
\bigg)= \lam_{\BC/\BR}d_\ell^{-1}\v\Del\Theta^\el_\ell.
$$
Equivalently, 
$$
\phi_\ell = 
\CF_{\psi^{-1}}
\bigg(
(\cdot)
\CF_\psi
\big(
\lam_{\BC/\BR}d_\ell^{-1}\v\Del\Theta_\ell^\el
\big)
\bigg).
$$
Since the Fourier transform sends multiplication by $2\pi ix$ to derivative, the above identity is equivalent to the following 
$$
(2\pi i)d_\ell\cdot \phi_\ell = 
\frac{\ud
\big(
\lam_{\BC/\BR}\v\Del\Theta_\ell^\el
\big)
}{\ud x}
$$
From \cite[Lem.~1.2]{jlgl2}, $\lam_{\BC/\BR} = i$, therefore we are reduced to show 
$$
d_\ell\cdot \phi_\ell = \ell\cdot \phi_\ell=
\frac{\ud
\big(
\Del\Theta_\ell^\el
\big)
}{\ud x}.
$$
Write $x =e^{i\theta}+e^{-i\theta} = 2\cos \theta$, and $\Del(x) = |e^{i\theta}-e^{-i\theta}| = \sqrt{4-x^2}$, then 
$$
\phi_{\ell}(x) = \frac{e^{i\ell\theta}+e^{-i\ell\theta}}{\Del(x)} = \frac{2\cos \ell\theta}{\sqrt{4-x^2}}.
$$
It turns out that $\phi_\ell(x)$ is related to the well-known Chebyshev polynomials. We recall their definitions and basic properties below.

\begin{enumerate}
    \item 
    Chebyshev polynomials of the first kind $\RT_n$ are defined by 
    $$
    \RT_n(\cos \theta) = \cos(n\theta);
    $$
    Chebyshev polynomials of the second kind $\RU_n$ are defined by 
    $$
    \RU_n(\cos\theta) = \frac{\sin (n+1)\theta}{\sin \theta};
    $$

    \item
    Moreover, the derivative of $\RU_{n}(x)$ enjoys the following identity 
    $$
    \frac{\ud \RU_n(x)}{\ud x} = 
    \frac{(n+1)\RT_{n+1}(x)-x\RU_n(x)}{x^2-1}.
    $$
\end{enumerate}
These facts can be deduced by straightforward computation.

Based on the above facts we have
$$
d_\ell \phi_{\ell}(x) = 
\ell\cdot \frac{2\RT_\ell(x/2)}{\sqrt{4-x^2}},\quad |x|<2.
$$
Similarly, from \cite[(1.3.1)]{langlands-singularites-transfert}, 
$$
\Theta_\ell^\el(x) = -\frac{e^{i\ell\theta}-e^{-i\ell\theta}}{e^{i\theta}-e^{-i\theta}} = 
-\frac{\sin \ell\theta}{\sin \theta} = -\RU_{\ell-1}(x/2),\quad |x|<2,
$$
hence we are reduced to show the following identity
\begin{align*}
    2\ell\cdot \frac{\RT_\ell(x/2)}{\sqrt{4-x^2}} = 
    \frac{
   - \ud 
    \bigg(
\sqrt{4-x^2}
\RU_{\ell-1}(x/2)
    \bigg)
    }{\ud x},\quad |x|<2.
\end{align*}
By direct calculation, the right hand side is equal to 
$$
\frac{x}{\sqrt{4-x^2}}\RU_{\ell-1}(x/2)
-\frac{\sqrt{4-x^2}}{2}\RU_{\ell-1}^\p(x/2).
$$
Applying the derivative for Chebyshev polynomials above, we have
$$
\RU^\p_{\ell-1}(x/2) = \frac{\ell\RT_\ell(x/2)-(x/2)\RU_{\ell-1}(x/2)}{(x/2)^2-1}
$$
Therefore we are reduced to show 
$$
2\ell\cdot 
\frac{\RT_\ell(x/2)}{\sqrt{4-x^2}}
=
\frac{x}{\sqrt{4-x^2}}
\RU_{\ell-1}(x/2)
-\frac{\sqrt{4-x^2}}{2}
\bigg(
\frac{\ell\RT_\ell(x/2)-(x/2)\RU_{\ell-1}(x/2)}{(x/2)^2-1}
\bigg).
$$
After multiplying $\sqrt{4-x^2}$ on both sides, it is equivalent to the following identity 
$$
2\ell\RT_\ell(x/2)
=
x
\RU_{\ell-1}(x/2)
-
\bigg(
x\RU_{\ell-1}(x/2)-2\ell\RT_\ell(x/2)
\bigg)
$$
which holds automatically. We thus complete the proof of the theorem.
\end{proof}

\subsection{The image of the stable Lafforgue transform}

In this subsection, we describe the image of $\LafSec$ for $G=\SL_2$ over a non-archimedean local field of odd residual characteristic.

\begin{thm} \label{thm:image Laf}
	The image of $\LafSec: \CZ^{\rm st} \to \CC^{\rm st}$ is the subspace $\CC^{\rm st, \perp triv}$ of $\CC^{\rm st}$ orthogonal to the character of the trivial representation.
	\end{thm}
\begin{proof}
	We only need to establish the theorem for the non-unital subalgebra of finite support $\LafSec: \CZ^{\st,\fin}\to \CC^{\st,\fin}$. 

	First, by Part (2) of Theorem \ref{thm:vanishing:SBC}, for any $z=(\RJ_\Fc\circ \c)\ud^*g\in \CZ^{\st,\fin}$, $\CF_{\psi^{-1}}(\RJ_\Fc)(0)=0$. From here we deduce that $\mathrm{Laf}(z)$ is orthogonal to the trivial character, i.e. the Fourier transform of $\mathrm{Laf}(z)$, as a function on the Steinberg-Hitchin base, vanishes at zero. It follows that the image of $\CZ^{\st,\fin}$ lands in $\CC^{\rm st,fin,\perp triv}$. It remains to show the surjectivity. In order to show the surjectivity, not surprisingly we will use the trace Paley-Wiener theorem (\cite{Bernstein-Deligne-Kazhdan}). By the trace Paley-Wiener theorem, any element in the usual cocenter $\CC^{\fin} = \CH/[\CH,\CH]$ is uniquely characterized by linear functionals $\ell$ on $\Irr(G(F))$ such that for any character $\chi_0$ of $T(F)$, $\chi\in \BX^*(T)\otimes_\BZ \BC\mapsto \ell(\Ind\chi\otimes \chi_0)$ is a regular function. Here $T$ is the diagonal split torus of $G$, $\BX^*(T)$ is the rational character group of $T$, and $\Ind$ is the induced representation functor attached to the characters of $T(F)$. Passing to the stable cocenter, by Theorem \ref{thm:main:SBC}, since stable tempered characters are dense in the space of stable distributions, the stable cocenter can be characterized by linear functionals $\ell$ on $\Irr(G(F))$ that are in the usual cocenter and are constant on tempered local $L$-packets.

	Write $\Irr(G(F)) = \Irr(G(F))_{\scusp}\bigsqcup \Irr(G(F))_{\ind}$, where $\Irr(G(F))_{\scusp}$ are supercuspidal representations of $G(F)$, and $\Irr(G(F))_{\mathrm{ind}}$ are irreducible representations of $G(F)$ whose cuspidal support are not supercuspidal. By linearity, we can divide the discussion to linear functionals on $\Irr(G(F))$ that are supported on $\Irr(G(F))_{\scusp}$ and $\Irr(G(F))_{\ind}$ separately. We first treat the supercuspidal situation.

	By the trace Paley-Wiener theorem and linearity, for regular functions on $\Irr(G(F))$ supported on supercuspidal component only (actually they are indicator functions since $G=\SL_2$), they are given by the linear combination of the orbital integral of matrix coefficients of supercuspial representations. Passing to the stable cocenter, by Theorem \ref{thm:imageLafsupercuspidal} and its stable variant, they are given by the image of the stable Lafforgue transform of stable supercuspidal distribution characters. Hence the image of the stable Lafforgue transform contains all the elements in the stable cocenter supported on supercuspidal component only. 

	It remains to treat $\Irr(G(F))_{\ind}$. The tempered local $L$-packet for irreducible representations are given by (the restriction of) fully induced tempered principal series of $G(F)$, except the Steinberg representation. Hence when restricted to $\Irr(G(F))_{\ind}$, there are no difference between the stable cocenter and the usual cocenter. In particular, up to linear combination, after fixing a ramified (or trivial) character $\chi_0$ of $T(F)$, any linear functional $\ell$ on $\Irr(G(F))_{\ind}$ in the usual (stable) cocenter are determined uniquely by a regular function on $\BX^*(T)$, which in turn is uniquely determined by the same function on the Bernstein component attached to the ramified (or trivial) character $\chi_0$. It follows that the image of the stable Lafforgue transform also contains the (stable) cocenter supported on $\Irr(G(F))_{\ind}$. 

	It follows that we complete the proof of the theorem.
\end{proof}

\begin{rmk}
Similar results hold for $G=\GL_2$, except one need to add the twist by central characters for trivial representations and Steinberg representations.
\end{rmk}

\subsection{Analytic properties of the function $\RJ_\Fc$} \label{subsec:analytic-properties-Jc}

In this subsection, we establish some analytic properties of the function $\RJ_\Fc$ on the Steinberg-Hitchin base, which will be useful in proving the descent formula. These analytic properties are established by explicit calculations. When $F$ is non-archimedean, the calculation is based on explicit formulae of Moy and Tadic.  

\begin{thm}\label{thm:vanishing:SBC}
Let $G=\SL_2$ $($resp. $\GL_2)$ and the residual characteristic of $F$ is other than two, then the following statements hold.

\begin{enumerate}
\item For any $\RJ_G\in \CZ_G^{\st,\fin}$, there is an absolutely integrable function (resp. absolutely integrable function that is rapidly decreasing on the determinant factor) $\RJ_\Fc$ on $\Fc_G(F)\simeq F$ (resp. $\simeq F\times F^\times$), such that $\RJ_G = (\RJ_\Fc\circ \c)\ud^\times g$, where $\c:G\to (G\sslash \Ad(G)) =\Fc_G$ is the Chevalley quotient map given by the trace function (resp. given by the (trace, determinant) function);

\item The Fourier inversion of $\RJ_\Fc$ in trace variable, denoted as $\wh{\RJ}_\Fc = \CF_{\psi^{-1}}(\RJ_\Fc)$, is represented by a globally bounded smooth function vanishing at zero;

\item 
$\frac{\CF_{\psi^{-1}}(\RJ_\Fc)}{|\cdot|}$ is absolutely integrable on $F$. As a result, for any étale quadratic $F$-algebra $K$, $\CF_\psi\big( \frac{\eta_K }{|\cdot|} \CF_{\psi^{-1}}(\RJ_\Fc) \big)$ is globally bounded on $F$.
\end{enumerate}
\end{thm}
\begin{rmk}\label{rmk:descent:localintegrable}
Since the orbital integrals $($normalized by the square-root of the norm of Weyl discriminant$)$ are locally bounded $($\cite[\S 1.8]{beuzart2015local}$)$, Part $(1)$ ensures that the distribution $\RJ_G$ is locally integrable on $G(F)$.
\end{rmk}

Following the computation in \cite[\S 3]{Moy-Tadic-Bernstein}, when $G=\GL_2$ and $F$ is non-archimedean, the Mellin transform of a regular function along the center of $\GL_2(F)$ will cut out a compactly supported locus on the determinant factor. Similar statement is true for $F$ archimedean since the entire functions on the admissible dual are still rapidly decay along the vertical strip when restricted to the determinant factor. Hence we only prove the case for $G=\SL_2$ and leave the details for $G=\GL_2$ to the reader.

We first prove the theorem for $F$ non-archimedean.

\begin{proof}
By \cite[Thm.~2.5]{Moy-Tadic-Bernstein}, any distributions supported on finitely many Bernstein components of $G(F)$ are locally integrable. Moreover, \cite[\S 3]{Moy-Tadic-Bernstein} provides explicit bases for $\CZ_G^{\fin}$ which are presented as follows:
\begin{itemize}
\item Cuspidal component (\cite[\S 3.1]{Moy-Tadic-Bernstein}): For any supercuspidal representation $\pi$ of $G(F)$, a basis is given by $e_\pi = \check{\Theta}_\pi$ where $\check{\Theta}_\pi$ is the contragredient of the supercuspidal character of $\pi$;

\item Regular non-cuspidal component (\cite[\S 3.4]{Moy-Tadic-Bernstein}): For a character $\lam_0:\Fo^\times \to \BC^\times$ with $\lam_0^2\neq 1$, and any integer $n\in \BZ$, up to constant, the basis distribution $e_n$ is supported on split elements with
\begin{align*}
e_{n}(\diag(\vpi^ko,\vpi^{-k}o^{-1})) &= 
\Bigg\{
\begin{matrix}
0 & k\neq \pm n\\
\frac{\lam_0(o^{-1})}{|\vpi^{k}o^2-\vpi^{-k}|} & k=n\\
\frac{\lam_0(o)}{|\vpi^ko^2-\vpi^{-k}|} & k=-n
\end{matrix}\quad \text{if $n\neq 0$}
\\
e_0(\diag(\vpi^ko,\vpi^{-k}o^{-1}))
&=\Bigg\{
\begin{matrix}
0 & k\neq 0 \\
\frac{\lam_0(o)+\lam_0(o^{-1})}{|o^2-1|} & k=0
\end{matrix}
\end{align*}
where $o\in \Fo^\times$;

\item Irregular ramified principal series component (\cite[\S 3.5]{Moy-Tadic-Bernstein}): Let $\lam_0$ be a nontrivial character of $\Fo^\times$ of order two. For any positive integer $n\in \BN$, up to constant, the basis distribution $e_n$ is supported on split elements with 
$$
e_n(\diag(\vpi^ko,\vpi^{-k}o^{-1})) = 
\Bigg\{
\begin{matrix}
0 & k\neq \pm n\\
\frac{\lam_0(o)}{|\vpi^ko^2-\vpi^{-k}|} & k=\pm n
\end{matrix}
$$

\item Unramified principal series component (\cite[\S 3.6]{Moy-Tadic-Bernstein} and corrections in \cite{Moy-Tadic-Bernstein-erratum}): For any non-negative integer $n\geq 0$, up to constant, the basis distribution $e_n = e^\p_n+e^{\p\p}_n$ is given by 
\begin{align*}
e^{\p\p}_{n}(g) 
&= (q-1)(q^n+q^{-n})
\Theta_{\mathrm{St}_{\SL_2}}(g)
\\
&=
\Bigg\{
\begin{matrix}
(q-1)(q^{n}+q^{-n})\big(
\frac{q^{k}+q^{-k}}{|\vpi^{k}o^2-\vpi^{-k}|}-1
\big), & g\text{ conjugate to }\diag(\vpi^{k}o,\vpi^{-k}o^{-1})\\
-(q-1)(q^{n}+q^{-n}), & \text{otherwise}
\end{matrix}
\end{align*}
and $e^{\p}_{n}$ supported on split elements given by
\begin{align*}
e^\p_{n}(\diag(\vpi^{k}o,\vpi^{-k}o^{-1}))
=\frac{q(q+1)}{|\vpi^{k}o^2-\vpi^{-k}|}
(I_{n+k}+I_{n-k}),\quad 
I_l = 
\bigg\{
\begin{matrix}
-q^{-|l|-1}\frac{(q-1)}{(q+1)} & l\neq 0 \\
\frac{2}{q(q+1)} & l=0
\end{matrix}
\end{align*}
\end{itemize}
Hence any distributions in $\CZ^{\st,\fin}_G$ are actually represented by stable functions, and we derive the following explicit formulas for a basis of distributions in $\CZ^{\st,\fin}_G$.

\begin{thm}[\cite{Moy-Tadic-Bernstein}\cite{Moy-Tadic-Bernstein-erratum}]\label{thm:explicitbasisstable}
We have the following basis family of locally integrable stable distributions in $\CZ_G^{\st,\fin}$:
\begin{itemize}
\item Cuspidal component: For any supercuspidal representation $\pi$ of $G(F)$, $b_\pi = $ stable distribution character of $\pi$.

\item Ramified principal series component: Fix a nontrivial character $\lam_0:\Fo^\times \to \BC^\times$. If $\lam_0^2$ is nontrivial, then for any nonzero integer $n\in \BZ$, $b_{\lam_0,n}$ is supported on split elements with 
$$
b_{\lam_0,n}(\diag(\vpi^{k}o,\vpi^{-k}o^{-1})) = 
\Bigg\{
\begin{matrix}
0 & k\neq \pm n\\
\frac{\lam_0(o^{-1})}{|\vpi^{k}o^2-\vpi^{-k}|} & k= n\\
\frac{\lam_0(o)}{|\vpi^ko^2-\vpi^{-k}|} & k=-n
\end{matrix}
$$
If $\lam_0^2$ is nontrivial and $n=0$, or $\lam_0^2$ is trivial, 
$$
b_{\lam_0,n}(\diag(\vpi^{k}o,\vpi^{-k}o^{-1})) = 
\Bigg\{
\begin{matrix}
0 & k\neq \pm n\\
\frac{\lam_0(o^{-1})+\lam_0(o)}{|\vpi^{k}o^2-\vpi^{-k}|} & k= \pm n\\
\end{matrix}
$$

\item Unramified principal series component: For any non-negative integer $n\geq 0$, $b_n = b^\p_n+b^{\p\p}_n$ with 
\begin{align*}
b^{\p\p}_{n}(g) 
&= (q-1)(q^n+q^{-n})
\Theta_{\mathrm{St}_{\SL_2}}(g)
\\
&=
\Bigg\{
\begin{matrix}
(q-1)(q^{n}+q^{-n})\big(
\frac{q^{k}+q^{-k}}{|\vpi^{k}o^2-\vpi^{-k}|}-1
\big), & g\text{ conjugate to }\diag(\vpi^{k}o,\vpi^{-k}o^{-1})\\
-(q-1)(q^{n}+q^{-n}), & \text{otherwise}
\end{matrix}
\end{align*}
and $b^{\p}_{n}$ supported on split elements given by
\begin{align*}
b^\p_{n}(\diag(\vpi^{k}o,\vpi^{-k}o^{-1}))
=\frac{q(q+1)}{|\vpi^{k}o^2-\vpi^{-k}|}
(I_{n+k}+I_{n-k}),\quad 
I_l = 
\bigg\{
\begin{matrix}
-q^{-|l|-1}\frac{(q-1)}{(q+1)} & l\neq 0 \\
\frac{2}{q(q+1)} & l=0
\end{matrix}
\end{align*}
\end{itemize}
\end{thm}
Notice that $|\vpi^{k}o^2-\vpi^{-k}| = \Del(\diag(\vpi^ko,\vpi^{-k}o^{-1}))$ where $\Del$ is the square-root of the norm of the Weyl discrimiant.

Based on Theorem \ref{thm:explicitbasisstable}, we only need to establish Theorem \ref{thm:vanishing:SBC} for each case individually.

For Part (1) of Theorem \ref{thm:vanishing:SBC}, it follows from the above explicit formula. Precisely, based on the above explicit basis, for any $\RJ_G = \big(\RJ_\Fc\circ \c\big)\ud^* g$, $\Del\RJ_\Fc$ is both compactly supported and globally bounded. Notice that for cuspidal component, one can find the explicit supercuspidal character table for $\SL_2(F)$ from \cite[p.6,~p.7]{characterSL2}, which shows that they are supported on compact locus and are globally bounded after multiplying by $\Del$. Hence the absolute integrability of $\RJ_\Fc$ follows from the local integrability of $\frac{1}{\Del}$. As a result $\CF_{\psi^{-1}}(\RJ_\Fc)$ is continuous and globally bounded.

It remains to establish Part (2) and Part (3) of Theorem \ref{thm:vanishing:SBC}.

(I) We first treat the cuspidal component. Since we assume that the residual characteristic of $F$ is not two, following subsection \ref{subsec:ggpsstable:gl2}, up to constant, we may assume that $\RJ_\Fc = \CG\CG_E(\phi_\chi)$ where $E$ is an étale quadratic field extension of $F$ and $\chi:E^1\to \BC^\times$ is a nontrivial character. From the work of Harish-Chandra, we know that $\CG\CG_E(\phi_\chi)$ is locally integrable and smooth over generic locus.

From Remark \ref{rmk:ggps:gl2}, $\CF_{\psi^{-1}}(\RJ_\Fc) =\eta_E \CF_{\psi^{-1}}(\phi_\chi)$. Since $\frac{1}{\Del} = \frac{1}{|\Tr^2-4|^{1/2}}$ is locally integrable on $F$, $\phi_\chi$ is a compactly supported and locally integrable function. Therefore $\CF_{\psi^{-1}}(\phi_\chi)$ is represented by a smooth function on $F$. Moreover, by definition, $\CF_{\psi^{-1}}(\phi_\chi)(0) = \int_{F}\phi_\chi(x)\ud x=\int_{E_1}\chi(e)\ud_{E,1}e = 0$. Hence $\CF_{\psi^{-1}}(\phi_\chi)$ is a smooth function on $F$ vanishing at zero. Since we are working with non-archimedean situation, it indicates that $\CF_{\psi^{-1}}(\phi_\chi)$ vanishes in a neighborhood of zero. In particular, $\frac{\CF_{\psi^{-1}}(\RJ_\Fc)}{|\cdot|}$ is still a smooth function on $F$ that vanishes at zero. This justifies the second part of the theorem.

Write $\wh{\RJ}_\Fc = \CF_{\psi^{-1}}(\RJ_\Fc)$ and $\wh{\RJ}^\p_\Fc = \frac{\wh{\RJ}_\Fc}{|\cdot|}$. By the above argument, for any étale quadratic $F$-algebra $K$, $\eta_K\wh{\RJ}^\p_\Fc$ is locally integrable in a neighborhood of zero and globally bounded, hence $\CF_{\psi}(\eta_K\cdot \wh{\RJ}^\p_\Fc)$ is a well-defined distribution. It remains to show that $\CF_\psi(\eta_K\cdot \wh{\RJ}^\p_\Fc)$ is locally integrable. To see this, since $\frac{1}{\Del} = \frac{1}{|\Tr^2-4|^{1/2}}$, $\phi_\chi$ lies in $L^p(F,\ud x)$ for any $1\leq p<2$. By Hausdorff-Young inequality (\cite[Thm.~4.27]{FGAbharmonic}), $\wh{\RJ}_\Fc = \eta_E\cdot \CF_{\psi^{-1}}(\phi_\chi)$ lies $L^q(F,\ud x)$ with $1/p+1/q = 1$. It follows that by Hölder's inequality, for sufficiently small positive constant $\veps>0$, 
$$
\int_{|\xi|\geq \veps}
\bigg|\frac{\wh{\RJ}_\Fc(\xi)}{|\xi|}
\eta_K(\xi)
\bigg|\ud \xi
\leq 
\bigg(
\int_{|\xi|\geq\veps}
|\wh{\RJ}_\Fc|^q
\ud x
\bigg)^{1/q}
\cdot 
\bigg(
\int_{|\xi|\geq\veps}
\frac{1}{|x|^p}
\ud x
\bigg)^{1/p}
<\infty
$$
for any $1<p<2$ i.e. the Fourier transform of $\wh{\RJ}^\p_\Fc\cdot \eta_K\cdot \mathrm{ch}_{|\xi|\geq\veps}$ is uniformly bounded, where $\mathrm{ch}_{|\xi|\geq \veps}$ is the (unnormalized) characteristic function of $\{\xi\in F\mid |\xi|\geq \veps\}$. On the other hand, due to the integrability of $\wh{\RJ}^\p_\Fc\cdot \eta_K\cdot \mathrm{ch}_{|\xi|\leq \veps}$, the Fourier transform of $\wh{\RJ}^\p_\Fc\cdot \eta_K\cdot \mathrm{ch}_{|\xi|\leq \veps}$ is a globally bounded smooth function. Therefore $\frac{\wh{\RJ}_\Fc}{|\cdot|}$ is the summation of two integrable function which is therefore still integrable. Hence $\CF_{\psi}(\eta_K\cdot \wh{\RJ}^\p_\Fc)$ is a globally bounded function which is clearly locally integrable.

Notice that the above proof also works for the stable distribution characters of the discrete series representations of $G(\BR)$.

(II) Next we treat the ramified principal series component. Write $\RJ_\Fc = b_{\lam_0,n}$. The local integrability of $\RJ_\Fc$ follows from the local integrability of $\frac{1}{\Del}$. Moreover, in this situation, $\RJ_\Fc$ is actually compactly supported on $F$, and the integration of $\RJ_\Fc$ on $F$ can be pulled back to the integration of the character $\lam_0$ on the split torus using \eqref{eq:ggps:measure}. The vanishing of the integral $\int_{F}\RJ_\Fc(\xi)\ud \xi$ follows from the fact that $\lam_0$ is nontrivial. The rest arguments are the same as the cuspidal situation and we omit.

(III) Finally we treat the unramified principal series component. After dividing $b_n$ by $-(q-1)(q^{n}+q^{-n})$, we may write 
\begin{align*}
\RJ_\Fc(\tr(g)) = &
\Bigg\{
\begin{matrix}
\big(
1-\frac{q^{k}+q^{-k}}{|\vpi^{k}o^2-\vpi^{-k}|}
\big)+\frac{\wt{I}_{n+k}+\wt{I}_{n-k}}{|\vpi^{k}o^2-\vpi^{-k}|}, & g=\diag(\vpi^ko,\vpi^{-k}o^{-1})
\\
1 & g\text{ elliptic}
\end{matrix}
\end{align*}
where 
$$
\wt{I}_l = 
\bigg\{
\begin{matrix}
\frac{q^{-|l|}}{q^{n}+q^{-n}} & l\neq 0\\
-\frac{2}{(q-1)(q^{n}+q^{-n})} & l=0
\end{matrix}
$$
We first show that $\RJ_\Fc$ is actually compactly supported. Essentially we only need to show that $\RJ_\Fc$ is compactly supported on split locus. By symmetry let us assume that $k>n$ (or $k<-n$), then $\RJ_\Fc(\tr(\diag(\vpi^{k}o,\vpi^{-k}o^{-1})))$ is equal to
$$
= 
\big(
1-\frac{q^k+q^{-k}}{q^{k}}
\big)
+q^{-k}\big(
\frac{q^{-(n+k)}}{q^{n}+q^{-n}}
+\frac{q^{-(k-n)}}{q^{n}+q^{-n}}
\big)
=-q^{-2k}+q^{-2k}=0.
$$
It follows that $\RJ_\Fc$ is compactly supported on $F$. Moreover it is integrable on $F$ since $\Del \RJ_\Fc$ is globally bounded. Hence $\CF_{\psi^{-1}}(\RJ_\Fc)$ is represented by a smooth function.

Next we show that $\CF_{\psi^{-1}}(\RJ_\Fc)(0) = \int_{\xi\in F}\RJ_\Fc(\xi)\ud \xi =0$. The proof follows from tedious calculation. We first explain our idea for $n=0$. Using the above formula, when $n=0$
$$
\RJ_\Fc(\tr(g)) = 
\bigg\{
\begin{matrix}
1-\frac{2q}{(q-1)|o^2-1|} & g=\diag(o,o^{-1}),o\in \Fo^\times\\
1 & g\text{ elliptic}
\end{matrix}
$$
We can break $\RJ_\Fc$ as $\RJ_\Fc = \RJ_{\Fc,1}+\RJ_{\Fc,2}$ with $\RJ_{\Fc,1}$ is identically $=1$ on compact elements, and hence is equal to the characteristic function of $\Fo\subset F$, and $\RJ_{\Fc,2}$ is supported on (compact) split elements given by 
$$
\RJ_{\Fc,2}(\tr(g)) = 
\bigg\{
\begin{matrix}
-\frac{2q}{(q-1)|o^2-1|} & g=\diag(o,o^{-1}),o\in \Fo^\times\\
0 & g\text{ elliptic}
\end{matrix}
$$
Therefore after descending down to the trace variable and with the help of \eqref{eq:ggps:measure}, 
\begin{align*}
\int_{\xi\in F}
\RJ_{\Fc}(\xi)\ud \xi =& 
\int_{|\xi|\leq 1}
\RJ_{\Fc,1}(\xi)\ud \xi
-\int_{\xi=o+o^{-1}}
\frac{2q}{(q-1)|o^2-1|}\ud \xi
\\
=&\vol(\Fo)-
\frac{q}{(q-1)}
\int_{o\in \Fo^\times}
\ud^\times o
=1-\frac{q}{q-1}(1-q^{-1}) = 0.
\end{align*}
It follows that we complete the vanishing statement for $n=0$.

Next let us assume that $n\geq 1$. As the case of $n=0$, we break $\RJ_\Fc = \RJ_{\Fc,1}+\RJ_{\Fc,2}$ where $\RJ_{\Fc,1}$ is supported on compact elements given by 
$$
\RJ_{\Fc,1}(\tr(g)) = 
\bigg\{
\begin{matrix}
(1-\frac{2}{|o^2-1|})+\frac{2\wt{I}_n}{|o^2-1|},  & g=\diag(o,o^{-1})\\
1, & g \text{ elliptic}
\end{matrix}
$$
and $\RJ_{\Fc,2}$ is supported on non-compact elements given by 
$$
\RJ_{\Fc,2}(\tr(g)) = 
\bigg\{
\begin{matrix}
\big(
1-\frac{q^{k}+q^{-k}}{|\vpi^{k}o^2-\vpi^{-k}|}
\big)+\frac{\wt{I}_{n+k}+\wt{I}_{n-k}}{|\vpi^{k}o^2-\vpi^{-k}|} &  g=\diag(\vpi^ko,\vpi^{-k}o^{-1}),-n\leq k\leq n,k\neq 0\\
0  & g \text{ elliptic or }g=\diag(o,o^{-1})
\end{matrix}
$$
Explicitly, for $1\leq k\leq n-1$ and $g=\diag(\vpi^ko,\vpi^{-k}o^{-1})$,
$$
\RJ_{\Fc,2}(\tr(g)) = 
-q^{-2k}+\frac{q^{-k}\big(q^{-(n+k)}+q^{-(n-k)}\big)}{q^{n}+q^{-n}}
=
\frac{-q^{-2k+n}+q^{-n}}{q^{n}+q^{-n}}
$$
For $k=n$ and $g=\diag(\vpi^no,\vpi^{-n}o)$,
$$
\RJ_{\Fc,2}(\tr(g)) = 
-q^{-2n}
+q^{-n}
\bigg(
\frac{q^{-2n}}{q^{n}+q^{-n}}-\frac{2}{(q-1)(q^{n}+q^{-n})}  
\bigg).
$$
Therefore 
\begin{align*}
\int_{\xi\in F}
\RJ_{\Fc,2}(\xi)\ud \xi 
=&
\int_{\xi = t+t^{-1}}
\RJ_{\Fc,2}(\xi)\ud \xi 
=
\frac{1}{2}
\int_{t\in F^\times}
\RJ_{\Fc,2}(\diag(t,t^{-1}))\Del(\diag(t,t^{-1}))\ud^\times t
\\
=&
\frac{1}{2}
\sum_{k=-n,k\neq 0}^{n}
\int_{o\in \Fo^\times}
\RJ_{\Fc,2}(\diag(\vpi^{k}o,\vpi^{-k}o^{-1}))
\Del(\diag(\vpi^{k}o,\vpi^{-k}o^{-1}))\ud^\times o
\\
=&
\vol(\Fo^\times)
\cdot
\sum_{k=1}^{n}
\RJ_{\Fc,2}(\diag(\vpi^{k},\vpi^{-k}))
\Del(\diag(\vpi^{k},\vpi^{-k})).
\end{align*}
Notice that for $k\geq 1$, $\Del(\diag(\vpi^{k}o,\vpi^{-k}o^{-1})) = |\vpi^{k}o-\vpi^{-k}o^{-1}| = q^{k}$. Moreover
\begin{align*}
&\vol(\Fo^\times)
\cdot 
\RJ_{\Fc,2}(\diag(\vpi^{n},\vpi^{-n}))
\Del(\diag(\vpi^{n},\vpi^{-n})) 
\\
=& 
(1-q^{-1})
\bigg(
-q^{-2n}+q^{-n}
\big(
\frac{q^{-2n}}{q^{n}+q^{-n}}
-\frac{2}{(q-1)(q^{n}+q^{-n})}
\big)
\bigg)q^{n},
\end{align*}
and 
\begin{align*}
&\vol(\Fo^\times)
\cdot
\sum_{k=1}^{n-1}
\RJ_{\Fc,2}(\diag(\vpi^{k},\vpi^{-k}))
\Del(\diag(\vpi^{k},\vpi^{-k}))
\\
=&
(1-q^{-1})
\sum_{k=1}^{n-1}
\bigg(
\frac{q^{-n}-q^{n-2k}}{q^{n}+q^{-n}}
\bigg)q^{k}
=
\frac{(1-q^{-1})}{(q^{n}+q^{-n})}
\sum_{k=1}^{n}
\bigg(
q^{-n+k}-q^{n-k}
\bigg)
\\
=&
\frac{(1-q^{-1})}{(q^{n}+q^{-n})}
\cdot 
\bigg(
\frac{q^{1-n}-q-1+q^{n}}{1-q}
\bigg)
\end{align*}
It follows that 
\begin{align*}
\int_{\xi\in F}\RJ_{\Fc,2}(\xi)\ud \xi =&
(1-q^{-1})
\bigg(
-q^{-2n}+q^{-n}
\big(
\frac{q^{-2n}}{q^{n}+q^{-n}}
-\frac{2}{(q-1)(q^{n}+q^{-n})}
\big)
\bigg)q^{n}
\\
&+
\frac{(1-q^{-1})}{(q^{n}+q^{-n})}
\cdot 
\bigg(
\frac{q^{1-n}-q-1+q^{n}}{1-q}
\bigg)
\\
=&
\frac{(1-q^{-1})}{(q^n+q^{-n})}
\bigg(
\frac{q^{-n+1}+q^{n}}{1-q}
\bigg).
\end{align*}
Similarly, following the same idea as the case for $n=0$, 
\begin{align*}
\int_{\xi\in F}
\RJ_{\Fc,1}(\xi)\ud \xi =& 
\int_{|\xi|\leq 1}
\RJ_{\Fc,1}(\xi)\ud \xi = 
\int_{\xi\in \Fo}\ud\xi
+
\int_{\xi = o+o^{-1}}
\frac{2\wt{I}_n-2}{|o^2-1|}\ud\xi
\\
=&
1+\int_{o\in \Fo^\times}
\big(\frac{q^{-n}}{q^{n}+q^{-n}}-1\big)\ud^\times  o
=1-\frac{q^n}{q^{n}+q^{-n}}\vol(\Fo^\times) = 
1-\frac{q^{n}(1-q^{-1})}{q^{n}+q^{-n}}
\\
=&
\frac{(1-q^{-1})}{(q^{n}+q^{-n})}
\bigg(
\frac{q^{-n+1}+q^{n}}{q-1}
\bigg)
.
\end{align*}
It follows that 
$$
\int_{\xi\in F}
\RJ_{\Fc}(\xi)\ud \xi = 
\int_{\xi\in F}
\big(
\RJ_{\Fc,1}(\xi)+\RJ_{\Fc,2}(\xi)
\big)
\ud \xi = 0
$$
and hence $\CF_{\psi^{-1}}(\RJ_\Fc) = 0$. The rest argument is the same as cuspidal situation and we omit.
\end{proof}

\subsubsection{Archimedean variant}

Now we establish Theorem \ref{thm:vanishing:SBC} for $F$ archimedean and $G=\SL_2$. When $F=\BR$ and $\RJ_G = \big(\RJ_\Fc\circ \c\big)\ud^*g$ being equal to the stable discrete series character of $G(F)$, using the explicit character formula from \cite[1.3.1]{langlands-singularites-transfert}, we see that $\RJ_\Fc$ is absolutely integrable. The rest argument are the same as the proof of the cuspidal component in non-archimedean case. Notice that unlike the non-archimedean case where the vanishing of $\wh{\RJ}_\Fc$ at zero implies that $\wh{\RJ}_\Fc$ vanishes in a neighborhood of zero, for archimedean situation we only know that $\wh{\RJ}_\Fc$ vanishes at zero, while its derivatives might no longer vanish at zero.

It remains to treat $\RJ_G\in \CZ^{\st,\fin}_G$ supported on finitely many principal series component. 

\begin{proof}

We first treat the case when $F=\BC$. The tempered dual of $G(\BC)$ is parametrized by $\BZ\times i\BR$. For $(n,\lam)\in \BZ\times i\BR$, the corresponding tempered principal series representation $\pi_{n,\lam}\simeq \pi_{-n,-\lam}$ has distribution character given by 
$$
\Theta_{\lam,n}(\diag(e^{t}e^{i\theta},e^{-t}e^{-i\theta})) = 
\frac{e^{\lam t}e^{in\theta}+e^{-\lam t}e^{-in\theta}}{
|e^te^{i\theta}-e^{-t}e^{-i\theta}|
},\quad t\in \BR, \theta\in [0,2\pi).
$$
Here $e^te^{i\theta}\in \BC^\times \simeq \BR_{>0}\times \BS^1$. Any fixed pairs $(n,\lam)$ and $(-n,-\lam)$ uniquely determine a fixed connected component of the tempered dual of $G(\BC)$. By Weyl integration formula, for a test function $f$ on $G(\BC)$, 
\begin{equation}\label{eq:generalization}
\Theta_{\lam,n}(f) = 
\int_{t\in \BR}
\int_{\theta\in [0,2\pi)}
e^{in\theta}
e^{\lam t}
\RF_f(e^te^{i\theta})
\frac{\ud \theta}{2\pi}
\ud t
\end{equation}
where $\RF_f(e^te^{i\theta})$ is the orbital integral function (normalized by the square-root of Weyl discriminant) given by $= \Del(\diag(e^te^{i\theta},e^{-t}e^{-i\theta}))\int_{g\in G(\BC)/T(\BC)}f(g^{-1}\diag(e^te^{i\theta},e^{-t}e^{-i\theta}) g)\ud g$ and $T\subset G$ is the diagonal split torus. 
By \cite[p.123]{ggps}, up to constant, the Plancherel formula for $G(\BC)$ reads as follows:
$$
\del_e(g) = 
\sum_{n\in \BZ}
\int_{\BR}
\Theta_{ix,n}(g)(n^2+x^2)\ud x.
$$
Following the same idea as \cite[\S 3.4-\S 3.5]{Moy-Tadic-Bernstein}, for a fixed integer $n\in \BZ$ corresponding to components of tempered dual $\{n,-n\}\times \BR$, we are going to calculate the invariant distribution given by an entire function $\xi(\lam)$ that is rapidly decreasing function in vertical strip. Notice that when $n=0$, we need to assume that $k$ is even since the regular functions need to be invariant under the involution $(n,\lam)\mapsto (-n,-\lam)$. 

For notational convenience, we switch $\lam$ to $i\lam$ with $\lam\in \BR$. For any $n\in \BZ$ and $k\geq 0$, consider the invariant distribution $\FD_{n,\xi}$ on $G(\BC)$ that is given by 
\begin{align*}
\FD_{n,\xi}*f(1) =& 
\int_{\BR}
\bigg(
\xi(\lam)
\Theta_{i\lam,n}(f)
+
\xi(-\lam)
\Theta_{i\lam,-n}(f)
\bigg)
(n^2+\lam^2)\ud \lam
\\
=&
\bigg\{
\begin{matrix}
2
\int_{\BR}
\xi(\lam)
\Theta_{i\lam,n}(f)
(n^2+\lam^2)\ud \lam & (n\neq 0)\text{ or }(n=0,k\text{ even})
\\
0 & n=0, k \text{ odd}
\end{matrix}
\end{align*}
with $\xi(\lam)$ entire on $\BC$ and rapidly decay in bounded horizional strip (since we switch $\lam$ to $i\lam$). 
Using the formula \eqref{eq:generalization}, 
$$
\int_{\BR}
\xi(\lam)
\Theta_{i\lam,n}(f)
(n^2+\lam^2)\ud \lam
=
\int_{\BR}
\xi(\lam)
(n^2+\lam^2)
\ud \lam
\int_{t\in \BR}
\int_{\theta\in [0,2\pi)}
e^{in\theta}
e^{i\lam t}
\RF_f(e^te^{in\theta})
\frac{\ud\theta}{2\pi}\ud t.
$$
In other words, $\FD_{n,\xi}$ represents the following tempered distribution on $\Fc_G(F) = F$:
\begin{num}
\item\label{num:SBC:complex} As a tempered distribution on the trace variable $\in F = \BC$, $\FD_{n,\xi}$ represents the following tempered distribution: For any Schwartz function $\phi\in \CS(\BC)$, 
$$
\FD_{n,\xi}(\phi) = 
\int_{\BR}
\xi(\lam)
(n^2+\lam^2)
\ud \lam
\int_{t\in \BR}
\int_{\theta\in [0,2\pi)}
e^{in\theta}
e^{i\lam t}
\phi(e^te^{in\theta}+e^{-t}e^{-in\theta})
\frac{\theta}{2\pi}\ud t
$$
i.e. after pulling back $\phi$ along the trace map to get a Schwartz function $\phi(e^te^{in\theta}+e^{-t}e^{-in\theta})$ on $T(\BC)$, $\FD_{n,\xi}$ acts on $\phi$ via the above formula.
\end{num}
The integration in $\theta\in [0,2\pi)$ is compact and won't affect the analytical issue. We focus on $t\in \BR$. After changing variable, essentially we are looking at the distribution 
$$
\phi\in \CS(\BR)\mapsto 
\int_{\BR}
\xi(\lam)(n^2+\lam^2)
\ud \lam
\int_0^\infty|x|^{i\lam}
\phi(x+x^{-1})\ud^*x.
$$
Now by Mellin inversion (\cite[Defin.~2.2]{jiang2021certain}, \cite[Thm.~4.3]{igusa_higher}), 
$$
x\mapsto \int_{\BR}|x|^{i\lam}\xi(\lam)(n^2+\lam^2)\ud \lam
$$
represents a Schwartz function $f_\xi(x)$ on $\CS(\BR^\times)$, which is rapidly decreasing at both $0$ and $\infty$. Hence the distribution becomes
\begin{align}\label{eq:archimedean:vanishing:PScomplex}
\int_{0}^\infty
f_{\xi}(x)\phi(x+x^{-1})\ud^\times x = 
\int_{|y|>2}
\frac{f_\xi(\frac{-y+\sqrt{y^2-4}}{2})}{\sqrt{y^2-4}}
\phi(y)\ud y
\end{align}
Notice that the function $\frac{\mathbbm{1}_{|y|>2}f_\xi(\frac{-y+\sqrt{y^2-4}}{2})}{|\sqrt{y^2-4}|}$ is integrable near $y=2$, and smooth of rapidly decay near infinity. Hence its Fourier transform is represented by a smooth function that is globally bounded. To show that the Fourier transform vanishes at zero, Finally to show that the Fourier transform vanishes at zero, it suffices to see that in \eqref{num:SBC:complex}, if we plug $\phi$ to be the constant function, then for $n\neq 0$, the following integral vanishes identically
$$
\int_{0}^{2\pi}e^{in\theta}\ud \theta = 0.
$$
For $n=0$, we are reduced to show that 
$$
\int_{t\in \BR}
e^{i\lam t}\ud t
\int_\BR
\xi(\lam)\lam^2\ud \lam=0.
$$
But it follows from Fourier inversion directly. In conclusion we finished the proof for Part (2) of Theorem \ref{thm:vanishing:SBC} for this situation. For Part (3), following the same idea as before, up to the compact integration in $e^{i\theta}$, we only need to observe that in the equation \eqref{eq:archimedean:vanishing:PScomplex}, after multiplying $\sqrt{|y^2-4|}$, which is the square-root of Weyl discriminant, we get 
$$
\mathbbm{1}_{|y|>2}
f_\xi\big(
\frac{-y+\sqrt{y^2-4}}{2}
\big)
$$
that is globally bounded. Hence Part (3) of Theorem \ref{thm:vanishing:SBC} follows from the same argument as the Hausdorff-Young inequality in the non-archimedean case.

The case when $F=\BR$ can be treated in a similar vein as $F=\BC$ situation, and we give a brief sketch.
Up to constant, the Plancherel formula for $G(\BR)$ reads as follows (\cite[p.123]{ggps})
$$
\del_e(g) = 
(\mathrm{Dis})+
\int_{\lam\in \BR}
\Theta_{+,i\lam}(g)\lam \mathrm{tanh}\big(\frac{\pi \lam}{2}\big)
\ud \lam
+
\int_{\BR}
\Theta_{-,i\lam}
(g)
\lam\mathrm{coth}\big(\frac{\pi \lam}{2}\big)
\ud \lam
$$
where $(\mathrm{Dis})$ is the discrete series part, and for $a\in \BR^\times$, the distribution character is given as follows
$$
\Theta_{\pm ,i\lam}(\diag(a,a^{-1})) = 
\frac{
|a|^{i\lam}\sgn(a)^{n_\pm}+|a|^{-i\lam}\sgn(a)^{n_\pm}
}{|a-a^{-1}|}
$$
with $n_+ = 0$ and $n_- = 1$. It suffices to treat $\RJ_G =\big(\RJ_\Fc\circ \c\big)\ud^*g\in \CZ_G^{\st,\fin}$ supported on each principal series component individually. 

Fix an even entire function $\xi(\lam)$ on $\BC$ that is rapidly decay in bounded horizontal vertical strips. We are studying the following distributions on $f\in \CC^\infty_c(G(\BR))$
\begin{align*}
\FD_{+,\xi}(f) &= 
\int_{\lam\in \BR}
\Theta_{+,i\lam}(f)
\xi(\lam)
\lam 
\tanh\big(
\frac{\pi \lam}{2}
\big)\ud \lam,
\\
\FD_{-,\xi}(f) &= 
\int_{\lam\in \BR}
\Theta_{-,i\lam}(f)
\xi(\lam)
\lam 
\coth\big(
\frac{\pi \lam}{2}
\big)\ud \lam
\end{align*}
Following the complex case, let $\RF_f(\pm e^t) = 
\Del(\pm\diag(e^t,e^{-t}))
\int_{g\in G(\BR)/T(\BR)}
f(\pm g^{-1}
\diag(e^t,e^{-t})g)
\ud g
$
be the normalized orbital integral where $T\subset G$ is the diagonal split torus. Then 
\begin{align*}
\FD_{+,\xi}(f) &= 
\int_\BR
\xi(\lam)
\lam\tanh(\lam)
\ud \lam
\int_{t\in \BR}
e^{i\lam t}
\big(\RF_f(e^t)+\RF_f(-e^t)\big)\ud t
\\
\FD_{-,\xi}(f) &= 
\int_\BR
\xi(\lam)
\lam\coth(\lam)
\ud \lam
\int_{t\in \BR}
e^{i\lam t}
\big(\RF_f(e^t)-\RF_f(-e^t)\big)\ud t
\end{align*}
In other words, 
\begin{itemize}
\item 
As a tempered distribution on the trace variable $\in F=\BR$, $\FD_{+,\xi}$ represents the following tempered distribution: For any Schwartz function $\phi\in \CS(\BR)$,
\begin{align*}
\FD_{+,\xi}(\phi) =& 
\int_\BR
\xi(\lam)
\lam\tanh(\lam)
\ud \lam
\int_{t\in \BR}
e^{i\lam t}
\big(\phi(e^t+e^{-t})+\phi(-e^t-e^{-t})\big)\ud t
\\
=&
\int_\BR
\xi(\lam)
\lam\tanh(\lam)
\ud \lam
\int_{x\in \BR}
|x|^{i\lam}
\phi(x+x^{-1})\ud^*x
\end{align*}

\item 
As a tempered distribution on the trace variable $\in F=\BR$, $\FD_{-,\xi}$ represents the following tempered distribution: For any Schwartz function $\phi\in \CS(\BR)$,
$$
\FD_{-,\xi}(\phi) = 
\int_\BR
\xi(\lam)
\lam\coth(\lam)
\ud \lam
\int_{x\in \BR}
|x|^{i\lam}
\sgn(x)\phi(x+x^{-1})\ud^*x
$$
where $\sgn(x)$ is the sign character.
\end{itemize}

Following the same analytical estimations in the complex case, the results follow from the following two facts:

\begin{itemize}
    \item For the tempered distribution $\FD_{+,\xi}$, the Mellin transform of 
    $$
    \lam\in \BC\mapsto \xi(\lam)\lam \tanh(\lam)
    $$
    represents a Schwartz function on $\CS(\BR^\times)$. Moreover, $\xi(\lam)\lam \tanh(\lam)$ vanishes at zero which proves the vanishing statement in Part (2) of Theorem \ref{thm:vanishing:SBC}. The rest analytical estimations are the same as complex situation and we omit;

    \item For the tempered distribution $\FD_{-,\xi}$, the Mellin transform of 
    $$
    \lam\in \BC\mapsto \xi(\lam)\lam \coth(\lam)
    $$
    represents a Schwartz function on $\CS(\BR^\times)$ that is even (notice that $\lam\coth(\lam)$ is still an entire function). Moreover, after multiplying by the sign function $\sgn(x)$, the resulting function is an odd function on $\BR$ whose integration on $\BR$ automatically vanishes. This proves the desired vanishing property in 
    Part (2) of Theorem \ref{thm:vanishing:SBC}. The rest analytical estimations are the same as complex situation and we omit.
\end{itemize}
\end{proof}

% !TEX root = luo-ngo.tex

\section{Calculation of the descent}\label{sec:descent}

In this section, we calculate the action of an element in the stable Bernstein center of $\SL_2(F)$ or $\GL_2(F)$ on a dihedral representation over a local field $F$ whose residual characteristic is not equal to two. Each element $\RJ_G\in \CZ^{\rm st}$ acts on the dihedral representation $\CW_E(\chi)$ associated with a quadratic extension $E/F$ and a character $\chi$ of $E^1$ as a scalar $\gamma(\RJ_G,\CW_E(\chi))$. From the integral formula expressing the scalar $\gamma(\RJ_G,\CW_E(\chi))$, we derive Theorem \ref{thm:stableLaf} providing an integral formula for the stable Lafforgue transform.  We also establish supplementary results for the image of the Lafforgue transform for $G=\SL_2$.

\subsection{Descent formula}\label{subsec:proofsofdescent}
In this subsection, we will state and prove the following descent formula. We prove it for $G=\SL_2$. The case for $G=\GL_2$ is similar and we leave it to the reader.

\begin{thm}\label{thm:descent:sl2}
\begin{enumerate}
	\item
Let $G=\SL_2$. Let $\RJ_G = (\RJ_\Fc\circ \c)\ud^*g\in \CZ_G^{\st,\fin}$. Then $\RJ_G$ acts on the dihedral representation $\CS(\chi)\subset \CW_E(\chi)$ associated with an étale quadratic $F$-algebra $E$ and a character $\chi$ of $E^1$ via the scalar $\gam(\RJ_G,\CW_E(\chi))$ given by the formula
\begin{equation} \label{eq:descent:sl2}
\gam(\RJ_G,\CW_E(\chi)) = 
\lam_{E/F}
\int_{E^1}
\CF_{\psi}
\bigg(
\CF_{\psi^{-1}}(\RJ_\Fc)\frac{\eta}{|\cdot|}
\bigg)(\tr(e))
\chi(e)\ud_{E^1}e.
\end{equation}

	\item
Let $G=\GL_2$. Let $\RJ_G = (\RJ_\Fc\circ \c)\ud^*g\in \CZ^{\st,\fin}_G$. Then $\RJ_G$ acts on the dihedral representation $\CW_E(\chi)$ associated with an étale quadratic $F$-algebra $E$ and a character $\chi$ of $E^\times$ via the scalar $\gam(\RJ_G,\CW_E(\chi))$ given by the formula
\begin{equation}\label{eq:descent:gl2}
\gam(\RJ_G,\CW_E(\chi)) = 
\lam_{E/F}
\int_{E^\times}
\CF_{\psi}
\bigg(
\CF_{\psi^{-1}}(\RJ_\Fc)\frac{\eta}{|\cdot|}
\bigg)(\tr(e))
\chi(e)|\Nr(e)|^{-\frac{1}{2}}\ud_{E^\times}e.
\end{equation}
\end{enumerate}
\end{thm}

Based on the Gelfand-Graev character formula \eqref{eq:ggps:gl2:interpret:sl2} and the self-adjointness of the Fourier transform, the following corollary holds, which confirms Theorem \ref{thm:stableLaf}.

\begin{cor}\label{cor:descent:sl2}
    With the above notation the following identity holds
    $$
\gam(\RJ_G,\CW_E(\chi))
=\int_{\xi\in \Fc_G(F)}
\CF_{\psi}
\bigg(
\frac{\CF_{\psi^{-1}}(\RJ_\Fc)}{|\cdot|}
\bigg)(\xi)
\theta^\st_\chi(\xi)
\ud \xi.
    $$
Here $\ud \xi$ is the additive Haar measure introduced in subsection \ref{subsec:measure} when $G=\SL_2$, and is the product of additive Haar measure on trace variable with the multiplicative Haar measure on determinant variable when $G=\GL_2$.
\end{cor}

We will use three different approaches to establish the theorem for different types of dihedral representations:

\begin{enumerate}
    \item 
    
    For $F$ non-archimedean and $\chi$ a nontrivial character of $E^1$, so that $\CW_E(\chi)$ is supercuspidal, we establish Theorem \ref{thm:descent:sl2} using the so called \emph{Bruhat regularization};

    \item For principal series representations, we establish Theorem \ref{thm:descent:sl2} through studying the action of $\RJ_G$ on a dense subspace of the \emph{universal principal series model};
     
    \item Finally, for $F$ the real field and $\chi$ a nontrivial character of $E^1$, so that $\CW_E(\chi)$ is a discrete series representation, we establish Theorem \ref{thm:descent:sl2} through connecting the stable discrete series character with the \emph{Chebyshev polynomials}.
\end{enumerate}
The discrete series representations are subrepresentations of principal series representations, hence a prior the third situation follows from the second. However we still find it   fascinating to highlight the connection with the Chebyshev polynomials.

\subsection{Supercuspidal representations and Bruhat regularization}\label{subsec:descent:supercuspidal}

Let $F$ be a non-archimedean local field, $E$ a quadratic extension of $F$, and $\chi:E^1\to \BC^\times$ a nontrivial character. In this case, $\pi=\CW_E(\chi)$ can be identified as the subspace of functions in $\CW_E$ that are $\chi$-equivariant. Moreover, the matrix coefficients for $\CW_E(\chi)$ are compactly supported. Therefore for $\RJ_G = \big(\RJ_\Fc\circ \c\big)\ud^*g\in \CZ_G^{\st,\fin}$, the following limit  
\begin{equation}\label{eq:descent:supercuspidal:Bernsteinreg}
\lim_{n\mapsto \infty}
\int_{g\in G(F)}
\big(
\RJ_G*e_{K_n}
\big)(g)
\CW_E(\chi)(g)v\ud^*g,\quad v\in \CW_E(\chi)
\end{equation}
is convergent to the following absolutely convergent integral 
$$
\int_{g\in G(F)}
\RJ_\Fc\circ \c(g)
\CW_E(\chi)(g)v\ud^*g.
$$
Here we use the fact that $\RJ_G$ is locally integrable from Remark \ref{rmk:descent:localintegrable}. By the absolute convergence, we can restrict the integration to the open dense Bruhat cell based on \eqref{eq:notation:sl2mes:2} and the following lemma.

\begin{lem}\label{lem:ell:sl2:coordinates}
The following identity holds
$$
\RJ_G
\bigg(
\begin{pmatrix}
1 & \\
v & 1
\end{pmatrix}
\begin{pmatrix}
t & \\
 & t^{-1}
\end{pmatrix}
\begin{pmatrix}
1 & u\\
 & 1
\end{pmatrix}
\bigg)
=\RJ_\Fc(t+t^{-1}+tuv) |t|^2 \ud^*  t
\ud u \ud v
$$
\end{lem}
\begin{proof}
    It follows from \eqref{eq:notation:sl2mes:2} and the following identity 
    $$
\tr
\bigg(
\begin{pmatrix}
1 & \\
v & 1
\end{pmatrix}
\begin{pmatrix}
t & \\
 & t^{-1}
\end{pmatrix}
\begin{pmatrix}
1 & u\\
 & 1
\end{pmatrix}
\bigg)
=t+t^{-1}+tuv.
    $$
\end{proof}
It follows that for any $\phi\in \CW_E(\chi)$ and $z\in E$, 
\begin{align*}
\gam(\RJ_G,\CW_E(\chi))\phi(z)=
\int_{u,v\in F,t\in F^\times}
&\RJ_\Fc(t+t^{-1}+tuv)
\\
&\bigg(
\CW_E(\chi)
\bigg(
\begin{pmatrix}
1 & \\
v & 1
\end{pmatrix}
\begin{pmatrix}
t & \\
 & t^{-1}
\end{pmatrix}
\begin{pmatrix}
1 & u\\
 & 1
\end{pmatrix}
\bigg)
\phi
\bigg)(z)
|t|^2\ud^*t\ud u\ud v.
\end{align*}

Furthermore, by Lemma \ref{lem:ell:sl2:3}, the following corollary holds.

\begin{cor}\label{cor:ell:sl2:coordinates}
With the above notation the following identity holds for any $\phi\in \CW_E(\chi)$,
\begin{align*}
\gam(\RJ_G,\CW_E(\chi))\phi(z)=
\lam_{E/F}
\int_{u,v\in F,t\in F^\times}
&\RJ_\Fc(t+t^{-1}+tuv)
\frac{\eta(-v)}{|v|}
\int_{x\in E}
\psi
\bigg(
\frac{\Nr(z-x)}{v}
\bigg)
\\
&
\eta(t)|t|
\psi\big(ut^2\Nr(x)\big)\phi(tx)\ud_Ex
|t|^2\ud^*t\ud u\ud v.
\end{align*}
\end{cor}

In the following, in order to establish Theorem \ref{thm:descent:sl2}, we introduce the so called \emph{Bruhat regularization} for the above right hand side integral. We outline the procedure in the following steps:

\begin{enumerate}
    \item Let $\Lam = \BN^3$ be the index set parameterizing triples of natural numbers. Fix a family of test functions $\{\mathbbm{1}_n\}_{n\geq 1}\subset \CC^\infty_c(F^\times)$ such that $\lim_{n\mapsto \infty}\mathbbm{1}_n$ tends to the identity function on $F^\times$. For $\lam = (i,j,k)\in \BN^3$, let 
    $$
    \mathbbm{1}_\lam(v,t,u) = \mathbbm{1}_i(v)
    \mathbbm{1}_j(t)\mathbbm{1}_k(u).
    $$

    \item 
By absolute convergence, Corollary \ref{cor:ell:sl2:coordinates} can be rewritten as follows 
\begin{align*}
\gam(\RJ_G,\CW_E(\chi))
\phi(z) = 
\lam_{E/F}
\lim_{i\mapsto \infty}
\lim_{j\mapsto \infty}
\lim_{k\mapsto \infty}
\int_{u,v\in F,t\in F^\times}
&
\mathbbm{1}_\lam(v,t,u)
\RJ_{\Fc}(t+t^{-1}+tuv)
\frac{\eta(-v)}{|v|}
\\
\int_{x\in E}&
\psi\bigg(
\frac{\Nr(z-x)}{v}
\bigg)
\eta(t)
|t|
\psi\big(
ut^2\Nr(x)
\big)
\phi(tx)
\\
&\ud_E x
|t|^2
\ud^*t\ud u\ud v.
\end{align*}

\item 
For $\Re(s_1),\Re(s_2)\geq 0$ and $z\in E^\times$, consider the following complex deformation of above integral
\begin{align*}
\RI_{\phi,z}(\lam,s_1,s_2)
=
\lam_{E/F}
\int_{u,v\in F,t\in F^\times}
&
\mathbbm{1}_\lam(v,t,u)
|v|^{s_1}|t|^{s_2}
\RJ_{\Fc}(t+t^{-1}+tuv)
\frac{\eta(-v)}{|v|}
\\
\int_{x\in E}&
\psi\bigg(
\frac{\Nr(z-x)}{v}
\bigg)
\eta(t)
|t|
\psi\big(
ut^2\Nr(x)
\big)
\phi(tx)
\\
&\ud_E x
|t|^2
\ud^*t\ud u\ud v
\end{align*}
and set 
$$
\RI_{\phi,z}(s_1,s_2) = 
\lim_{i\mapsto \infty}
\lim_{j\mapsto \infty}
\lim_{k\mapsto \infty}
\RI_{\phi,z}(\lam,s_1,s_2).
$$
By absolute convergence, $\RI_{\phi,z}(s_1,s_2)$ is holomorphic for $\Re(s_1),\Re(s_2)\geq 0$. In particular 
$$
\gam(\RJ_G,\CW_E(\chi))\phi(z) = \RI_{\phi,z}(0,0).
$$
\end{enumerate}

Now we are ready to introduce the Bruhat regularization.

\begin{defin}[Bruhat regularization]\label{defin:descent:Bruhatreg}
Let $F$ be any local field and $E$ an étale quadratic $F$-algebra. For any $\phi\in \CW_E$, $z\in E^\times$ and $\Re(s_1),\Re(s_2)$ to be sufficiently large, let
\begin{align*}
\RI_{\phi,z}(\lam,s_1,s_2)
=
\lam_{E/F}
\int_{u,v\in F,t\in F^\times}
&
\mathbbm{1}_\lam(v,t,u)
|v|^{s_1}|t|^{s_2}
\RJ_{\Fc}(t+t^{-1}+tuv)
\frac{\eta(-v)}{|v|}
\\
\int_{x\in E}&
\psi\bigg(
\frac{\Nr(z-x)}{v}
\bigg)
\eta(t)
|t|
\psi\big(
ut^2\Nr(x)
\big)
\phi(tx)
\\
&\ud_E x
|t|^2
\ud^*t\ud u\ud v
\end{align*}
and set 
$$
\RI_{\phi,z}(s_1,s_2) = 
\lim_{i\mapsto \infty}
\lim_{j\mapsto \infty}
\lim_{k\mapsto \infty}
\RI_{\phi,z}(\lam,s_1,s_2),\quad 
\RI_{\phi,z} = \RI_{\phi,z}(0,0)
$$
whenever the limit exists and $\RI_{\phi,z}(s_1,s_2)$ has a holomorphic continuation to $s_1=s_2=0$.
\end{defin}

The main theorem in this subsection is the following.

\begin{thm}\label{thm:descent:Bruhatreg}
With the notation in Definition \ref{defin:descent:Bruhatreg}, $\RI_{\phi,z}(s_1,s_2)$ is holomorphic when $s_1=s_2=s$ satisfying $0<\Re(s)<1$ with a holomorphic continuation to $s=0$, and the following identity holds
$$
\RI_{\phi,z} = 
\lam_{E/F}
\int_{E^1}
\CF_{\psi}
\bigg(
\CF_{\psi^{-1}}(\RJ_\Fc)\frac{\eta}{|\cdot|}
\bigg)(\tr(e))\phi(ez)
\ud_{E^1}e.
$$
\end{thm}
It is clear that Theorem \ref{thm:descent:Bruhatreg} implies Theorem \ref{thm:descent:sl2} for supercuspidal representations. On the other hand, it is worth mentioning that Theorem \ref{thm:descent:Bruhatreg} also works for $E$ split over $F$. However to derive Theorem \ref{thm:descent:sl2} for $E$ split over $F$, or say the principal series representations, one need to show that the regularization \eqref{eq:descent:supercuspidal:Bernsteinreg} is equal to $\RI_{\phi,z}$, which is unknown to us when $E$ splits over $F$. In the next subsection, we will use another strategy to resolve the split situation.

In the following, we are going to establish Theorem \ref{thm:descent:Bruhatreg}.

\begin{proof}

For each fixed $\lam\in \BN^3$, by Part (1) of Theorem \ref{thm:vanishing:SBC}, the integral defining $\RI_{\phi,z}(\lam,s_1,s_2)$ is absolutely convergent as a quadruple integral whenever $\Re(s_1),\Re(s_2)\geq 0$. We first do a series of changing variables to simplify the integral:

\begin{enumerate}
    \item 
    Write 
\begin{align*}
\RI_{\phi,z}(\lam,s_1,s_2)
=
\lam_{E/F}
\int_{u,v\in F,t\in F^\times}
&
\mathbbm{1}_\lam(v,t,u)
|v|^{s_1}|t|^{s_2}
\RJ_{\Fc}(t+t^{-1}+tuv)
\frac{\eta(-v)}{|v|}
\\
\int_{x\in E}&
\psi\bigg(
\frac{\Nr(z-x)}{v}
\bigg)
\eta(t)
|t|
\psi\big(
ut^2\Nr(x)
\big)
\phi(tx)
\\
&\ud_E x
|t|^2
\ud^*t\ud u\ud v.
\end{align*}
We change
$$
\mu = t+t^{-1}+tuv\longleftrightarrow 
u = \frac{\mu-t-t^{-1}}{tv}
$$
and get
\begin{align*}
\RI_{\phi,z}(\lam,s_1,s_2)
=
\lam_{E/F}
\int_{u,v\in F,t\in F^\times}
&
\mathbbm{1}_i(v)\mathbbm{1}_j(t)\mathbbm{1}_k\bigg(\frac{\mu-t-t^{-1}}{tv}\bigg)
|v|^{s_1}|t|^{s_2}
\RJ_{\Fc}(\mu)
\frac{\eta(-v)}{|v|^2}
\\
\int_{x\in E}&
\psi\bigg(
\frac{\Nr(z-x)}{v}
\bigg)
\eta(t)
\psi\bigg(
\frac{(\mu t-t^2-1)\Nr(x)}{v}
\bigg)
\phi(tx)
\\
&\ud_E x
|t|^2
\ud^*t\ud \mu\ud v.
\end{align*}

\item 
Change variable $x\mapsto x/t$ and get
\begin{align*}
\RI_{\phi,z}(\lam,s_1,s_2)
=
\lam_{E/F}
\int_{u,v\in F,t\in F^\times}
&
\mathbbm{1}_i(v)\mathbbm{1}_j(t)\mathbbm{1}_k\bigg(\frac{\mu-t-t^{-1}}{tv}\bigg)
|v|^{s_1}|t|^{s_2}
\RJ_{\Fc}(\mu)
\frac{\eta(-v)}{|v|^2}
\\
\int_{x\in E}&
\psi\bigg(
\frac{\Nr(z-\frac{x}{t})}{v}
\bigg)
\eta(t)
\psi\bigg(
\frac{(\frac{\mu}{t} -1-\frac{1}{t^2})\Nr(x)}{v}
\bigg)
\phi(x)
\\
&\ud_E x
\ud^*t\ud \mu\ud v.
\end{align*}
The term in the additive character can be calculated as follows
\begin{align*}
&\psi\bigg(
\frac{
\Nr(z)-
\tr(\bar{z}x/t)
+\frac{\Nr(x)}{t^2}
}{v}
\bigg)
\psi\bigg(
\frac{(\frac{\mu \Nr(x)}{t} -\Nr(x)-\frac{\Nr(x)}{t^2})}{v}
\bigg)
\\
=&
\psi\bigg(
\frac{
\Nr(z)-\tr(\bar{z}x/t)
+\frac{\mu \Nr(x)}{t}
-\Nr(x)
}{v}
\bigg).
\end{align*}
Therefore 
\begin{align*}
\RI_{\phi,z}(\lam,s_1,s_2)
=
\lam_{E/F}
\int_{u,v\in F,t\in F^\times}
&
\mathbbm{1}_i(v)\mathbbm{1}_j(t)\mathbbm{1}_k\bigg(\frac{\mu-t-t^{-1}}{tv}\bigg)
|v|^{s_1}|t|^{s_2}
\RJ_{\Fc}(\mu)
\frac{\eta(-v)}{|v|^2}
\\
\int_{x\in E}&
\eta(t)
\psi\bigg(
\frac{
\Nr(z)-\tr(\bar{z}x/t)
+\frac{\mu \Nr(x)}{t}
-\Nr(x)
}{v}
\bigg)
\phi(x)
\\
&\ud_E x
\ud^*t\ud \mu\ud v.
\end{align*}

\item 
Change $x\mapsto xz$ and get
\begin{align*}
\RI_{\phi,z}(\lam,s_1,s_2)
=
\lam_{E/F}
\int_{u,v\in F,t\in F^\times}
&
\mathbbm{1}_i(v)\mathbbm{1}_j(t)\mathbbm{1}_k\bigg(\frac{\mu-t-t^{-1}}{tv}\bigg)
|v|^{s_1}|t|^{s_2}
\RJ_{\Fc}(\mu)
\frac{\eta(-v)}{|v|^2}
\\
\int_{x\in E}&
\eta(t)
\psi\bigg(
\frac{
\Nr(z)-\Nr(z)\tr(x/t)
+\frac{\mu \Nr(x)\Nr(z)}{t}
-\Nr(x)\Nr(z)
}{v}
\bigg)
\\
&
\phi(xz)
|\Nr(z)|\ud_E x
\ud^*t\ud \mu\ud v.
\end{align*}

\item 
Change $v\mapsto v\Nr(z)$ and get 
\begin{align*}
\RI_{\phi,z}(\lam,s_1,s_2)
=
\lam_{E/F}
\int_{u,v\in F,t\in F^\times}
&
\mathbbm{1}_i\big(v\Nr(z)
\big)\mathbbm{1}_j(t)\mathbbm{1}_k\bigg(\frac{\mu-t-t^{-1}}{tv\Nr(z)}\bigg)
\big|v\Nr(z)\big|^{s_1}|t|^{s_2}
\\
&\RJ_{\Fc}(\mu)
\frac{\eta(-v)}{|v|^2}
\int_{x\in E}
\eta(t)
\psi\bigg(
\frac{
1-\tr(x/t)
+\frac{\mu \Nr(x)}{t}
-\Nr(x)
}{v}
\bigg)
\\
&
\phi(xz)
\ud_E x
\ud^*t\ud \mu\ud v.
\end{align*}

\item 
Change $t\mapsto -\frac{t\Nr(x)}{v}$ and get 
\begin{align*}
\RI_{\phi,z}(\lam,s_1,s_2)
=
\lam_{E/F}
\int_{u,v\in F,t\in F^\times}
&
\mathbbm{1}_i\big(v\Nr(z)
\big)\mathbbm{1}_j
\big(-\frac{t\Nr(x)}{v}\big)\mathbbm{1}_k\bigg(-\frac{\mu+\frac{t\Nr(x)}{v}+\frac{v}{\Nr(x)t}}{t\Nr(x)\Nr(z)}\bigg)
\\
&
\big|v\Nr(z)\big|^{s_1}
\big|\frac{t\Nr(x)}{v}\big|^{s_2}
\RJ_{\Fc}(\mu)
\frac{\eta(t)}{|v|^2}
\\
&
\int_{x\in E}
\psi\bigg(
\frac{
1
-\Nr(x)
}{v}
+\tr\big(\frac{1}{tx}\big)
-\frac{\mu}{t}
\bigg)
\phi(xz)
\ud_E x
\ud^*t\ud \mu\ud v.
\end{align*}

\item 
Finally we change $t\mapsto t^{-1}$ and $v\mapsto v^{-1}$, and get
\begin{align*}
\RI_{\phi,z}(\lam,s_1,s_2)
=
\lam_{E/F}
\int_{u,v\in F,t\in F^\times}
&
\mathbbm{1}_i\big(\frac{\Nr(z)}{v}
\big)\mathbbm{1}_j
\big(-\frac{v\Nr(x)}{t}\big)\mathbbm{1}_k\bigg(-\frac{\mu t+v\Nr(x)+\frac{t^2}{\Nr(x)v}}{\Nr(x)\Nr(z)}\bigg)
\\
&
\big|\frac{\Nr(z)}{v}\big|^{s_1}
\big|\frac{v\Nr(x)}{t}\big|^{s_2}
\RJ_{\Fc}(\mu)
\eta(t)
\\
&
\int_{x\in E}
\psi\bigg(
\big(1
-\Nr(x)\big)v
+\tr\big(\frac{t}{x}\big)
-\mu t
\bigg)
\phi(xz)
\ud_E x
\ud^*t\ud \mu\ud v.
\end{align*}
\end{enumerate}

After the tedious variable changes we are ready to take the iterated limit 
$$
\RI_{\phi,z}(s_1,s_2) = \lim_{i\mapsto \infty}
\lim_{j\mapsto \infty}
\lim_{k\mapsto \infty}
\RI_{\phi,z}(\lam,s_1,s_2).
$$

We first evaluate the limit 
$$
\lim_{k\mapsto \infty}
\RI_{\phi,z}(\lam,s_1,s_2).
$$
Since $\{\mathbbm{1}_n\}\subset \CC^\infty_c(F^\times)$, for any $s_1,s_2\in \BC$, the integration in $v,t$ and $x$ are all absolutely convergent. By Part (1) of Theorem \ref{thm:descent:sl2}, $\RJ_\Fc(\mu)$ is absolutely integrable. Hence by the dominated convergence theorem we can switch the integration in $u,v,x$ and $\lim_{k\mapsto \infty}$ to get 
$$
\lim_{k\mapsto \infty}
\int_{\mu\in F}
\mathbbm{1}_k\bigg(-\frac{\mu t+v\Nr(x)+\frac{t^2}{\Nr(x)v}}{\Nr(x)\Nr(z)}\bigg)
\psi(-\mu t)
\RJ_\Fc(\mu)\ud \mu = 
\CF_{\psi^{-1}}(\RJ_\Fc)(t).
$$
Hence 
\begin{align}\label{eq:sl2:descent:Bruhat:1}
    \lim_{k\mapsto \infty}
\RI_{\phi,z}(\lam,s_1,s_2)
=
\lam_{E/F}
\int_{u,v\in F,t\in F^\times} \nonumber
&
\mathbbm{1}_i\big(\frac{\Nr(z)}{v}
\big)\mathbbm{1}_j
\big(-\frac{v\Nr(x)}{t}\big)
\CF_{\psi^{-1}}(\RJ_\Fc)(t)
\big|\frac{\Nr(z)}{v}\big|^{s_1} \nonumber
\\
&
\big|\frac{v\Nr(x)}{t}\big|^{s_2}
\eta(t)
\int_{x\in E}
\psi\bigg(
\big(1
-\Nr(x)\big)v
+\tr\big(\frac{t}{x}\big)
\bigg) \nonumber
\\
&\phi(xz)
\ud_E x
\ud^*t\ud v.
\end{align}
Now by Part (2) of Theorem \ref{thm:descent:sl2}, $\CF_{\psi^{-1}}(\RJ_\Fc)(t)$ is smooth and globally bounded. Moreover $\CF_{\psi^{-1}}(\RJ_\Fc)$ vanishes at zero, it implies the following fact:
\begin{itemize}
    \item When $F$ is non-archimedean, the globally bounded smooth function $\CF_{\psi^{-1}}(\RJ_\Fc)$ vanishes in a neighborhood of zero. Therefore $\CF_{\psi^{-1}}(\RJ_\Fc)(t)|t|^{-s}$ is absolutely integrable for $\Re(s)>1$;

    \item When $F$ is archimedean, the globally bounded smooth function $\CF_{\psi^{-1}}(\RJ_\Fc)$ vanishes at zero. Therefore $\CF_{\psi^{-1}}(\RJ_\Fc)(t)|t|^{-s}$ is 
    \begin{itemize}
    \item
    absolutely integrable near zero for $\Re(s)<2$,
    
    \item and absolutely integrable near infinity for $\Re(s)>1$.
    \end{itemize}
    In conclusion $\CF_{\psi^{-1}}(\RJ_\Fc)(t)|t|^{-s}$ is absolutely integrable for $1<\Re(s)<2$.
\end{itemize}
Therefore whenever $1<\Re(s_2+1)<2$ (notice that we use the multiplicative Haar measure $\ud^*t$), the integration in $u,v,x$ are all absolutely convergent. Hence we can apply the dominated convergence theorem again to evaluate the limit $\lim_{j\mapsto \infty}$:
$$
\lim_{j\mapsto \infty}
\int_{t\in F^\times}
\mathbbm{1}_j
\big(-
\frac{v\Nr(x)}{t}
\big)
\CF_{\psi^{-1}}(\RJ_\Fc)(t)
|t|^{-s_2-1}\eta(t)
\psi
\bigg(
\tr\big(
\frac{t}{x}
\big)
\bigg)\ud t = 
\CF_{\psi}
\bigg(
\frac{\eta
\CF_{\psi^{-1}}(\RJ_\Fc)
}{|\cdot|^{-s_2-1}}
\bigg)
\big(
\tr(
\frac{1}{x})
\big).
$$
In other words, for $0<\Re(s_2)<1$ and any $s_1\in \BC$,
\begin{align*}
\lim_{j\mapsto \infty}
\lim_{k\mapsto \infty}
\RI_{\phi,z}(\lam,s_1,s_2) =& 
\lam_{E/F}
\int_{v\in F}
\mathbbm{1}_i
\big(
\frac{\Nr(z)}{v}
\big)
\big|
\frac{\Nr(z)}{v}
\big|^{s_1}
\big|
v\Nr(x)
\big|^{s_2}
\\
&\int_{x\in E}
\CF_{\psi}
\bigg(
\frac{\eta
\CF_{\psi^{-1}}(\RJ_\Fc)
}{|\cdot|^{-s_2-1}}
\bigg)
\big(
\tr(
\frac{1}{x}
)
\big)
\psi
\bigg(
\big(
1-\Nr(x)
\big)v
\bigg)
\phi(xz)\ud_E x\ud v.
\end{align*}
Notice that the right hand side above is indeed absolutely convergent, since $\mathbbm{1}_i\in \CC^\infty_c(F^\times)$, and $\CF_{\psi}
\bigg(
\frac{\eta
\CF_{\psi^{-1}}(\RJ_\Fc)
}{|\cdot|^{-s_2-1}}
\bigg)$ is globally bounded for $0<\Re(s_2)<1$. 

Finally we evaluate the limit $\lim_{i\mapsto \infty}$. Since the right hand side integral is holomorphic for any $s_1\in \BC$, we can let $s_1=s_2$, and break the integral in $x\in E$ according to the short exact sequence $1\to E^1\to E^\times \to \Nr(E^\times)\to 1$ to get
\begin{align*}
\lim_{j\mapsto \infty}
\lim_{k\mapsto \infty}
\RI_{\phi,z}(\lam,s_1,s_2) =
\lam_{E/F}
&\int_{v\in F}
\mathbbm{1}_i
\big(
\frac{\Nr(z)}{v}
\big)
\psi\big(
(1-a)v
\big)
\ud v
\int_{a\in \Nr(E^\times)}
\big|\Nr(z)a\big|^{s_2}\ud a
\\
&
\int_{e\in E^1}
\CF_{\psi}
\bigg(
\frac{\eta
\CF_{\psi^{-1}}(\RJ_\Fc)
}{|\cdot|^{-s_2-1}}
\bigg)
\big(
\tr(
\frac{1}{e_ae}
)
\big)
\phi(e_aez)\ud_{E^1} e
\end{align*}
where $e_a\in E^\times$ such that $\Nr(e_a) = a$.
For abbreviation, let $\RF(a)$ be the function supported on $\Nr(E^\times)$ given by
$$
\RF(a) = 
\psi(-av)
\big|\Nr(z)a\big|^{s_2}
\int_{e\in E^1}
\CF_{\psi}
\bigg(
\frac{\eta
\CF_{\psi^{-1}}(\RJ_\Fc)
}{|\cdot|^{-s_2-1}}
\bigg)
\big(
\tr(
\frac{1}{e_ae}
)
\big)
\phi(e_aez)\ud_{E^1} e.
$$
Then the above integral can be rewritten as 
\begin{align*}
\lam_{E/F}
\int_{v\in F}
\mathbbm{1}_i
\big(
\frac{\Nr(z)}{v}
\big)
\psi(v)
\ud v
\int_{a\in F}
\psi(-av)
\RF(a)\ud a.
\end{align*}
By Fourier inversion, 
$$
\lim_{i\mapsto \infty}
\int_{v\in F}
\mathbbm{1}_i
\big(
\frac{\Nr(z)}{v}
\big)
\psi(v)
\ud v
\int_{a\in F}
\psi(-av)
\RF(a)\ud a
=\RF(1).
$$
Therefore 
$$
\lim_{i\mapsto \infty}
\lim_{j\mapsto \infty}
\lim_{k\mapsto \infty}
\RI_{\phi,z}(\lam,s_2,s_2) = 
\lam_{E/F}
\big|
\Nr(z)
\big|^{s_2}
\int_{e\in E^1}
\CF_{\psi}
\bigg(
\frac{\eta \CF_{\psi^{-1}}(\RJ_\Fc)}{|\cdot|^{-s_2-1}}
\bigg)
\big(
\tr(e)
\big)
\phi(ez)\ud_{E^1}e.
$$
Now $\phi|_{E^1}$ is a test function on $E^1$, hence by Part (3) of Theorem \ref{thm:descent:sl2}, the right hand side has a holomorphic continuation to $s_2=0$ in the sense of tempered distributions. It follows that we complete the proof.

\end{proof}

\subsection{Principal series representations}
\label{subsec:descent:PS}

In this subsection, we establish the theorem for principal series representations of $G=\SL_2$.

Let $F$ be a local field. The principal series reprsentations of $G(F)=\SL_2(F)$ can be realized in the following induced representation model
$$
i_\chi = \Ind^G_B(\chi) = \{f:\CC^\infty(G(F))\to \BC\mid f(bg) = \del^{1/2}_B(b)\chi(b)f(g)\}
$$
where $B=TU$ is the upper triangular Borel subgroup of $G$, $\del_B$ is the modular character of $B(F)$, and $\chi:B(F)\to \BC^\times$ is a quasi-character factoring through the diagonal torus $T(F)$. As explained in \cite{jlgl2}, the relation between $i_\chi$ and those constructed via Weil representations are connected by the Whittaker models for $i_\chi$. In the following, instead of working with a particular principal series, we use the idea of Theorem \ref{thm:descent:Bruhatreg} to work with the universal principal series model. Precisely, we consider the following right $G$-equivariant commutative diagram
\begin{align}\label{eq:descent:PS:diagram}
\xymatrix{
\CS(G(F))\ar[r]^{\RJ_G} \ar[d]^{r_U}& \CS(G(F))  \ar[d]^{r_U} \\
i_U \ar[r]^{\RJ_{G/U}} \ar[d]^{\CP_\chi} & i_U\ar[d]^{\CP_\chi} \\
i_\chi \ar[r]^{\gam(\RJ_G,i_\chi)} & i_\chi
}
\end{align}
Here $\CS(G(F)) = \CC^\infty_c(G(F))$ when $F$ is non-archimedean, and is the Schwartz algebra of $G(F)$ when $F$ is archimedean. The space $\CS(G(F))$ has a natural $G$-equivariant projection $r_U$: 
$$
i_U = \ind^G_U= \big\{r_U(f)(g) = 
g\in G(F)\mapsto \int_{u\in U(F)}f(ug)\ud u\mid f\in \CS(G(F))
\big\}.
$$
The projection $\CP_\chi$ is defined via the following integral
\begin{align}\label{eq:descent:operatorCPchi}
\CP_\chi: i_U &\to i_\chi \nonumber
\\
\phi &\mapsto \bigg(g\mapsto \int_{t\in T(F)}
\chi^{-1}(t)\del_B^{-1/2}(t)
\phi(tg)\ud^*t\bigg).
\end{align}
The integrals above are all absolutely convergent since the restriction of Schwartz functions on $G(F)$ to closed subgroups are still Schwartz (\cite{BZ76}\cite{AGSchwartznashmd}). In particular the map $\CP_{\chi}$ is entire in $\chi$. 

The following lemma shows that the diagram \eqref{eq:descent:PS:diagram} is well-defined.

\begin{lem}\label{lem:descent:PS:diagram:Welldefined}
The operators $\RJ_{G/U}$ is well-defined. Moreover, under the right $G$-equivariant projection $\CP_\chi$, it descends down to the multiplication by the function $\chi\in \wh{T}(F)\mapsto \gam(\RJ_G,i_\chi)$ between $i_{\chi}$. 
\end{lem}
\begin{proof}
    To show that $\RJ_{G/U}$ is well-defined, we only need to show that for $f\in \CS(G(F))$, if $r_U(f)=0$, then $r_U(\RJ_G*f) = 0$. Since $G(F)$ is semi-simple, $\CS(G(F))$ is a subspace of $\CC(G(F))$, the space of Harish-Chandra Schwartz functions on $G(F)$ (\cite[Lem.~1.1]{Delorme-limites} for archimedean. The result is manifest for non-archimedean). In particular, $f$ is a strongly cuspidal function in the sense of \cite[\S 5.1]{beuzart2015local}. By \cite[(5.3.1)]{beuzart2015local}, $f$ is strongly cuspidal if and only if the operator $i_\chi(f) = 0$ for any (tempered) character of $T(F)$. Since $i_{\chi}(\RJ_G*f) = \gam(\RJ_G,i_\chi)i_\chi(f)$, it follows that $i_\chi(\RJ_G*f)=0$ and hence $\RJ_G*f$ is also strongly cuspidal. Hence $r_U(\RJ_G*f) = 0$ identically, from which we deduce that the operator $\RJ_{G/U}$ is well-defined.

    Next we show that under the right $G$-equivariant projection $\CP_\chi$, $\RJ_{G/U}$ descends to the multiplication by the function $\chi\in \wh{T}(F)\mapsto \gam(\RJ_G,i_\chi)$. 
    
    First, we show that the map is well-defined. Essentially we need to show that for $\phi\in i_U$, if $\CP_{\chi}(\phi) = 0$ identically, then $\CP_\chi\big(\RJ_{G/U}(\phi)\big)=0$ identically. By Mellin transform, $\CP_\chi(\phi) = 0$ for any $\chi$ implies that $\phi = 0$. Therefore $\RJ_{G/U}(\phi) =0$ which implies that $\CP_{\chi}\big(\RJ_{G/U}(\phi)\big)=0$. Hence, the bottom map is well-defined. 
    It remains to show that the bottom map is the multiplication by the function $\chi\in \wh{T}(F)\mapsto \gam(\RJ_G,i_\chi)$. For generic $\chi$, $i_\chi$ is irreducible. Hence by Schur's lemma, the bottom map is given by multiplying a constant. It suffices to show that the constant is equal to $\gam(\RJ_G,i_\chi)$. 

    For any $f\in \CS(G(F))$, fix a test function $e_K\in \CC^\infty_c(G(F))$. Let us compute $$\CP_\chi\bigg(r_U\big(( \RJ_G*e_K)*f\big)\bigg).$$ By definition, it is equal to the following iterated integral
    \begin{align*}
    &\CP_\chi\bigg(r_U\big( (\RJ_G*e_K)*f\big)\bigg)(g) = 
    \int_{t\in T(F)}
    \chi^{-1}(t)\del^{-1/2}_B(t)
    r_U\big(
    (\RJ_G*e_K)*f
    \big)(tg)\ud^*t 
    \\
    &= 
     \int_{t\in T(F)}
    \chi^{-1}(t)\del^{-1/2}_B(t)
    \int_{u\in U(F)}
    \big((\RJ_G*e_K)*f\big)(utg)\ud u \ud^*t
    \\
    &=
    \int_{t\in T(F)}
    \chi^{-1}(t)
    \del_B^{-1/2}(t)
    \ud^*t
    \int_{u\in U(F)}
    \ud u
    \int_{h\in G(F)}
    (\RJ_G*e_K)(h)f(h^{-1}utg)\ud^*h.
    \end{align*}
Let $g$ be the identity element, and put $f^\vee(g) = f(g^{-1})$, then the above identity becomes 
\begin{align*}
=& 
\int_{t\in T(F)}
    \chi^{-1}(t)
    \del_B^{-1/2}(t)
    \ud^*t
    \int_{u\in U(F)}
    \ud u
    \int_{h\in G(F)}
    (\RJ_G*e_K)(h)f^\vee(t^{-1}u^{-1}h)\ud^*h
    \\
    =&
    \int_{t\in T(F)}
    \chi^{-1}(t)
    \del_B^{-1/2}(t)
    \ud^*t
    \int_{u\in U(F)}
    \ud u
    \int_{h\in G(F)}
    (\RJ_G*e_K)(h)f^\vee(t^{-1}uh)\ud^*h.
\end{align*}
The above integration is absolutely convergent since $\RJ_G*e_K\in \CS(G(F))$, hence we may change variable $u\mapsto tut^{-1}$ and get 
\begin{align*}
=&   \int_{t\in T(F)}
    \chi^{-1}(t)
    \del_B^{1/2}(t)
    \ud^*t
    \int_{u\in U(F)}
    \ud u
    \int_{h\in G(F)}
    (\RJ_G*e_K)(h)f^\vee(ut^{-1}h)\ud^*h
    \\
=&
\int_{h\in G(F)}
(\RJ_G*e_K)(h)
\CP_{\chi^{-1}}\big(r_U(f^\vee)\big)(h)\ud^*h.
\end{align*}
In conclusion, we have established the following identity
\begin{align*}
\CP_{\chi}
\bigg(
r_U
\big(
(\RJ_G*e_K)*f
\big)
\bigg)(\Id)=
\int_{h\in G(F)}
(\RJ_G*e_K)(h)
\CP_{\chi^{-1}}
\big(
r_U(f^\vee)
\big)(h)\ud^*h.
\end{align*}
Now we choose $\{e_{K_n}\}_{n\geq 1}\subset \CC^\infty_c(G(F))$ to be a delta sequence tending to the identity element at $G(F)$, then $(\RJ_G*e_{K_n})*f = e_{K_n}*(\RJ_G*f)$ is stably convergent to $\RJ_G*f$ if $F$ is non-archimedean, and is an approximation to identity if $F$ is archimedean. Precisely, when $F$ is archimedean, $e_{K_n}*(\RJ_G*f)$ converges to $\RJ_G*f$, at least under the $L^1$-norm on $G(F)$. As a result, by Fubini's theorem and smoothness, $\CP_{\chi}\big(r_U(e_{K_n}*\RJ_G*f) \big)$ converges to $\CP_{\chi}\big(r_U(\RJ_G*f) \big)$ pointwise.

Hence by continuity, the left hand side converges to 
$$
\CP_{\chi}
\bigg(
r_U
\big(
(\RJ_G*f)
\big)
\bigg)(\Id) 
$$
and the right hand side converges to 
$$
\gam(\RJ_G,i_{\chi^{-1}})
\CP_{\chi^{-1}}
\big(
r_U(f^\vee)
\big)(\Id) = \gam(\RJ_G,i_\chi)
\CP_{\chi}
\big(
r_U(f)
\big)(\Id),\quad i_\chi\simeq i_{\chi^{-1}}.
$$
We complete the proof of the theorem.
\end{proof}

Based on Lemma \ref{lem:descent:PS:diagram:Welldefined}, to examine the action of $\RJ_G$ on the principal series $\{i_\chi\}_{\chi\in \wh{T}}$ which provides the constant $\gam(\RJ_G,i_\chi)$, it suffices to examine the action of $\RJ_{G/U}$ on some special class of vectors in $i_U$, which we view as the \emph{universal principal series model}. 

We are going to establish Theorem \ref{thm:descent:sl2} for principal series representations. We choose functions $f\in \CS(G(F))$ that are supported on the open dense Bruhat cell given by 
\begin{align}\label{eq:descent:PS:bruhatcoordinate:3}
f
\bigg(
\begin{pmatrix}
    1 & \\
    v & 1
\end{pmatrix}
\begin{pmatrix}
t & \\
  &t^{-1}
\end{pmatrix}
\begin{pmatrix}
1 &u \\
  &1
\end{pmatrix}
\bigg) = \phi_0(v)\phi_1(t)\phi_2(u)
\end{align}
such that $\phi_0,\phi_2\in \CC^\infty_c(F)$ and $\phi_1\in \CC^\infty_c(F^\times)$.
It is clear that $\CP_{\chi}\big(r_U(f^\vee) \big)$ maps into a dense subspace of $i_\chi$. In the following, we examine the action of $\RJ_{G/U}$ on $\big(r_U(f^\vee)\big)$. 

By the commutativity of the right $G$-equivariant diagram \eqref{eq:descent:PS:diagram}, it suffices to calculate $r_U\big( \RJ_G*f\big)(\Id)$. By definition, 
\begin{align*}
r_U\big( \RJ_G*f\big)(\Id)
=\int_{\wt{u}\in U(F)}
\ud \wt{u}
\int_{h\in G(F)}
\RJ_\Fc\circ \c(\wt{u}h^{-1})f(h)
\ud^*h.
\end{align*}
It is worth mentioning that in general, the above integral might only be convergent as an iterated integral. Write $h$ via the coordinate \eqref{eq:descent:PS:bruhatcoordinate:3}, then 
$r_U\big(\RJ_G*f \big)(\Id)$ is equal to 
\begin{align}\label{eq:descent:PS:afterbruhat}
    =&
    \int_{\wt{u}\in F}
    \ud\wt{u}
    \int_{v,u\in F,t\in F^\times}
    \RJ_\Fc
    \big(
    t^{-1}+t+tv(u-\wt{u})
    \big)
    \phi_0(v)\phi_1(t)\phi_2(u)
    |t|^2\ud^*t\ud u\ud v \nonumber
    \\
    =&
    \int_{\wt{u}\in F}
    \ud\wt{u}
    \int_{v,u\in F,t\in F^\times}
    \RJ_\Fc
    (t^{-1}+t+tvu)
    \phi_0(v)\phi_1(t)\phi_2(u+\wt{u})
    |t|^2\ud^*t\ud u\ud v\quad (u\mapsto u+\wt{u}).
\end{align}
Here we abuse the notation and write 
$
\wt{u} = 
\left( 
\begin{smallmatrix}
    1 & \wt{u}\\
      & 1
\end{smallmatrix}   
\right).
$
By the unitarity of Fourier transform, we can write 
\begin{align*}
\int_{v\in F}
\RJ_\Fc(t+t^{-1}+tvu)
\phi_0(v)
\ud v = 
\int_{\xi\in F}
\CF_{\psi^{-1}}(\RJ_\Fc)
(\xi)
\psi
\big(
\xi(t+t^{-1})
\big)
\CF_{\psi}(\phi_0)(tu\xi)
\ud \xi.
\end{align*}
Hence \eqref{eq:descent:PS:afterbruhat} can be rewritten as the following iterated integral
\begin{align*}
r_U\big(\RJ_G*f\big)(\Id) =\int_{\wt{u}\in F}
\ud \wt{u}
&\int_{u\in F,t\in F^\times}
\phi_1(t)
|t|^2
\phi_2(u+\wt{u})
\ud^*t
\ud u
\\
&\int_{\xi\in F}
\CF_{\psi^{-1}}(\RJ_\Fc)
(\xi)
\psi
\big(
\xi(t+t^{-1})
\big)
\CF_{\psi}(\phi_0)(tu\xi)
\ud \xi.
\end{align*}
Motivated from the proof for the supercupsidal case, we introduce the following complex deformation of the above integral. Precisely, for $0<\Re(s)<1$, consider the following integral
\begin{align}\label{eq:descent:complexdeform:5}
    \RI(s) = 
    \int_{\wt{u}\in F}
\ud \wt{u}
&\int_{u\in F,t\in F^\times}
\phi_1(t)
|t|^2
\phi_2(u+\wt{u})
\ud^*t
\ud u   \nonumber
\\
&\int_{\xi\in F}
|\xi|^{-s}
\CF_{\psi^{-1}}(\RJ_\Fc)
(\xi)
\psi
\big(
\xi(t+t^{-1})
\big)
\CF_{\psi}(\phi_0)(tu\xi)
\ud \xi.
\end{align}
It is clear that $\RI(0) = r_U\big(\RJ_G*f \big)(\Id)$.

We establish the following theorem, which is one of the main results in this subsection.

\begin{thm}\label{thm:descent:PS:1}
For $0\leq \Re(s)<1$, the quadruple integral defining $\RI(s)$ is absolutely convergent. Moreover, the following identity holds
\begin{align*}
\RI(s) = 
\int_{t\in F^\times}
\CF_{\psi}
\bigg(
\frac{\CF_{\psi^{-1}}(\RJ_\Fc)}{|\cdot|^{s+1}}
\bigg)(t+t^{-1})
|t|^{-1}
r_U(f^\vee)
\begin{pmatrix}
    t & \\
      &t^{-1}
\end{pmatrix}
\ud^*t,\quad 0\leq \Re(s)<1.
\end{align*}
\end{thm}
\begin{proof}
We first show that the quadruple integral defining $\RI(s)$ is absolutely convergent for $0\leq \Re(s)<1$. Changing variable $u\mapsto u/(t\xi)$, we get
\begin{align*}
    \RI(s) = 
    \int_{\wt{u}\in F}
\ud \wt{u}
&\int_{u\in F,t\in F^\times}
\phi_1(t)
|t|
\phi_2\big(\frac{u}{t\xi}+\wt{u}\big)
\ud^*t
\ud u
\\
&\int_{\xi\in F}
|\xi|^{-s-1}
\CF_{\psi^{-1}}(\RJ_\Fc)
(\xi)
\psi
\big(
\xi(t+t^{-1})
\big)
\CF_{\psi}(\phi_0)(u)
\ud \xi.
\end{align*}
The integration in $\wt{u}$ and $u$ are clearly absolutely integrable since $\phi_0,\phi_2$ are both Schwartz-Bruhat functions. Similarly the integration in $t$-variable is also absolutely convergent since $\phi_1\in \CC^\infty_c(F^\times)$. Finally, based on Theorem \ref{thm:vanishing:SBC}, $\CF_{\psi^{-1}}(\RJ_\Fc)(\xi)|\xi|^{-(s+1)}$ is absolutely integrable when $0\leq \Re(s)<1$. It follows that the above quadruple integral defining $\RI(s)$ is indeed absolutely convergent.

Now the integration in $\wt{u}$ provides $\int_{\wt{u}\in F}\phi_2(\wt{u})\ud \wt{u}$, the integration in $u$ provides $\phi_0(0)$ by Fourier inversion, and the integration in $\xi$ yields a Fourier transform, which, in conclusion, provides 
\begin{align*}
\RI(s) =& 
    \phi_0(0)
    \int_{u\in F}
    \phi_2(u)\ud u
    \int_{t\in F^\times}
    \phi_1(t)|t|
    \CF_{\psi}
    \bigg(
    \frac{\CF_{\psi^{-1}}(\RJ_\Fc)}{|\cdot|^{s+1}}
    \bigg)(t+t^{-1})
        \ud^*t
        \\
    =&
    \int_{t\in F^\times}
\CF_{\psi}
\bigg(
\frac{\CF_{\psi^{-1}}(\RJ_\Fc)}{|\cdot|^{s+1}}
\bigg)(t+t^{-1})
|t|^{-1}
r_U(f^\vee)
\begin{pmatrix}
    t & \\
      &t^{-1}
\end{pmatrix}
\ud^*t.
\end{align*}
It follows that we complete the proof of the proposition.
\end{proof}

Now to finalize the proof of Theorem \ref{thm:descent:sl2} for principal series representations, it suffices to show that both sides of the equality in Theorem \ref{thm:descent:PS:1} has a continuous extension to $s=0$, which is the following lemma.

\begin{lem}\label{lem:descent:PS:continuousextentozero}
Both sides of the equality in Theorem \ref{thm:descent:PS:1} has a continuous extension to $s=0$.
\end{lem}
\begin{proof}
It suffices to show that the right hand side has a continuous extension to $s=0$. In other words, 
\begin{align*}
\lim_{s\mapsto 0^+}
&\int_{t\in F^\times}
\CF_{\psi}
\bigg(
\frac{\CF_{\psi^{-1}}(\RJ_\Fc)}{|\cdot|^{s+1}}
\bigg)
(t+t^{-1})
|t|^{-1}
r_U(f^\vee)
\begin{pmatrix}
    t & \\
      &t^{-1}
\end{pmatrix}
\ud^*t
\\
=&
\int_{t\in F^\times}
\CF_{\psi}
\bigg(
\frac{\CF_{\psi^{-1}}(\RJ_\Fc)}{|\cdot|}
\bigg)
(t+t^{-1})
|t|^{-1}
r_U(f^\vee)
\begin{pmatrix}
    t & \\
      &t^{-1}
\end{pmatrix}
\ud^*t.
\end{align*}
Precisely, the function $r_U(f^\vee)$, when restricted to the diagonal torus, is a test function on $F^\times$. Hence the above identity follows from the continuity of Fourier transform and the continuous $s$-family of tempered distributions $\CF_{\psi^{-1}}(\RJ_\Fc)/|\cdot|^{s+1}$ for $0\leq \Re(s)<1$.

\end{proof}

% !TEX root = luo-ngo.tex

\section{Bernstein decomposition of the cocenter} \label{sec:summands}

The stable transfer gives rise to certain quotients of the stable cocenter ${\rm SO}(\CH)$. From Corollary \ref{cor:SBC=SBCprime}, at least when $G=\SL_2$ or $\GL_2$, we know that ${\rm SO}(\CH)$ is a module over the stable Bernstein center $\CZ^{\rm st}$. It turns out that some of those quotients are canonically direct summands of ${\rm SO}(\CH)$. In this section, we construct these summands by explicit sections and integral transforms based on the orthogonality of elliptic characters for $G=\SL_2$ and $\GL_2$. The existence of explicit sections realizing these direct summands was inspired by the work of Johnstone and the first-named author (\cite{DanielZhilin}).

\subsection{Total transfer to tori}\label{subsec:totaltransfertori}

In this subsection, following Remark \ref{rmk:GGformulaforstabletransfer}, we describe the total stable transfer map explicitly. Let $F$ be a local field whose residual characteristic is not equal to two. Following \eqref{indexset:parametrizequadratic} up to stable conjugacy, there are four, two or one maximal tori in $G=\SL_2$ and $\GL_2$ depending on $F$ being non-archimedean, real, or complex. As in \eqref{indexset:parametrizequadratic} we let $\CI$ be the index set parametrizing étale quadratic $F$-algebras and let $\{T_\alp\}_{\alp\in \CI}$ be the set of maximal tori in $G(F)$ up to stable conjugacy. Let $\eta_\alp$ be the corresponding quadratic characters of $F^\times$ attached to $E_\alp$. For $\alp\in \CI$, there is a homomorphism of algebras
$$
\rho_\alp^*:\CZ^{\rm fin,st} \to \CZ_\alpha^{\rm fin,st}=\CZ(T_\alp(F))^{\rm fin, st}
$$
whose image is the subring $(\CZ_\alpha^{\rm fin,st})^{\tau_\alpha}$ of $\tau_\alpha$-invariant elements, where $\tau_\alp$ is the unique nontrivial involution over $F$ for $E_\alp$. Its kernel are given by elements $J\in \CZ^{\rm fin, st}$ such that $\gamma(J,\pi)=0$ for all representations $\pi$ belonging to an $L$-packet $[\pi]$ coming from $T_\alpha$ by dihedral lifting.

We also have the Langlands stable transfer map 
$$\CT_\alpha:{\rm SO}(\CH) \to \CC^\infty_c(T_\alp(F))$$
whose image is the space of $\tau_\alpha$-invariants $\CC^\infty_c(T_\alp(F))^{\tau_\alpha}$. When $F$ is archimedean, the Langlands stable transfer map also extends to the Schwartz algebra $\CS$
$$
\CT_\alp:\RS\RO(\CS)\to \CS(T_\alp(F))^{\tau_\alp}
$$
which follows from the following adjoint identity and the Paley-Wiener theorem for the Schwartz algebras (\cite[p.100]{Barker}\cite[Chap.~IV,~\S 5.6]{GGV} and the variant for $\GL_2$)
$$
\langle \CT_\alp(f),\chi\rangle = 
\langle f,\theta^\st_\chi\rangle,\quad f\in \RS\RO(\CS)\supset \RS\RO(\CH).
$$

In this subsection, we consider the total transfer to tori
\begin{equation} \label{eq:TTS}
	\CT_\oplus=\bigoplus_{\alpha \in \CI} \CT_{\alpha}:{\rm SO}(\CH) \to \bigoplus_{\alpha \in \CI} \CC_c^\infty(T_\alp(F))^{\tau_\alp}
\end{equation}

The following theorem reflects our knowledge of $L$-packets of $G(F)$ by means of the dihedral representations, which are reviewed in subsection \ref{subsec:dihedral rep}, notably Proposition \ref{pro:dihedral:Cas72:1} and Theorem \ref{thm:L-packet} (\cite[Thm.~1.7]{Casselma-Quadratic}).

\begin{thm}\label{thm:total-transfer}
Let $F$ be non-archimedean of odd residual characteristic.
\begin{enumerate}
	\item
	Let $G=\SL_2$.	The kernel of the total transfer map $\CT_\oplus$ is the one-dimensional vector space generated by the normalized elliptic character of the Steinberg representation. 	
	The image of the total transfer $\CT_\oplus$ is the finite-codimensional subspace of $\bigoplus_{\alpha \in \CI} \mathcal C_c^\infty(T_\alp(F))^{\tau_\alp}$ consisting of quadruples $$(q_\alpha \mid \alpha\in \CI)\in \bigoplus_{\alpha \in\CI} \CC_c^\infty(T_\alp(F))^{\tau_\alp}$$ satisfying the equations
	\begin{enumerate}
		\item $\langle q_0,\eta_\alpha\rangle =\langle q_\alpha,1_{\alpha}\rangle $ for all $\alpha\in\CI\bs \{0\}$, where  $1_\alpha$ are the trivial characters of $T_\alp = E_\alpha^1$ and $\eta_\alpha$ is the quadratic character of $T_0 = F^\times=E_0^1$ corresponding to the quadratic extension $E_\alpha$ of $F$.   
		
		\item $\langle q_\alpha,\chi^\qd_{\alpha}\rangle= \langle q_\beta,\chi^\qd_{\beta}\rangle$ for all $\alpha,\beta\in \CI\bs \{0\}$, where $\chi^\qd_\alpha$ and $\chi^\qd_\beta$ are the unique nontrivial quadratic characters of $T_\alp = E_\alpha^1$ and $T_\bet= E_\beta^1$ respectively. 
	\end{enumerate}
	Here $\langle \cdot,\cdot \rangle$ is the standard pairing between characters and test functions on $T_\alp$.

\item 
Let $G=\GL_2$. The kernel of the total transfer map $\CT_\oplus$ is the one-parameter family generated by normalized elliptic characters of twisted Steinberg representations. The image of the total transfer $\CT_\oplus$ is the subspace of $\bigoplus_{\alp\in\CI }\CC^\infty_c(T_\alp(F))^{\tau_\alp}$ consisting of quadruples 
$$(q_\alpha \mid \alpha\in \CI)\in \bigoplus_{\alpha \in \CI} \mathcal \CC_c^\infty(T_\alp(F))^{\tau_\alp}$$ satisfying the equations
	\begin{enumerate}
		\item $\langle q_0,(\chi,\chi\otimes\eta_\alpha)\rangle =\langle q_\alpha,\chi\circ \Nr_\alp\rangle $ for all $\alpha\in \CI \bs \{0\}$, where  $\chi$ is a character of $F^\times$ and $\Nr_\alp:E^\times_\alp\to F^\times$ is the norm map.
		
		\item $\langle q_\alpha,\chi^\qd_{\alpha}\rangle= \langle q_\beta,\chi^\qd_{\beta}\rangle$ for all $\alpha,\beta\in \CI\bs \{0\}$, where $\chi^\qd_\alpha$ and $\chi^\qd_\beta$ are the characters of $E^\times_\alp$ and $E^\times_\bet$ such that $\chi_\alp$ and $\chi_\bet$ does not factor through the norm map but their square does, and their dihedral lifting to $G(F)$ are isomorphic.
	\end{enumerate}
\end{enumerate}
\end{thm}
\begin{proof}
	We only prove the statement for $G=\SL_2$ and the case for $G=\GL_2$ is similar. By the density of stable tempered characters from Theorem \ref{thm:main:SBC}, an element $f\in {\rm SO}(\CH)$ is zero if and only if its stable trace on tempered representations are zero. For $G=\SL_2$, it is equivalent to say that $f=0$ if and only if $\CT_\alpha(f)=0$ for all $\alpha\in \{0,1,\pm 1/2\}$ and if the trace $\tr{\rm St}(f)$ of $f$ on the Steinberg representation is zero. Thus the kernel of the total transfer $\CT_\oplus$ is at most one-dimensional generated by the (stable) orbital integral of the pseudo-coefficient of Steinberg representation, which is exactly equal to the normalized elliptic character of Steinberg representation. 

	The equations describing the compatibility between different tori follows from the explicit description of the dihedral lifting, in particular Theorem \ref{thm:L-packet} and \cite[Thm.~4.6]{jlgl2}.
\end{proof}

We have the parallel theorem for $F$ archimedean. The proof is identical and we omit.

\begin{thm}\label{thm:total-transfer:archimedean}
Let $F$ be archimedean.

\begin{enumerate}
\item When $F$ is the field of complex numbers, the stable transfer $\CT_0$ is the identity map. Hence $\CT_0:\RS\RO(\CH)\simeq \CC^\infty_c(T_0(F))^{\tau_0}$.

\item When $F$ is the field of real numbers, the total transfer map $\CT_\oplus$ is injective. The image of the total transfer $\CT_\oplus$ is the subspace of $\bigoplus_{\alp\in \CI}\CC^\infty_c(T_\alp(F))^{\tau_\alp}$ enjoying the same characterization as Theorem \ref{thm:total-transfer}.

\item The above statements also hold if we replace $\CH$ by $\CS$ and $\CC^\infty_c(T_\alp(F))^{\tau_\alp}$ by $\CS(T_\alp(F))^{\tau_\alp}$.
\end{enumerate}
\end{thm}

\subsection{The discrete part of the cocenter}\label{subsec:discrete-cocenter}

Let $F$ be non-archimedean. The action of the Bernstein center $\CZ$ on the Hecke algebra $\CH$ preserves the commutatator $[\CH,\CH]$ and hence induces an action on the cocenter ${\rm O}(\CH)=\CH/[\CH,\CH]$. Assuming the stable center conjecture and the identification of $\CZ^{\st,\prime}$ with the stable Bernstein center $\CZ^\st$ that is discussed in Remark \ref{rmk:anotherdef:SBC}, which in particular is valid for $G=\SL_2$ and $\GL_2$ by Theorem \ref{thm:main:SBC}, we derive an action of $\CZ^{\rm st}$ on the stable cocenter ${\rm SO}(\CH)$ which induces a decomposition of ${\rm SO}(\CH)$ according to the connected components of the stable Bernstein variety $\Omega^{\rm st}(G(F))$. In particular, for $G=\SL_2$ and $\GL_2$, the stable Bernstein variety $\Omega^{\rm st}(G(F))$ is the disjoint union of the supercuspidal components and the union of parabolically induced components,
$$\Omega^{\rm st}(G(F)) = \Omega^{\rm st}(G(F))_{\rm cus} \sqcup \Omega^{\rm st}(G(F))_{\rm ind},$$
which induces a decomposition of the stable cocenter ${\rm SO}(\CH)$ into a direct sum
$$
{\rm SO}(\CH)= {\rm SO}(\CH)_{\rm cus} \oplus {\rm SO}(\CH)_{\rm ind}.
$$

When $G=\SL_2$, ${\rm SO}(\CH)_{\rm cus}$ is supported by the union of zero-dimensional components and ${\rm SO}(\CH)_{\rm ind}$ is supported by the union of one-dimensional components. We stress that ${\rm SO}(\CH)_{\rm ind}$ may and does contain elements supported by zero-dimensional subschemes of $\Omega^{\rm st}(G(F))$ which are contained in one-dimensional components. These functions are multiple of the normalized elliptic character of the Steinberg representation.

We have the following theorem.
\begin{thm}\label{thm:discete-cocenter}
The kernel of the stable transfer factor for the split torus which is the restriction map
$$
\CT_0:{\rm SO}(\CH) \to \CC^\infty_c(T_0(F))^{\tau_0}
$$
consists of elements $f\in  {\rm SO}(\CH)$ which, as functions on the Steinberg-Hitchin base, are supported on the elliptic locus 
	$${\rm supp}(f) \subset \Fc(F)^{\rm ell}= \bigcup_{\alpha\in\CI\bs \{0\}} \pi_\alpha(T_\alp(F)).$$ 
We will call the discrete part of the cocenter the kernel of $\CT_0$
$${\rm ker}(\CT_0)={\rm SO}(\CH)_{\rm dis}.$$

When $F$ is non-archimedean, there is a direct sum decomposition
$$
{\rm SO}(\CH)_{\rm dis}= {\rm SO}(\CH)_{\rm cus} \oplus {\rm SO}(\CH)_{\rm St}.
$$
Moreover, 
${\rm SO}(\CH)_{\rm cus}$ is the direct summand of ${\rm SO}(\CH)$ supported by the cuspidal components of the stable Bernstein variety, and ${\rm SO}(\CH)_{\rm St}$ is 
the subspace of ${\rm SO}(\CH)$ that are annihilated by the stable dihedral characters, and is precisely generated by the normalized elliptic Steinberg characters $($twisted by $\CH(F^\times)$ along the center when $G=\GL_2$$)$. 
\end{thm} 
\begin{proof}
It suffices to prove the decomposition for $\RS\RO(\CH)_{\disc}$. We prove it for $G=\SL_2$ and the case for $G=\GL_2$ is similar. Let $F$ be local non-archimedean. For $J\in \CZ^\st$ and $f\in \RS\RO(\CH)$, by the spectral characterization of distributions in the (stable) Bernstein center, 
\begin{equation}\label{eq:discete-cocenter:1}
J(f) = 
\int_{\pi\in \wh{G}^{\mathrm{temp}}}
\gam(J,\pi)
\langle f,\theta^\st_\pi\rangle \ud \pi
\end{equation}
where $\langle f,\theta^\st_\pi\rangle$ is the stable trace pairing between $f$ and the stable tempered character $\theta^\st_\pi$. Here we use the fact that the Plancherel measure is constant on the tempered local $L$-packet of $\SL_2(F)$ (\cite{ggps}). For $f\in \RS\RO(\CH)_\disc$ vanishing on the split locus, its trace pairing vanishes for principal series representations. Hence \eqref{eq:discete-cocenter:1} can be rewritten as 
$$
J(f) = 
\int_{\pi\in \wh{G}^{\mathrm{disc}}}
\gam(J,\pi)
\langle f,\theta^\st_\pi\rangle \ud \pi.
$$
where $\wh{G}^{\disc}$ is the set of discrete series representations of $G(F)$ that can be written as the disjoint union $= \wh{G}^{\cusp}\sqcup \mathrm{St}$. Here $\wh{G}^{\mathrm{cusp}}$ is the set of supercuspidal representations of $G(F)$ and $\mathrm{St}$ is the Steinberg representation. Hence $\RS\RO(\CH)_\St$ is stable under the action of $\CZ^\st$ and is the one-dimensional complement of $\RS\RO(\CH)_\cusp$, the cuspidal support of $\Ome^\st(G(F))_\cusp$. It follows that we complete the proof of the theorem.
\end{proof}

Similarly, for $F$ archimedean, we also have a decomposition 
$$
\RS\RO(\CS) = \RS\RO(\CS)_\disc\oplus \RS\RO(\CS)_\ind
$$
with $\RS\RO(\CS)_\disc$ the subspace of functions supported on elliptic locus of $G(F)$, and $\RS\RO(\CS)_\ind$ the subspace supported on split locus only.

\subsection{Scalar product on the discrete part of the cocenter}\label{subsec:scalarproductdiscrete}

In this subsection, we will reformulate the orthogonality of elliptic characters (\cite{MR874042}\cite{MR1237898}) as a scalar product on the discrete part of the cocenter.

Let $G$ be a reductive group over a local field $F$. Let $Z$ be the split part of the center of $G$ over $F$. For two discrete series representations $\{\pi_i\}_{i=1}^2$ of $G(F)$ with the same unitary central character, let $\Theta_i$ be their distribution characters and $\Theta^C_i$ be the induced function on $C(F)$, the set of strongly regular semi-simple conjugacy classes of $G(F)$. Then the following identity holds
\begin{equation}\label{eq:scalarproductdiscrete:1}
\int_{\wb{C}(F)^\el}
\Theta^C_1(c)\wb{\Theta}^C_2(c)
\Del(c)\v(c)\ome_{\wb{C}} = 
\bigg\{
\begin{matrix}
1 & \pi_1\simeq \pi_2\\
0 & \text{otherwise}
\end{matrix}
.
\end{equation}
The identity follows from the orthogonality of elliptic tempered characters. Here following subsection \ref{subsec:measure}, $\wb{C}(F) = C(F)/Z(F)$ is the set of strongly regular semi-simple conjugacy classes of $\wb{G}(F) = G(F)/Z(F)$ equipped with the induced measure $\ome_{\wb{C}}$, $\wb{C}(F)^\el$ is the subset of elliptic locus, and $\v(c)$ is the inverse of the volume of the connected centralizer of $c\in \wb{C}(F)^\el$ in $\wb{G}(F)$. The pairing \eqref{eq:scalarproductdiscrete:1} can be upgraded to a scalar product on the discrete part of the cocenter as follows.

\begin{thm}\label{thm:scalarproductdiscrete}
For any unitary character $\chi$ of $Z(F)$, let $\CH_\chi$ be the space of smooth functions on $G(F)$ that are compactly supported modulo $Z(F)$, and are $(Z(F),\chi)$-equivariant. Let $\RO(\CH_\chi)$ be the $(Z(F),\chi)$-equivariant cocenter and let $\RO(\CH_\chi)_\disc$ be the subspace of $\RO(\CH_\chi)$ consisting of functions on $C(F)$ whose restriction to the hyperbolic locus vanishes. Then $\RO(\CH_\chi)_\disc$ is a $\CZ$-module. Moreover, $\RO(\CH_\chi)_\disc$ has a Hermitian definite inner product 
$$
(f_1,f_2) = 
\int_{\wb{C}(F)^\el}
f_1(c)
\wb{f}_2(c)
\Del(c)^{-1}
\v(c)\ome_{\wb{C}}.
$$
\end{thm}
Comparing with \eqref{eq:scalarproductdiscrete:1}, it is $\Del(\cdot)^{-1}$ rather than $\Del(\cdot)$ appearing in the above pairing because of the normalization in Corollary \ref{cor:normalizationoforb}.
\begin{proof}
It suffices to show that $\RO(\CH)_\disc$, the subspace of $\RO(\CH)$ consisting of functions vanishing on the hyperbolic locus, is a $\CZ$-module. By the density of principal series characters on the split locus, $f\in \RO(\CH)$ lies in $\RO(\CH)_\disc$ if and only if 
$$
\langle f,\Theta^C_\pi\rangle =0
$$
for any principal series representation $\pi$ of $G(F)$. But since $\CZ$ acts via scaling on any principal series representations, it follows that $\RO(\CH)_\disc$ is a $\CZ$-module. 
\end{proof}

Similarly, we can introduce a inner product pairing on the discrete part of the stable cocenter. The split center $Z$ acts on the Steinberg-Hitchin base $\Fc=\Fc_G$. Let $\wb{\Fc}(F) = \Fc_{\wb{G}}(F) = \Fc_G(F)/Z(F)$ which can be viewed as the Steinberg-Hitchin base for $\wb{G}$ equipped with the induced measure $\ome_{\wb{\Fc}}$. Let $\wb{\Fc}(F)^\el$ be the subset of $\wb{\Fc}(F)$ corresponding to elliptic locus. For $a\in \wb{\Fc}(F)^\el$, let $\v(a)$ be the inverse of the volume of the connected centralizer of $a\in \wb{\Fc}(F)^\el$ in $\wb{G}(F)$. For any unitary character $\chi$ of $Z(F)$, there is the following Hermitian definite inner product on $\RS\RO(\CH_\chi)_\disc$, which is the subspace of $\RS\RO(\CH_\chi)$ vanishing on the split locus of $\wb{\Fc}(F)$,
\begin{equation}\label{eq:scalarproductdiscrete:2}
(f_1,f_2) = 
\int_{\wb{\Fc}(F)^\el}
f_1(a)
\wb{f}_2(a)
\Del(a)^{-1}\v(a)\ome_{\wb{\Fc}}.
\end{equation}

\begin{rmk}\label{rmk:scalardiscreteproduct}
For the purpose of next subsection, we introduce the following inner product pairing that is parallel to \eqref{eq:scalarproductdiscrete:2}: For two stable distribution characters $\theta_1$, $\theta_2$ of $G(F)$ with the same unitary central character,
\begin{equation}\label{eq:scalarproductdiscrete:3}
(\theta_1,\theta_2)_{\wb{\Fc}^\el} = 
\int_{\wb{\Fc}(F)^\el}
\theta_1(a)\wb{\theta}_2(a)\Del(a)
\v(a)\ome_{\wb{\Fc}}.
\end{equation}
Similarly, for two functions $f_1,f_2\in \RS\RO(\CH)$, define
\begin{equation}\label{eq:scalarproductdiscrete:4}
(f_1,f_2)_{\Fc^\el} = 
\int_{\Fc(F)^\el}
f_1(a)\wb{f}_2(a)\Del(a)
\v(a)\ome_{\Fc}.
\end{equation}
\end{rmk}

\subsection{Another unitary structure on the affine line}\label{subsec:anotherunitary}

Starting from this subsection, we focus on $G=\SL_2$ and $\GL_2$ over a local field $F$. Let $Z$ be the split center of $G$ and $\wb{G}(F) = G(F)/Z(F)$. In particular, when $G=\SL_2$, $\wb{G}(F)=G(F)$.

The Gelfand-Graev transform for $G(F)$ defined by the formulae \eqref{eq:ggps:interpret:1} \eqref{eq:ggps:interpret:2} is clearly unitary with respect to the standard measure on the affine line since both the Fourier transform and the multiplication by the quadratic characters are unitary. It is a remarkable fact that it is unitary with respect to another measure derived from the orthogonality of elliptic tempered characters, which is stated below. 

\begin{thm}\label{thm:anotherunitary:sl2}
Let $G=\SL_2$ over a local field $F$ of residual characteristic other than two. For any $\alp\in \CI$, let 
\begin{align*}
\CG\CG_{E_\alp} = \CG\CG_\alp: L^2(F,\ud x) &\to L^2(F,\ud x)
\\
\phi&\mapsto \lam_{\alp}\big(\CF_{\psi^{-1}}(\eta_\alp\cdot \CF_{\psi}(\phi)) \big)
\end{align*}
Then the following statements hold.
\begin{enumerate}
\item Fix a quadratic field extension $E=E_\alp$ of $F$ with $\alp\in \CI\bs \{0\}$. Let $\phi_i$ $(i=1,2)$ be two functions supported on $\tr(E^1)$ that are square-integrable with respect to the measure $\Del(a)\ud a$, and $(\phi_i,\phi_{\mathbbm{1}_{\alp}})_{\wb{\Fc}^\el} =0$ for $i=1$ or $2$, where $\mathbbm{1}_{\alp}$ is the trivial character of $E^1_\alp$ and $\phi_{\mathbbm{1}}$ is the function supported on $\tr(E^1_\alp)\subset \Fc_G(F)=F$ defined by \eqref{eq:phi-chi}, i.e. the constant Fourier coefficient of $\phi_1$ or $\phi_2$, viewed as a function on $E^1_\alp$, vanishes. Then the following identity holds.
$$
\big(\CG\CG_\alp(\phi_1),\CG\CG_\alp(\phi_2)\big)_{\wb{\Fc}^\el}
= 
(\phi_1,\phi_2)_{\wb{\Fc}^\el}.
$$

\item 
Fix $\alp,\bet\in \CI\bs \{0\}$ with $\alp\neq \bet$. For $\phi$ supported on $\tr(E^1_\alp)$ that is square-integrable with respect to the measure $\Del(a)\ud a$, the following identity holds 
$$
\big(\CG\CG_{\alp}(\phi_1),\CG\CG_{\bet}(\phi_{\chi^\qd_{\bet}})\big)_{\wb{\Fc}^\el} = 
(\phi,\phi_{\chi^\qd_{\alp}})_{\wb{\Fc}^\el}
$$
where $\chi^\qd_{*}$ is the unique non-trivial character of $E^1_*$ whose square is trivial. In particular, it happens only when $F$ is non-archimedean.
\end{enumerate}
\end{thm}
We first treat the case when $F=\BR$, $\alp = 1$ and hence $E_\alp = \BC$. For two characters $\chi_1,\chi_2$ of $E^1_\alp$, let $\theta_{\chi_i}^\st$ be the stable character for the discrete series character for $G(F)$. From the orthogonality of elliptic discrete series characters for $\SL_2(\BR)$, or based on the stable character formula in \cite[(1.3.1)]{langlands-singularites-transfert}, the following identity holds
$$
\int_{\wb{C}(F)^\el = C(F)^\el}
\theta^\st_\chi(c)
\wb{\theta}^\st_\chi(c)
\Del(c)
\v(c)
\ome_{\wb{C}} = 
\bigg\{
\begin{matrix}
2 & \phi_{\chi_1} = \phi_{\chi_2} \neq \phi_{\mathbbm{1}} \\
0 & \text{otherwise}
\end{matrix}
$$
where $\mathbbm{1}$ is the trivial character of $E^1_\alp = \BC^1$ and $\phi_\chi$ is defined in \eqref{eq:phi-chi} attached to $\chi$. Descending down to the Steinberg-Hitchin base, and using the fact that the trace map is generically a $2$-fold cover, the elliptic orthogonality can be reformulated as follows
\begin{align*}
(\theta_{\chi_1}^\st,\theta_{\chi_2}^\st)_{\wb{\Fc}^\el}= 
\bigg\{
\begin{matrix}
1 & \phi_{\chi_1} = \phi_{\chi_2}\neq \phi_{\mathbbm{1}}\\
0 & \text{otherwise}
\end{matrix}
\end{align*}
By the Gelfand-Graev stable character formula \eqref{eq:ggps:interpret:1}, the above identity can be reformulated as follows
\begin{align}\label{eq:ggps:sl2:orthogonality:real}
\big(
\CG\CG_\alp(\phi_{\chi_1}),
\CG\CG_\alp(\phi_{\chi_2})
\big)_{\wb{\Fc}^\el} = 
\bigg\{
\begin{matrix}
(\phi_{\chi_1},\phi_{\chi_2})_{\wb{\Fc}^\el} & \chi_1=\chi_2=\mathbbm{1}\\
0 & \text{otherwise}
\end{matrix}
\end{align}
By Mellin transform (which boils down to Fourier series in this case), the characters $\{\chi\}_{\chi\in \wh{E}^1_\alp}$ form an $L^2$-basis for square integrable functions on $E^1_\alp$. In particular, for functions descending down to $\tr(E^1_\alp)$, $\{\phi_\chi\}_{\chi\in \wh{E}^1_\alp}$ is a spanning set. Hence as a corollary, for two smooth (or even square-integrable) functions on $E^1_\alp$, by linearity and continuity, the following identity holds as long as $(\phi_i,\phi_{\mathbbm{1}})_{\wb{\Fc}^\el} = 0$ for $i=1$ or $2$ (i.e. the constant Fourier coefficient of $\phi_1$ or $\phi_2$ vanishes),
\begin{align}\label{eq:ggps:sl2:orthogonality:real:3}
\big(
\CG\CG_\alp(\phi_1),\CG\CG_\alp(\phi_2)
\big)_{\wb{\Fc}^\el} = 
(\phi_1,\phi_2)_{\wb{\Fc}^\el},\quad \text{if} \quad 
(\phi_i,\phi_{\mathbbm{1}})_{\wb{\Fc}^\el} = 0\text{ for $i=1$ or $2$}.
\end{align}
As a consequence, we have established the following theorem.

\begin{thm}\label{thm:ggps:sl2:orthognality:real:summary}
For $F=\BR$ and $\alp\in \CI\bs \{0\}$ and hence $E_\alp = \BC$, the operator $\CE_\BC := \CG\CG_\BC\cdot \mathbbm{1}_{\tr(\BC^1)}$ is unitary on
$$
L^2_0(\tr(\BC^1),\Del(a)\ud a) := 
\{\phi\in L^2(\tr(\BC^1),\Del(a)\ud a) \mid (\phi,\phi_{\mathbbm{1}}) = 0\}.
$$
Here $\mathbbm{1}_{\tr(\BC^1)}$ is the characteristic function of $\tr(\BC^1)\subset \BR$.
\end{thm}
Notice that when $F=\BR$, up to stable conjugacy, there is only one elliptic torus in $\SL_2(\BR)$. Hence we can drop the volume factor $\v(a)$. 

\begin{rmk}\label{rmk:ggps:sl2:orthognality:real:extra}
It would be interesting to find a direct proof of Theorem \ref{thm:ggps:sl2:orthognality:real:summary} without referring to the orthogonality of elliptic tempered characters. 
\end{rmk}

Next, we assume that $F$ is non-archimedean with odd residual characteristic. There is a parallel unitary structure on $F$ with mild modification (the modification is needed due to the existence of the unique supercuspidal local $L$-packet of $\SL_2(F)$ of cardinality four).

We follow the same idea as the real case. Let $E$ be a quadratic $F$-algebra. The stable distribution character of $\SL_2(F)$ attached to a character $\chi:E^1\to \BC^\times$ is given by 
\begin{equation}\label{eq:ggps:padic:sl2:char}
\theta^\st_\chi(g) = \CG\CG_E(\phi_\chi)(\tr(g)),\quad g\in \SL_2(F)
\end{equation}

From \cite[Thm.~4.6]{jlgl2}\cite[Thm.~1.9]{Casselma-Quadratic}\cite[\S 11]{shelstad_notes}, the stable distribution character  $\theta^\st_\chi$ admits the following descriptions:
\begin{num}
\item\label{num:inv:sl2:padic:basic}
$
\theta^\st_\chi = \theta^\st_{\chi^{-1}} = \theta^\st_{\chi^\iota}
$ for any $\chi:E^1\to \BC^\times$. Here $\iota = \iota_E$ is the unique nontrivial involution of $E$ over $F$;

\item\label{num:inv:sl2:padic:ps} $\theta^\st_\chi$ is the restriction of the distribution character of a principal series of $\GL_2(F)$ when $\chi = \mathbbm{1}$, which in particular vanishes over elliptic locus;

\item\label{num:inv:sl2:padic:supcusp} $\theta^\st_\chi$ is the restriction of the distribution character of a supercuspidal representation of $\GL_2(F)$ when $\chi\neq \mathbbm{1}$. Moreover, the following statements hold:
\begin{enumerate}
\item[(i)] If $\chi^2\neq \mathbbm{1}$, then $\theta^\st_\chi$ is the sum of the distribution characters of two supercuspidal representations of $\SL_2(F)$ lying in the same local $L$-packet of cardinality two with the local $L$-parameter of dihedral type determined by $\chi$. For $\phi_{\chi_1}\neq \phi_{\chi_2}$ with $\chi_i^2\neq \mathbbm{1}$, the corresponding supercuspidal representations of $\SL_2(F)$ are not isomorphic to each other; 

\item[(ii)] If $\chi$ is the unique nontrivial character of $E^1$ with $\chi^2 =\mathbbm{1}$, then $\theta^\st_\chi$ is the sum of the distribution characters of four supercuspidal representations of $\SL_2(F)$ lying in the unique local $L$-packet with four elements. Moreover, for two different quadratic extensions $E_\alp$, $E_\bet$ of $F$ with $\alp,\bet\in \CI\bs \{0\}$ and $\alp\neq \bet$, let $\chi_*$ be the unique character of $E^1_*$ such that $\chi_*\neq \mathbbm{1}$ and $\chi_*^2 = \mathbbm{1}$, then $\theta^\st_{\chi_\alp} = \theta^\st_{\chi_\bet}$.
\end{enumerate}
\end{num}
The elliptic orthogonality for the stable distribution characters provides the following identity
$$
\int_{\wb{C}(F)^\el}
\theta^\st_{\chi_1}(c)
\wb{\theta}^\st_{\chi_2}(c)
\Del(c)
\v(c)
\ome_{\wb{C}}
=
\Bigg\{
\begin{matrix}
2 & \phi_{\chi_1} = \phi_{\chi_2}\text{ and }\chi^2_i\neq \mathbbm{1} \\
4 & \chi_i\neq \mathbbm{1} \text{ and }\chi^2_i=\mathbbm{1}\\
0 & \text{otherwise}
\end{matrix}
$$
From \cite[p.335]{Moy-Tadic-Conj-Orb}, up to conjugation by $\SL_2(F)$ there are four (resp. six) elliptic tori in $\SL_2(F)$ within three stable conjugacy classes parametrized by quadratic extensions $E/F$ when $-1\notin F^{\times 2}$ (resp. $-1\in F^{\times 2}$). Moreover the number of tori within its stable conjugacy class is equal to the cardinality of the Weyl group of the tori. Hence, by \eqref{eq:ggps:padic:sl2:char} and the fact that the trace map is generically a $2$-fold cover, the elliptic orthogonality can be reformulated as follows: 
\begin{itemize}
\item For $*=\alp,\bet\in \CI\bs \{0\}$, 
\begin{align}\label{eq:ggps:sl2:padic:refororthogonal}
\big(
\CG\CG_\alp(\phi_{\chi_\alp}),\CG\CG_{\bet}(\phi_{\chi_\bet})
\big)_{\wb{\Fc}^\el} 
=& 
\Bigg\{
\begin{matrix}
1 & \phi_{\chi_\alp} = \phi_{\chi_\bet} \text{ and }\chi^2_*\neq \mathbbm{1}   \\
2 & \chi_*\neq \mathbbm{1} \text{ and }\chi^2_*=\mathbbm{1}\\
0 & \text{otherwise}
\end{matrix}
\\
=&
\Bigg\{
\begin{matrix}
(\phi_{\chi_\alp},\phi_{\chi_\bet})_{\wb{\Fc}^\el} & \phi_{\chi_\alp} = \phi_{\chi_\bet} \text{ and }\chi^2_*\neq \mathbbm{1}   \\
2 & \chi_*\neq \mathbbm{1} \text{ and }\chi^2_*=\mathbbm{1}\\
0 & \text{otherwise}
\end{matrix} 
\nonumber
\end{align}
Notice that when $\chi^2_* =\mathbbm{1}$, $\chi_* = \chi_*^{-1}$.
\end{itemize}
For $\alp\in \CI\bs \{0\}$, following the same idea as the real case, the set of characters of $E^1_\alp$, $\{\chi\}_{\chi\in \wh{E}^1_\alp}$, forms an $L^2$-basis for square-integrable functions on $E^1_\alp$. Hence for smooth functions on $E^1_\alp$ descending down to $\tr(E^1_\alp)$, $\{\phi_\chi\}_{\chi\in \wh{E}^1_\alp}$ is a generating set. Therefore for smooth functions $\phi_1,\phi_2$ supported on $\tr(E^1_\alp)$, by linearity and continuity, the following identity holds as long as $(\phi_i,\phi_{\mathbbm{1}})_{\wb{\Fc}^\el} = 0$ for $i=1$ or $2$ (i.e. the constant Fourier coefficient of $\phi_1$ or $\phi_2$ vanishes)
\begin{align}\label{eq:ggps:sl2:ggps:inv:main:1}
\big(
\CG\CG_\alp(\phi_1),
\CG\CG_\alp(\phi_2)
\big)_{\wb{\Fc}^\el} = 
(\phi_1,\phi_2)_{\wb{\Fc}^\el}.
\end{align}
In other words, the operator $\CG\CG_\alp$ is unitary on the affine line with respect to the measure $\mathbbm{1}_\el(a)\Del(a)\v(a)\ud a$ for functions supported on $\tr(E^1_\alp)$ with vanishing constant Fourier coefficient.

On the other hand, when $*=\alp,\bet\in \CI\bs \{0\}$ and $\alp\neq \bet$, let $\chi_*^\qd$ be the unique nontrivial character of $E^1_*$ whose square is trivial, by \eqref{eq:ggps:sl2:padic:refororthogonal} and Mellin inversion, for any smooth function $\phi$ on $\tr(E^1_\alp)$,
\begin{equation}\label{eq:ggps:sl2:ggps:inv:main:2}
\big(
\CG\CG_\alp(\phi),
\CG\CG_\bet(\phi_{\chi_\bet^\qd})
\big)_{\wb{\Fc}^\el}
=
\big(
\CG\CG_\alp(\phi),
\CG\CG_\alp(\phi_{\chi_\alp^\qd})
\big)_{\wb{\Fc}^\el}
=
(\phi,\phi_{\chi^\qd_\alp})_{\wb{\Fc}^\el}.
\end{equation}

\subsubsection{Variant for $G=\GL_2$}

Finally we make a remark for the elliptic orthogonality of $G=\GL_2$ over a local field $F$ of residual characteristic not equal to two. Notice that there are no stability issue for $G$, up to conjugation, the maximal tori are parametrized by quadratic extensions of $F$, and the local $L$-packets for $G(F)$ are singleton. We are working with irreducible admissible representations of $G(F)$ with a fixed unitary central character $\chi:Z(F) = F^\times \to \BC^\times$. For two irreducible admissible square-integrable (modulo the center) representations $\pi_1,\pi_2$ of $G(F)$ with the same central character $\chi$, let $\theta_i$ $(i=1,2)$ be their distribution character, then the orthogonality of elliptic tempered characters show that 
$$
\int_{\wb{C}(F)^\el}
\theta_1(c)\wb{\theta}_2(c)
\Del(c)\v(c)\ome_{\wb{C}} = 
\bigg\{
\begin{matrix}
1 & \pi_1\simeq \pi_2\\
0 & \text{otherwise}
\end{matrix}
$$
Descending down the above identity to $\wb{\Fc}(F)^\el$, the Gelfand-Graev character identity \eqref{eq:ggps:interpret:2} shows that the following statements hold.

\begin{num}
\item\label{num:ellorthogonalgl2:1} Fix a quadratic extension $E=E_\alp$ of $F$ with $\alp\in \CI\bs \{0\}$. For a character $\chi_\alp:E^\times_\alp\to \BC^\times$. The central character of the distribution character of the dihedral lifting $\theta_{\chi_\alp}$ is $\chi_{\alp}|_{F^\times}\eta_\alp$. Fix two unitary characters $\chi_{\alp,i}:E_\alp^\times\to \BC^\times$ $(i=1,2)$ that both do not factor through the norm map (hence the corresponding dihedral representations are discrete series, which are supercuspidal when $F$ is non-archimedean) such that $\chi_{\alp,1}|_{F^\times} = \chi_{\alp,2}|_{F^\times}$,
$$
\big(
\CG\CG_\alp(\phi_{\chi_{\alp,1}}),\CG\CG_\alp(\phi_{\chi_{\alp,2}})
\big)_{\wb{\Fc}^\el} = 
(\phi_{\chi_1},\phi_{\chi_2})_{\wb{\Fc}^\el}
=
\bigg\{
\begin{matrix}
1 & \phi_{\chi_1} = \phi_{\chi_2}\\
0 & \text{otherwise}
\end{matrix}
$$
The identity shows that the operator $\CG\CG_\alp$ is unitary on $\wb{\Fc}(F)$ with respect to the measure $\mathbbm{1}_\el(a) \Del(a)\v(a)\ud a$ for functions supportd on $\c(E^\times_\alp/Z(F))$ with vanishing Fourier coefficient for characters that factor through the norm (and modulo the center). 

To be precise, following the same argument as the $\SL_2(F)$ case, for two smooth and compactly supported (even square-integrable modulo the center) functions $\phi_i$ on $E_\alp^\times$, let $\wb{\phi}_i$ be the Mellin transform of $\phi_i$ along the center $Z(F)$ against the character $\chi$. Suppose that for $i=1$ or $2$, $(\wb{\phi}_i,\phi_{\xi\circ \Nr_\alp})_{\wb{\Fc}^\el} = 0$ for any character $\xi$ of $F^\times$ such that $\xi\circ \Nr_\alp|_{F^\times} = \chi$, then the following identity is true
$$
\big(
\CG\CG_\alp(\wb{\phi}_1),\CG\CG_\alp(\wb{\phi}_2)
\big)_{\wb{\Fc}^\el} = 
(\wb{\phi}_{1},\wb{\phi}_{2})_{\wb{\Fc}^\el}.
$$
Similarly, with the above assumption, after extracting the central character for $\wb{\phi}_2$, we also have the following identity
$$
\big(
\CG\CG_\alp(\phi_1),
\CG\CG_\alp(\wb{\phi}_2)\big)_{\Fc^\el} = 
(\phi_1,\wb{\phi}_2)_{\Fc^\el}.
$$

\item\label{num:ellorthogonalgl2:2}  Fix $*=\alp, \bet\in \CI\bs \{0\}$ and $\alp\neq \bet$. For characters $\chi_*:E^\times_*\to \BC^\times$ both do not factor through the norm map and $\chi_\alp|_{F^\times}\eta_\alp = \chi_\bet|_{F^\times}\eta_\bet$ (and hence their dihedral liftings share the same central character)
\begin{align*}
\big(
\CG\CG_\alp(\phi_{\chi_\alp}),
\CG\CG_\bet(\phi_{\chi_\bet})
\big)_{\wb{\Fc}^\el}
=
\bigg\{
\begin{matrix}
1 & \text{$\chi^2_*$ factor through the norm map}\\
0 & \text{otherwise}
\end{matrix}
\end{align*} 
Notice that the above identity is parallel to the existence of the unique supercuspidal local $L$-packet with four elements in the $\SL_2(F)$ situation.
\end{num}

\subsection{Decomposition of the stable cocenter: motivation}\label{subsec:decompositionstablecocenter}

As in last subsection we take $G=\SL_2$ or $\GL_2$ over a local field of residual characteristic not equal to two. Recall that in subsection \ref{subsec:totaltransfertori} we have defined a total transfer map 
$$
\CT_\oplus: \RS\RO(\CH) \to \bigoplus_{\alp\in \CI} \CC^\infty_c(T_\alp(F))^{\tau_\alp}
$$
that is $\CZ^\st$-linear, with the $\CZ^\st$-module on $\CC^\infty_c(T_\alp(F))^{\tau_\alp}$ given by the homomorphism $\rho_\alp^*:\CZ^\st\to \CZ_\alp^{\tau_\alp}$. It is natural to ask the following question:
\begin{quest}\label{quest:inversion}
Given $f\in \RS\RO(\CH)$ with $\CT_{\oplus}(f) = (f_\alp)_{\alp\in \CI}\in \bigoplus_{\alp\in \CI}\CC^\infty_c(T_\alp(F))^{\tau_\alp}$ enjoying the characterization in Theorem \ref{thm:total-transfer}, can we reconstruct $f$ explicitly from the datum $(f_\alp)_{\alp\in \CI}$? 
\end{quest}
Roughly speaking, an answer to the above question would provide a decomposition of the stable cocenter by $F$-quadratic algebras, i.e.
$$
\RS\RO(\CH)\simeq \bigoplus_{\alp\in \CI}
\RS\RO(\CH)_{\alp}
$$
with $\RS\RO(\CH)_\alp$ being the image of $\CC^\infty_c(T_\alp(F))^{\tau_\alp}$ under the section from $\CC^\infty_c(T_\alp(F))^{\tau_\alp}$ to $\RS\RO(\CH)_\alp$. Before answering Question \ref{quest:inversion}, we would like to ask if the datum $(f_\alp)_{\alp\in \CI}$ determines $f\in \RS\RO(\CH)$ uniquely, which is summarized as follows.

\begin{num}
\item\label{num:stableBC::dense:cplx} When $F=\BC$ and hence $\CT_\oplus =\CT_0$ is the identity map, $f=f_0$;

\item\label{num:stableBC:dense:real} When $F=\BR$, the collection of stable tempered characters from the dihedral lifting $\{\theta^\st_{\chi_\alp}\}_{\chi_\alp,\alp\in \CI}$ is dense in the space of stable distributions (Theorem \ref{thm:main:SBC}), hence by adjunction, $(f_\alp)_{\alp\in \CI}$ is sufficient to determine $f$;

\item\label{num:stableBC:dense:nonarchi} When $F$ is non-archimedean of odd residual characteristic, $\{\theta^\st_{\chi_\alp}\}_{\chi_\alp,\alp\in \CI}$ is \textbf{not} dense in the space of stable distributions. One need to take into account of the distribution character of (twisted) Steinberg representations.
\end{num}

We introduce the following two subspaces of $\CC^\infty_c(T_\alp(F))^{\tau_\alp}$ for $\alp\in \CI\bs \{0\}$.
\begin{defin}\label{defin:directsummand:subspace}
Fix $\alp\in \CI\bs \{0\}$.
\begin{enumerate}
	\item Let $\CC^\infty_c(T_\alp(F))^{\tau_\alp}_+$ be the subspace of functions $\phi_\alp\in \CC^\infty_c(T_\alp(F))^{\tau_\alp}$ with the following properties:
	\begin{enumerate}
	\item If $G=\SL_2$, then $\langle \phi_\alp,\mathbbm{1}_{T_\alp}\rangle = 0$, i.e. the constant Fourier coefficient of $\phi_\alp$ vanishes; 

	\item If $G=\GL_2$, then $\langle \phi_\alp,\chi\circ \Nr_{\alp}\rangle =0$ for any $\chi$ a character of $F^\times$;
	\end{enumerate}

	\item Let $\CC^\infty_c(T_\alp(F))^{\tau_\alp}_{++}$ be the subspace of functions $\phi_\alp\in \CC^\infty_c(T_\alp(F))^{\tau_\alp}_{+}$ with the following properties. In particular the situation occurs only when $F$ is non-archimedean:
	\begin{enumerate}
	\item If $G=\SL_2$, then $\langle \phi_\alp, \chi_{\alp}^\qd\rangle =0$ where $\chi^\qd_\alp$ is the unique nontrivial character of $E^1_\alp$ whose square is trivial;

	\item If $G=\GL_2$, then $\langle \phi_\alp, \chi^\qd_\alp\rangle =0$ where $\chi^\qd_\alp$ runs over all the characters of $E^\times_\alp$ that do not factor through the norm map, but their square factor through the norm map.
	\end{enumerate}

	\item For $J_\alp\in \CC^\infty_c(T_\alp(F))^{\tau_\alp}$, let $J_{\alp}^+$ and $J^{++}_\alp$ be the image of the natural projection from $\CC^\infty_c(T_\alp(F))^{\tau_\alp}$ onto $\CC^\infty_c(T_\alp(F))^{\tau_\alp}_+$ and $\CC^\infty_c(T_\alp(F))^{\tau_\alp}_{++}$ through subtracting the corresponding Fourier coefficients.
\end{enumerate}
We also introduce the parallel subspaces $\CS(T_\alp(F))^{\tau_\alp}_+$ and $\CS(T_\alp(F))^{\tau_\alp}_{++}$ for $\CS(T_\alp(F))^{\tau_\alp}$. 
\end{defin}
These subspaces are introduced to avoid the overlap between dihedral liftings from different tori of $G(F)$.

When $F$ is non-archimedean, for any $\alp\in \CI$, $\CC^\infty_c(T_\alp(F))^{\tau_\alp}$ is a $\CZ_\alp^{\tau_\alp} = \CZ(T_\alp(F))^{\tau_\alp}$-module since for abelian group the center and the cocenter are the same. As a result, $\CC^\infty_c(T_\alp(F))^{\tau_\alp}_+$ and $\CC^\infty_c(T_\alp(F))^{\tau_\alp}_{++}$ are both $\CZ_\alp^{\tau_\alp}$-modules. Following the discussion in subsection \ref{subsec:totaltransfertori}, the homomorphism 
$$
\rho_\alp^*:\CZ^\st\to \CZ_\alp^{\tau_\alp}
$$
equips $\CC^\infty_c(T_\alp(F))^{\tau_\alp}$, $\CC^\infty_c(T_\alp(F))^{\tau_\alp}_+$ and $\CC^\infty_c(T_\alp(F))^{\tau_\alp}_{++}$ with a $\CZ^\st$-module structure via Mellin transform. The same statements are true for $F$ archimedean after replacing smooth and compactly supported functions by Schwartz functions.

The main theorem that we are going to establish in the next subsections is the following, which provides a decomposition for $\RS\RO(\CH)$ parametrized by quadratic $F$-algebras.

\begin{thm}\label{thm:explicitsection}
Let $F$ be a local field of residual characteristic not equal to two, and $G=\SL_2$ or $\GL_2$. 
\begin{enumerate}
\item \rm{(Nonsplit summands)}

\begin{enumerate}
\item For $\alp\in \CI\bs\{0\}$, consider the map
\begin{align*}
\CE_\alp:\CC^\infty_c(T_\alp(F))^{\tau_\alp}&\to \RS\RO(\CH)
\\
J_\alp &\mapsto \CE_\alp(J_\alp) = \mathbbm{1}_\el \cdot 
\v\Del\cdot 
\CG\CG_\alp
\bigg(
\frac{\nu_\alp(J_\alp)}{\v\Del}
\bigg)
\end{align*}
When restricted to $\CC^\infty_c(T_\alp(F))^{\tau_\alp}_+$, $\CE^+_\alp = \CE_\alp|_{\CC^\infty_c(T_\alp(F))^{\tau_\alp}_+}$ is a homomorphism of $\CZ^\st$-modules that provides a section for $\CT_\alp$. Here $\nu_\alp(J_\alp)$ is the unique function on $\pi_\alp(T_\alp(F))\subset \Fc_G(F)$ such that $\pi^*_\alp(\nu_\alp(J_\alp)) = J_\alp$. Moreover, for $\bet\in \CI\bs \{\alp\}$, $\CT_\bet\circ \CE^+_\alp =0$;

\item For $J_\alp\in \CC^\infty_c(T_\alp(F))^{\tau_\alp}$, $\CE_\alp(J_\alp) = \CE_\alp^+(J^+_\alp)$;

\item 
When $F$ is non-archimedean, the operator $\CE_\cusp:=\CE_1^+\oplus \CE^{++}_{1/2}\oplus \CE^{++}_{-1/2}$ gives rise to a $\CZ^\st$-linear isomorphism 
$$
\CE_\cusp:\CC^\infty_c(T_1(F))^{\tau_1}_+\oplus 
\CC^\infty_c(T_{1/2}(F))^{\tau_{1/2}}_{++}
\oplus 
\CC^\infty_c(T_{-1/2}(F))^{\tau_{-1/2}}_{++}\to \RS\RO(\CH)_\cusp.
$$
Here $\CE^{++}_\alp=\CE_{\alp}|_{\CC^\infty_c(T_\alp(F))^{\tau_\alp}_{++}}$.

\item 
When $F$ is real, the operator $\CE_1^+$ provides a $\CZ^\st$-linear isomorphism 
$$
\CE_1^+:\CC^\infty_c(T_1(F))^{\tau_1}_+\simeq \RS\RO(\CH)_\disc
$$
\end{enumerate}

Replacing $\CH$ by the Schwartz algebra $\CS$ when $F=\BR$, the statements are still true.

\item \rm{(Split summands)}

Consider the map $\CE_0$ from $\CC^\infty_c(T_0(F))^{\tau_0}$ to functions on $\Fc(F)$ that is given by the following formulas
\begin{align*}
J_0\mapsto 
\Bigg\{
\begin{matrix}
\nu_0(J_0) & F=\BC\\
\nu_0(J_0)-\CE_1(\CT_1(\nu_0(J_0))) & F=\BR\\
\nu_0(J_0)-\CE_1(\CT_1(\nu_0(J_0)))
-\sum_{\alp=\pm 1/2}
\CE^{++}_\alp(\CT_{\alp}(\nu_0(J_0))^{++}) & F\text{ non-archi.}
\end{matrix}
\end{align*}
Then the following statement is true: For any $J\in \RS\RO(\CH)$ with $\CT_{\oplus}(J) = (J_\alp)_{\alp\in \CI}$, we have the following facts:
\begin{enumerate}
\item When $F=\BC$, $\CT_0\circ \CE_0 = \Id$;

\item When $F=\BR$, 
$
J = \sum_{\alp\in \CI}\CE_\alp(J_\alp);
$

\item When $F$ is non-archimedean,
$$
J-J^\el_\St\cdot \v\Del \Theta^\el_\St = 
\sum_{\alp\in \CI}
\CE_\alp(J_\alp)
$$
where 
\begin{itemize}
\item when $G=\SL_2$, $J^\el_\St$ is constant given by
$$
J^\el_{\St} = \langle J,\Theta^\el_\St\rangle = 
\int_{\Fc(F)}
J(c)
\Theta^\el_\St(c)
\ome_{\Fc};
$$

\item 
when $G=\GL_2$, $J^\el_\St$ factors through the determinant map by
$$
J^\el_\St(g) = 
\int_{\Fc_{\SL_2}(F)}
J(c,\det g)
\Theta^\el_\St(c)\ome_{\Fc_{\SL_2}};
$$
\end{itemize}
\end{enumerate}
Replacing $\CH$ by the Schwartz algebra $\CS$ when $F$ is archimedean, the same statement is true.
\end{enumerate}
\end{thm}
When $F$ is non-archimedean, $\CE_0(\CC^\infty_c(T_0(F))^{\tau_0})$ lands in $\RS\RO(\CH)$. Precisely, by the trace Paley-Wiener theorem, $\nu_0(J_0)$ lands in $\RS\RO(\CH)$ since it can be realized as regular functions on $\Irr(G(F))$ supported on principal series components only. Therefore $\CT_\alp(\nu_0(J_0))\in \CC^\infty_c(T_\alp(F))^{\tau_\alp}$. 

The theorem is manifest for $F=\BC$. Hence in the next two subsections we will assume that $F$ is either real or non-archimedean. The rest of the section will be devoted to the proof of the theorem.

\subsection{Summands associated with the nonsplit tori} \label{subsec:command-nonsplit}

In this subsection, we are going to establish Part (1) of Theorem \ref{thm:explicitsection} based on the elliptic orthogonality formulated in subsection \ref{subsec:anotherunitary}.

In the following let us establish the theorem for $G=\SL_2$ over real or non-archimedean local field of odd residual characteristic. The situation for $G=\GL_2$ is similar and we omit.

We first prove Part (1) (a) of the theorem. For the fact that $\CE_\alp^+$ is a homomorphism of $\CZ^\st$-module, it follows from the following observation. By Mellin transform and inversion, for any nontrivial character $\chi_\alp$ of $E^1_\alp$ with $\phi_{\chi_\alp}$ given by \eqref{eq:phi-chi}, by the Gelfand-Graev character formula, $\CE^+_{\alp}(\phi_{\chi_\alp}) = \mathbbm{1}_\el\cdot \v\Del \theta^{\st}_{\chi_\alp}$ is the stable normalized elliptic character of the stable character of the dihedral lifting $\theta_{\chi_\alp}^{\st}$. In particular this is also equal to the stable orbital integral of the matrix coefficient of the supercuspidal representation attached to the dihedral lifting of $\chi_\alp$. Hence both the character $\chi_\alp$ and the stable normalized elliptic character corresponds uniquely to the (stable) supercuspidal component in $\CZ^\st$ determined by $\CW(\chi_\alp)$. It follows that $\CE^+_\alp$ is $\CZ^\st$-linear. 

It remains to show that $\CE^+_\alp$ is a section. For $\bet\in \CI\bs \{\alp\}$, if $\bet=0$ then $\CT_0\circ \CE^+_\alp = 0$ since the image of $\CE^+_\alp$ is supported on elliptic locus. For $\bet\neq 0$ and $J^+_\alp\in \CC^\infty_c(T_\alp(F))^{\tau_\alp}_+$, for any character $\chi_\bet$ of $E^1_\bet$, 
\begin{align*}
\langle\CT_\bet\circ \CE_\alp^+(J^+_\alp),\chi_\bet \rangle 
&= 
\langle \CE^+_\alp(J^+_\alp), \CG\CG_\bet(\phi_{\chi_\bet})\rangle
=
\big(
\CG\CG_\alp
\bigg(
\frac{\nu_\alp(J^+_\alp)}{\v\Del}
\bigg),\CG\CG_\bet(\phi_{\chi_\bet})
\big)_{\wb{\Fc}^\el} 
\\
&= 
\bigg\{
\begin{matrix}
\langle J^+_\alp,\chi_\bet\rangle & \alp=\bet\\
0 & \alp\neq \bet
\end{matrix}
\end{align*}
Here $\langle\cdot,\cdot \rangle$ is the standard pairing between characters and functions, and the last identity follows from Theorem \ref{thm:anotherunitary:sl2}, i.e. the unitarity of the Gelfand-Graev transform on elliptic locus. It follows that by Mellin inversion we have proved Part (1) (a) of the theorem.

For Part (1) (b), it follows from the fact that $J_\alp-J^+_\alp$ is supported on trivial characters of $T_\alp(F)$, whose distribution character for the dihedral lifting is a principal series character, and hence vanishes on the elliptic locus.

For Part (1) (c), it follows from the observation that up to a finite linear combination, $\RS\RO(\CH)_\cusp$ are spanned by the stable orbital integral of matrix coefficients of supercuspidal representations of $G(F)$, which are equal to the stable normalized elliptic characters of supercuspidal representations of $G(F)$. Hence they are precisely given by the image of $\CE_\cusp$. Notice that only the subspaces $\CC^\infty_c(T_\alp(F))^{\tau_\alp}_{++}$ for $\alp=\pm 1/2$ show up since the dihedral lifting from different tori of the supercuspidal representations of $G(F)$ have unique overlap at the nontrivial character whose square becomes trivial. Finally Part (1) (c) follows from the same argument as the non-archimedean case, after applying Mellin transform to functions in $\CC^\infty_c(T_1(F))^{\tau_1}_+$.

\subsection{Summands associated with the split torus}\label{subsec:summand-split}

In this subsection, we are going to establish Part (2) of Theorem \ref{thm:explicitsection}. Again we establish the theorem for $G=\SL_2$ over real and non-archimedean local field of odd residual characteristic. The situation for $G=\GL_2$ is similar and we omit.

We first show the case when $F=\BR$. To show that $J = \sum_{\alp\in \CI}\CE_\alp(J_\alp)$, it suffices to show that $J$ and $\sum_{\alp\in \CI}\CE_\alp(J_\alp)$ are the same on the split and elliptic locus separately. By the definition of $\CE_0$, as $\CE_1$ has image supported on elliptic locus only, we only need to verify the equality over elliptic locus. Over the elliptic locus, by definition, $\nu_0(J_0)$ vanishes identically, and hence
$$
\big(\sum_{\alp\in \CI}
\CE_\alp(J_\alp)\big)|_{\Fc(F)^\el} = 
\CE_1(J_1)-\CE_1(\CT_1(\nu_0(J_0))).
$$
Therefore it remains to show that over the elliptic locus, 
$$
J|_{\Fc(F)^\el} = \CE_1(J_1-\CT_1(\nu_0(J_0))).
$$
By the density of stable tempered characters, and the fact that the principal series characters are all supported on split locus, it suffices to show that for any character $\chi:T_1(F)\to \BC$ with stable lifting $\theta^\st_{\chi} = \CG\CG_1(\phi_{\chi})$, 
\begin{align}\label{eq:summandsplit:1}
\langle J,\CG\CG_1(\phi_{\chi})\rangle
=
\langle \CE_1(J_1-\CT_1(\nu_0)(J_0)),\CG\CG_1(\phi_{\chi})\rangle
\end{align}
where $\langle\cdot,\cdot\rangle$ is the standard pairing between stable orbital integrals and stable characters that is given by the integration over the Steinberg-Hitchin base with respect to the usual additive Haar measure. The left hand side of \eqref{eq:summandsplit:1}, by adjunction, can be rewritten as $\langle J_1,\chi\rangle$. 

It turns out that to show \eqref{eq:summandsplit:1}, it suffices to notice that based on the definition of $\CE_1$, the right hand side can be rewritten as 
$$
\big(
\CG\CG_1(J_1-\CT_1(\nu_0(J_0))),\CG\CG_1(\phi_\chi)
\big)_{\wb{\Fc}^\el}.
$$
Moreover, the function $J_1-\CT_1(\nu_0(J_0))$ enjoys the following characterization
\begin{itemize}
	\item Viewed as a function on $T_1(F)$, the constant Fourier coefficient vanishes
$$
\langle J_1-\CT_1(\nu_0(J_0)),\mathbbm{1}_{T_1}\rangle
=\langle J_1,\mathbbm{1}_{T_1}\rangle
-
\langle \nu_0(J_0),\theta_{\mathbbm{1}_{T_1}}^\st\rangle=0.
$$
Notice that the last identity follows from the fact that $\theta_{\mathbbm{1}_{T_1}}^\st = \eta_1$ and Part (1) of Theorem \ref{thm:total-transfer:archimedean}.
\end{itemize}
It follows that based on Theorem \ref{thm:anotherunitary:sl2}, 
\begin{align*}
\big(
\CG\CG_1(J_1-\CT_1(\nu_0(J_0))),\CG\CG_1(\phi_\chi)
\big)_{\wb{\Fc}^\el} =& 
\big(
J_1-\CT_1(\nu_0(J_0)),\phi_{\chi}\big)_{\wb{\Fc}^\el}
\\
=&
\langle J_1-\CT_1(\nu_0(J_0)),\chi\rangle=
\langle J_1,\chi\rangle
\\
=&
\langle J,\theta^\st_\chi\rangle = 
\langle J,\CG\CG_1(\phi_\chi)\rangle
\end{align*}
which is exactly equal to the left hand side. It follows that we complete the proof of the theorem for $F=\BR$ situation.

It remains to treat the case when $F$ is non-archimedean. As usual, using the density of stable tempered characters, we need to show that the left hand side and the right hand side share the same value after pairing with all the stable tempered characters of $G(F)$. Following the same idea as the archimedean case, both sides share the same value over the split locus. It suffices to verify that both sides share the same value after pairing with all dihedral stable tempered characters attached to quadratic field extensions and Steinberg characters. 

First, by elliptic orthogonality, 
$$
\langle J-J^\el_{\St}\cdot \v\Del\Theta^\el_\St,\Theta_\St\rangle = 
\langle \nu_0(J_0),\Theta_\St\rangle.
$$
Similarly,
\begin{align*}
\langle \sum_{\alp\in \CI}
\CE_\alp(J_\alp),\Theta_\St\rangle 
=
\langle \nu_0(J_0),\Theta_\St\rangle+&
\sum_{\alp\neq 0}
\langle \CE_\alp(J_\alp),\Theta_\St\rangle 
-
\langle \CE_1(\CT_1(\nu_0(J_0))),\Theta_\St\rangle
\\
-&
\sum_{\alp=\pm 1/2}
\langle \CE_\alp^{++}(\CT_\alp(\nu_0(J_0))^{++}),\Theta_\St\rangle.
\end{align*}
By Mellin transform and inversion, up to a finite linear combination, for $\alp\neq 0$, the image of $\CE_\alp$ and $\CE^{++}_\alp$ are all given by the normalized elliptic stable character of supercuspidal representations. Hence the elliptic orthogonality shows that 
\begin{align*}
\langle \sum_{\alp\in \CI}
\CE_\alp(J_\alp),\Theta_\St\rangle 
=\langle \nu_0(J_0),\Theta_\St\rangle.
\end{align*}
In conclusion both sides share the same paring with Steinberg character. We are left to show that both sides share the same paring with stable dihedral supercuspidal characters. The idea is similar to the archimedean case and we will use the elliptic orthogonality reformulated in Theorem \ref{thm:anotherunitary:sl2}. For any nontrivial character $\chi_\bet:T_\bet(F)\to \BC^\times$ with $\bet\neq 0$, we need to show that 
$$
\langle J,\theta_{\chi_\bet}\rangle = 
\langle \sum_{\alp\in \CI}
\CE_\alp(J_\alp),\theta_{\chi_\bet}\rangle
$$
Equivalently, 
$$
\langle J, \CG\CG_\bet(\phi_{\chi_\bet})\rangle 
=\sum_{\alp\in \CI}
\langle 
\CE_\alp(J_\alp),\CG\CG_{\bet}(\phi_{\bet})\rangle.
$$
But it follows from the same argument as the archimedea case with the aid of Theorem \ref{thm:anotherunitary:sl2}.
It follows that we complete the proof of the theorem.

\begin{rmk}\label{rmk:decompositionsections}

Let $F$ be non-archimedean. Recall that we have a decomposition of the stable cocenter $\RS\RO(\CH)$:
$$
\RS\RO(\CH) = \RS\RO(\CH)_\cusp\oplus \RS\RO(\CH)_\ind
$$
where the cuspidal part $\RS\RO(\CH)_\cusp$ is supported by the union of zero dimensional (modulo the center when $G=\GL_2$) components of the stable Bernstein varieties and the induced part $\RS\RO(\CH)_\ind$ is supported by the union of positive-dimensional (modulo the center when $G=\GL_2$) components.
By restricting the transfer attached to the split torus $\CT_0$ to the induced part we obtain a surjective map
$$
\CT_0: \RS\RO(\CH)_\ind \to \CC^\infty_c(T_0(F))^{\tau_0}
$$
whose kernel is generated by the normalized elliptic character of Steinberg representation twisted by compactly supported functions along the split center, i.e.
$$
\RS\RO(\CH)_\St=
\ker \CT_0|_{\RS\RO(\CH)_\ind} = 
\big\{
\big(
\v\Del\Theta^\el_\St
\big)
\cdot 
\phi\circ \det
\mid 
\phi\in \CH(Z).
\big\}
$$
In particular, when $G=\SL_2$ and hence $Z$ is trivial, the above kernel is exactly given by the scalar multiples of the normalized elliptic Steinberg character. We get an exact sequence of $\CZ^\st$-module
\begin{equation} \label{eq:ind-seq}
	0 \to  \RS\RO(\CH)_\St \to {\rm SO}(\CH)_{\rm ind} \to \CC^\infty_c(T_0(F))^{\tau_0} \to 0.
\end{equation}
As a result, $\RS\RO(\CH)_\St$ can be viewed as the torsion part of $\RS\RO(\CH)_\ind$ and its free part is isomorphic to the structural sheaf of the union of positive-dimensional components of the stable Bernstein variety (modulo the center when $G=\GL_2$).

When $F$ is archimedean, one can show that the section $\CE_0$ is $\CZ^\st$-linear, in the sense that for any $z\in \CZ^\st$ and $J_0\in \CC^\infty_c(T_0(F))^{\tau_0}$, for every tempered stable character $\Theta_{[\pi]}$ the following identity holds
$$
\langle \CE_0(z_0J_0),\Theta_{[\pi]}\rangle = 
\gam(z,[\pi])
\langle \CE_0(J_0),\Theta_{[\pi]}
\rangle
$$
where $z_0 = \rho_0(z)\in \CC^\infty_c(T_0(F))^{\tau_0}$ is the image of $z$ under the descent map. However, when $F$ is non-archimedean, $\CE_0$ is no longer $\CZ^{\st}$-linear due to the existence of Steinberg representation. 
\end{rmk}

\begin{rmk}\label{rmk:two-distinguished-sections}
Let $F$ be non-archimedean of odd residual characteristic and $G=\SL_2$ or $\GL_2$. For the purpose of next two subsections, we introduce the following two subspaces of $\RS\RO(\CH)$:
\begin{defin}
\begin{enumerate}
\item Let $\RS\RO(\CH)^{\perp \St}$ be the subspace of $\RS\RO(\CH)$ such that the following function in $g\in G(F)$ vanishes identically
$$
J_{\St}(g) = 
\int_{\Fc_{\SL_2}(F)}
J(c,\det g)\Theta_\St(c)\ome_{\Fc_{\SL_2}}
$$

\item Let $\RS\RO(\CH)^{\perp \triv}$ be the subspace of $\RS\RO(\CH)$ such that the following function in $g\in G(F)$ vanishes identically
$$
J_{\triv}(g) = 
\int_{\Fc_{\SL_2}(F)}
J(c,\det g)\ome_{\Fc_{\SL_2}}.
$$
\end{enumerate}
\end{defin}
In other words, $\RS\RO(\CH)^{\perp \St}$ (resp. $\RS\RO(\CH)^{\perp \triv}$) is the subspace of $\RS\RO(\CH)$ that are orthogonal to (twisted) Steinberg characters (resp. one-dimensional characters). 

Recall that the Steinberg character is given as follows (which in particular is independent of $\det g$):
\begin{align}\label{eq:Steinbergcharformula}
\Theta_\St(g) = 
\bigg\{
\begin{matrix}
\phi_{\del_B^{1/2}}(\tr g,\det g)-1 & \text{$g$ split}\\
-1 & \text{$g$ elliptic}
\end{matrix}
\end{align}
where $\phi_{\del_B^{1/2}}(\tr g,\det g) = \frac{\del_B^{1/2}(g)+\del_B^{-1/2}(g)}{\Del(g)}$. Based on the explicit character formula for $\Theta_\St$, the contribution $J_\St^\el$ in Part (2) (c) of Theorem \ref{thm:explicitsection} can be rewritten as follows
\begin{enumerate}
\item For $J\in \RS\RO(\CH)^{\perp \St}$, 
$$
J^{\el}_\St(g) = 
\int_{\Fc_{\SL_2}(F)}
\nu_0(J_0)(c,\det g)(1-\phi_{\del_B^{1/2}})(c,1)
\ome_{\Fc_{\SL_2}}.
$$
For abbreviation, we write the right hand side as $J^{\el}_{\St}(g) = \langle \nu_0(J_0),1-\phi_{\del_B^{1/2}}\rangle_{\SL_2}(g)$;

\item For $J\in \RS\RO(\CH)^{\perp \triv}$,
$$
J^{\el}_\St(g) = 
\int_{\Fc_{\SL_2}(F)}
\nu_0(J_0)(c,\det g)
\ome_{\Fc_{\SL_2}}.
$$
The right hand side above is abbreviated as $J^{\el}_\St(g) = \langle \nu_0(J_0),1\rangle_{\SL_2}(g)$;
\end{enumerate}
In other words, for the above two situations, $J^\el_\St$ can be rewritten as a linear form in $\nu_0(J_0)$. Using the fact that $\Theta^\el_\St = -\mathbbm{1}_\el$, it follows that Part (2) (c) of Theorem \ref{thm:explicitsection} can be reformulated as follows:

\begin{thm}\label{thm:explicitsection:distinguishedtwo}
With the notation and convention in Theorem \ref{thm:explicitsection}, when $F$ is non-archimedean of odd residual characteristic, define the following two extended sections from $J_0\in \CC^\infty_c(T_0(F))^{\tau_0}$ to functions on $\Fc(F)$:
\begin{align*}
\CE_0^{\perp \St}(J_0)
&=
\CE_0(J_0)-
\v\Del \mathbbm{1}_\el 
\cdot \langle \nu_0(J_0),1-\phi_{\del_B^{1/2}}\rangle_{\SL_2}
\\
\CE_0^{\perp \triv}(J_0) &= \CE_0(J_0)
-\v\Del \mathbbm{1}_\el
\langle \nu_0(J_0),1\rangle_{\SL_2}
\end{align*}
Then 
\begin{enumerate}
\item For $J\in \RS\RO(\CH)^{\perp \St}$, 
$$
J = \CE_0^{\perp \St}(J_0)+\sum_{\alp\neq 0}\CE_\alp(J_\alp);
$$

\item For $J\in \RS\RO(\CH)^{\perp \triv}$,
$$
J = \CE_0^{\perp \triv}(J_0)+\sum_{\alp\neq 0}\CE_\alp(J_\alp).
$$
\end{enumerate}
\end{thm}
These two extended sections will be used in the next two subsections to construct inversion formulae for $\CZ^{\st,\fin}$ and the stable orbital integral of spherical functions.
\end{rmk}

% !TEX root = luo-ngo.tex

\section{Inversion formula for $\CZ^{\st,\fin}$}\label{sec:inv}

In this section, over a local field $F$ of odd residual characteristic, we provide inversion formulas for the descent formulas in Theorem \ref{thm:descent:sl2}. The inversion formula will be an immediate corollary of Theorem \ref{thm:explicitsection:distinguishedtwo}.  

\subsection{The problem of inversion and its solution}

Let $G=\SL_2$ or $\GL_2$ over a local field $F$ of residual characteristic not equal to two. The dihedral lifting in subsection \ref{subsec:dihedral rep} yields a transfer of characters of $T_\alp(F)$ to stable characters of irreducible representations of $G(F)$. This gives rise to a homomorphism of algebras $\rho_\alp^*:\CZ^{\st,\fin}\to \CC^\infty_c(T_\alp(F))^{\tau_\alp}$ (when $F$ is archimedean, we work with $\CS(T_\alp(F))^{\tau_\alp}$). The inversion problem that we are going to study in this section is the following:

\begin{quest}\label{quest:inverseproblem}
Given $(J_\alp)_{\alp\in \CI}$ with $J_\alp\in \CC^\infty_c(T_\alp(F))^{\tau_\alp}$ $($resp. $\CS(T_\alp(F))^{\tau_\alp}$ for $F$ archimedean$)$, under what condition can we ensure that there exists $J\in \CZ^{\st,\fin}$ such that $\rho_\alp^*(J) = J_\alp$, and if one can find an explicit formula for $J$ as a function on the Steinberg-Hitchin base?
\end{quest}
For motivational purpose, let us assume that $F$ is non-archimedean of odd residual characteristic. Based on Theorem \ref{thm:descent:sl2}, the descent homomorphism $\rho_\alp^*$ factorizes as the composition $\rho_\alp^* = \CT_\alp\circ \LafSec$, with the Lafforgue transform $\LafSec:\CZ^{\st,\fin}\to \RS\RO(\CH)$ and the Langlands stable transfer $\CT_\alp:\RS\RO(\CH)\to \CC^\infty_c(T_\alp(F))^{\tau_\alp}$. When $G=\SL_2$, as functions on the Steinberg-Hitchin base, the Lafforgue transform is given by the formula \eqref{eq:Laf formula}
$$
\LafSec({\bf c}^*J_\Fc \ud g)=\CF_\psi\left(\frac{\CF_{\psi^{-1}}(J_\Fc)}{|.|}\right),\quad 
\c^*J_\Fc\ud g\in \CZ^{\st,\fin}
$$
and the Langlands stable transfer is given by the Gelfand-Graev formula \eqref{eq:GG formula}
$$
\CT_\alpha(f) = \lambda_{\alp} \pi_\alpha^* \left(\CF_{\psi} \big(
\eta_{\alpha} \CF_{\psi^{-1}}(f)\big) \right)
\quad 
f\in \RS\RO(\CH)
.
$$
Their composition, the descent to $\rho_\alp^*: \CZ^{\rm st,fin}\to \CC^\infty_c(E_\alp^1)^{\tau_\alpha}$ is given by the descent formula
$$
\rho_\alp^*({\bf c}^*J_\Fc \ud g)= \lambda_{\alp} \pi_\alpha^* \left(\CF_{\psi} \left(\frac{\eta_{\alpha} \CF_{\psi^{-1}}(J_\Fc)}{|.|}
\right) \right).
$$
By Theorem \ref{thm:image Laf}, the image of $\LafSec:\CZ^{\st,\fin}\to \RS\RO(\CH)$ lands in $\RS\RO(\CH)^{\perp\triv}$ that is orthogonal to all the one-dimensional characters of $G(F)$. On the other hand, Theorem \ref{thm:total-transfer} also describes the image of the total transfer to tori $\CT_{\oplus}: \RS\RO(\CH)\to \bigoplus_{\alp\in \CI} \CC^\infty_c(T_\alp(F))^{\tau_\alp}$. It turns out that the restriction of $\CT_\oplus$ to the image of the Lafforgue transform $\RS\RO(\CH)^{\perp \triv}$ induces an isomorphism of $\CZ^\st$-modules to $\im(\CT_\oplus)$. Combining Theorem \ref{thm:explicitsection} and Theorem \ref{thm:explicitsection:distinguishedtwo}, we summarize the discussion in the following theorem.

\begin{thm}\label{thm:inv-for-Zstfin}
Let $G=\SL_2$ or $\GL_2$ over a local field $F$ of residual characteristic not equal to two. Let $\rho_{\oplus}:\CZ^{\st,\fin}\to \bigoplus_{\alp\in \CI}\CZ_\alp^{\tau_\alp}$ be the total descent. The image of $\rho_{\oplus}$ is the subspace of $\bigoplus_{\alp\in \CI}\CZ_\alp^{\tau_\alp}$ consisting of pairs 
$$
(J_\alp\mid \alp\in \CI)\in \bigoplus_{\alp\in \CI}
\CZ_\alp^{\tau_\alp}
$$
satisfying the following compatibility conditions
\begin{enumerate}
\item If $G=\SL_2$, then 
\begin{enumerate}
\item $\langle J_0,\eta_\alp\rangle = \langle q_\alp,1_\alp\rangle$ for $\alp\in \CI\bs \{0\}$, where $1_\alp$ is the trivial character of $T_\alp$ and $\eta_\alp$ is the quadratic character of $F^\times$ attached to $E_\alp/F$.

\item $\langle J_\alp,\chi_\alp^\qd\rangle = \langle J_\bet,\chi^\qd_\bet\rangle$ for any $\alp,\bet\in \CI\bs \{0\}$, where $\chi^\qd_\alp$ and $\chi^\qd_\bet$ are the unique nontrivial character of $T_\alp = E^1_\alp$ and $T_\bet = E^1_\bet$ whose square is trivial.
\end{enumerate}

\item If $G=\GL_2$, then 
\begin{enumerate}
\item $\langle J_0,(\chi,\chi\otimes \eta_\alp\rangle = \langle J_\alp,\chi\circ \Nr_\alp\rangle$ for $\alp\in \CI\bs \{0\}$, where $\chi$ is a character of $F^\times$ and $\Nr_\alp:E^\times_\alp\to F^\times$ is the norm map.

\item $\langle J_\alp,\chi^\qd_\alp\rangle = \langle J_\bet,\chi^\qd_\bet\rangle$ for any $\alp,\bet\in \CI\bs \{0\}$, where $\chi^\qd_\alp$ and $\chi^\qd_\bet$ are the characters of $E^\times_\alp$ and $E^\times_\bet$ such that $\chi_\alp$ and $\chi_\bet$ does not factor through the norm map but their square does, and their dihedral lifting to $G(F)$ are isomorphic.
\end{enumerate}
\end{enumerate}
For $(J_\alp\mid \alp\in \CI)$ satisfying the above compatibility condition, there exists a unique $J\in \CZ^{\st,\fin}$ such that $\rho_\alp^*(J) = J_\alp$ for any $\alp\in \CI$. Moreover, $J = (\c^*J_\Fc)\ud g$ is given by the following formulae:
\begin{enumerate}
\item When $F=\BC$, 
$$
J = \Laf(\nu_0(J_0));
$$

\item When $F=\BR$, 
$$
J = \Laf(\sum_{\alp\in \CI}\CE_\alp(J_\alp));
$$

\item When $F$ is non-archimedean of odd residual characteristic,
\begin{align*}
J = 
\Laf
\bigg(
\CE_0^{\perp \triv}(J_0)+
\sum_{\alp\in \CI\bs \{0\}}
\CE_\alp(J_\alp)
\bigg)
\end{align*}
\end{enumerate}
Here for a function $z$ on $\Fc(F)$, $\Laf(z) = \CF_{\psi}\big(|\cdot|\CF_{\psi^{-1}}(z)\big)$ where the Fourier transforms and the absolute value norm are taken in trace variable, and the sections $\CE_\alp$ and $\CE^{\perp \triv}_0$ are defined in Theorem \ref{thm:explicitsection} and Theorem \ref{thm:explicitsection:distinguishedtwo}.
\end{thm}

%\input{Cp_7_Kernel}

% !TEX root = luo-ngo.tex

\section{Stable orbital integral of spherical functions}\label{sec:stableorbit}

Another important consequence of the Bernstein decomposition of the stable cocenter is an explicit formula for stable orbital integrals of spherical functions. As a corollary we derive a formula for the stable orbital integral of the basic function. 

\subsection{Unramified stable orbital integrals}\label{subsec:unramifiedSOI}

Throughout the section we let $G=\SL_2$ or $\GL_2$ over a local field of residual characteristic not equal to two. Let $\CH_\unr$ (resp. $\CS_\unr$) be the subalgebra of $\CH$ (resp. $\CS$) that are bi-$K$-invariant, where 
$$
K = 
\Bigg\{
\begin{matrix}
G(\Fo) & F\text{ non. archi.}\\
\RU_2\cap G(F) & F=\BC\\
\RO_2\cap G(F) & F=\BR 
\end{matrix}
$$
is a fixed maximal compact subgroup of $G(F)$ that is also open when $F$ is non-archimedean. Here $\RU_2$ (resp. $\RO_2$) is the corresponding unitary group (resp. orthogonal group) in $\GL_2(\BC)$ (resp. $\GL_2(\BR)$). 

By Satake isomorphism and its archimedean variant given by the Harish-Chandra isomorphism (\cite{satake63}\cite{anker1991spherical}), we have isomorphisms
$$
\Sat: \CH_\unr \quad (\text{resp. } \CS_\unr) 
\simeq \CC^\infty_c(T_0(F))^{\tau_0}_\unr \quad  (\text{resp. } \CS(T_0(F))_{\unr}^{\tau_0})
$$
For $h\in \CH_{\unr}$ and $\CS_{\unr}$, let $\RS\RO(h)$ be its stable orbital integral. From \cite[Lem.~4.3]{satake63}, 
$$
\RS\RO(h)|_{\c(T_0)}=\CT_0(\RS\RO(h)) = \Sat(h).
$$
On the other hand, for $F$ non-archimedean of odd residual characteristic and $\alp\in \CI\bs \{0,1\}$, 
$$
\CT_\alp(\RS\RO(h)) = 0
$$ 
which follows from the following adjoint identity for any character $\chi_\alp$ of $T_\alp$
$$
\langle \CT_\alp(\RS\RO(h)),\chi_\alp\rangle 
=\langle \RS\RO(h), \theta^\st_{\chi_\alp}\rangle 
$$
and the fact that the dihedral lifting of $\chi_\alp$ to $G(F)$ is ramified. It follows that we derive the following lemma.

\begin{lem}\label{lem:unramified:descent}
For any $\alp\in \CI$, let $h_\alp = \CT_\alp(\RS\RO(h))$. Then 
\begin{enumerate}
\item $h_0 = \Sat(h)$;

\item $h_\alp = 0$ for $\alp\in \CI\bs \{0,1\}$;

\item When $F$ is non-archimedean of odd residual characteristic, $\RS\RO(h)\in \RS\RO(\CH)^{\perp,\St}$.
\end{enumerate}
\end{lem}
\begin{proof}
It suffices to establish Part (3) of the lemma, which follows from the fact that (twisted) Steinberg representations are all ramified.
\end{proof}

As a corollary, based on Theorem \ref{thm:explicitsection} and Theorem \ref{thm:explicitsection:distinguishedtwo}, the following statements hold.

\begin{thm}\label{thm:explicitsphericalfun}
With the notation and convention from Theorem \ref{thm:explicitsection} and Theorem \ref{thm:explicitsection:distinguishedtwo}, the following statements hold: For $h\in \CH_{\unr}$ $($resp. $\CS_{\unr}$$)$, let $h_\alp = \CT_\alp(\RS\RO(h))$ for $\alp\in \CI$. Then 
\begin{enumerate}
\item $\RS\RO(h) = h_0 = \Sat(h)$ if $F=\BC$;

\item $\RS\RO(h) = \CE_0(h_0)$ if $F=\BR$;

\item $\RS\RO(h) = \CE_0^{\perp, \St}(h_0)$ if $F$ is non-archimedean of odd residual characteristic.
\end{enumerate}
\end{thm}
\begin{proof}
It suffices to establish the theorem for $F\neq \BC$. By Lemma \ref{lem:unramified:descent} and Theorem \ref{thm:explicitsection:distinguishedtwo}, 
\begin{enumerate}
\item $\RS\RO(h) = \CE_0(h_0)+\CE_1(h_1)$ if $F=\BR$;

\item $\RS\RO(h) = \CE_0^{\perp, \St}(h_0)+\CE_1(h_1)$ if $F$ is non-archimedean of odd residual characteristic.
\end{enumerate}
Hence it suffices to show that $\CE_1(h_1) =0$. By Theorem \ref{thm:explicitsection}, $\CE_1(h_1) = \CE^+_1(h_1^+)$. But by definition, for any character $\chi_1$ of $T_1(F)$,
$$
\langle h_1,\chi_1\rangle = 
\langle \RS\RO(h),\theta^{\st}_{\chi_1}\rangle
$$
which vanishes unless $\CW(\chi_1)$ is an unramified representation. But this happens only when $\chi_1$ is the trivial character for $G=\SL_2$, and factors through the norm map when $G=\GL_2$. It follows that $h_1^+ = 0$ and hence $\CE_1(h_1) = \CE^+_1(h^+_1) = 0$.
\end{proof}

\begin{rmk}\label{rmk:intermediateextension}
The fact that $\RS\RO(h) = \CE^{\perp,\St}_0(h_0)$ and $h_0 = \RS\RO(h)|_{\c(T_0)}$ reflects the heuristic that, at least over the local function field, there should exists a perverse $t$-structure on the stacky quotient $[G/\Ad(G)]$, such that $\RS\RO(h)$ can be realized as the trace of Frobenius of a perverse sheaf on $[G/\Ad(G)]$ which is the intermediate extension of its restriction to the split locus (\cite{BKVPevInf}). The question will be addressed in a forthcoming work.
\end{rmk}

\subsection{Stable orbital integral of the basic function}\label{subsec:SOI-of-basic}

As a corollary to Theorem \ref{thm:explicitsphericalfun}, we derive an explicit formula for the stable orbital integral of the basic function on $G(F)$ when $G=\GL_2$ or $\BG_m\times \SL_2$ in turns of the basic function on the maximal tori. Notice that when $G=\SL_2$, to make sense of the basic function we need to enlarge it to be $\BG_m\times \SL_2$ to make sense of the unramified twist. 
Recall from Conjecture \ref{conjec:bknproposal}, for a pair $(G,\rho)$, the basic function $\BL^\rho_G$ is the unique spherical function on $G(F)$ such that 
$$
\CZ(s,\BL^\rho_G,\vphi^\circ_\pi) = L(s,\pi,\rho)
$$
where $\vphi^\circ_\pi$ is the zonal spherical function of an unramified representation $\pi$ of $G(F)$. Following the discussion in \cite[(4.15)]{ngo2016hankel} (see also \cite{MR3990815}), the above equation can also be reformulated as follows:
\begin{num}
\item For any irreducible admissible representation $\pi$ of $G(F)$, $\tr\pi(\BL^\rho)$ is nonzero if and only if $\pi$ is unramified. When $\pi$ is unramified, as a meromorphic function in the Satake parameter of $\pi$,
$$
\tr\pi(\BL^\rho) = L(-\frac{n_\rho-1}{2},\pi,\rho).
$$
\end{num}

\begin{rmk}\label{rmk:stableorbit:basicfun:1}
Following \cite[Defin.~1.3.1, Thm~1.3.3, Lem.~2.4.4]{MR3990815}, we renormalize our basic function and Fourier kernel so that for any unramified representation $\pi$,
$$
\tr\pi(\BL^\rho) = L(\frac{1}{2},\pi,\rho)
$$
and 
$$
\pi(\RJ^\rho_\psi) = \gam(\frac{1}{2},\pi^\vee,\rho,\psi).
$$
Then $\RJ^\rho_\psi*(\BL^\rho)^\vee = \BL^\rho$.
\end{rmk}
Following the discussion in subsection \ref{subsec:toruscase}, through restricting $\rho$ to the dual group of the split and unramified tori of $G(F)$, we let $\BL^\rho_0$ and $\BL^\rho_1$ be the basic function attached to the split and unramified maximal tori of $G$.

The main result in this subsection is the following theorem.

\begin{thm}\label{thm:explicit:SOB}
The following identities hold.
\begin{enumerate}
\item $\RS\RO(\BL^\rho) = \BL^\rho_0$ if $F=\BC$;

\item $\RS\RO(\BL^\rho) = \CE_0(\BL^\rho_0)$ if $F=\BR$;

\item $\RS\RO(\BL^\rho) = \CE^{\perp,\St}_0(\BL^\rho_0)$ if $F$ is non-archimedean of odd residual characteristic.
\end{enumerate}
\end{thm}

\begin{rmk}
Actually Part (3) of Theorem \ref{thm:explicit:SOB} is also true when the residual characteristic of $F$ is even, since $\RS\RO(\BL^\rho)$ annihilates all the supercuspidal representations.
\end{rmk}

It is clear that Theorem \ref{thm:explicit:SOB} follows from the same discussion as Theorem \ref{thm:explicitsphericalfun}, which is an immediate corollary of the following lemma.

\begin{lem}\label{lem:descent-of-SOB}
For any $\alp\in \CI$, 
$$
\CT_\alp(\RS\RO(\BL^\rho)) = 
\bigg\{
\begin{matrix}
\BL^\rho_\alp & \alp =0,1\\
0 & \alp\neq 0,1
\end{matrix}.
$$
\end{lem}
\begin{proof}
Although the basic functions are no longer compactly supported or Schwartz on $G(F)$, after multiplying a smooth and compactly supported function along the determinant, they lie in $\CH$ and $\CS$ respectively that do not affect the discussion of the stable transfer map.

For $\chi_\alp$ a character of $T_\alp$, let $\CW(\chi_\alp)$ be its (stable) dihedral lifting. By the adjoint identity 
$$
\langle \CT_\alp(\RS\RO(\BL^\rho)),\chi_\alp\rangle = 
\langle \RS\RO(\BL^\rho), \theta^\st_{\chi_\alp}\rangle = \tr \CW(\chi_\alp)(\BL^\rho),
$$
it suffices to determine $\tr \CW(\chi_\alp)(\BL^\rho)$. 
By \cite[\S 3]{jlgl2} and \cite[p.76]{jlgl2}, the dihedral lifting $\CW(\chi_\alp)$ is unramified only when the following conditions are satisfied:

\begin{enumerate}
\item For $G=\GL_2$, $\CW(\chi_\alp)$ is irreducible. Moreover, it is unramified only when 
\begin{enumerate}
\item $\alp =0$ and hence $E=F\oplus F$. In this case $\chi_0 = (\chi_1,\chi_2)$ is an unramified character of $E_0^\times$. When it is the case, $\CW(\chi_0)\simeq \Ind_B^G(\chi)$ and
$$
\tr\CW(\chi_0)(\BL^\rho) = L(\frac{1}{2},\CW(\chi),\rho) = L(\frac{1}{2},\chi,\rho);
$$
where the last identity follows from the compatibility between the unramified local $L$-factor and automorphic induction (\cite{gl2llc});

\item $\alp=1$ and hence $E_1$ is the unique unramified quadratic extension of $F$, and $\chi_1 = \chi\circ \Nr_{1}$ with $\chi$ an unramified character of $F^\times$. When it is the case, $\CW(\chi_1)\simeq \Ind_B^G(\chi,\chi\eta_0)$ and 
$$
\tr\CW(\chi_1)(\BL^\rho) = L(\frac{1}{2},(\chi,\chi\eta_1),\rho);
$$

\item For all other situations, $\tr\CW(\chi_\alp)(\BL^\rho) = 0$;
\end{enumerate}
As a result, 
$$
\langle \CT_\alp(\RS\RO(\BL^\rho)),\chi_\alp\rangle = 
\bigg\{
\begin{matrix}
\langle \BL^\rho_\alp ,\chi_\alp\rangle & \alp=0,1\\
0 & \alp\neq 0,1
\end{matrix}
$$
from which we deduce the lemma.

\item 
For $G=\BG_m\times \SL_2$ and $\chi\times \chi_\alp$ a character of $F^\times \times E^1_\alp$, $\CW(\chi\times \chi_\alp)$ is unramified only when

\begin{enumerate}
    \item $\alp=0$ and both
$\chi$ and $\chi_{0}$ are unramified. When it is the case, 
$$
\tr\CW(\chi)(\BL^\rho) = 
L(
\frac{1}{2}, \chi\times \chi_{0},\rho
);
$$

    \item $\alp=1$ and $\chi\times \chi_{1}$ with $\chi$ unramified and $\chi_{1}$ trivial. When it is the case, 
$$
\tr\CW(\chi)(\BL^\rho) = 
L(
\frac{1}{2},\chi\times (\eta_1\times \mathbbm{1}),\rho
);
$$
    \item
For all other situations, $\tr\CW(\chi)(\BL^\rho) = 0$;
\end{enumerate}
As a result, 
$$
\langle \CT_\alp(\RS\RO(\BL^\rho)),\chi_\alp\rangle = 
\bigg\{
\begin{matrix}
\langle \BL^\rho_\alp ,\chi_\alp\rangle & \alp=0,1\\
0 & \alp\neq 0,1
\end{matrix}
$$
from which we deduce the lemma.
\end{enumerate}
It follows that we complete the proof of the lemma.
\end{proof}

% !TEX root = luo-ngo.tex

\section{Orbital Hankel transform}\label{sec:HKT}

In this section, we discuss the orbital Hankel transform for $G=\SL_2$ or $\GL_2$ considered by the second author (\cite{ngo2016hankel}). Over a non-archimedean local field, the stable Bernstein center $\CZ^{\st}$ (which in particular contains the $\rho$-Fourier kernel after multiplying a smooth and compactly supported function along the determinant) intertwines the Hecke algebra $\CH$, which induces orbital Hankel transforms on the space of stable orbital integrals $\RS\RO(\CH)$. We are going to show that the orbital Hankel transform is compatible with the Langlands stable transfer. In particular, the $\rho$-orbital Hankel transform, which is the descent of the $\rho$-Fourier transform, descends to the abelian $\rho$-Fourier transform on maximal tori under the Langlands stable transfer. Moreover, based on the decomposition of the stable cocenter studied in Section \ref{sec:summands}, we provide explicit formulas for the action of $\CZ^\st$ on $\RS\RO(\CH)$ as functions on the Steinberg-Hitchin base.

\subsection{A commutative diagram}\label{subsec:HKT:diagram}

Let $\CS^\rho(G(F))$ be the $\rho$-Schwartz space conjectured in subsection \ref{subsec:proposal} which is intertwined by the conjectural $\rho$-Fourier transform $\CF^\rho_\psi$. In \cite{ngo2016hankel}, the second author asks if one plugs the global $\rho$-Schwartz functions into the Arthur-Selberg trace formula, can one establish the functional equation of the global automorphic $L$-functions attached to $(G,\rho)$ from the geometric side of the trace formula. 

To understand the problem one first need to understand the following commutative diagram
\begin{align}\label{eq:orbhankel:1}
\xymatrix{
\CS^\rho(G(F))\ar[r]^{\CF^\rho_\psi} \ar[d]^{\mathrm{st. orb.}} & \CS^\rho(G(F))\ar[d]^{\mathrm{st. orb.}}
\\
\CS\CO^\rho(G(F)) \ar[r]^{\CH^{\rho}_\psi} & \CS\CO^\rho(G(F))
}
\end{align}
Here the vertical arrows are given by taking the stable orbital integrals, and $\CS\CO^\rho(G(F))$ is defined to be the space of stable orbital integrals of the conjectural $\rho$-Schwartz functions with the normalization introduced in subsection \ref{subsec:normalization-SOI-OI}. The lower horizontal arrow is given by the induced transform $\CH^\rho_\psi$, which we refer to the \emph{$\rho$-orbital Hankel transform}. The goal of this section is to show that the operator $\CH^\rho_\psi$ is well-defined for $G=\SL_2$ (actually $\BG_m\times \SL_2$) or $\GL_2$, and show that the following diagram is commutative for any étale quadratic $F$-algebra $E_\alp$ with $\alp\in \CI$,
$$
\xymatrix{
\CS\CO^\rho(G(F))\ar[d]^{\CT_\alp} \ar[r]^{\CH^{\rho}_\psi} & \CS\CO^\rho(G(F)) \ar[d]^{\CT_\alp}\\
\CS^\rho(T_\alp(F)) \ar[r]^{\CF^{\rho}_{\psi,T_\alp}} & \CS^\rho(T_\alp(F))
}
$$
where $\CF^\rho_{\psi,T_\alp}$ is the abelian $\rho$-Fourier transform on $T_\alp$ constructed in subsection \ref{subsec:toruscase}. This indicates the compatibility between the $\rho$-Hankel transform and the Langlands' stable transfer.

\begin{rmk}
Here the kernel function for $\CF^\rho_{\psi,T_\alp}$ is given by $\wt{\RJ}^\rho_{\psi,\alp} = \RJ^{\rho}_{\psi,\alp}\cdot \kappa^\rho_\alp$ rather than $\RJ^{\rho}_{\psi,\alp}$ which is introduced in subsection \ref{subsec:aformulaFK}. The constant $\kappa^\rho_\alp$ is needed due to the discrepency between the Langlands local gamma factors for dihedral liftings.
\end{rmk}

Following Conjecture \ref{conjec:bknproposal} and Remark \ref{rmk:proposal}, up to multiplying a smooth and compactly supported function on the determinant factor, we may substitute $\CS^\rho(G(F))$ by $\CS(G(F))$ which is the space of test functions when $F$ is non-archimedean, and the space of rapidly decreasing functions when $F$ is archimedean. We also substitute the $\rho$-Fourier kernel by a stable distribution $J\in \CZ^\st$. 

The following theorem is the main result of this subsection.

\begin{thm}\label{thm:orbHKT:commutativediagram}
Fix $J\in \CZ^\st$. Consider the following commutative diagram
\begin{align*}
\xymatrix{
\CS(G(F)) \ar[r]^{J} \ar[d]^{\text{st. orb.}} & \CS(G(F)) \ar[d]^{\text{st. orb.}}
\\
\CS\CO(G(F)) \ar[d]^{\CT_\alp}\ar[r]^{\CH_J} & \CS\CO(G(F)) \ar[d]^{\CT_\alp}\\
\CS(T_\alp(F)) \ar[r]^{J_\alp} &\CS(T_\alp(F))
}
\end{align*}
where $\CS(G(F))=\CH$ the space of smooth and compactly supported functions on $G(F)$ when $F$ is non-archimedean, and is the space of Schwartz algebra $\CS$ when $F$ is archimedean. The upper $($resp. lower$)$ horizontal operator is given by 
$$
h\in \CS(G(F))\mapsto J*h^\vee
\qquad 
(\text{resp. } h_\alp\in \CS(T_\alp(F))\mapsto J_\alp*h_\alp^\vee)
$$
where $J_\alp = \rho_\alp^*(J)\in \CZ_\alp^{\tau_\alp}$. Then the following statements hold:
\begin{enumerate}
\item The operator $\CH_J$ is well-defined, and hence the top square is commutative;

\item The bottom square is commutative.

\end{enumerate}
\end{thm}
\begin{proof}
Notice that by the discussion in subsection \ref{subsec:totaltransfertori}, indeed, $\CT_\alp(\CS\CO(G(F)))\subset \CS(T_\alp(F))^{\tau_\alp}$. 

(1) To show that the operator $\CH_J$ is well-defined, it suffices to show that for any $h\in \CS(G(F))$, if $\RS\RO(h) =0$ identically, then $\RS\RO(J*h^\vee) =0$ identically. Notice that $\RS\RO(h) = 0$ if and only if $\RS\RO(h^\vee) = 0$. By Theorem \ref{thm:main:SBC}, the stable tempered characters are dense in the space of stable distributions, hence it suffices to show that for any tempered local $L$-packet $[\pi]$ of $G(F)$ with stable tempered character $\Theta_{[\pi]}^{\Fc}$ as a function on $\Fc(F)$, the following term vanishes identically 
$$
\langle \Theta^{\Fc}_{[\pi]},\RS\RO(J*h^\vee)\rangle 
=
\int_{g\in G(F)}
\Theta^\Fc_{[\pi]}(\c(g))
J*h^\vee(g)\ud g = 
\gam(J,[\pi])\langle \Theta^\Fc_{[\pi]},\RS\RO(h^\vee)\rangle
$$
Here the last identity follows from the fact that the regular function $\gam(J,\cdot)$ is constant on tempered local $L$-packets, which is established in the proof of Theorem \ref{thm:main:SBC}. But by definition $\RS\RO(h^\vee) = 0$ identically, it follows that $\RS\RO(J*h^\vee) = 0$.

(2) For any $f = \RS\RO(h)\in \CS\CO(G(F))$, we compute
\begin{align*}
\CT_\alp\circ \CH_J(f) = \CT_\alp\circ 
\CH_J(\RS\RO(h)).
\end{align*}
By Part (1), the above term is equal to 
\begin{align*}
=
\CT_\alp\circ 
\RS\RO(J*h^\vee).
\end{align*}
For any character $\chi_\alp$ of $T_\alp$, 
\begin{align*}
\langle \CT_\alp\circ \RS\RO(J*h^\vee),\chi_\alp\rangle 
=
\langle \RS\RO(J*h^\vee),\theta^\st_{\chi_\alp}\rangle =
\gam(J,\CW(\chi_\alp))
\langle \RS\RO(h^\vee),\theta^\st_{\chi_\alp}\rangle.
\end{align*}
By definition, $\gam(J,\CW(\chi_\alp)) = \gam(J_\alp,\chi_\alp) = \langle J_\alp,\chi_\alp\rangle$ with $J_\alp = \rho_\alp^*(J)$. Moreover, 
\begin{align*}
\langle \RS\RO(h^\vee),\theta^\st_{\chi_\alp}\rangle 
=\langle \CT_\alp(\RS\RO(h^\vee)),\chi_\alp\rangle
\end{align*}
But by definition, as a function on the torus $T_\alp$, 
$$
\CT_\alp(\RS\RO(h^\vee)) = \big(\CT_\alp(\RS\RO(h))\big)^\vee
$$
it follows that we arrive that the following identity 
$$
\langle 
\CT_\alp\circ 
\RS\RO(J*h^\vee),\chi_\alp\rangle 
=
\langle J_\alp,\chi_\alp\rangle
\cdot 
\langle \big(\CT_\alp(\RS\RO(h))\big)^\vee,\chi_\alp\rangle.
$$
Applying the Mellin inversion in $\chi_\alp$, we deduce Part (2) of the theorem.
\end{proof}

\subsection{An explicit formula}\label{subsec:orbHKT:explicitformula}

Theorem \ref{thm:orbHKT:commutativediagram} shows that the following diagram is commutative 
$$
\xymatrix{
\CS\CO(G(F))\ar[r]^{\CH_{J}} \ar[d]^{\CT_\alp} & \CS\CO(G(F)) \ar[d]^{\CT_\alp}
\\
\CS(T_\alp(F)) \ar[r]^{J_\alp} & \CS(T_\alp(F))
}
$$
In other words, for any $f\in \CS\CO(G(F))$, $\CH_J(f)$ enjoys the following characterization:
\begin{num}
\item\label{num:orbHKT:explicit:1} For any étale quadratic $F$-algebra $E_\alp$ with $\alp\in \CI$, 
$$
\CT_\alp\circ \big(\CH_{J}(f) \big) = 
J_\alp*
\big(
\CT_\alp(f)
\big)^\vee.
$$
\end{num}
Following the notation from Theorem \ref{thm:explicitsection}, $\big(J_\alp*\big(
\CT_\alp(f)
\big)^\vee\big)_{\alp\in \CI}$ is the image of the total transfer $\CT_{\oplus}\big(\CH_J(f) \big)$. From Theorem \ref{thm:explicitsection}, we can derive an explicit formula for $\CH_J(f)$ from the datum $\big(J_\alp*\big(
\CT_\alp(f)
\big)^\vee\big)_{\alp\in \CI}$, which is summarized in the following theorem.

\begin{thm}\label{thm:orbHKT:explicitformula}
With the notation from above and Theorem \ref{thm:explicitsection}, the following statements hold:

\begin{enumerate}
\item When $F=\BC$, 
$$
\CH_J(f) = \nu_0\big(J_0*\CT_0(f)^\vee\big);
$$

\item When $F=\BR$, 
$$
\CH_J(f) = 
\sum_{\alp\in \CI}
\CE_\alp
\big(J_\alp*
(
\CT_\alp(f)
)^\vee
\big);
$$

\item When $F$ is non-archimedean of odd residual characteristic,
$$
\CH_J(f) = 
\CE_0^{\rm HK}
(J_0,f)
+
\sum_{\alp\in \CI\bs\{0\}}
\CE_\alp
\big(J_\alp*
(
\CT_\alp(f)
)^\vee
\big)
$$
where 
$$
\CE^{\rm HK}_0(J_0,f) = 
\CE_0
\big(
J_0*(\CT_0(f))^\vee
\big)
-
\big(\CH_J(f)\big)^\el_\St
\cdot 
\v\Del \mathbbm{1}_\el.
$$
The function $\big(\CH_J(f)\big)^\el_\St$ is the pull-back along the determinant $($which in particular is constant when $G=\SL_2$$)$ of the following functions:
\begin{itemize}
\item When $G=\SL_2$, 
$$
=
\gam(J_0,\del_B^{1/2})
\langle f,\Theta_\St\rangle 
-
\langle 
\nu_0
\big(
J_0*\CT_0(f)^\vee
\big),
\phi_{\del^{1/2}_B}-1\rangle;
$$

\item When $G=\GL_2$,
$$
=
(J_0\del_B^{1/2})*\langle f,\Theta_\St\rangle_{\SL_2}
-
\langle 
\nu_0
\big(
J_0*\CT_0(f)^\vee
\big),(\phi_{\del_B^{1/2}}-1)
\rangle_{\SL_2}
$$
where the convolution $*$ is taken on $F^\times$, and $\langle \cdot,\cdot\rangle_{\SL_2}$ is defined above Theorem \ref{thm:explicitsection:distinguishedtwo}. 
\end{itemize}
\end{enumerate}
\end{thm}
As a result, $\CH_J(f)$ is expressed as the abelian datum $\{(J_\alp*\CT_\alp(f))^\vee\}_{\alp\in \CI}$. For the situation of $\rho$-orbital Hankel transforms, $\CH^\rho_\psi(f)$ is expressed as the abelian Fourier transforms $\{\CF^\rho_{\psi,T_\alp}(\CT_\alp(f))\}_{\alp\in \CI}$.

\begin{proof}
By Theorem \ref{thm:explicitsection}, it suffices to treat the case when $F$ is non-archimedean of odd residual characteristic. In particular, we only need to compute 
$$
\big(\CH_J(f)\big)^\el_\St.
$$
We first treat the case when $G=\SL_2$. By definition and the explicit formula for the Steinberg character from \eqref{eq:Steinbergcharformula}, 
\begin{align*}
\big(\CH_J(f)\big)^\el_\St
=
\langle \CH_J(f),\Theta^\el_\St\rangle 
=
\langle \CH_J(f),\Theta_\St\rangle 
-
\langle \nu_0\big(J_0*\CT_0(f)^\vee\big),\phi_{\del_B^{1/2}}-1\rangle.
\end{align*}
Write $f= \RS\RO(h)$, then 
$$
\CH_J(f) = \RS\RO(J*h^\vee)
$$
and hence 
$$
\langle \CH_J(f),\Theta_\St\rangle 
=\gam(J,\St)\langle \RS\RO(h),\Theta_\St\rangle 
=\gam(J_0,\del_B^{1/2})
\langle f,\Theta_\St\rangle,
$$
In conclusion
$$
\big(\CH_J(f)\big)^\el_\St
=\gam(J_0,\del_B^{1/2})
\langle f,\Theta_\St\rangle
-\langle \nu_0\big(J_0*\CT_0(f)^\vee\big),\phi_{\del_B^{1/2}}-1\rangle.
$$

Next let us take $G=\GL_2$ so that $\big(\CH_J(f)\big)^{\el}_\St$ is the pull-back of a function on $F^\times$ along the determinant map. For any character $\chi$ of $F^\times$, the Mellin transform of $\big(\CH_J(f)\big)^{\el}_\St$ against the character can be rewritten as follows 
\begin{align*}
\langle 
\big(\CH_J(f)\big)^{\el}_\St,\chi
\rangle 
=
\langle \CH_J(f),\Theta^\el_{\St\otimes \chi}\rangle 
=\langle 
\CH_J(f),\Theta_{\St\otimes \chi}\rangle 
-
\langle 
\nu_0
\big(
J_0*\CT_0(f)^\vee
\big),(\phi_{\del_B^{1/2}}-1)\otimes \chi
\rangle 
\end{align*}
Following the same reason as before, 
$$
\langle \CH_J(f),\Theta_{\St\otimes \chi}
\rangle 
=\gam(J_0,\del_B^{1/2}\otimes \chi)
\langle f,\Theta_{\St\otimes \chi}\rangle.
$$
In conclusion,
\begin{align*}
\langle 
\big(\CH_J(f)\big)^{\el}_\St,\chi
\rangle 
=\gam(J_0,\del_B^{1/2}\otimes \chi)
\langle f,\Theta_{\St\otimes \chi}\rangle
-
\langle 
\nu_0
\big(
J_0*\CT_0(f)^\vee
\big),(\phi_{\del_B^{1/2}}-1)\otimes \chi
\rangle.
\end{align*}
Taking the Mellin inversion in $\chi$ we get
\begin{align*}
\big(
\CH_J(f)
\big)^\el_\St = 
\big((J_0\del_B^{1/2})\circ \det\big)*\langle f,\Theta_\St\rangle_{\SL_2}
-
\langle 
\nu_0
\big(
J_0*\CT_0(f)^\vee
\big),(\phi_{\del_B^{1/2}}-1)
\rangle_{\SL_2}.
\end{align*}
\end{proof}

\subsection{Global speculation}\label{subsec:HKT:globalspeculation}

Finally we make a brief remark for the global speculation for the $\rho$-orbital Hankel transform. This subsection is largely conjectural.

Following the last section of \cite{ngo2016hankel}, we fix a pair $(G,\rho)$ where $G$ is a reductive group over a \textbf{global} field $k$, and $\rho$ is a representation of the Langlands dual group of $G$. 

For any local place $\nu\in |k|$, following Conjecture \ref{conjec:bknproposal}, there should exists a pair $(\CS^\rho(G(k_\nu)),\CF^\rho_{\psi_\nu})$ of $\rho$-Schwartz space and a $\rho$-Fourier transform intertwining the $\rho$-Schwartz space. One can form the global $\rho$-Schwartz space
$$
\CS^\rho(G(\BA_k)) = 
\bigotimes^\p_{\nu\in |k|}
\CS^\rho(G(k_\nu)) \curvearrowleft 
\CF^\rho_{\psi} = \bigotimes_{\nu\in |k|} 
\CF^{\rho}_{\psi,\nu}
$$
where the right hand side is the restricted tensor product with respect to the $\rho$-basic functions $\{\BL^\rho_\nu\}_{\nu\in |k|}$. The $\rho$-Poisson summation conjecture proposed in \cite{BK00} predicts that the following identity holds 
$$
\sum_{\gam\in G(k)}
h(x\gam y^{-1}) = 
\sum_{\gam\in G(k)}
\CF^\rho_{\psi}(h)(y\gam x^{-1}),\quad h\in \CS^\rho(G(\BA_k))
$$
at least when we put mild local constraints to the function (see \cite[Conj.~5.3.1]{Zhilinsingapore} also). In \cite{ngo2016hankel}, the second author suggests to take $x=y$ and integrate both sides of the Poisson summation identity against the automorphic quotient of $G$, which, by trace formula, would give us a rough expansion of the following form. Here $\CA_\cusp$ is the set of cuspidal automorphic representations of $G$ and $L(\pi,\rho)$ is the global automorphic $L$-function for the pair $(G,\rho)$, $k^\qd/k$ runs over all the global quadratic algebras of $k$, $\RS\RO_h(e)$ is the global stable orbital integral of $h$ evaluated at $e\in (k^\qd)^\times$, and $\vol(e)$ is the Tamagawa number for the elliptic torus attached to $e$ when $e$ is elliptic
\begin{align*}
&\int_{x\in [G]}
\sum_{\gam\in G(k)}
h(x\gam x^{-1})\ud x
\\
\sim 
&\bigg\{
\begin{matrix}
\text{Spectral expansion:} & \sum_{\pi\in \CA_\cusp}
L(\pi,\rho) +\text{ (Extra)}\\
\text{Geometric expansion:} &
\sum_{k^\qd/k}
\sum_{e\in (k^\qd)^\times}
\vol(e)
\RS\RO_h(e)+\text{ (Extra)}
\end{matrix}
\end{align*}
By no means the above expression can be an exact identity, as one need to analyze the contribution from (Extra), which usually come residual spectrum and Eisenstein contribution on the spectral side, and endoscopic groups from the geometric side. 

In \cite{ngo2016hankel}, as a Corollary of the conjectural $\rho$-Poisson summation, the second author asks if one can establish the following Poisson summation identity for the $\rho$-orbital Hankel transform, which should be simpler than the original $\rho$-Poisson summation conjecture, but still provides a nontrivial symmetry for the geometric side of the trace formula and hence deduce the meromorphic continuation and functional equation of the automorphic $L$-functions $L(\pi,\rho)$ on the spectral side:
\begin{align}\label{eq:global:poissonbase}
\sum_{k^\qd/k}
\sum_{e\in (k^\qd)^\times}
\vol(e)
\RS\RO_h(e)+(\text{Extra})
\sim 
\sum_{k^\qd/k}
\sum_{e\in (k^\qd)^\times}
\vol(e)
\CH^\rho_\psi(\RS\RO_h)(e)+(\text{Extra})
\end{align}

Based on the commutative diagram established in subsection \ref{subsec:HKT:diagram}, after applying the (global) Langlands stable transfer $\CT_{k^\qd}$, indeed we have the following Poisson summation identity for each fixed $k^\qd/k$,
$$
\sum_{e\in (k^\qd)^\times}
\CT_{k^\qd}
(
\RS\RO_h)(e)
\sim 
\sum_{e\in (k^\qd)^\times}
\CT_{k^\qd}\big(\CH^\rho_\psi(\RS\RO_h)\big)(e)
$$
which follows from the $\rho$-Poisson summation identity for the $\rho$-abelian Fourier transform for $(\CS^\rho(T_\alp(\BA_k)),\CF^\rho_{\psi,T_\alp})$, and hence the following identity holds
\begin{align}\label{eq:global:abeliandescent}
\sum_{k^\qd/k}
\sum_{e\in (k^\qd)^\times}
\CT_{k^\qd}
(
\RS\RO_h)(e)
\sim 
\sum_{k^\qd/k}
\sum_{e\in (k^\qd)^\times}
\CT_{k^\qd}\big(\CH^\rho_\psi(\RS\RO_h)\big)(e).
\end{align}
As a result, with a sufficiently sophisticated inversion formula for the total transfer $\CT_{\oplus}$, we might dream that the identity \eqref{eq:global:abeliandescent} can be used to deduce the conjectural Poisson summation in \eqref{eq:global:poissonbase}. In other words, the conjectural Poisson summation for the $\rho$-orbital Hankel transform might be deduced from the abelian one in \eqref{eq:global:abeliandescent}. In Theorem \ref{thm:explicitsection} and Theorem \ref{thm:orbHKT:explicitformula}, we provide one  inversion formula for $\CT_\oplus$, yet more sophisticated inversion formulas are to be explored.

\bibliographystyle{amsalpha}
\bibliography{luo-ngo.bib}

\end{document}